\documentclass[1 [leqno,10pt]{amsart}
\usepackage{amsmath,mathrsfs}
\usepackage{amssymb,amsmath}
\usepackage{graphicx}
\usepackage{pdfsync}
\usepackage{latexsym,color}
\usepackage[mathscr]{eucal}
\numberwithin{equation}{section}

\newcommand{\ld}{\lambda}

\newcommand{\p}{\partial}
\newcommand{\vp}{\varphi}
\newcommand{\vep}{\varepsilon}
\newcommand{\og}{\omega}
\newcommand{\Og}{\Omega}
\newcommand{\sg}{\sigma}

\newcommand{\gm}{\gamma}
\newcommand{\Gm}{\Gamma}
\newcommand{\dt}{\delta}
\newcommand{\Dt}{\Delta}
\newcommand{\fr}{\frac}
\newcommand{\wt}{\widetilde}
\newcommand{\wh}{\widehat}

\newcommand{\bR}{{\mathbb R}^3 }
\newcommand{\bS}{{\mathbb S}^2 }
\newcommand{\bSS}{{\bS}\times{\bS}}

\newcommand{\bRR}{{\bR}\times{\bR}}

\newcommand{\bRS}{{\bR}\times {\mathbb S}^2 }
\newcommand{\bRRS}{{\bRR}\times{\mathbb S}^2 }

\newcommand{\bRd}{{\mathbb R}^d}
\newcommand{\la}{\langle}
\newcommand{\ra}{\rangle}
\newcommand{\mR}{{\mathbb R}}
\newcommand{\mN}{{\mathbb N}}

\newcommand{\mS}{{\mathbb S}}
\newcommand{\be}{\begin{equation}}
\newcommand{\ee}{\end{equation}}
\newcommand{\bes}{\begin{eqnarray}}
\newcommand{\ees}{\end{eqnarray}}
\newcommand{\beas}{\begin{eqnarray*}}
\newcommand{\eeas}{\end{eqnarray*}}
\newcommand{\lb}{\label}
\newcommand{\cal}{\mathcal}
\newcounter{thm}
\newtheorem{theorem}{Theorem}[section]
\newtheorem{proposition}[theorem]{Proposition}
\newtheorem{definition}[theorem]{Definition}
\newtheorem{lemma}[theorem]{Lemma}
\newtheorem{remark}[theorem]{Remark}

\newtheorem{corollary}[theorem]{\indent Corollary}

\begin{document}
\baselineskip 13pt

\title[Semi-classical limit of homogeneous Boltzmann equation]
{On semi-classical limit of spatially homogeneous quantum Boltzmann equation: weak convergence}

\author[L. -B. He, X. Lu, M. Pulvirenti]{Ling-Bing He,  Xuguang Lu,  Mario Pulvirenti}
\address[L.-B. He]{Department of Mathematical Sciences, Tsinghua University\\
Beijing 100084,  P. R.  China.} \email{hlb@mail.tsinghua.edu.cn}
\address[X. Lu]{Department of Mathematical Sciences, Tsinghua University\\
Beijing 100084,  P. R.  China.} \email{xglu@mail.tsinghua.edu.cn}
 \address[M. Pulvirenti]{Dipartimento di Matematica, Universit\`{a} di Roma La Sapienza, Piazzale Aldo Moro 5,
00185 Rome, Italy;
International Research Center M $\&$ MOCS, Universit\`{a} dell'Aquila, Palazzo Caetani, Cisterna di Latina,
(LT) 04012, Italy}  \email{pulviren@mat.uniroma1.it}

\begin{abstract} It is expected in physics that the homogeneous quantum Boltzmann equation with Fermi-Dirac or
Bose-Einstein statistics  and with   Maxwell-Boltzmann operator (neglecting effect of the statistics) for the weak coupled gases will  converge to the homogeneous Fokker-Planck-Landau equation as the Planck constant  $\hbar$ tends to zero. In this paper and the upcoming work \cite{HLP2}, we will  provide a mathematical justification on this semi-classical limit.  Key ingredients into the proofs   are  the new framework to catch the {\it weak projection gradient}, which is motivated by Villani \cite{V1} to identify the $H$-solutions for Fokker-Planck-Landau equation, and the symmetric structure inside the cubic terms of the collision operators.
 \medskip


\end{abstract}
\maketitle

 \section{Introduction}

The quantum Boltzmann equations for  Fermi-Dirac and
Bose-Einstein statistics were proposed
by Uehling and Uhlembeck in \cite {UU} (after Nordheim \cite {N}) on the basis of heuristic arguments. In contrast with the usual classical Boltzmann equation valid for a rarefied gas in the so called Boltzmann-Grad limit, such equations should be derived from the evolution of real Fermion   and Bosons  in the so called weak-coupling limit (see \cite{Bal} and \cite {BPCE}). From a physical point of view,
it is expected  that the semi-classical limit of the quantum Boltzmann equations is the Fokker-Planck-Landau (FPL) equation.
This is mainly because that from the point of view of particle system, FPL equation is the effective equation associated with a dense and weakly interacting gas of classical particles (see \cite {ICM}, \cite {BPS}).
Therefore it is not surprising that the semi-classical limits of the solutions of quantum Boltzmann equations is expected to be solutions of the FPL equation.

It should be stressed that in the weak-coupling limit, the kinetic picture for quantum systems yields a Boltzmann type evolution, basically a jump process, and not a diffusion in velocity as for classical systems. This is because of the tunnel effect which does assign a finite probability  to a test particle to go freely in the scattering process. However when $\hbar$ goes to zero, this effect is negligible and the diffusion is recovered in the limit, provided that the diffusion coefficient
$M_{\widehat{\phi}}$ (defined in \eqref{ker2} below) is finite.

\vskip1mm

  The aim of the present work and the forthcoming work \cite{HLP2} is to provide a mathematical
  justification for such semi-classical limit in the  so-called homogeneous case.
\subsection{Setting of the problem} In this subsection, we will first introduce the quantum Boltzmann equation and Fokker-Planck-Landau equation and then list some basic properties of the solutions.
\subsubsection{Introduction of equations}
The spatially homogeneous quantum Boltzmann equation reads
 \be {\p}_tf(t,{\bf v})=Q_{\ld}^{\vep}(f)(t,{\bf v}),\qquad (t, {\bf v})\in (0,\infty)\times{\bR}\lb{1.1}\ee
where
 $\vep:=2\pi\hbar$, $\hbar>0$ is the Planck constant,
 $Q_{\ld}^{\vep}(f)(t,{\bf v})=Q_{\ld}^{\vep}(f(t,\cdot))({\bf v})$
 and
\beas
Q_{\ld}^{\vep}(f)({\bf v})&=&\int_{{\mathbb R}^3\times {\mathbb S}^2}
B_{\ld}^{\vep}({\bf v}-{\bf v}_*,\og)\\
&\times& \Big( f'f_*'(1+\ld\vep^3 f)(1+\ld\vep^3 f_*)
-ff_*(1+\ld \vep^3 f')(1+\ld\vep^3 f_*')\Big)
{\rm d}\og{\rm d}{\bf v}_*.\eeas
Several comments on the collision operator $Q_{\ld}^{\vep}$ are in order:
\begin{enumerate}
	\item[(1)] $\ld\in\{0,-1, +1\}$

\item [(2)] We use  standard shorthand notations  $
 f=f({\bf v}),\, f_*=f({\bf v}_*),\, f'=f({\bf v}'),\, f_*'=f({\bf v}_*'),$ where $({\bf v},{\bf v}_*)$ and $({\bf v}',{\bf v}_*')$ are the velocities of particles before and after their collision, and ${\bf v}'$, ${\bf v}_*'$ are given by the {\it $\og$-representation}
\bes && {\bf v}'={\bf v}-(({\bf v}-{\bf v}_*)\cdot\og) \og,
\quad {\bf v}_*'={\bf v}_*+(({\bf v}-{\bf v}_*)\cdot\og) \og,\quad {\bf v}, {\bf v}_*\in{\mR}^3,\quad
 \og\in {\mathbb S}^2, \lb{1.Omega} \ees
\item[(3)] $B^{\vep}_{\ld}({\bf v}-{\bf v}_*,\og)$ is defined by (see e.g. \cite{BPCE})
 \bes
&&
B^{\vep}_{\ld}({\bf v}-{\bf v}_*,\og)=
\left\{\begin{array}
{ll}\displaystyle
\fr{|({\bf v}-{\bf v}_*)\cdot\og|}{\vep^4}
\bigg\{\widehat{\phi}\Big(\fr{|{\bf v}-{\bf v}'|}{\vep}\Big)^2+
\widehat{\phi}\Big(\fr{|{\bf v}-{\bf v}_*'|}{\vep}\Big)^2\bigg\}
\qquad {\rm if}\quad \ld=0\\ \\ \displaystyle
\fr{|({\bf v}-{\bf v}_*)\cdot\og|}{\vep^4}\bigg(
\widehat{\phi}\Big(\fr{|{\bf v}-{\bf v}'|}{\vep}\Big)+\ld
\widehat{\phi}\Big(\fr{|{\bf v}-{\bf v}_*'|}{\vep}\Big)\bigg)^2
\qquad  {\rm if}\quad \ld=\pm 1
\end{array}\right.\lb{ker1}\ees
\item[(4)] $\widehat{\phi}(r):=\widehat{\phi(|\cdot|)}(\xi)$ with $|\xi|=r$, where
the function $\widehat{\phi(|\cdot|)}(\xi)$ is the Fourier transform of the particle interaction potential
${\bf x}\mapsto \phi(|{\bf x}|)\in {\mR}$.
\end{enumerate}
In particular, we specify that

 \noindent $\bullet$ if $\lambda=0$, \eqref{1.1} corresponds to the spatially homogeneous Maxwell-Boltzmann equation (MB) with a semi-classical collision kernel:
$$
{\p}_tf(t,{\bf v})=Q^{\vep}_0(f)(t,{\bf v}),\quad (t,{\bf v})
 \in (0,\infty)\times {\mR}^3.\eqno({\rm MB})$$

\noindent$\bullet$  if $\lambda=-1$, \eqref{1.1} corresponds to  the spatially homogeneous  Boltzmann equation for Fermi-Dirac particles (FD):
$${\p}_tf(t,{\bf v})=Q_{-1}^{\vep}(f)(t,{\bf v}),\quad (t,{\bf v})\in
 (0,\infty)\times {\mR}^3. \eqno({\rm FD})$$

\noindent$\bullet$  if $\lambda=1$, \eqref{1.1} corresponds to  the spatially homogeneous  Boltzmann equation for Bose-Einstein particles (BE):
$${\p}_t f(t,{\bf v})=Q_{+1}^{\vep}(f)(t,{\bf v}), \quad (t,{\bf v})\in
(0,\infty)\times {\mR}^3. \eqno({\rm BE})$$

\begin{remark}\label{Remark1.2}{\rm
In the present paper we call  {\rm MB} collision operator the one obtained from the {\rm FD} or {\rm BE} operators by neglecting the cubic terms (i.e. the case $\ld=0$). We note that for any suitable (e.g. integrable or nonnegative) Borel measurable function $\Psi$ on ${\bRR}\times {\bRR}$ satisfying
$\Psi({\bf w},{\bf w}_*, {\bf v}, {\bf v}_*)=\Psi({\bf w}_*,{\bf w},{\bf v}, {\bf v}_*)$, applying the formula (\ref{1.RR})  twice
with the reflection $\sg\to -\sg$ in between, one has the identity:
\bes&&\lb{1.New}
\int_{{\bS}}|({\bf v}-{\bf v}_*)\cdot\og|
\widehat{\phi}\Big(\fr{|{\bf v}-{\bf v}_*'|}{\vep}\Big)^2\Psi({\bf v}', {\bf v}_*',{\bf v}, {\bf v}_*){\rm d}\og\\
&&=
\int_{{\bS}}|({\bf v}-{\bf v}_*)\cdot\og|
\widehat{\phi}\Big(\fr{|{\bf v}-{\bf v}'|}{\vep}\Big)^2\Psi({\bf v}', {\bf v}_*',{\bf v}, {\bf v}_*){\rm d}\og.\nonumber\ees
This implies that the kernel $B_{\ld}^{\vep}({\bf v}-{\bf v}_*,\og)$ in the collision operator $Q_{\ld}^{\vep}(f)({\bf v})$ can also be replaced by
 \be
\label{asymm}
\wt{B}_{\ld}^{\vep}({\bf v}-{\bf v}_*,\og)=\fr{|({\bf v}-{\bf v}_*)\cdot\og|}{\vep^4}
\bigg\{
2\widehat{\phi}\Big(\fr{|{\bf v}-{\bf v}'|}{\vep}\Big)^2+
\ld 2\widehat{\phi}\Big(\fr{|{\bf v}-{\bf v}'|}{\vep}\Big)
\widehat{\phi}\Big(\fr{|{\bf v}-{\bf v}_*'|}{\vep}\Big)\bigg\}
\ee
and thus for the case $\ld=0$ we recover the {\rm MB} cross section in \cite{BP}.
According to  \cite {BCEP04} and \cite {BCEP08} , the case $\ld=0$  is just the cross section
arising from the weak coupling limit of quantum particles without statistics.
For the case $\ld=\pm 1$, the kernel $B_{\ld}^{\vep}({\bf v}-{\bf v}_*,\og)$
is still nonnegative while $\wt{B}_{\ld}^{\vep}({\bf v}-{\bf v}_*,\og)$ in (\ref{asymm})
is not, but
the second term in the righthand side of \eqref{asymm} is expected to vanish in the limit $\vep \to 0$ because it is the product of two terms which are going to concentrate in different points. Therefore the effect of the statistics should be negligible in this limit, up to the control of the cubic terms.
We also remark that the cross-section entering in the quantum Boltzmann equation is the one obtained in the Born approximation, because of the weakness of the potential.}
\end{remark}

The Fokker-Planck-Landau(FPL) equation was originally obtained by Landau from classical Boltzmann equation with cutoff Rutherford cross section (see \cite{LP}).
Mathematically the spatially homogeneous Fokker-Planck-Landau equation
associated to the Coulomb potential reads
$$ {\p}_t f(t,{\bf v})=Q_{L}(f)(t,{\bf v}), \quad (t,{\bf v})\in (0,\infty)\times {\mR}^3\eqno({\rm FPL})$$
where $
Q_{L}(f)(t,{\bf v})=Q_{L}(f(t,\cdot))({\bf v})$,
\bes&&
Q_{L}(f)({\bf v})=  M_{\widehat{\phi}}\nabla_{{\bf v}}\cdot\int_{{\mathbb R}^3}\fr{1}{|{\bf v-v}_*|}
\Pi({\bf v-v}_*)\big(\nabla_{{\bf v}}-\nabla_{{\bf v}_*}\big)
f({\bf v})f({\bf v}_*) {\rm d}{\bf v}_*,  \lb{Landau1}\\
&&M_{\widehat{\phi}}\, \,{\rm is\, \,defined\, \, in\, \eqref{ker2}},\quad
\Pi({\bf z})={\rm I}-{\bf n}\otimes {\bf n}\in
{\mR}^{3\times 3}, \quad {\bf n}={\bf z}/|{\bf z}|\,\,{\rm for}\,\, {\bf z}\in {\bR}\setminus\{{\bf 0}\},\lb{Landau2}\\
&&{\rm I}=(\dt_{ij})_{3\times 3},\,\, {\rm the \,\, unit\,\, matrix},\quad
 {\bf n}\otimes {\bf n}={\bf n}{\bf n}^{\tau}=(n_in_j)_{3\times 3}\quad {\rm for}\,\,
 {\bf n}=(n_1,n_2, n_3)^{\tau}\in {\mR}^3,
 \nonumber\\
 &&
\nabla_{{\bf v}}=({\p}_{v_1}, {\p}_{v_2}, {\p}_{v_3})^{\tau}. \nonumber\ees
\begin{remark}{\rm
 The word {\it Coulomb}  usually used for the FPL equation with the kernel singularity
 $\frac 1 {|{\bf v-v}_*|}$ is somehow misleading.   Such a singularity has nothing to do with the Coulomb interaction.  Indeed the FPL equation with different diffusive coefficient $M_{\widehat{\phi}}$ can also be derived even though the underlying particle system evolves under the action of a two-body smooth potential. For Coulomb potential, we have $\fr{1}{2\pi}M_{\widehat{\phi}}=\log(\Lambda),$ which is called Coulomb logarithm.}
\end{remark}

In \cite{BP}, Benedetto and Pulvirenti proved the
operator convergence $\lim\limits_{\vep\to 0^+}Q_{\ld}^{\vep}(f)=Q_{L}(f)$ in $\mathcal{S}'$,
for every $\ld\in \{0,-1, +1\}$ and for a suitable class of integrable functions $f$,  i.e.
$$\lim\limits_{\vep\to 0^+}\int_{{\bR}}\psi({\bf v})Q_{\ld}^{\vep}(f)({\bf v})
{\rm d}{\bf v}=\int_{{\bR}}\psi({\bf v})Q_{L}(f)({\bf v})
{\rm d}{\bf v}\quad \forall\, \psi\in \mathcal{S}$$
where $\mathcal{S}=\mathcal{S}({\bR})$ is the class of Schwartz functions.
\smallskip

 In this paper, we prove a further result: roughly speaking,
 solutions (strong, mild or weak) of Eq.(MB), Eq.(FD), and Eq.(BE)
converge weakly  to certain weak solutions of Eq.(FPL) as (up to subsequences) $\vep=\vep_n\to 0\,(n\to\infty).$  We do these under the following assumptions (\ref{ker2})-(\ref{ker3}).
\vskip2mm

\noindent{\bf Basic assumptions on the Fourier transform of the interaction potential:}
\begin{enumerate}
\item[$(\mathbf{A1})$.] For simplicity, we assume that the diffusive coefficient
$M_{\widehat{\phi}}$ in \eqref{Landau1}
is normalized, i.e.
\be M_{\widehat{\phi}}:=2\pi\int_{0}^{\infty}r^3
|\widehat{\phi}(r)|^2{\rm d}r=1\lb{ker2}.\ee
\item[$(\mathbf{A2})$.]  In order to study the convergence from Eq.(FD) to Eq.(FPL) and from Eq.(BE) to Eq.(FPL), we  further assume that
\be  \widehat{\phi}\in C_b({\mR}_{\ge 0}),\quad
A_{\widehat{\phi}}:=\sup_{r\ge 0} r|\widehat{\phi}(r)|<\infty.\lb{ker3}\ee
\end{enumerate}

It should be noted that Proposition \ref{Prop6.2} in the Appendix shows that
if there exists $r_0\ge 0$ such that
\be  r\mapsto r \widehat{\phi}(r) \,\,\,{\rm is\,\,\,monotone\,\,\, in}\,\,\, (r_0,\infty)\lb{1.po1}\ee
then the function ${\bf x}\mapsto \phi(|{\bf x}|)$ defined by
\be\phi(\rho)=\lim_{R\to\infty}\fr{1}{2\pi^2\rho}\int_{0}^{R}r\widehat{\phi}(r)\sin(\rho r){\rm d}r,\quad \rho>0\lb{1.po2}\ee
is the corresponding interaction potential. Furthermore if $r_0=0$ and $\widehat{\phi}(r)\ge 0$ in
$(0,\infty)$, then $\phi(\rho)\ge 0$ in
$(0,\infty).$

\subsubsection{Basic properties of the equations} First of all we note that  if $f$ is a solution of
Eq.(\ref{1.1})
or Eq.(FPL), then it enjoys the conservation of mass, momentum and the energy, i.e.
\be\int_{{\mR}^3}\left(\begin{array}{lll}\, 1 \\
\,{\bf v} \\
\,|{\bf v}|^2\end{array}\right)f(t,{\bf v}){\rm d}{\bf v}=
\int_{{\mR}^3}\left(\begin{array}{lll}\, 1 \\
\,{\bf v} \\
\,|{\bf v}|^2\end{array}\right)f(0,{\bf v}){\rm d}{\bf v}\qquad \forall\, t\ge 0.\lb{1.Conserv}\ee
Now let us introduce the entropy and the famous $H$-theorem\footnote{In this paper, ``the $H$-theorem" means that
the solution $f$ satisfies either the entropy identity like (\ref{entropy-identity-landau}) or the entropy inequality like (\ref{entropyinequality1}).  } for Eq.(\ref{1.1}) and
Eq.(FPL).
\vskip2mm
 \noindent $\bullet$ For MB model ($\lambda=0$) and FPL model, the corresponding entropies
 $H_0(f)$ and $H_L(f)$ are given by
 \beas&& H_{0}(f)=H_L(f)=H(f):=\int_{{\mathbb R}^3}
f({\bf v})\log f({\bf v}){\rm d}{\bf v}. \eeas

\noindent$\bullet$ For FD model and BE model (i.e. $\lambda\in \{-1, +1\}$), the entropies are defined by
\beas
&&
H_{\ld\vep^3}(f):=\int_{{\mathbb R}^3}\Big(f\log f-
\fr{1}{\ld\vep^3}(1+\ld\vep^3 f)\log(1+\ld\vep^3 f)
\Big){\rm d}{\bf v},\quad \ld\in\{-1,+1\}.\eeas

Let $f^{\vep}$ be a solution of Eq.(MB), Eq.(FD), Eq.(BE) (i.e. $\ld=0, -1, +1$) respectively
with the initial datum $f_0^{\vep}$ that belongs to
$L^1_2\cap L\log L ({\bR})$,  $L^1_2\cap L^{\infty}({\bR})$ (with $f_0^{\vep}\le \vep^{-3}$), and
$L^1_2({\bR})$   respectively.
Then, by formal calculation,  the $H$-theorem for $f^{\vep}$ can be stated as follows:
\be H_{\ld\vep^3}(f^{\vep}(t))+\int_{0}^{t}D^{\vep}_{\ld}(f^{\vep}(s)){\rm d}s=H_{\ld\vep^3}(f^{\vep}_0)
\qquad \forall\, t\ge 0
\lb{entropy-identity}\ee
where
\bes\lb{1.endi}  D^{\vep}_{\ld}(f)&=&\fr{1}{4}\int_{{\mathbb R}^3\times {\mathbb R}^3\times {\mathbb S}^2}
B_{\ld}^{\vep}({\bf v}-{\bf v}_*,\og)\\
&\times & \Gm\Big(f'f_*'(1+\ld\vep^3 f)(1+\ld\vep^3 f_*),\,
ff_*(1+\ld\vep^3 f')(1+\ld\vep^3 f_*')\Big){\rm d}\og{\rm d}{\bf v}{\rm d}{\bf v}_*,
\nonumber\ees
\be
\Gm(a,b)=\left\{\begin{array}
{ll}\displaystyle (a-b)\log (\fr{a}{b}),\quad \qquad  {\rm if}\quad a,b>0;
\\ \displaystyle
\,\,\infty, \,\,\,\,\qquad \qquad \quad \qquad {\rm if}\quad a=0<b\,\,\,{\rm or}\,\,\,b=0<a;
\\
 \displaystyle
\,\,\,0,\,\qquad \qquad \qquad \qquad {\rm if}\quad a=b=0\,.
\end{array}\right.\lb{1.Gamma}\ee
Let $f$  be a solution of Eq.(FPL)
with the initial datum $f_0$ that belongs to
$L^1_2\cap L\log L ({\bR})$.
By formal calculation,  the $H$-theorem for solutions of Eq.(FPL) can be stated as follows:
\be H(f(t))+\int_{0}^{t}D_{L}(f(s)){\rm d}s=H(f_0)
\qquad \forall\, t\ge 0
\lb{entropy-identity-landau}\ee
with
\be\label{D(f)}
D_L(f) =
2\int_{{\bRR}}\fr{1}{|{\bf v-v}_*|}\big|\Pi({\bf v-v}_*)(\nabla_{\bf v}-\nabla_{{\bf v}_*})
\sqrt{f({\bf v}) f({\bf v}_*)}
\big|^2
{\rm d}{\bf v}{\rm d}{\bf v}_*.\ee
Here and below we denote as usual for functionals of ${\bf v}\mapsto f(t,{\bf v})$ as
$$H(f(t))=H(f(t,\cdot)),\quad  D_L(f(t))=D_L(f(t,\cdot)),\quad \|f(t)\|=\|f(t,\cdot)\|,\quad {\rm etc.}$$
\medskip
Several remarks are in order:
\begin{remark}{\rm  According to the physical meaning
of the Fermi-Dirac model, the factor $1-\vep^3f(t,{\bf v})$ should be nonnegative, i.e.
$0\le f(t,{\bf v})\le \vep^{-3}$.
It can be proved
that if initially $0\le f_0\le \vep^{-3}$, then $0\le f(t,\cdot)\le \vep^{-3}$ for all $t\ge 0$. Thus the entropy $H_{-\vep^3}(f(t))$ for Eq.(FD) is always finite for all $t\ge0$.}
\end{remark}

\begin{remark}{\rm
 The entropy $H(f)$ for Eq.(MB) and Eq.(FPL), and the entropy $H_{+\vep^3}(f)$ for Eq.(BE) are all
 finite if initially they are finite.}
\end{remark}

\begin{remark}{\rm
In our  proof, the $H$-theorem is very useful for the convergence from
Eq.(MB) to Eq.(FPL) and  from Eq.(FD) to Eq.(FPL). But we have no idea how to apply it
to prove the convergence  from Eq.(BE) to  Eq.(FPL) because the $H$-theorem for Eq.(BE) does not
provide the $L^1$ weak compactness of solutions.}
\end{remark}

\subsection{Basic notations} Let us first introduce some function spaces and norms which will be used throughout the paper:

\noindent $\bullet$  $L^1({\bR})=L^1_0({\bR})$,\,\,
 $ L^1_k({\mR}^3)=\Big\{ f\in L^1({\mR}^3)\,\,\big|\,\, \|f\|_{L^1_k}:=\int_{{\mR}^3}(1+|{\bf v}|^2)^{k/2}|f({\bf v})|
 {\rm d}{\bf v}<\infty\Big\},\,\,k\in{\mR}$;

 \noindent $\bullet$
$L^1_k\cap L\log L ({\bR})=\Big\{ 0\le f\in L^1_k({\bR})\,\,\big|\,\,
  \int_{{\bR}}f({\bf v})|\log f({\bf v})|{\rm d}{\bf v}<\infty\Big\}$ ;

\noindent $\bullet$  $L^{\infty}([0,\infty), L^1_k({\bR}))=
 \Big\{ f \,\,{\rm Lebesgue\,\, measurable\,\,on}\,\,[0,\infty)\times {\bR}\,
 \,\big|\,\, \sup\limits_{t\ge 0}\|f(t)\|_{L^1_k}<\infty\Big\}$;

 \noindent $\bullet$
$L^{\infty}([0,\infty), L^1_k\cap L\log L ({\bR}))
= \Big\{ 0\le f\in L^{\infty}([0,\infty), L^1_k({\bR}))
 \,\big|\,\,
\sup\limits_{t\ge 0}\|f(t)\log f(t)\|_{L^1}<\infty\Big\};$

 \noindent $\bullet$ $\|f\|_{\infty}=\sup\limits_{{\bf x}\in E}|f({\bf x})|$ if $f$ is bounded on its domain of definition $E$.

\subsection{Main results, difficulties and   strategies of the problem}

In this subsection, we will   give the precise statement for our main results,  explain the difficulties of the proof for the weak convergence, and sketch our strategy of the proof.

\subsubsection{From  Eq.(MB) to Eq.(FPL)}  Our goal is to prove that,
 starting from   Eq.(MB), the weak solutions of  Eq.(MB)  (up to subsequences) converge to an {\it H-solution} of Eq.(FPL).  Before addressing the difficulties and the new ideas for this problem, we begin with some definitions on weak solutions of associated equations.

\begin{definition}\lb{weakMB} Given $\vep=2\pi\hbar>0$.
Let
$B^{\vep}_{0}({\bf v}-{\bf v}_*,\og), D^{\vep}_{0}(f)$ be given by (\ref{ker1}),(\ref{ker2}),(\ref{1.endi}) with $\ld=0$.
We say that a function  $f^{\vep}\in L^{\infty}([0,\infty), L^1_2\cap L\log L ({\bR}))$
is a weak solution of  Eq.(MB) if
$f^{\vep}$ satisfies the following {\rm (i),(ii)}:

{\rm (i)}  $f^{\vep}$ conserves the mass, momentum, and energy, and satisfies the entropy inequality
\be  H(f^{\vep}(t))+\int_{0}^{t}D^{\vep}_{0}(f^{\vep}(s)){\rm d}s\le
H(f^{\vep}_0),\quad t\ge 0.
\lb{entropy-inequality}\ee

{\rm (ii)} For any $\psi\in C_c^2( {\bR})$,
$t\mapsto \int_{{\mathbb R}^3}\psi({\bf v}) f^{\vep}(t,{\bf v}){\rm d}{\bf v}$
is absolutely continuous on
$[0, \infty)$ and the differential equality\footnote{In this paper we say that a function $g(t)$ is
is absolutely continuous on
$[0, \infty)$ if $g$ is absolutely
continuous on every bounded interval $[a,b]\subset [0, \infty)$. Recalling that
this is equivalent to that $\fr{{\rm d}}{{\rm d} t}g\in L^1_{loc}([0,\infty))$ and
$g(b)-g(a)=\int_{a}^{b}\fr{{\rm d}}{{\rm d} t}g(t){\rm d}t$ for all
$0\le a<b<\infty$.
}
\bes \lb{1.WW}&&
\fr{{\rm d}}{{\rm d}t}
\int_{{\mathbb R}^3}\psi({\bf v})f^{\vep}(t,{\bf v}){\rm d}{\bf v}
 \\
&&= \fr{1}{4}\int_{{\bRRS}}B_{0}^{\vep}({\bf v}- {\bf v}_*,\og)
\big({f^{\vep}}'{f_*^{\vep}}'-{f^{\vep}}{f_*^{\vep}}\big)\big(\psi+\psi_*-\psi'-\psi_*'\big)
{\rm d}{\bf v}{\rm d}{\bf v}_*{\rm d}\og\nonumber\ees
holds for almost every $t\in(0,\infty)$.
\end{definition}
  \begin{remark}{\rm The weak solutions of Eq.(MB)   has been
considered in \cite{CCL} for the so-called very soft potentials that includes our MB model because
 \be \int_{{\mathbb S}^2}
B_{0}^{\vep}({\bf v}-{\bf v}_*,\og)
\sin^2(\theta)\cos^2(\theta)\big|_{\theta=\arccos({\bf n}\cdot\og)}{\rm d}\og\le \fr{4}{|{\bf v}-{\bf v}_*|^3}\lb{1.BB}\ee which comes from the assumption (\ref{ker2}). Here  ${\bf n}=({\bf v-v}_*)/|{\bf v-v}_*|$
and throughout this paper we define ${\bf n}=(1,0,0)^{\tau}$ for ${\bf v-v}_*
={\bf 0}$. Note that as in \cite{CCL}, the {\it $\sg$-representation} will be used:
\be {\bf v}'=\fr{{\bf v+v}_*}{2}+\fr{|{\bf v-v}_*|}{2}\sg, \quad
{\bf v}'_*=\fr{{\bf v+v}_*}{2}-\fr{|{\bf v-v}_*|}{2}\sg,\quad \sg\in {\bS}.\lb{1.Sigma}\ee
The $\og$-representation (\ref{1.Omega}) and the $\sg$-representation (\ref{1.Sigma})
have the following relation (see e.g. section 4 of chapter 1 in \cite{V2}):
\be
\int_{{\bS}}|({\bf v-v}_*)\cdot\og|\Psi({\bf v}', {\bf v}_*')
\big|_{\og-{\rm rep.}}{\rm d}\og=\fr{1}{2}
\int_{{\bS}}|{\bf v-v}_*|\Psi({\bf v}', {\bf v}_*')\big|_{\sg-{\rm rep.}}{\rm d}\sg. \lb{1.RR}
\ee
From this, it is easy to see that the inequality (\ref{1.BB})
(with the $\og$-representation) is just the inequality (1.6) in \cite{CCL} (with the $\sg$-representation) with $N=3, \gm=-3$, and a different constant.  Therefore by entropy control,
the righthand side of (\ref{1.WW}) is absolutely convergent (see also (\ref{3.EE}) below).} \end{remark}

The following is a definition of $H$-solutions in spirit of Villani \cite{V1}.

\begin{definition}\lb{weakFPL}
We say that a function $f\in L^{\infty}([0,\infty), L^1_2\cap L\log L({\bR}))$
 is an $H$-solution of Eq.(FPL) if
$f$ satisfies the following {\rm (i),(ii), (iii)}:

{\rm (i)}  $f$ conserves the mass, momentum, and energy.

{\rm (ii)} The function $ (t,{\bf v},{\bf v}_*)\mapsto \sqrt{f(t,{\bf v}) f(t,{\bf v}_*)/|{\bf v-v}_*|}$ has
the weak projection gradient \\ $ \Pi({\bf v-v}_*)\nabla_{{\bf
v-v}_*}\sqrt{f(t,{\bf v}) f(t,{\bf v}_*)/|{\bf v-v}_*|}$ in ${\bf v-v_*}\neq {\bf 0}$ defined in Definition \ref{weak-diff-2} in Appendix, and
$$\int_{0}^{\infty}D_L(f(t)){\rm d}t<\infty$$
where
\be D_L(f(t)):=2\int_{{\bRR}}\big|\Pi({\bf v-v}_*)\nabla_{{\bf v-v}_*}
\sqrt{f(t,{\bf v}) f(t,{\bf v}_*)/|{\bf v-v}_*|}
\big|^2
{\rm d}{\bf v}{\rm d}{\bf v}_*.\lb{1.DD}\ee

{\rm (iii)} For any $\psi\in C_c^2( {\bR})$,
$t\mapsto \int_{{\mathbb R}^3}\psi({\bf v}) f(t,{\bf v}){\rm d}{\bf v}$
is absolutely continuous on $[0, \infty)$  and the differential equality
\be \fr{{\rm d}}{{\rm d} t}\int_{{\mathbb R}^3}\psi({\bf v}) f(t,{\bf v}){\rm d}{\bf v}
=\int_{{\bR}}\psi({\bf v})Q_{L}(f)(t,{\bf v})
{\rm d}{\bf v}\lb{Lweak}\ee
holds for almost every $t\in[0,\infty)$, where the righthand side of
(\ref{Lweak}) is defined by
\bes \lb{H-solu}&&
\int_{{\mathbb R}^3}\psi({\bf v})Q_{L}(f)(t,{\bf v}){\rm d}{\bf v}
 \\
&&:=-\fr{1}{2}
\int_{{\bRR}}\fr{\nabla \psi({\bf v})-
\nabla \psi({\bf v}_*)}{|{\bf v-v}_*|}\cdot\Pi({\bf v-v}_*)
\nabla_{{\bf v-v}_*}\big(f(t,{\bf v}) f(t,{\bf v}_*)\big)\,
{\rm d}{\bf v}{\rm d}{\bf v}_*.\qquad \nonumber\ees
Besides, if an {\it H}-solution $f$ also satisfies
\be H(f(t))+\int_{0}^{t}D_L(f(s)){\rm d}s\le H(f(0))\qquad \forall\, t\ge 0
\lb{entropyinequality1}\ee
then we say that $f$ satisfies the entropy inequality.
\end{definition}

Some remarks are in order:

\begin{remark} {\rm The weak projection gradient operator $\Pi({\bf v-v}_*)\nabla_{{\bf
v-v}_*}$ is introduced in Definition \ref{weak-diff-2} in Appendix. We emphasize that $\Pi({\bf v-v}_*)\nabla_{{\bf
v-v}_*}F({\bf v, v}_*)$ does not equal $\Pi({\bf v-v}_*)(\nabla_{\bf v}-\nabla_{{\bf v}_*})F({\bf v, v}_*)$ unless the usual gradient operators $\nabla_{\bf v}$ and $\nabla_{{\bf v}_*} $ are both well-defined for
$F({\bf v, v}_*)$.}
\end{remark}

\begin{remark}{\rm
The  assumption (ii) in Definition \ref{weakFPL} and Lemma \ref{Lemma6.2}(b), Lemma \ref{Lemma6.7} in Appendix imply that
 $\sqrt{ff_*}$ and  $ff_*$
 also have  weak  projection gradients $\Pi({\bf v-v}_*)\nabla_{{\bf
v-v}_*}\sqrt{ff_*}$ and
$\Pi({\bf v-v}_*)\nabla_{{\bf
v-v}_*}(ff_*)$ in ${\bf v-v}_*\neq {\bf 0}$, and they have the following relations
\beas&&\fr{1}{\sqrt{|{\bf v-v}_*|}}\Pi({\bf v-v}_*)\nabla_{{\bf v-v}_*}
\sqrt{ff_*}
=\Pi({\bf v-v}_*)\nabla_{{\bf v-v}_*}
\sqrt{ff_*/|{\bf v-v}_*|}\,
,\\
&&\Pi({\bf v-v}_*)
\nabla_{{\bf v-v}_*}(ff_*)
=2\sqrt{ff_*}\Pi({\bf v-v}_*)
\nabla_{{\bf v-v}_*}\sqrt{ff_*}\,.\eeas
Since
$|\nabla \psi({\bf v})
-\nabla \psi({\bf v}_*)|\le \|D^2\psi\|_{\infty}|{\bf v-v}_*|$, it follows
from Cauchy-Schwarz inequality that for almost all $t\in(0,\infty)$ the integral in the righthand side of
(\ref{H-solu}) is absolutely convergent and
\beas&&\bigg|\int_{{\mathbb R}^3}\psi({\bf v})Q_{L}(f)(t,{\bf v}){\rm d}{\bf v}
\bigg|\\
&&\le \int_{{\bRR}}\fr{|\nabla \psi({\bf v})-
\nabla \psi({\bf v}_*)|}{|{\bf v-v}_*|}
\sqrt{ ff_*|{\bf v-v}_*|}\,\big|\Pi({\bf v-v}_*)
\nabla_{{\bf v-v}_*}\sqrt{ ff_*/|{\bf v-v}_*|}\big|\,
{\rm d}{\bf v}{\rm d}{\bf v}_*
\\
&&\le C_{f_0}\|D^2\psi\|_{\infty}\sqrt{D_L(f(t))}\eeas
where the constant $C_{f_0}<\infty$ depends only on the mass and energy of $f_0$.}
\end{remark}

\begin{remark} {\rm Recall that the usual weak form of Eq.(FPL) is defined by
\beas
\fr{{\rm d}}{{\rm d} t}\int_{{\mathbb R}^3}\psi({\bf v}) f(t,{\bf v}){\rm d}{\bf v}
=\int_{{\bRR}}L[\psi]({\bf v,v}_*) f(t,{\bf v}) f(t,{\bf v}_*)
{\rm d}{\bf v}{\rm d}{\bf v}_*,\quad t\in (0,\infty)\eeas
for all $\psi\in C_c^2( {\bR})$, where
\be L[\psi]({\bf v,v}_*)=\fr{1}{2|{\bf v-v}_*|}(\nabla_{{\bf v}}-\nabla_{{\bf v}_*})\cdot
\Pi({\bf v-v}_*)\big(\nabla \psi({\bf v})-\nabla\psi({\bf v}_*)\big)\lb{1.LL}.\ee
 The existence of weak solutions of Eq.(FPL) has been a problem because
$ L[\psi]({\bf v,v}_*)=O(\fr{1}{|{\bf v-v}_*|})$ and $ff_*/|{\bf v-v}_*|$ may be non-integrable on ${\bRR}$. This is the main motivation in \cite{V1}  to introduce the {\it H-solution} to Eq.(FPL).}
\end{remark}

  To prove the convergence from the weak solution of Eq.(MB) to $H$-solution of Eq.(FPL), the main  difficulty lies in the  derivation of the existence of {\it weak projection gradient} of solutions of Eq.(PFL) by taking the weak limit
  with respect to $\vep\to 0^+$.  Note that the Maxwell-Boltzmann kernel $B_{0}^{\vep}({\bf v}- {\bf v}_*,\og)$ in Eq.(MB) with the assumption
  (\ref{ker2}) can not be written as the product form $b(\cos\theta)\Phi(|v-v_*|)$ used in  \cite{V1}, and this is one reason that we cannot directly follow the proof used there. We will instead establish a framework to deal with the {\it weak projection gradient} and this can be regarded as one of our contributions in this paper:

 $\bullet$ We introduce
 new types of test function spaces to define the  weak projection gradient which are the keys to identify the $H$-solution for Fokker-Planck-Landau equation. This is motivated by Villani \cite{V1}.  In this way, we clarify that generally, $\Pi({\bf v-v}_*)\nabla_{{\bf v-v}_*}\neq \Pi({\bf v-v}_*)(\nabla_{{\bf v}}-\nabla_{{\bf v}_*})$.

  $\bullet$  We introduce tools such as the change of variables(see Section \ref{COV}) and a suitable  measure space (see (\ref{3.Measure})) to prove the existence of weak projection gradients
  and to derive the integrability of $D_L(f(t))$ from $D_0^\vep(f^{\vep}(t))$ in the new framework.
\vskip1mm
Our first main result of the paper is as follows:

\begin{theorem}\lb{Theorem1}{\rm[From  Eq.(MB) to Eq.(FPL)]} Let
$B^{\vep}_{0}({\bf v}-{\bf v}_*,\og)$ be given by (\ref{ker1}) and (\ref{ker2}) with $\ld=0$.
Let $\vep_0>0, f_0^{\vep}\in L^1_2\cap L\log L({\mR}^3)$ satisfy
\be \sup_{0<\vep\le \vep_0}\int_{{\mathbb R}^3}f^{\vep}_0({\bf v})(1+|{\bf v}|^2+
|\log f^{\vep}_0({\bf v})|){\rm d}{\bf v}<\infty.\ee
Let
 $f^{\vep}\in L^{\infty}([0, \infty),
L^1_2\cap L\log L({\mR}^3))
$ be weak solutions of the Eq.({\rm MB})
with the initial data $f_0^{\vep}$. Then for any sequence
${\cal E}=\{\vep_n\}_{n=1}^{\infty}\subset(0,\vep_0]$ satisfying $\vep_n\to 0\,(n\to\infty)$,
there exist a subsequence, still denote it as
${\cal E}$, and an $H$-solution $f$ of Eq.({\rm FPL}),
such that $f^{\vep}(t,\cdot)\rightharpoonup f(t,\cdot)$ weakly in $L^1({\bR})$  (as ${\cal E}\ni \vep\to 0$)
for every $t\in [0,\infty)$. More precisely we have
\be \lim_{{\cal E}\ni \vep\to 0}\int_{{\bR}}\psi({\bf v})f^{\vep}(t,{\bf v}){\rm d}{\bf v}
=\int_{{\bR}}\psi({\bf v})f(t,{\bf v}){\rm d}{\bf v}\quad \forall\, \psi\in L^{\infty}({\bR})
\quad \forall\, t\in[0,\infty) \lb{L1}\ee
and
\be \lim_{{\cal E}\ni \vep\to 0}\int_{0}^{T}{\rm d}t
\int_{{\bR}}\psi({\bf v})Q_{0}^{\vep}(f^{\vep})(t,{\bf v})
{\rm d}{\bf v}=\int_{0}^{T}{\rm d}t\int_{{\bR}}\psi({\bf v})Q_{L}(f)(t,{\bf v})
{\rm d}{\bf v} \lb{L2}\ee
for all $ \psi\in C_c^2({\bR})$ and all
$T\in(0,\infty)$, where the righthand side of (\ref{L2}) is defined by (\ref{H-solu}).
Besides, it also holds
\be H(f(t))+
\int_{0}^{t}D_L(f(s)){\rm d}s\le  \liminf_{{\cal E}\ni \vep\to 0}
 H(f^{\vep}_0)\qquad \forall\, t>0.\lb{entropyinequality2}\ee
In particular, if in addition that
$|f_0^{\vep}-f_0|\big(1+\log^{+}(|f_0^{\vep}-f_0|)\big)\to 0$ in $L^1({\bR})$ as
$\vep\to 0^{+}$ for some $f_0\in L^1_2({\bR})\cap L\log L({\bR})$ (for instance
$f_0^{\vep}\equiv f_0$), then this
$H$-solution $f$ with the initial datum $f(0,\cdot)=f_0$ also satisfies the entropy inequality (\ref{entropyinequality1}).
\end{theorem}

\begin{remark} {\rm This result indicates that starting from any smooth short-range potential function $\phi$, the Maxwell-Boltzmann equation will converge to the Fokker-Planck-Landau equation in the weak coupling limit, which is expected in physics and is consistent with Remark \ref{Remark1.2}.   Moreover comparing to the grazing collisions limit proposed in \cite{V1}, the weak coupling limit enables us to get the right diffusion coefficient $M_{\wh{\phi}}$ for the resulting equation.}
\end{remark}

\begin{remark} {\rm As emphasized above, here we have only proved that the function $\sqrt{ff_*/|{\bf v}-{\bf v}_*|}$ has the weak projection gradient
$\Pi({\bf v-v}_*)\nabla_{{\bf v-v}_*}\sqrt{ff_*/|{\bf v}-{\bf v}_*|}$ in ${\bf v-v}_*\neq {\bf 0}$, and this does not directly imply
the existence of the usual weak gradient $\nabla_{{\bf v}} \sqrt{f(t,{\bf v})}$. Thus the result in \cite{D1}
 cannot be directly applied to derive the weighted $L^p({\bR})\,(p>1)$  estimate for $H$-solutions, and so it remains unclear whether an $H$-solution can be a usual weak solution.}
\end{remark}

\subsubsection{  From Eq.(FD) or Eq.(BE) to Eq.(FPL)}
Let us first introduce the definition of weak solutions of  Eq.(FD) and Eq.(BE).

\begin{definition}\lb{weakFDBE} Given $\vep=2\pi\hbar>0$.
Let $B^{\vep}_{\ld}({\bf v}-{\bf v}_*,\og)$ be given by (\ref{ker1}), (\ref{ker2}), (\ref{ker3}) with $\ld=\pm 1$.
We say that a function  $0\le f^{\vep}\in L^{\infty}([0,\infty), L^1_2({\bR}))$
is a weak solution of Eq.(FD) ($\ld=-1$)
and Eq.(BE) ($\ld=+1$) respectively  if
$f^{\vep}$ satisfies the following {\rm (i),(ii)} and for the case $\ld=-1$ assuming further that
$f^{\vep}\le \vep^{-3}$ on $[0,\infty)\times {\bR}$:

{\rm (i)}  $f^{\vep}$ conserves the mass, momentum, and energy.

{\rm (ii)} For any $\psi\in C_c^2( {\bR})$,
$t\mapsto \int_{{\mathbb R}^3}\psi({\bf v}) f^{\vep}(t,{\bf v}){\rm d}{\bf v}$
is absolutely continuous on
$[0, \infty)$ and the differential equality
\be
\fr{{\rm d}}{{\rm d}t}
\int_{{\mathbb R}^3}\psi f^{\vep}{\rm d}{\bf v}
= \int_{{\bRR}}L^{\vep}_{\ld}[\psi]
f^{\vep}f^{\vep}_*{\rm d}{\bf v}{\rm d}{\bf v}_*
+R_{\ld}^{\vep}[\psi, f^{\vep}(t)]\lb{MBFDBE2}\ee
holds for almost every $t\in(0,\infty)$, where
\bes && L^{\vep}_{\ld}[\psi]({\bf v}, {\bf v}_*)=\fr{1}{2}\int_{{\mathbb S}^2}
B_{\ld}^{\vep}({\bf v}-{\bf v}_*,\og)\big(\psi'+\psi_*'
-\psi-\psi_*\big){\rm d}\og,\lb{1.LLL}\\
&& R_{\ld}^{\vep}[\psi, f]
=\ld \vep^3 \int_{{\mathbb R}^3\times {\mathbb R}^3\times {\mathbb S}^2}
B_{\ld}^{\vep}({\bf v}-{\bf v}_*,\og)
\big(\psi'+\psi_*'
-\psi-\psi_*\big)f'ff_*
{\rm d}\og{\rm d}{\bf v}{\rm d}{\bf v}_*.\lb{1.RRR}\ees
\end{definition}

\begin{remark}\label{Remark1.1}  {\rm Notice that the assumption (\ref{ker3}) with
definition (\ref{ker1}) implies that
$B^{\vep}_{\ld}({\bf v}-{\bf v}_*,\og)\le \fr{2}{\vep^3}
A_{\widehat{\phi}}\|\widehat{\phi}\|_{\infty}(1+\fr{\cos(\theta)}{\sin(\theta)})$ so that
\beas\sup_{{\bf v},{\bf v}_*\in {\bR}}\int_{\bS}
B^{\vep}_{\ld}({\bf v}-{\bf v}_*,\og){\rm d}\og\le \fr{16\pi}{\vep^3}
A_{\widehat{\phi}}\|\widehat{\phi}\|_{\infty}<\infty\eeas
where $ \theta=\arccos(|{\bf n}\cdot\og|)$, ${\bf n}=({\bf v-v}_*)/|{\bf v-v}_*|$. This ensures that
the integrals in the righthand side of (\ref{MBFDBE2}) are absolutely convergent for
the case $\ld=-1$ and thus the weak form (\ref{MBFDBE2}) is rigorous. For $\ld=+1$, however, the weak form (\ref{MBFDBE2}) maybe holds rigorously only for short time interval (see e.g. \cite{Briant-Einav} for local in time existence) or holds only formally for large time interval because (due to the effect of velocity concentration) there has been no results
on the global in time existence of weak solutions  $f^{\vep}$ in
$L^{\infty}([0,\infty), L^1_2({\bR}))$ for general initial data. So far most available general results for $\ld=+1$ have been only proven for isotropic measure-valued weak solutions (see below).}
\end{remark}

In the study of the convergence from Eq.(FD) and Eq.(BE) to Eq.(FPL), the most difficult problem is how to prove the zero limit of the cubic term:
\be\lim_{\vep\to 0^+}\int_{0}^{T}R_{\ld}^{\vep}[\psi, f^{\vep}(t)]{\rm d}t=0\qquad \forall\, T\in (0,\infty),
\,\,\,\forall\, \psi\in C_c^2( {\bR}).\lb{cubicconverg}\ee
Indeed, from the analysis in Section 2 and Section 4 one can see that it is very
different from the isotropic case that
for anisotropic weak solutions $f^{\vep}$  of Eq.(FD) or Eq.(BE),
it is even hard to prove the $L^{\infty}$-boundedness of $\{R_{\ld}^{\vep}[\psi, f^{\vep}(t)]\,\,|\,\,\vep\in (0, \vep_0], t\in [0, T]\}$.
So far we can only prove the zero limit (\ref{cubicconverg}) for
 isotropic solutions because in the isotropic case there are more symmetric and cancellation properties
in the cubic terms (see Step 3 in the proof of Lemma \ref{Lemma4.3}).
By the way, if the zero limit (\ref{cubicconverg}) can be proven to hold for
 general weak solutions $f^{\vep}$ of Eq.(FD), then such a model convergence to Eq.(FPL) maybe holds also (for instance) for the quantum Boltzmann equation for Haldane exclusion statistics (see e.g. \cite{BBM},\cite{Ark2010}), provided that the corresponding
collision kernel has a similar structure as those obtained from the weak-coupling limit.
These are of course completely new and more involved. \vskip2mm

Now let us introduce definitions about isotropic solutions.

\begin{definition} If a solution $f$
is a function of $(t,|{\bf v}|^2/2)$ only,
then we say that $f$ is an isotropic solution (or radially symmetric solution).
\end{definition}

For convenience of analysis we will write an isotropic solution as
$f=f(t, |{\bf v}|^2/2)=f(t,x)$ with $x=|{\bf v}|^2/2$, and
the measure element
$f(t,|{\bf v}|^2/2){\rm d}{\bf v}$
will be replaced by $4\pi\sqrt{2}f(t,x)\sqrt{x}\,{\rm d}x$. The corresponding
test functions are chosen from $ C_c^2({\mR}_{\ge 0})$.
\bigskip

 {\it {\bf Case 1}: From Eq.(FD) to Eq.(FPL).}
It is easily calculated that the isotropic form of Eq.(FPL) is
\be {\p}_tf(t,x)=
\fr{16\pi}{3\sqrt{x}}\fr{\p}{\p x}\int_{{\mR}_{\ge 0}}
\big(f(t,y){\p}_{x}f(t, x)-f(t,x){\p}_{y}f(t, y)\big)(x\wedge y)^{3/2}
{\rm d}y,\quad (t,x)\in {\mR}_{>0}^2.\lb{1.isotrop1}\ee
Here and below
$$a\wedge b=\min\{a,b\},\quad a\vee b=\max\{a,b\}.$$

\begin{definition}{\rm[Isotropic weak solutions of Eq.(FPL)]}
A function
$0\le f\in L^{\infty}\big([0,\infty), L^1({\mR}_{\ge 0}, (1+x)\sqrt{x}{\rm d}x)\big)$
is called an isotropic weak solution of Eq.(FPL) (or equivalently a weak solution of Eq.(\ref{1.isotrop1})) if

{\rm (i)} $f$ conserves the mass and energy:
$\int_{{\mathbb R}_{\ge 0}}(1,x)f(t,x)\sqrt{x}\,{\rm d}x\equiv
\int_{{\mathbb R}_{\ge 0}}(1,x)f(0,x)\sqrt{x}\,{\rm d}x,\,t\ge 0.$

{\rm (ii)}  For any
$\vp\in C_c^2({\mR}_{\ge 0})$, it holds
\be
\fr{{\rm d}}{{\rm d}t}
\int_{{\mR}_{\ge 0}}\vp(x) f(t,x)\sqrt{x}\,{\rm d}x
=\int_{{\mR}_{\ge 0}^2}
L[\vp](x,y)f(t,x)f(t,y)\sqrt{x}\sqrt{y}\,{\rm d}x{\rm d}y,\quad t\in[0,\infty)\lb{1.isotrop2}\ee
where
\be L[\vp](x,y)=\fr{4\pi}{\sqrt{x\vee y}}
\Big(\fr{2}{3}\big(\vp''(x)+\vp''(y)\big)
(x\wedge y)-\big(\vp'(x)-\vp'(y)\big)
{\rm sgn}(x-y)
\Big)\lb{Lker1}.\ee
\end{definition}
\begin{remark}{\rm Setting $L[\vp](0,0)=0$, the function $(x,y)\mapsto L[\vp](x,y)$
is continuous in ${\mR}_{\ge 0}^2$ and
\be |L[\vp](x,y)|\le C_0\|\vp''\|_{\infty}(\sqrt{x}+\sqrt{y}),\quad (x,y)\in{\mR}_{\ge 0}^2\lb{Lker2}\ee
where $\|\vp''\|_{\infty}=\sup\limits_{x\ge 0}|\vp''(x)|$ and $C_0\in(0,\infty)$ denotes an absolute constant.}
\end{remark}

We also need to define the isotropic weak solutions for  Eq.(FD) and Eq.(BE) (see \cite{Lu2004} and Appendix in \cite{CaiLu}). We first introduce, for any
given $\Phi\in C_b({\mR}_{\ge 0}^2)$,  a function $W_{\Phi}(x,y,z)$ as follows:
\be
 W_{\Phi}(x,y,z):=\fr{4\pi }{\sqrt{xyz}}
\int_{|\sqrt{x}-\sqrt{y}|\vee |\sqrt{x_*}-\sqrt{z}|}
^{(\sqrt{x}+\sqrt{y})\wedge(\sqrt{x_*}+\sqrt{z})}{\rm d}s
\int_{0}^{2\pi}\Phi(\sqrt{2}s, \sqrt{2} Y_* ){\rm d}\theta
\quad {\rm if}\quad x_*,x,y,z>0, \lb{WPhi}\ee
\be
 W_{\Phi}(x,y,z):=\left\{\begin{array}
{ll}
 \displaystyle
\fr{(4\pi)^2}{\sqrt{yz}}
\Phi\big(\sqrt{2y}, \sqrt{2z}\,\big)
\qquad\,\, \qquad\qquad{\rm if}\quad  x=0,\,y>0,\, z>0 \\ \\ \displaystyle
\fr{(4\pi)^2}{\sqrt{xz}}
\Phi\big(\sqrt{2x}, \sqrt{2(z-x)}\,\big)
\qquad \qquad {\rm if}\quad  y=0,\, z>x>0
\\ \\ \displaystyle
\fr{(4\pi)^2}{\sqrt{xy}}
\Phi\big(\sqrt{2(y-x)},\sqrt{2x}\,\big)
\qquad \quad \quad{\rm if}\quad  z=0,\,y>x>0,
\\ \\ \displaystyle
0\qquad \quad \qquad {\rm others}\end{array}\right.\lb{WPhi2}\ee

\be Y_*:=Y_*(x,y,z,s,\theta)=\left\{\begin{array}
{ll}\displaystyle
\left|\sqrt{\Big(z-\fr{(x-y+s^2)^2}{4s^2}
\Big)_{+}}
+e^{{\rm i}\theta}\sqrt{\Big(x-\fr{(x-y+s^2)^2}{4s^2}\Big)_{+}
}\,\right|\quad {\rm if}\quad s>0\\ \\ \displaystyle
0\qquad\qquad  {\rm if}\quad s=0
\end{array}\right.\lb{Y}\ee
for all $x,y,z\ge 0$, where $(u)_{+}=\max\{u,0\},$
${\rm i}=\sqrt{-1}$.

\begin{definition}{\rm[Isotropic weak solutions of  Eq.(FD) and Eq.(BE)]}
A function $0\le f^{\vep}\in L^{\infty}\big([0,\infty), L^1({\mR}_{\ge 0}, \\ (1+x)\sqrt{x}{\rm d}x)\big)$
is called an isotropic weak solution of Eq.(FD)($\lambda=-1$) or Eq.(BE)($\lambda=+1$) if

{\rm (i)} $f^{\vep}$ conserves the mass and energy:
$\int_{{\mathbb R}_{\ge 0}}(1,x)f^{\vep}(t,x)\sqrt{x}\,{\rm d}x=
\int_{{\mathbb R}_{\ge 0}}(1,x)f^{\vep}(0,x)\sqrt{x}\,{\rm d}x
$ for all $t\ge 0$, and for the case $\ld=-1$ it also requires that $f^{\vep}\le \vep^{-3}$ on $[0,\infty)\times {\mR}_{\ge 0}$.

{\rm (ii)} For any
$\vp\in C_c^2({\mR}_{\ge 0})$, it holds
\bes\lb{weak1} &&\fr{{\rm d}}{{\rm d}t}\int_{{\mR}_{\ge 0}}\vp(x)f^{\vep}(t,x)\sqrt{x}{\rm d}x=
\int_{{\mR}_{\ge 0}^2}{\cal J}_{\ld}^{\vep}[\vp](y,z)f^{\vep}(t,y)f^{\vep}(t,z)\sqrt{yz}
{\rm d}y{\rm d}z  \\
&& \qquad \qquad +\ld\vep^3 \int_{{\mR}_{\ge 0}^3}
{\cal K}_{\ld}^{\vep}[\vp](x,y,z) f^{\vep}(t,x)f^{\vep}(t,y)f^{\vep}(t,z)\sqrt{xyz}
{\rm d}x{\rm d}y{\rm d}z\qquad \forall\, t\in (0,\infty)\nonumber\ees
where
\be {\cal J}_{\ld}^{\vep}[\vp](y,z)=\fr{1}{2}
\int_{0}^{y+z}{\cal K}_{\ld}^{\vep}[\vp](x,y,z)\sqrt{x}{\rm d}x,\lb{JPhi}\ee
\be {\cal K}_{\ld}^{\vep}[\vp](x,y,z)=W_{\Phi_{\ld}^{\vep}}(x,y,z)
\big(\vp(x)+\vp(x_*)-\vp(y)-\vp(z)\big),
 \lb{KPhi}\ee
$x,y,z\in{\mR}_{\ge 0}$, $x_*=(y+z-x)_{+}$, and
\be \Phi_{\ld}^{\vep}(r, \rho)
=\fr{1}{\vep^4}\Big(\widehat{\phi}\big(\fr{r}{\vep}\big)+\ld\widehat{\phi}\big(\fr{\rho}{\vep}\big)
\Big )^2\quad with\quad \widehat{\phi}\in C_b({\mR}_{\ge 0}),\quad \ld\in\{-1, +1\}.\lb{ker4}\ee
\end{definition}
\begin{remark}{\rm
Thanks to the identity
$$(\sqrt{x}+\sqrt{y})\wedge (\sqrt{x_*}+\sqrt{z})-|\sqrt{x}-\sqrt{y}|\vee |\sqrt{x_*}-\sqrt{z}|
=2\min\{\sqrt{x},\sqrt{x_*},\sqrt{y},\sqrt{z}\,\},$$
we have
\beas\big|\vp(x)+\vp(x_*)-\vp(y)-\vp(z)\big|
\le\min\Big\{4\|\vp\|_{\infty},\, 2\|\vp'\|_{\infty}(
|x-y|\wedge|x-z|),\, \|\vp''\|_{\infty}|x-y||x-z|\Big\}.\eeas
From this together with $\Phi_{\ld}^{\vep}\in C_b({\mR}_{\ge 0}^2)$, it is easily seen that for any $\vp\in C^2_c({\mathbb R}_{\ge 0})$,
the function $(y,z)\mapsto (1+\sqrt{y}+\sqrt{z})^{-1}{\cal J}_{\ld}^{\vep}[\vp](y,z)$
belongs to $C_b({\mR}_{\ge 0}^2)$ and $(x,y,z)\mapsto {\cal K}_{\ld}^{\vep}[\vp](x,y,z)$
belongs to $C_b({\mR}_{\ge 0}^3)$.  These ensure that the integrals in the righthand side of
 Eq.(\ref{weak1}) are absolutely convergent for all $0\le f^{\vep}\in
 L^{\infty}\big([0,\infty), L^1({\mR}_{\ge 0}, (1+x)\sqrt{x}{\rm d}x)\big)$.}
 \end{remark}

In this paper, for any sequence ${\cal E}=\{\vep_n\}_{n=1}^{\infty}\subset
(0, \vep_0]$ satisfying $\vep_n\to 0\,(n\to\infty)$, we assume without loss of generality that
$\vep_n$ is strictly decreasing, so that the sequence  ${\cal E}$ can be treated as subset of
$(0, \vep_0]$, and thus ``$\vep_n\to 0\,(n\to\infty)$" and
``${\cal E}\ni \vep\to 0 $" have the same meaning, so that we can denote
\be\lb{1.UpLw}\lim_{{\cal E}\ni\vep\to 0}a_{\vep}=\lim_{n\to\infty}a_{\vep_n},
\quad\liminf_{{\cal E}\ni\vep\to 0}a_{\vep}=\lim_{k\to\infty}\inf_{n\ge k}a_{\vep_n},\quad
\limsup_{{\cal E}\ni\vep\to 0}a_{\vep}=\lim_{k\to\infty}\sup_{n\ge k}a_{\vep_n}.\ee

\begin{theorem}\lb{Theorem2}{\rm[From Eq.(FD) to Eq.(FPL)]} Let
$B^{\vep}_{-1}({\bf v}-{\bf v}_*,\og)$ be given by (\ref{ker1}), (\ref{ker2}), (\ref{ker3})
with $\ld=-1$.
Let $\vep_0>0, 0\le f_0^{\vep}\in L^1({\mR}_{\ge 0}, (1+x)\sqrt{x}{\rm d}x)
\cap L^{\infty}({\mR}_{\ge 0})$  satisfy
\be f^{\vep}_0\le \vep^{-3},\quad \sup_{0<\vep\le \vep_0}
\int_{{\mathbb R}_{\ge 0}}f^{\vep}_0(x)(1+x+
|\log f^{\vep}_0(x)|)\sqrt{x}{\rm d}x<\infty.\lb{1.FD}\ee
Let
 $f^{\vep}(t,x)$ with the initial data $f_0^{\vep}(x)$ be isotropic weak solutions of the Eq.({\rm FD})
satisfying the entropy equality (\ref{entropy-identity}) with $\ld=-1$ and
$f^{\vep}=f^{\vep}(t, |{\bf v}|^2/2)$. Then for any sequence
${\cal E}=\{\vep_n\}_{n=1}^{\infty}\subset(0,\vep_0]$ satisfying $\vep_n\to 0\,(n\to\infty)$,
there exist a subsequence, still denote it as
${\cal E}$, and an isotropic weak solution
$f \in L^{\infty}([0, \infty), L^1({\mR}_{\ge 0}, (1+x)\sqrt{x}{\rm d}x))$ of Eq.({\rm FPL}),
such that $f^{\vep}(t,\cdot)\rightharpoonup f(t,\cdot)$ weakly in $L^1({\mR}_{\ge 0}, \sqrt{x}{\rm d}x)$
(as ${\cal E}\ni \vep\to 0$) for every $t\in [0,\infty)$.  More precisely we have for all $t\in [0,\infty)$
\beas&&\lim_{{\cal E}\ni\vep\to 0}\int_{{\mR}_{\ge 0}}\vp(x)f^{\vep}(t,x)\sqrt{x}\,{\rm d}x
=\int_{{\mR}_{\ge 0}}\vp(x)f(t,x)\sqrt{x}\,{\rm d}x\qquad \forall\,\vp\in L^{\infty}({\mR}_{\ge 0});
\\
&& \lim_{{\cal E}\ni \vep\to 0}\int_{{\mR}_{\ge 0}^2}{\cal J}_{-1}^{\vep}[\vp](y,z)
f^{\vep}(t,y)f^{\vep}(t,z)\sqrt{yz}
{\rm d}y{\rm d}z=\int_{{\mR}_{\ge 0}^2}L[\vp](y,z)f(t,y)f(t,z)\sqrt{yz}
{\rm d}y{\rm d}z, \\
&&
\lim_{\vep\to 0^{+}}\vep^3\int_{{\mR}_{\ge 0}^3}{\cal K}_{-1}^{\vep}[\vp](x,y,z)
f^{\vep}(t,x)f^{\vep}(t,y)f^{\vep}(t,z)\sqrt{xyz}
{\rm d}x{\rm d}y{\rm d}z=0\eeas
for all $\vp\in C_c^2({\mR}_{\ge 0}).$
\end{theorem}
\begin{remark}{\rm To prove the zero limit of the cubic terms with respect to $\vep\to 0^+$, our key observation lies in the symmetric structure inside the cubic terms (see the proof of Lemma \ref{Lemma4.3} ). Such a symmetric property is  also observed  when we consider the asymptotic expansion for the same limit (see \cite{HLP2}).}
\end{remark}

 {\it {\bf Case 2}: From  Eq.(BE) to Eq.(FPL).}
For Eq.(BE) (i.e. the case $\ld=+1$), so far we could only consider measure-valued weak solutions.
In fact, global in time solutions of Eq.(BE) with the collision kernel
satisfying the above conditions (\ref{ker1}), (\ref{ker2}), (\ref{ker3}) are known to exist only for measure-valued isotropic
weak solutions. (For the existence of local in time anisotropic solutions, see [10]). This is a consequence of
the dynamical Bose condensation for which one may have mass concentration in finite time (at least for some type of collision kernels).
It is also the physical counterpart of the fact that the entropy control is not enough to guarantee the
$L^1$-weak compactness of the solution even for the case of high temperature.

Let us introduce measure spaces used for defining the measure-valued
isotropic weak solutions of Eq.(BE) and Eq.(FPL).

\noindent $\bullet$   Let ${\cal B}({\mathbb R}_{\ge 0})$ be the class of signed real Borel measures $F$ on ${\mathbb R}_{\ge 0}$ satisfying
$\int_{{\mathbb R}_{\ge 0}}{\rm d}|F|(x)<\infty$ where
$|F|$ is the total variation of $F$. For any $k\ge 0$ let
\beas&&{\cal B}_k({\mathbb R}_{\ge 0})=
\Big\{ F\in {\cal B}({\mathbb R}_{\ge 0})\,\,\big|\,\,\|F\|_k:=
\int_{{\mathbb R}_{\ge 0}}(1+x)^k{\rm d}|F|(x)< \infty\Big\},\\
&&{\cal B}_k^{+}({\mathbb R}_{\ge 0})=
\{ F\in{\cal B}_k({\mathbb R}_{\ge 0}) \,\,|\, \, F\ge 0\}, \quad
{\cal B}^{+}({\mathbb R}_{\ge 0})={\cal B}^{+}_0({\mathbb R}_{\ge 0}).\eeas
Let $F\otimes F, F\otimes F\otimes F$ be product measures for $F\in {\cal B}^{+}({\mathbb R}_{\ge 0})$. To shorten notation, we denote
${\rm d}^2F={\rm d}(F\otimes F), {\rm d}^3F={\rm d}(F\otimes F\otimes F)$, i.e.
${\rm d}^2F(x,y)={\rm d}F(x){\rm d}F(y), {\rm d}^3F(x,y,z)={\rm d}F(x){\rm d}F(y){\rm d}F(z).$

\begin{definition}\lb{measBE}  Given $\vep=2\pi\hbar>0$, let
$B^{\vep}_{+1}({\bf v}-{\bf v}_*,\og)$ be given by (\ref{ker1}), (\ref{ker2}), (\ref{ker3}) with $\ld=+1$.
Let $\{F_t^{\vep}\}_{t\ge 0}\subset {\cal B}^{+}_1({\mathbb R}_{\ge 0})$.
We say that $\{F_t^{\vep}\}_{t\ge 0}$, or simply $F_t^{\vep}$, is a measure-valued
isotropic weak solution
of Eq.(BE) if $F_t^{\vep}$ satisfies the following {\rm (i),(ii)}:

{\rm (i)} $F_t^{\vep}$ conserves the mass and energy:
$\int_{{\mathbb R}_{\ge 0}}(1,x){\rm d}F_t^{\vep}(x)=\int_{{\mathbb R}_{\ge 0}}
(1,x){\rm d}F_0^{\vep}(x)$ for all $t\ge 0 .$

{\rm (ii)} For any $\vp\in C_c^2({\mathbb R}_{\ge 0})$,
the function $t\mapsto \int_{{\mathbb R}_{\ge 0}}\vp{\rm d}F_t^{\vep}$ belongs to
$C^1([0,\infty))$ and
\be\fr{{\rm d}}{{\rm d}t}\int_{{\mR}_{\ge 0}}\vp{\rm d}F_t^{\vep}
=\int_{{\mR}_{\ge 0}^2}{\cal J}_{+1}^{\vep}[\vp]{\rm d}^2F_t^{\vep}
+\vep^3\int_{{\mR}_{\ge 0}^3}{\cal K}_{+1}^{\vep}[\vp]
{\rm d}^3F_t^{\vep}
\qquad \forall\, t\in[0,\infty).\lb{weak2}\ee
\end{definition}

\begin{definition}\lb{measoluFPL}
Let $\{F_t\}_{t\ge 0}\subset {\cal B}_{1}^{+}({\mathbb R}_{\ge 0})$.
We say that $\{F_t\}_{t\ge 0}$, or simply $F_t$, is a  measure-valued isotropic weak solution
of Eq.(FPL) if $F_t$ satisfies the following {\rm (i),(ii)}:

{\rm (i)} $F_t$ conserves the mass and energy:
$\int_{{\mathbb R}_{\ge 0}}(1,x){\rm d}F_t(x)=\int_{{\mathbb R}_{\ge 0}}(1,x){\rm d}F_0(x)$
for all $t\ge 0.$

{\rm (ii)} For any $\vp\in C_c^2({\mathbb R}_{\ge 0})$,
$t\mapsto \int_{{\mathbb R}_{\ge 0}}\vp{\rm d}F_t$ belongs to
$C^1([0,\infty))$  and
\be\fr{{\rm d}}{{\rm d}t}\int_{{\mathbb R}_{\ge 0}}\vp(x){\rm
d}F_t(x)= \int_{{\mR}_{\ge 0}^2}L[\vp](x,y){\rm d}F_t(x){\rm d}F_t(y),\quad t\in[0,\infty)
\lb{L-measure}\ee
where $L[\vp]$ is defined in
(\ref{Lker1}).
\end{definition}

\begin{remark}\lb{Remark1.5}{\rm Our test functions
$\psi({\bf v}),\vp(x)$ are chosen to be independent of time-variable $t$.
This simplified version of weak solutions is equivalent to the general version.
For instance if $f$ is an $H$-solution of Eq.(FPL) in Definition \ref{weakFPL}, then by Proposition \ref{Prop6.1} in Appendix, we have,
 for any $\psi\in C^2_c([0,\infty)\times {\bR})$,
\beas\fr{{\rm d}}{{\rm d}t}\int_{{\mathbb R}^3}\psi(t,{\bf v}) f(t,{\bf v}){\rm d}{\bf v}
=\int_{{\mathbb R}^3}f(t,{\bf v}){\p}_t\psi(t,{\bf v}){\rm d}{\bf v}+
\int_{{\mathbb R}^3}\psi(t,{\bf v})Q_{L}(f)(t,{\bf v})
{\rm d}{\bf v}\eeas
for all $t\in [0,\infty).$}
\end{remark}

\begin{theorem}\lb{Theorem3}{\rm [From Eq.(BE) to Eq.(FPL)]}  Let
$B^{\vep}_{+1}({\bf v}-{\bf v}_*,\og)$ be given by (\ref{ker1}), (\ref{ker2}), (\ref{ker3})
with $\ld=+1$.
Let $\vep_0>0, F^{\vep}_0\in {\cal B}_1^{+}({\mR}_{\ge 0})$ satisfy
$$\sup_{0<\vep\le \vep_0}\int_{{\mR}_{\ge 0}}(1+x){\rm d}F_0^{\vep}(x)<\infty.$$
Let $F^{\vep}_t\in {\cal B}_1^{+}({\mR}_{\ge 0})$ be
 measure-valued isotropic weak solution of Eq.(BE)
with the initial data $F_0^{\vep}$.
  Then for any sequence ${\cal E}=\{\vep_n\}_{n=1}^{\infty}
\subset (0,\vep_0]$ satisfying $\vep_n\to 0\,(n\to\infty)$
there exist a subsequence, still denote it as
${\cal E}$, and  a
measure-valued isotropic weak solution
$F_t$ of Eq.({\rm FPL}) with the same initial datum $F_0$, such that
$F^{\vep}_t\rightharpoonup F_t$ weakly in measure space
(as ${\cal E}\ni\vep\to 0$) for all $t\ge 0$.
More precisely, for every $t\in
[0,\infty)$,
\beas&& \lim_{{\cal E}\ni\vep\to 0}\int_{{\mR}_{\ge 0}}\vp(x)F^{\vep}_t(x)
=\int_{{\mR}_{\ge 0}}\vp(x){\rm d}F_t(x)\qquad \forall\, \vp\in C_c({\mR}_{\ge 0});\\
&&\lim_{{\cal E}\ni\vep\to 0}\int_{{\mR}_{\ge 0}^2}{\cal J}_{+1}^{\vep}[\vp](y,z)
{\rm d}F^{\vep}_t(y){\rm d}F^{\vep}_t(z)
 =\int_{{\mR}_{\ge 0}^2}L[\vp](y,z){\rm d}F_t(y){\rm d}F_t(z),\\
&& \lim_{\vep\to 0^{+}}\vep^3\int_{{\mR}_{\ge 0}^3}{\cal K}_{+1}^{\vep}[\vp]
(x,y,z)
{\rm d}F^{\vep}_t(x)
{\rm d}F^{\vep}_t(y)
{\rm d}F^{\vep}_t(z)=0\eeas
for all $\vp\in C_c^2({\mR}_{\ge 0})$.
\end{theorem}

 \begin{remark}{\rm For a
 given isotropic initial datum
$0\le f_0\in L^1_2\cap L\log L({\mR}^3)$, $f_0=f_0(|{\bf v}|^2/2)$,
we have two solutions of Eq.(FPL) with the same initial datum $f_0$. One is the weak limit of $f^{\vep}$ of Eq.(FD) in Theorem \ref{Theorem2}, which is denote by $f(t,\cdot)$. Another one comes from the weak limit of $F^{\vep}_t$ of Eq.(BE)
in Theorem \ref{Theorem3} which is a measure-valued solution $F_t\in {\cal B}^{+}_{1}({\mR}_{\ge 0})$ with the initial datum $F_0$
defined by ${\rm d}F_0(x)=f_0(x)\sqrt{x}{\rm d}x$. An interesting question is to ask  whether  $F_t$ coincides with $f(t,\cdot)$, namely
$$
{\rm d}F_t(x)=f(t,x)\sqrt{x}{\rm d}x\quad \forall\, t\in(0,\infty)\quad ?
$$
This is of course closely related to the uniqueness problem of weak solutions. We believe that the answer will be yes since there is no reason for conceiving concentrations in the FPL dynamics. But we have no idea how to prove it at the moment. } \end{remark}

Finally let us mention  the existence results for the initial value problems associated to the models Eq.(MB), Eq.(FD) and Eq.(BE), which are the cornerstone for the semi-classical limit.

\begin{theorem}\lb{existencetheorem}{\rm[Existence \cite{CCL},\cite{Lu2001},\cite{Lu2004}]}.  Given $\vep=2\pi\hbar>0$.

(1)\,  Let
$B^{\vep}_{0}({\bf v}-{\bf v}_*,\og)$ be given by (\ref{ker1}),(\ref{ker2}) with $\ld=0$. Then
for any $0\le f_0^{\vep}\in L^1_2\cap L\log L({\mR}^3)$, there exists a weak solution
$f^{\vep}$ of Eq.(MB) with the initial datum $f^{\vep}|_{t=0}=f_0^{\vep}$.

(2)\, Let
$B^{\vep}_{-1}({\bf v}-{\bf v}_*,\og)$ be given by (\ref{ker1}),(\ref{ker2}),(\ref{ker3})
with $\ld=-1$. Then for any $f_0^{\vep}\in L^1_2({\mR}^3)$ satisfy $0\le f_0^{\vep}\le \vep^{-3}$ on ${\mR}^3$,
there exists a weak solution $f^{\vep}$ of
Eq.({\rm FD}) with the initial datum $f^{\vep}|_{t=0}=f_0^{\vep}$,
and $f^{\vep}$ satisfies the entropy identity (\ref{entropy-identity}) with $\ld=-1$.
Besides if the initial datum $f_0^{\vep}$ is isotropic, i.e.
$f_0^{\vep}=f_0^{\vep}(|{\bf v}|^2/2)$, then the corresponding solution $f^{\vep}$ is also
isotropic: $f^{\vep}=f^{\vep}(t,|{\bf v}|^2/2)$ for all $t\ge 0$.

(3)\,  Let
$B^{\vep}_{+1}({\bf v}-{\bf v}_*,\og)$ be given by (\ref{ker1}),(\ref{ker2}),(\ref{ker3})
with $\ld=+1$. Then for any $F_0^{\vep}\in {\cal B}_1^{+}({\mR}_{\ge 0})$,
there exists a measure-valued isotropic weak solution $F_t^{\vep}$ of
Eq.(BE)  with the initial datum
$F_t^{\vep}|_{t=0}=F_0^{\vep}$.
\end{theorem}

\begin{remark}{\rm Part (1) of the above theorem can be proved  by Proposition 1.2 in \cite{CCL}
 because   (\ref{1.BB}) is almost equivalent to
the inequality (1.6) in \cite{CCL} with $N=3, \gm=-3$.

Part (2) was essentially proved in \cite{Lu2001}. In fact,
as mentioned in Section 1, the assumption \eqref{ker3} implies that
$\int_{\bS}
B^{\vep}_{-1}({\bf v}-{\bf v}_*,\og){\rm d}\og$ is bounded in ${\bf v-v}_*$.
Thus from the proof in \cite{Lu2001} one sees that
the Eq.(FD) admits a unique mild solution $f^{\vep}\in L^{\infty}([0,\infty), L^1_2\cap L^{\infty}({\bR}))$
with the initial datum $f_0^{\vep}$ such that $0\le f^{\vep}\le \vep^{-3}$
and $f^{\vep}$ conserves the mass, momentum, and energy, and satisfies
the entropy identity (\ref{entropy-identity}) with $\ld=-1$. Moreover if $f_0^{\vep}$ is isotropic, then it is easily proved that the solution $f^{\vep}$ is also isotropic.
The only difference from the classical version is that here the collision kernel $B^{\vep}_{-1}(|{\bf v}-{\bf v}_*|,\cos(\theta))$ can not be written as the product form (e.g.) $b(\cos\theta)|{\bf v}-{\bf v}_*|^{\gm}$.
But this is not a problem for the existence result.

Part (3) has been proven in \cite{Lu2004}. In fact
if we write
$B^{\vep}_{\ld}({\bf v}-{\bf v}_*,\og)=B^{\vep}_{\ld}(|{\bf v-v}_*|,\cos(\theta))
$
where $ \theta=\arccos(|{\bf n}\cdot\og|)$, ${\bf n}=({\bf v-v}_*)/|{\bf v-v}_*|$, then  from the assumption (\ref{ker3}) we see that
the function $(V,\tau)\mapsto B^{\vep}_{\ld}(V,\tau)$ is continuous on
${\mR}_{\ge 0}\times [0,1]$ and $B^{\vep}_{\ld}(V,\tau)
\le C_{\phi,\vep}V\tau$ for all $(V,\tau)\in {\mR}_{\ge 0}\times [0,1]$, where
$C_{\phi,\vep}=\fr{2}{\vep^4}\|\widehat{\phi}\|_{\infty}^2<\infty.$
Thus according to \cite{Lu2004} and the Appendix in \cite{CaiLu}, we get the existence in the theorem.
Note also that in \cite{Lu2004} it is also assumed that
$\int_{0}^{1}B^{\vep}_{\ld}(V,\tau){\rm d}\tau$ is strictly positive for all $V>0$, but this is only used to determine the uniqueness of equilibrium. Since in this paper we do not consider things about
equilibrium, there is no problem for the truth of part (3).
} \end{remark}

\subsection{Organization of the paper}  In Section 2, we prove some preliminary estimates.
 Section 3 is devoted to the proof of
  the convergence from Eq.(MB) to Eq.(FPL) for general initial data.
In Section 4, we prove crucial lemmas for isotropic solutions of Eq.(FD) and Eq.(BE).
Section 5 is devoted to the proof of the convergence from Eq.(FD) to Eq.(FPL) and from Eq.(BE) to Eq.(FPL).
Some general and technical parts including those about the weak projection gradient are put in the Appendix.

 \section{Preliminary lemmas for weak convergence from Eq.(MB) to Eq.(FDL).}

In this section, we provide some technical lemmas which will be used to prove the convergence of weak solutions of Eq.(MB) to $H$-solutions of Eq.(FPL).
 \bigskip

 \subsection{Change of variables} \label{COV} To adapt the definition of  weak projection gradient into the proof of the main result,  in this subsection, we will focus on  the change of variables in the estimates which plays the essential role in catching the entropy production.

 \subsubsection{Change of spherical variable} We first introduce a change of variable for the unit vector
 $\og\in {\bS}$:
\be \og=\cos(\theta){\bf n}+\sin(\theta)\sg,
\quad \theta\in [0,\pi],\quad  \sg\in {\mS}^{1}({\bf n})
\lb{2.OG}\ee where ${\bf n}\in {\bS}$ is any given vector and
$${\mS}^{1}({\bf n})=\{\sg\in {\bS}\,\,|\,\,\sg\,\bot\, {\bf n}\,\}.$$
Then accordingly we have,  for any integrable or nonnegative Borel measurable function $\Psi$ on ${\bS}$,
that
\be \int_{{\bS}}\Psi(\og){\rm d}\og=\int_{0}^{\pi}\sin\theta \bigg(\int_{{\mS}^{1}({\bf
n})}\Psi\big(\cos(\theta) {\bf n}+\sin(\theta)\sg\big){\rm d}\sg\bigg) {\rm d}\theta\,.\lb{2.Sym1}\ee
Notice that we also have
\be\int_{{\mS}^{1}({\bf
n})} \psi(\sg){\rm d}\sg
=\int_{{\mS}^{1}({\bf
n})} \psi(-\sg){\rm d}\sg.\lb{2.Reflect}\ee
Using (\ref{2.OG}),
we may rewrite the $\og$-representation (\ref{1.Omega}) as follows:
\be
\left\{\begin{array}
{ll}
\displaystyle
{\bf v}'=\sin^2(\theta){\bf v}
+\cos^2(\theta){\bf v}_*-|{\bf v}-{\bf v}_*|\cos(\theta)\sin(\theta)\sg,\\   \displaystyle
 {\bf v}'_*=\cos^2(\theta){\bf v}
+\sin^2(\theta){\bf v}_*+|{\bf v}-{\bf v}_*|\cos(\theta)\sin(\theta)\sg,
\\   \displaystyle
\theta\in[0,\pi],\,\,\,\sg\in{\mS}^1({\bf n}),\,\,\,{\bf n}=({\bf v-v}_*)/|{\bf v-v}_*|.
\end{array}\right. \lb{2.VV}\ee
From this and (\ref{2.Reflect}) it is easily seen that
for any integrable or nonnegative Borel measurable function $\Psi$ satisfying $\Psi({\bf v},{\bf v}_*) \equiv \Psi({\bf v}_*, {\bf v})$
on ${\bRR}$, we have
 \be\int_{{\mS}^1({\bf n})}\Psi({\bf v}',{\bf v}_*')
\big|_{\og=\sin(\theta){\bf n}+\cos(\theta)\sg}{\rm d}\sg
=\int_{{\mS}^1({\bf n})}\Psi({\bf v}',{\bf v}_*')
\big|_{\og=\cos(\theta){\bf n}+\sin(\theta)\sg}{\rm d}\sg\quad \forall\, \theta\in[0,\pi/2].
\lb{2.Sym}\ee
\subsubsection{Change of variables in the integrals}
We will present two key lemmas on the change of variables.  To do that,
  we introduce the following
variables defined for every $r>0, {\bf z}\in {\bR}\setminus\{{\bf 0}\}$ and $\sg\in {\bS}$:
\bes&&{\bf z}^*_{r}({\bf z}, \sg)\equiv {\bf z}_{r}({\bf z}, -\sg); \lb{2.HK1}\\
&&{\bf z}_{r}({\bf z}, \sg)={\bf z},\quad \sg_{r}({\bf z}, \sg)=\sg\quad
{\rm for}\,\,\, |{\bf z}|<r,  \lb{2.HK2}\\
&&
{\bf z}_{r}({\bf z}, \sg)=
\big(1-2(r/|{\bf z}|)^2\big){\bf z}+2r\sqrt{1-
(r/|{\bf z}|)^2}\sg\quad {\rm for}\,\,\,|{\bf z}|\ge r,  \lb{2.HK3}\\
&&\sg_{r}({\bf z}, \sg)=
2(r/|{\bf z}|)\sqrt{1-(r/|{\bf z}|)^2}\,{\bf n}+
\big(1-2(r/|{\bf z}|)^2\big)\sg
\quad {\rm for}\,\,\,|{\bf z}|\ge r \lb{2.HK4}\ees
with ${\bf n}={\bf z}/|{\bf z}|$.
It is  easy to check that for all $r>0, {\bf z}\in {\bR}\setminus\{{\bf 0}\}, \sg\in {\mS}^1({\bf n})$,
\be |{\bf z}_{r}({\bf z},\sg)|=|{\bf z}|,\quad
|{\bf z}_{r}({\bf z},\sg)-{\bf z}|=2r1_{\{|{\bf z}|\ge r\}},\quad
|\sg_{r}({\bf z}, \sg)-\sg|=\fr{2r}{|{\bf z}|}1_{\{|{\bf z}|\ge r\}}.\lb{2.HK5}\ee
 One may see that   \eqref{2.HK1}-\eqref{2.HK4} exactly come from the change of variables in the associated integrals:

\begin{lemma}\lb{Prop2.3}
Let $\Psi({\bf w},{\bf w}_*, {\bf v}, {\bf v}_*)\ge 0$ be a Borel measurable functions on $ {\bRR}\times {\bRR}$. Suppose that  $\Psi$ is symmetric with respective to the first two variables:
$\Psi({\bf w},{\bf w}_*, {\bf v}, {\bf v}_*)=\Psi({\bf w}_*,{\bf w}, {\bf v}, {\bf v}_*).$
Then, for any $\vep>0$,
\bes\lb{2.14}&&\int_{{\bRRS}}B^{\vep}_0({\bf v}-{\bf v}_*,\og)
\Psi({\bf v}', {\bf v}_*', {\bf v}, {\bf v}_*)\big|_{\og-{\rm rep.}}{\rm d}\og
{\rm d}{\bf v}{\rm d}{\bf v}_*
 \\
&&=\fr{1}{\pi}\int_{{\bRR}}\int_{0}^{\infty}\fr{1_{\{|{\bf z}|\ge\vep r\}}}{|{\bf z}|(2\vep r)^2}
\varrho(r)
\int_{{\mS}^1({\bf n})}\Psi\Big(\fr{{\bf w}+ {\bf z}_{\vep r}}{2},
\fr{{\bf w}- {\bf z}_{\vep r}}{2},\fr{{\bf w+z}}{2},\fr{{\bf w-z}}{2}\Big)
{\rm d}\sg
{\rm d}r
{\rm d}{\bf z}{\rm d}{\bf w}\nonumber\ees
where ${\bf n}={\bf z}/|{\bf z}|,$ ${\bf z}_{\vep r}={\bf z}_{\vep r}({\bf z},\sg)$ is defined in
(\ref{2.HK2})-(\ref{2.HK3}), and, according to (\ref{ker2}), the function
\be \lb{varrho}  \varrho(r):=2\pi r^3
|\widehat{\phi}(r)|^2\qquad  \mbox{satisfies}\,\,\quad \int_{0}^{\infty}\varrho(r){\rm d}r=1.\ee
Furthermore, if inaddition that $\Psi({\bf v},{\bf v}_*, {\bf v}, {\bf v}_*)\equiv 0$, then
\bes\lb{2.14*}&&\int_{{\bRRS}}B^{\vep}_0({\bf v}-{\bf v}_*,\og)
\Psi({\bf v}', {\bf v}_*', {\bf v}, {\bf v}_*)\big|_{\og-{\rm rep.}}{\rm d}\og
{\rm d}{\bf v}{\rm d}{\bf v}_*
\\
&&=\fr{1}{\pi}\int_{{\bRR}}\int_{0}^{\infty}\fr{1}{|{\bf z}|(2\vep r)^2}
\varrho(r)
\int_{{\mS}^1({\bf n})}\Psi\Big(\fr{{\bf w}+ {\bf z}_{\vep r}}{2},
\fr{{\bf w}- {\bf z}_{\vep r}}{2},\fr{{\bf w+z}}{2},\fr{{\bf w-z}}{2}\Big)
{\rm d}\sg
{\rm d}r
{\rm d}{\bf z}{\rm d}{\bf w}. \nonumber \ees
\end{lemma}

\begin{proof}  Thanks to (\ref{1.New}), (\ref{1.RR}),  and the assumption on $\Psi$, and
using change of variables
$({\bf v},{\bf v}_*)=\fr{1}{2}({\bf w}+{\bf z},{\bf w}-{\bf z})$, we have
\beas&& {\rm the \,\,l.h.s.\,\, of\,\,(\ref{2.14})}
=\int_{{\bRRS}}\fr{| {\bf v}-{\bf v}_*|}{\vep^4}
\widehat{\phi}\Big(\fr{|{\bf v}-{\bf v}'|}{\vep}\Big)^2\Psi({\bf v}', {\bf v}_*', {\bf v}, {\bf v}_*)\big|_{\sg-{\rm rep.}}{\rm d}\sg
{\rm d}{\bf v}{\rm d}{\bf v}_*\\
&&=\fr{1}{8}\int_{{\bRR}}\bigg(\underbrace{\int_{{\bS}}
\fr{|{\bf z}|}{\vep^4}
\widehat{\phi}\Big(\fr{|{\bf z}|}{\vep}\sqrt{\fr{1}{2}(1-{\bf n}\cdot\sg)}\,\Big)^2
\Psi\Big(
\fr{{\bf w}+|{\bf z}|\sg}{2},
\fr{{\bf w}-|{\bf z}|\sg}{2}, \fr{{\bf w}+{\bf z}}{2}, \fr{{\bf w-z}}{2}\Big){\rm d}\sg}_{\mathcal{I}}
\bigg){\rm d}{\bf w}
 {\rm d}{\bf z},\eeas
 where ${\bf n}={\bf z}/|{\bf z}|$.
For the inner integral $\mathcal{I}$,  we have with \eqref{2.Sym1} that
\beas&&\mathcal{I}
=\int_{0}^{\pi}
\sin(\theta)
\fr{|{\bf z}|}{\vep^4}
\widehat{\phi}\Big(\fr{|{\bf z}|}{\vep}\sin(\theta/2)\Big)^2\\
&&\times
\int_{{\mS}^1({\bf n})}
\Psi\Big(\fr{{\bf w}+ |{\bf z}|(\cos(\theta){\bf n}+\sin(\theta)\sg)}{2},
\fr{{\bf w}- |{\bf z}|(\cos(\theta){\bf n}+\sin(\theta)\sg)}{2},
\fr{{\bf w+z}}{2},\fr{{\bf w-z}}{2}\Big)
{\rm d}\sg{\rm d}\theta
\\
&&=4\int_{0}^{1}s
\fr{|{\bf z}|}{\vep^4}
\widehat{\phi}\Big(\fr{|{\bf z}|}{\vep}s\Big)^2\\
&&\times
\int_{{\mS}^1({\bf n})} \Psi\Big(\fr{{\bf w}+ (1-2s^2){\bf z}+2|{\bf z}|\sqrt{1-s^2}s\sg}{2},
\fr{{\bf w}-((1-2s^2){\bf z}+2|{\bf z}|\sqrt{1-s^2}s\sg)}{2}
,\fr{{\bf w+z}}{2},\fr{{\bf w-z}}{2}\Big)
{\rm d}\sg
{\rm d}s
\\
&&=4\int_{0}^{\fr{|{\bf z}|}{\vep}}\fr{1}{(\vep r)^2}r^3
\widehat{\phi}(r)^2\fr{1}{|{\bf z}|}
\int_{{\mS}^1({\bf n})} \Psi\Big(\fr{{\bf w}+{\bf z}_{\vep r}}{2},\fr{{\bf w}-{\bf z}_{\vep r}}{2}
, \fr{{\bf w+z}}{2},\fr{{\bf w- z}}{2}\Big)
{\rm d}\sg{\rm d}r\\
&& =\fr{8}{\pi}\int_{0}^{\infty}\fr{1_{\{|{\bf z}|\ge\vep r \}}}{(2\vep r)^2}
\varrho(r)\fr{1}{|{\bf z}|}
\int_{{\mS}^1({\bf n})} \Psi\Big(\fr{{\bf w}+ {\bf z}_{\vep r}}{2},
\fr{{\bf w}- {\bf z}_{\vep r}}{2},\fr{{\bf w+z}}{2},\fr{{\bf w-z}}{2}\Big)
{\rm d}\sg
{\rm d}r.\eeas
Thus (\ref{2.14}) holds true. The equality (\ref{2.14*}) is due to the
fact that if $|{\bf z}|<\vep r$ then ${\bf z}_{\vep r}={\bf z}$ so that the integrand is zero.
\end{proof}

\begin{lemma}\lb{Prop2.2}
Let $\Psi({\bf z},\sg, {\bf u}), \psi(r, {\bf z})$ be Borel measurable functions on ${\bRS}\times {\bR}$
and $(0,\infty)\times{\bR}$
respectively, and they are either nonnegative or such that the following integrals are absolutely convergent for all $r>0$. Then for almost every
$r> 0$,
\bes&&\int_{{\bR}}\int_{{\mS}^1({\bf n})}
\Psi({\bf z},\sg, {\bf z}_r)
{\rm d}\sg{\rm d}{\bf z}=\int_{{\bR}}
\int_{{\mS}^1({\bf n})} \Psi({\bf z}^*_{r},\sg_r, {\bf z}){\rm d}\sg{\rm d}{\bf z},\lb{2.13}\\&&
\int_{{\bR}}
\int_{{\mS}^1({\bf n})}
\psi(|{\bf z}|, {\bf z}_r)
{\rm d}\sg{\rm d}{\bf z}=2\pi \int_{{\bR}}\psi(|{\bf z}|,{\bf z}){\rm d}{\bf z},\lb{2.13R}\ees
where  ${\bf n}={\bf z}/|{\bf z}|$ and ${\bf z}_{r}={\bf z}_{r}({\bf z}, \sg),
{\bf z}^*_{r}={\bf z}^*_{r}({\bf z}, \sg), \sg_{r}=\sg_{r}({\bf z}, \sg)$ are defined in (\ref{2.HK1})-(\ref{2.HK4}).
\end{lemma}

\begin{proof}
 The second equality (\ref{2.13R}) follows from (\ref{2.13})
and the identity $|{\bf z}^*_r|
=|{\bf z}|.$ So we need only to prove the first equality (\ref{2.13}).

Without loss of generality we may assume that $\Psi$ is nonnegative.
By monotone approximation $\Psi_n({\bf z},\sg,{\bf u})=(\Psi({\bf z}, \sg, {\bf u})
\wedge n)e^{-(|{\bf z}|+|{\bf u}|)/n}$, we may also assume that
the integrals in (\ref{2.13}) are bounded in $r\in (0,\infty)$.
For any Borel measurable function $\vp\in L^1([0,\infty))$, we have by
spherical coordinate transform
 (with ${\bf n}={\bf z}/|{\bf z}|$) that
\beas I &:= & \int_{{\bR}}|{\bf z}|
\int_{{\bS}}\vp\Big(|{\bf z}|\sqrt{\fr{1-{\bf n}\cdot\og}{2}}\Big)\Psi\Big({\bf z},
\fr{\og-({\bf n}\cdot\og) {\bf n}}{\sqrt{1-({\bf n}\cdot\og)^2}},|{\bf z}|\og \Big){\rm d}\og{\rm d}{\bf z}\\
&=&\int_{{\bR}}|{\bf z}|
\int_{{\bS}}\vp\Big(|{\bf z}|\sqrt{\fr{1- \og\cdot{\bf n}}{2}}\Big)\Psi\Big(|{\bf z}|\og,
\fr{{\bf n}-(\og\cdot{\bf n}) \og}{\sqrt{1-(\og\cdot{\bf n})^2}}, {\bf z}\Big)
{\rm d}\og{\rm d}{\bf z}.\eeas
Then using \eqref{2.Sym1} and change of variables $ t=\fr{r}{|{\bf z}|}$ and $\sg\to -\sg$, we
compute
\beas
I&=&\int_{{\bR}}|{\bf z}|
\int_{0}^{\pi}\sin(\theta)\vp(|{\bf z}|\sin(\theta/2))
\int_{{\mS}^1({\bf n})}\Psi\Big(|{\bf z}|\big(\cos(\theta){\bf n}+\sin(\theta)\sg\big),
{\bf n}\sin(\theta)-\cos(\theta)\sg,{\bf z}\Big)
{\rm d}\sg{\rm d}\theta{\rm d}{\bf z}
\\
&=&4\int_{{\bR}}|{\bf z}|
\int_{0}^{1}t\vp(|{\bf z}|t)
\int_{{\mS}^1({\bf n})}\Psi\Big(|{\bf z}|\big((1-2t^2){\bf n}+2t\sqrt{1-t^2}\sg\big),
{\bf n}2t\sqrt{1-t^2}-(1-2t^2)\sg,{\bf z}\Big)
{\rm d}\sg{\rm d}t{\rm d}{\bf z}\\
&=&4\int_{0}^{\infty}\vp(r)\bigg(\int_{{\bR}}1_{\{|{\bf z}|\ge r\}}
\int_{{\mS}^1({\bf n})} \Psi({\bf z}^*_r,\sg_r,{\bf z}){\rm d}\sg{\rm d}{\bf z}\bigg){\rm d}r.\eeas
On the other hand using (\ref{2.Sym1}) we also have
\beas I&=&\int_{{\bR}}|{\bf z}|
\int_{0}^{\pi}\sin(\theta)\vp(|{\bf z}|\sin(\theta/2))\int_{{\mS}^1({\bf n})}
  \Psi\big({\bf z},\sg, |{\bf z}|(\cos(\theta){\bf n}+\sin(\theta)\sg) \big)
 {\rm d}\sg{\rm d}\theta {\rm d}{\bf z}
\\
&=&4\int_{{\bR}}|{\bf z}|
\int_{0}^{1}t\vp(|{\bf z}|t)
\int_{{\mS}^1({\bf n})} \Psi\Big({\bf z},\sg, |{\bf z}|\big((1-2t^2){\bf n}+2t\sqrt{1-t^2}\sg\big)\Big)
{\rm d}\sg{\rm d}t{\rm d}{\bf z}
\\
&=&4\int_{0}^{\infty}\vp(r)\bigg( \int_{{\bR}}1_{\{|{\bf z}|\ge r\}}
\int_{{\mS}^1({\bf n})} \Psi({\bf z},\sg,{\bf z}_r)
{\rm d}\sg{\rm d}{\bf z}\bigg){\rm d}r.\eeas
Since $\vp\in L^1([0,\infty))$ is arbitrary, this implies that
 $$
 \int_{{\bR}}1_{\{|{\bf z}|\ge r\}}
\int_{{\mS}^1({\bf n})} \Psi({\bf z},\sg,{\bf z}_r)
{\rm d}\sg{\rm d}{\bf z}= \int_{{\bR}}1_{\{|{\bf z}|\ge r\}}
\int_{{\mS}^1({\bf n})} \Psi({\bf z}^*_r,\sg_r,{\bf z}){\rm d}\sg{\rm d}{\bf z} $$
for almost every $r>0$. By definition of
${\bf z}_r, {\bf z}^*_r, \sg_r$ we have ${\bf z}_r= {\bf z}^*_r={\bf z}, \sg_r
=\sg$ for all $|{\bf z}|< r$, so this
gives the equality (\ref{2.13}).
\end{proof}

\subsection{Convergence of $L^{\vep}_\lambda[\psi]$} We begin with two lemmas.

\begin{lemma}\lb{Lemma2.2}Let $\psi\in C^2_b({\bR})$,
$\Dt\psi({\bf v}',{\bf v}_*',{\bf v}, {\bf v}_*)=\psi({\bf v}')
+\psi({\bf v}_*')
-\psi({\bf v})
-\psi({\bf v}_*),$
\beas&& \|D^2\psi\|_{*}:=\sup_{{\bf v}, {\bf v}_*\in {\bR}}
\fr{1}{(1+|{\bf v}|^2+|{\bf v}_*|^2)^{1/4}}\sup_{|{\bf u}|\le \sqrt{|{\bf v}|^2+|{\bf v}_*|^2}}
\big|D^2\psi({\bf u})-\fr{1}{3}{\rm Tr}(D^2\psi({\bf 0})) {\rm I}\big|,\\
&&
C_{\psi}^*({\bf v,v}_*):=
\min\big\{ \|D^2\psi\|_{\infty},\, \|D^2\psi\|_{*}
(1+|{\bf v}|^2+|{\bf v}_*|^2)^{1/4}\big\}\eeas
where ${\rm I}=(\dt_{ij})_{3\times 3}$ is the
unit matrix.
 Then for all ${\bf v}, {\bf v}_*\in{\bR}$ we have
$$|\Dt\psi({\bf v}',{\bf v}_*',{\bf v}, {\bf v}_*)|\le C_{\psi}^*({\bf v,v}_*)
|{\bf v}-{\bf v}_*|^2|\sin(\theta)\cos(\theta)|\quad \forall\, \og\in{\bS},$$
where $ \theta=\arccos({\bf n}\cdot\og)\in[0,\pi]$, ${\bf n}=({\bf v-v}_*)/|{\bf v-v}_*|$. Moreover,
with (\ref{2.VV}) we have
\beas&&
\bigg|\int_{{\mS}^{1}({\bf
n})}\Dt\psi({\bf v}',{\bf v}_*',{\bf v}, {\bf v}_*){\rm d}\sg\bigg|
\le
8\pi C_{\psi}^*({\bf v,v}_*)|{\bf v}-{\bf v}_*|^{2}\sin^2(\theta)\cos^2(\theta)
\quad \forall\, \theta\in[0,\pi].
\eeas
\end{lemma}

\begin{proof} Let
$\psi_{h}({\bf v})=\psi({\bf v})-h |{\bf v}|^2/2$ with $h\in {\mR}.$
Then
\bes&&\Dt\psi_{h}({\bf v}',{\bf v}_*',{\bf v}, {\bf v}_*)=\Dt\psi({\bf v}',{\bf v}_*',{\bf v}, {\bf v}_*),\quad\nabla \psi_{h}({\bf v})=\nabla\psi({\bf v})-h {\bf v},\nonumber
\\
&&
D^2\psi_{h}({\bf v})=D^2\psi ({\bf v})-h {\rm I}\label{D2psi}.\ees
By writing
$\Dt\psi_{h}=({\psi_{h}}'-\psi_{h})-({\psi_{h}}_*-{\psi_{h}}_*')$ and using ${\bf v}_*-{\bf v}_*'={\bf v}'-{\bf v}$, we derive that
\beas\Dt\psi_{h}({\bf v}',{\bf v}_*',{\bf v}, {\bf v}_*)
=\int_{0}^{1}\!\!\!\int_{0}^{1}({\bf v}-{\bf v}_*')^{\tau}D^2\psi_{h}(\xi_{t,s})
({\bf v}'-{\bf v}){\rm d}s{\rm d}t\eeas
where $
\xi_{t,s}={\bf v}_*'+t({\bf v}'-{\bf v})+s({\bf v}-{\bf v}_*'), |\xi_{t, s}|\le \max\{|{\bf v}|, |{\bf v}'|, |{\bf v}_*|,
|{\bf v}_*'|\}\le  \sqrt{|{\bf v}|^2+|{\bf v}_*|^2}\,.$
Since
$|{\bf v}_*'-{\bf v}||{\bf v}'-{\bf v}|
=|{\bf v}-{\bf v}_*|^2|\sin(\theta)|\cos(\theta)|$, this gives
\beas|\Dt\psi({\bf v}',{\bf v}_*',{\bf v}, {\bf v}_*)|\le
 \Big(\sup_{|{\bf u}|\le\sqrt{|{\bf v}|^2+|{\bf v}_*|^2} }
|D^2\psi_{h}({\bf u})|\Big)|{\bf v}-{\bf v}_*|^2|\sin(\theta)\cos(\theta)|.\eeas
From this and (\ref{D2psi}) one sees that the first inequality in the lemma follows from
taking $h=0$ and using the following inequality
\be \lb{2.DD}
\sup_{|{\bf u}|\le\sqrt{|{\bf v}|^2+|{\bf v}_*|^2} }
|D^2\psi_{h}({\bf u})|
\le \|D^2\psi\|_{*}(1+|{\bf v}|^2+|{\bf v}_*|^2)^{1/4}\quad {\rm with}\quad h=\fr{1}{3}{\rm Tr}(D^2\psi({\bf 0})).\ee

To prove the second inequality in the lemma,  we set $\og=\cos(\theta){\bf n}+\sin(\theta)\sg$ with $\theta\in[0,\pi]$ and $\sg\in{\mS}^1({\bf n})$.
By writing
$\Dt\psi_{h}=(\psi_{h}'-\psi_{h})+({\psi_{h}}_*'-{\psi_{h}}_*)$ and using
${\bf v}_*'-{\bf v}_*=-({\bf v}'
-{\bf v})$, we derive that
\beas \Dt\psi_{h}({\bf v}',{\bf v}_*',{\bf v}, {\bf v}_*) &= &
\big(\nabla\psi_{h} ({\bf v})-\nabla \psi_{h} ({\bf v}_*)\big)\cdot ({\bf v}'-{\bf v})\\
&+&\int_{0}^{1}(1-t)({\bf v}'-{\bf v})^{\tau}D^2\psi_{h}({\bf v}+t({\bf v}'-{\bf v}))
({\bf v}'-{\bf v}){\rm d}t\\
&+&\int_{0}^{1}(1-t)({\bf v}_*'-{\bf v}_*)^{\tau}D^2\psi_{h}({\bf v}_*+t({\bf v}_*'-{\bf v}_*))
({\bf v}_*'-{\bf v}_*){\rm d}t\,.
\eeas
Using (\ref{2.VV}) we have ${\bf v}'-{\bf v}=-({\bf v-v}_*)\cos^2(\theta)-|{\bf v-v}_*|\cos(\theta)\sin(\theta)\sg$. Since
$\int_{{\mS}^{1}({\bf n})}{\bf a}\cdot\sg {\rm d}\sg=0$, it follows that
\bes\label{trianglepsi}&& \int_{{\mS}^{1}({\bf
n})}\Dt\psi({\bf v}',{\bf v}_*',{\bf v}, {\bf v}_*){\rm d}\sg
=-2\pi\big(\nabla\psi_{h} ({\bf v})-\nabla \psi_{h} ({\bf v}_*)\big)
\cdot({\bf v}-{\bf v}_*)\cos^2(\theta)\\
&&+\int_{{\mS}^{1}({\bf
n})}\int_{0}^{1}(1-t)({\bf v}'-{\bf v})^{\tau}D^2\psi_{h}({\bf v}+t({\bf v}'-{\bf v}))
({\bf v}'-{\bf v}){\rm d}t  {\rm d}\sg\nonumber\\
&&+\int_{{\mS}^{1}({\bf
n})}\int_{0}^{1}(1-t)({\bf v}_*'-{\bf v}_*)^{\tau}
D^2\psi_{h}({\bf v}_*+t({\bf v}_*'-{\bf v}_*))({\bf v}_*'-{\bf v}_*){\rm d}t {\rm d}\sg
\nonumber\ees
and thus
\beas\bigg|\int_{{\mS}^{1}({\bf
n})}\Dt\psi({\bf v}',{\bf v}_*',{\bf v}, {\bf v}_*)
{\rm d}\sg
\bigg|\
\le 4\pi\Big(\sup_{|{\bf u}|\le \sqrt{|{\bf v}|^2+|{\bf v}_*|^2}}
|D^2\psi_{h}({\bf u})|\Big)|{\bf v}-{\bf v}_*|^2\cos^2(\theta).\eeas
Similar argument applied to the decomposition
$\Dt\psi_{h}=(\psi_{h}'-{\psi_{h}}_*)+({\psi_{h}}_*'-\psi_{h})$ also gives
\beas\bigg|\int_{{\mS}^{1}({\bf
n})}\Dt\psi({\bf v}',{\bf v}_*',{\bf v}, {\bf v}_*)
{\rm d}\sg
\Big|\
\le 4\pi\Big(\sup_{|{\bf u}|\le \sqrt{|{\bf v}|^2+|{\bf v}_*|^2}}
|D^2\psi_{h}({\bf u})|\Big)|{\bf v}-{\bf v}_*|^2\sin^2(\theta).\eeas
Since $\min\{\cos^2(\theta),\sin^2(\theta)\}
\le 2\cos^2(\theta)\sin^2(\theta)$, this together with (\ref{2.DD}) gives the second inequality.
\end{proof}

It should be noted that for an isotropic function
$\psi({\bf v})=\vp(|{\bf v}|^2/2)$ with  $\vp\in C^2_b({\mR}_{\ge 0})$,  we have
$$
D^2\psi({\bf v})-\fr{1}{3}{\rm Tr}(D^2\psi({\bf 0})){\rm I}=\vp''(|{\bf v}|^2/2){\bf v}\otimes {\bf v}+
[\vp'(|{\bf v}|^2/2)-\vp'(0)]{\rm I}$$
Using
the inequality $|\vp'(\rho)-\vp'(0)|
\le  \int_{0}^{\rho}|\vp''(r)|{\rm d}r
\le
\big(\sup\limits_{0\le r\le \rho}|\vp''(r)|r^{3/4}\big)4 \rho^{1/4}$ gives
\beas
\big|D^2\psi({\bf u})-\fr{1}{3}{\rm Tr}(D^2\psi({\bf 0})){\rm I}\big|
\le 8\Big(\sup_{r\ge 0 }|\vp''(r)|r^{3/4}\Big)|{\bf u}|^{1/2}
\eeas
hence, by definition of $\|D^2\psi\|_{*}$,
\be \|D^2\psi\|_{*}\le
8\sup_{r\ge 0 }|\vp''(r)|r^{3/4}.\lb{2.DE}\ee
This inequality will be used to prove the conservation of energy.
\vskip1mm

\begin{lemma}\lb{Lemma2.4*} Let $\widehat{\phi}$ satisfy (\ref{ker2}) and let
$$A^*_{\widehat{\phi}}(\vep)=\sup_{\rho\ge 0}(\rho^{1/2}\wedge 1)
\int_{\fr{\rho}{\vep}}^{\infty} r^3|
\widehat{\phi}(r)|^2{\rm d}r,\quad A^*_{\widehat{\phi},\alpha}(\vep)=\sup_{\rho> 0}(\rho^{1/2}\wedge 1)
\int_{0}^{\fr{\rho}{\vep}}\Big(\fr{\vep}{\rho}r\Big)^{\alpha} r^3|
\widehat{\phi}(r)|^2{\rm d}r$$
where $\alpha>0$ is a constant. Then
$0\le A^*_{\widehat{\phi}}(\vep), A^*_{\widehat{\phi},\alpha}(\vep)\le \fr{1}{2\pi}$ for all $\vep>0,$
\be \lim_{\vep\to 0^+}A^*_{\widehat{\phi}}(\vep)=0,\qquad
\lim_{\vep\to 0^+}A^*_{\widehat{\phi},\alpha}(\vep)=0\lb{4.AA1},\ee
and for all ${\bf z}\in {\bR}\setminus\{{\bf 0}\}$,
\be \lb{4.AA2} 0\le R_{\vep}(|{\bf z}|):=\fr{1}{2\pi|{\bf z}|^3}-\int_{0}^{\pi/2}
\fr{|{\bf z}|\cos^3(\theta)\sin(\theta)}{\vep^4}
\Big|
\widehat{\phi}\Big(\fr{|{\bf z}|\cos(\theta)}{\vep}\Big)\Big|^2{\rm d}\theta\le \fr{A_{\widehat{\phi}}^*(\vep)}{|{\bf z}|^3(|{\bf z}|^{1/2}\wedge 1)},\ee
\be \int_{0}^{\pi/2}
\fr{|{\bf z}|\cos^{3+\alpha}(\theta)\sin(\theta)}{\vep^4}
\Big|
\widehat{\phi}\Big(\fr{|{\bf z}|\cos(\theta)}{\vep}\Big)\Big|^2{\rm d}\theta
\le \fr{A_{\widehat{\phi},\alpha}^*(\vep)}{|{\bf z}|^3(|{\bf z}|^{1/2}\wedge 1)}.
\lb{4.AA3}\ee
\end{lemma}

\begin{proof} The convergence (\ref{4.AA1}) is easily proved by using the integrability
$\int_{0}^{\infty}r^3|
\widehat{\phi}(r)|^2{\rm d}r=\fr{1}{2\pi}$ and dominated convergence theorem.
The inequalities (\ref{4.AA2}) and (\ref{4.AA3}) follow from change of variables and the definition of
$A_{\widehat{\phi}}^*(\vep), A_{\widehat{\phi},\alpha}^*(\vep)$.
\end{proof}

Now we are in a position to prove the convergence of $L^{\vep}_{\ld}[\psi]$.
\begin{lemma}\lb{Lemma2.4}  Let
$B^{\vep}_{\ld}({\bf v}-{\bf v}_*,\og)$ be given by (\ref{ker1}), (\ref{ker2}) with $\ld
\in\{0,-1,+1\}$, and let $L[\psi]({\bf v,v}_*), L^{\vep}_{\ld}[\psi]({\bf v}, {\bf v}_*)$ be defined in (\ref{1.LL}), (\ref{1.LLL}) for $\psi\in C_c^2({\bR})$.
Then for  all ${\bf v}, {\bf v}_*\in {\bR}$ with ${\bf v}\neq {\bf v}_*$,
\bes &&
\lim_{\vep\to 0^+}L^{\vep}_{\ld}[\psi]({\bf v}, {\bf v}_*)=
\lim_{\vep\to 0^+}L^{\vep}_{0}[\psi]({\bf v}, {\bf v}_*)=L[\psi]({\bf v}, {\bf v}_*),
\lb{2.19}\\
&&
 \sup_{\vep>0}\{|L^{\vep}_{\ld}[\psi]({\bf v}, {\bf v}_*)|,
\,|L[\psi]({\bf v}, {\bf v}_*)|\}\le
16\|D^2\psi\|_{\infty}\fr{1}{|{\bf v-v}_*|}. \lb{2.20}\ees
\end{lemma}

\begin{proof} By definition of $B^{\vep}_{\ld}({\bf v}-{\bf v}_*,\og)$ and (\ref{1.LLL}) we have
\be L^{\vep}_{\ld}[\psi]({\bf v}, {\bf v}_*)
=L^{\vep}_{0}[\psi]({\bf v}, {\bf v}_*)+\ld E^{\vep}[\psi]({\bf v}, {\bf v}_*)\lb{2.LE}\ee
where
$$E^{\vep}[\psi]({\bf v}, {\bf v}_*)=\int_{{\mathbb S}^2}\fr{|({\bf v}-{\bf v}_*)\cdot \og|}{\vep^4}
\widehat{\phi}\Big(\fr{|{\bf v}'-{\bf v}|}{\vep}\Big)\widehat{\phi}
\Big(\fr{|{\bf v}_*'-{\bf v}|}{\vep}\Big)\Dt\psi({\bf v}',{\bf v}_*',{\bf v}, {\bf v}_*){\rm d}\og.$$
Let ${\bf z}={\bf v-v}_*$. Making change of variable (\ref{2.OG}) with ${\bf n}={\bf z}/|{\bf z}|$ and using (\ref{2.VV}), (\ref{2.Sym1}), and (\ref{2.Sym})
we have
\be \lb{2.L} L^{\vep}_0[\psi]({\bf v}, {\bf v}_*)
=2
\int_{0}^{\pi/2}
\fr{|{\bf z}|\cos(\theta)\sin(\theta)}{\vep^4}
\widehat{\phi}\Big(\fr{|{\bf z}|\cos(\theta)}{\vep}\Big)^2
{\rm d}\theta\int_{{\mS}^1({\bf n})}\Dt\psi({\bf v}',{\bf v}_*',{\bf v}, {\bf v}_*)
{\rm d}\sg,
\ee
\be
\lb{2.E} E^{\vep}[\psi]({\bf v}, {\bf v}_*) = 2\int_{0}^{\pi/2}
\fr{|{\bf z}|\cos(\theta)\sin(\theta)}{\vep^4}
\widehat{\phi}\Big(\fr{|{\bf z}|\cos(\theta)}{\vep}\Big)
\widehat{\phi}\Big(\fr{|{\bf z}|\sin(\theta)}{\vep}\Big){\rm d}\theta
\int_{{\mS}^1({\bf n})}\Dt\psi({\bf v}',{\bf v}_*',{\bf v}, {\bf v}_*)
{\rm d}\sg.\ee
From \eqref{trianglepsi} we have
\bes \lb{2.EEE} &&
\int_{{\mS}^1({\bf n})}\Dt\psi({\bf v}',{\bf v}_*',{\bf v}, {\bf v}_*)\\
&&=-2\pi \big(\nabla\psi ({\bf v})-\nabla \psi ({\bf v}_*)\big)\cdot {\bf z}\cos^2(\theta)
+|{\bf z}|^2\cos^2(\theta)\sin^2(\theta)
\int_{{\mS}^1({\bf n})}\sg^{\tau}\fr{D^2\psi({\bf v})
+D^2\psi({\bf v}_*)}{2}\sg{\rm d}\sg\nonumber\\
&&+|{\bf z}|^2\cos^2(\theta)\int_{{\mS}^1({\bf n})}\int_{0}^{1}(1-t)E_{\psi}(t,\theta,
\sg,{\bf v}, {\bf v}_*){\rm d}t{\rm d}\sg \nonumber\\
&&=-2\pi \big(\nabla\psi ({\bf v})-\nabla \psi ({\bf v}_*)\big)\cdot {\bf z}\cos^2(\theta)
+\pi |{\bf z}|^2\cos^2(\theta)\sin^2(\theta)\fr{1}{\pi}
\int_{{\mS}^1({\bf n})}\sg^{\tau} H_{\psi} \sg{\rm d}\sg\nonumber\\
&&+|{\bf z}|^2\cos^2(\theta)\int_{{\mS}^1({\bf n})}\int_{0}^{1}(1-t)E_{\psi}(t,\theta,
\sg,{\bf v}, {\bf v}_*){\rm d}t{\rm d}\sg\nonumber
\ees
where $H_{\psi}=\fr{1}{2}(D^2\psi({\bf v}))+D^2\psi({\bf v}_*))$ and $E_{\psi}(t,\theta,
\sg,{\bf v}, {\bf v}_*)$ satisfies
\beas\big|E_{\psi}(t,\theta, \sg,{\bf v}, {\bf v}_*)\big|
\le
6\|D^2\psi\|_{\infty}\cos(\theta)+
2\Lambda_{\psi}\big(|{\bf z}|\cos(\theta)\big)\sin^2(\theta)
\eeas
\beas{\rm with}\quad \Lambda_{\psi}(\dt)=\sup_{{\bf x,y}\in{\bR}, |{\bf x-y}|\le \dt}|D^2\psi({\bf x})
-D^2\psi({\bf y})|,\quad \dt\ge 0.\eeas
These imply that ( with $E_{\psi}=E_{\psi}(t,\theta,
\sg,{\bf v}, {\bf v}_*)$)
\beas &&L^{\vep}_0[\psi]({\bf v}, {\bf v}_*)
=2
\int_{0}^{\pi/2}
\fr{|{\bf z}|\cos(\theta)\sin(\theta)}{\vep^4}
\widehat{\phi}\Big(\fr{|{\bf z}|\cos(\theta)}{\vep}\Big)^2
\bigg\{-2\pi \big(\nabla\psi ({\bf v})-\nabla \psi ({\bf v}_*)\big)\cdot {\bf z}\cos^2(\theta)
\\&&+\pi |{\bf z}|^2\cos^2(\theta)\sin^2(\theta)\fr{1}{\pi}
\int_{{\mS}^1({\bf n})}\sg^{\tau} H_{\psi} \sg{\rm d}\sg
+|{\bf z}|^2\cos^2(\theta)\int_{{\mS}^1({\bf n})}\int_{0}^{1}(1-t)E_{\psi}{\rm d}t{\rm d}\sg\bigg\}{\rm d} \theta \eeas
Using  Lemma \ref{Prop6.3}, (\ref{ker2}),  and Lemma \ref{Lemma2.4*}  we have
\bes &&\fr{1}{\pi}\int_{{\mS}^1({\bf n})}\sg^{\tau} H_{\psi} \sg{\rm d}\sg=
{\rm Tr}(H_{\psi})-{\bf n}^{\tau}H_{\psi}{\bf n},\lb{2.Matr}\\&&\lim_{\vep\to 0^+}\int_{0}^{\pi/2}
\fr{|{\bf z}|\cos^3(\theta)\sin(\theta)}{\vep^4}
\widehat{\phi}\Big(\fr{|{\bf z}|\cos(\theta)}{\vep}\Big)^2{\rm d}\theta
= \fr{1}{2\pi|{\bf z}|^3},\lb{2.phi}\\
&&\lim_{\vep\to 0^+}
\int_{0}^{\pi/2}
\fr{|{\bf z}|\cos^{3+\alpha}(\theta)\sin(\theta)}{\vep^4}
\Big|
\widehat{\phi}\Big(\fr{|{\bf z}|\cos(\theta)}{\vep}\Big)\Big|^2{\rm d}\theta
=0,\qquad \forall\, \alpha>0, \lb{2.Zero} \\
&&\lim_{\vep\to 0^+}
\int_{0}^{\pi/2}\Lambda_{\psi}\big(|{\bf z}|\cos(\theta)\big)
\fr{|{\bf z}|\cos^{3}(\theta)\sin(\theta)}{\vep^4}
\Big|
\widehat{\phi}\Big(\fr{|{\bf z}|\cos(\theta)}{\vep}\Big)\Big|^2{\rm d}\theta
=0.\nonumber\ees
From these and writing
$\cos^2(\theta)\sin^2(\theta)= \cos^2(\theta)-\cos^4(\theta)$ we derive  that
\beas&&
\lim_{\vep\to 0^+}L^{\vep}_{0}[\psi]({\bf v}, {\bf v}_*)
=-2\fr{(\nabla\psi
({\bf v})-\nabla\psi({\bf v}_*))\cdot{\bf z}}{|{\bf z}|^3}
+\fr{1}{|{\bf z}|}\big({\rm Tr}(H_{\psi})-{\bf n}^{\tau}H_{\psi}{\bf n}\big)
\\
&&=\fr{1}{|{\bf v-v}_*|}
\Big( {\rm Tr}(H_{\psi})-{\bf n}^{\tau}H_{\psi}{\bf n}-
2\fr{(\nabla\psi
({\bf v})-\nabla\psi({\bf v}_*))\cdot{\bf n}}{|{\bf v-v}_*|}
\Big)
=L[\psi]({\bf v,v}_*).\eeas
Next we prove that
$\lim\limits_{\vep\to 0^+}E^{\vep}[\psi]({\bf v}, {\bf v}_*)=0$ so that (\ref{2.19})
holds true. In fact, by Lemma \ref{Lemma2.2}, we first have
$$|E^{\vep}[\psi]({\bf v}, {\bf v}_*)|\le 16\pi\|D^2\psi\|_{\infty}|{\bf z}|^2\int_{0}^{\pi/2}\fr{|{\bf z}|\cos^3(\theta)
\sin^3(\theta)}{\vep^4}
\Big|\widehat{\phi}\Big(\fr{|{\bf z}|\cos(\theta)}{\vep}\Big)
\widehat{\phi}\Big(\fr{|{\bf z}|\sin(\theta)}{\vep}\Big)\Big|
{\rm d}\theta.$$
Then writing $\cos^3(\theta)
\sin^3(\theta)=\cos^{5/2}(\theta)\sqrt{\sin(\theta)}\cdot
\sin^{5/2}(\theta)\sqrt{\cos(\theta)}$ and using Cauchy-Schwarz inequality and (\ref{2.Zero}) we obtain
\beas&&|E^{\vep}[\psi]({\bf v}, {\bf v}_*)|\ \le 16\pi \|D^2\psi\|_{\infty}|{\bf z}|^2\int_{0}^{\pi/2}
\fr{|{\bf z}|\cos^{5}(\theta)
\sin(\theta)}{\vep^4}\Big|\widehat{\phi}\Big(\fr{|{\bf z}|\cos(\theta)}{\vep}\Big)\Big|^2
{\rm d}\theta\to 0\quad {\rm as}\quad \vep\to 0^+.\eeas
Finally to prove  (\ref{2.20}), we use
 $2|ab|\le a^2+b^2$, (\ref{2.Sym}), the second inequality in Lemma \ref{Lemma2.2},
 and (\ref{ker2}) to obtain that for all $
{\bf v}, {\bf v}_*\in{\bR}$ with ${\bf z}={\bf v}-{\bf v}_*\neq {\bf 0}$,
\beas&&|L^{\vep}_{\ld}[\psi]({\bf v}, {\bf v}_*)|\le
|L^{\vep}_{0}[\psi]({\bf v}, {\bf v}_*)|+|\ld||E^{\vep}[\psi]({\bf v}, {\bf v}_*)|
\\
&&\le 8
\int_{0}^{\pi/2}
\fr{|{\bf z}|\cos(\theta)\sin(\theta)}{\vep^4}
\widehat{\phi}\Big(\fr{|{\bf z}|\cos(\theta)}{\vep}\Big)^2
{\rm d}\theta\bigg|\int_{{\mS}^1({\bf n})}\Dt\psi({\bf v}',{\bf v}_*',{\bf v}, {\bf v}_*)
{\rm d}\sg\bigg|
\\
&&
\le
64\pi \|D^2\psi\|_{\infty}|{\bf z}|^{2}
\int_{0}^{\pi/2}
\fr{|{\bf z}|\cos^3(\theta)\sin^3(\theta)}{\vep^4}
\widehat{\phi}\Big(\fr{|{\bf z}|\cos(\theta)}{\vep}\Big)^2
{\rm d}\theta
\le 16\|D^2\psi\|_{\infty}\fr{1}{|{\bf z}|}.\eeas
This also gives
$ |L[\psi]({\bf v}, {\bf v}_*)|=\lim\limits_{\vep\to 0^+}|L^{\vep}_{\ld}[\psi]({\bf v}, {\bf v}_*)|\le
16\|D^2\psi\|_{\infty}\fr{1}{|{\bf v}-{\bf v}_*|}.$
\end{proof}

\subsection{A general lemma on weak convergence.}  The following lemma are often used
in dealing with weak convergence and it can be easily proved by induction on the number $k$ of factors.

\begin{lemma}\lb{Lemma2.5}(1)
Let $X\subset {\mR}^{d}$ be a Borel set, $\mu\ge 0$ a regular Borel
measure on $X$, and let
$f^n, f \in L^1(X, {\rm d}\mu)$ satisfy $f^n\rightharpoonup f \,(n\to\infty)$ weakly in
$L^1(X, {\rm d}\mu)$. Let
$\psi_n, \psi$ be Borel measurable functions on $X$ satisfying
$$\sup_{n\ge 1,{\bf x}\in X}
|\psi_n({\bf x})|<\infty,\quad \lim_{n\to\infty}\psi_n({\bf x})
=\psi({\bf x}),\quad \mu -{\rm a.e.}\quad  {\bf x}\in X.$$
Then
$$\lim_{n\to\infty}\int_{X}\psi_n({\bf x})f^n({\bf x})\mu({\bf x})
=\int_{X}\psi({\bf x})f({\bf x}){\rm d}\mu({\bf x}).$$

(2) Let $k, d_j\in{\mN}$,
and let $X_j\subset {\mR}^{d_j}$ be a Borel set, $\mu_j\ge 0$ a regular Borel
measure on $X_j$, and let
$f_j^n, f_j \in L^1(X_j, {\rm d}\mu_j)$ satisfy $f^n_j\rightharpoonup f_j \,(n\to\infty)$ weakly in
$L^1(X_j, {\rm d}\mu_j), j=1,2,..., k$. Let
$X=X_1\times X_2\times \cdots\times X_k$, $\mu=\mu_1\otimes\mu_2\otimes \cdots\otimes \mu_k$ (product measure), and let
$\Psi_n, \Psi$ be Borel measurable functions on $X$ satisfying
$$\sup_{n\ge 1,{\bf x}\in X}
|\Psi_n({\bf x})|<\infty,\quad \lim_{n\to\infty}\Psi_n({\bf x})
=\Psi({\bf x}),\quad \mu -{\rm a.e.}\quad  {\bf x}=({\bf x}_1, {\bf x}_2, ..., {\bf x}_k)\in X.$$
Then
$$\lim_{n\to\infty}\int_{X}\Psi_n({\bf x})f_1^n({\bf x}_1)f_2^n({\bf x}_2)
\cdots f_k^n({\bf x}_k){\rm d}\mu({\bf x})
=\int_{X}\Psi({\bf x})f_1({\bf x}_1)f_2({\bf x}_2)
\cdots f_k({\bf x}_k){\rm d}\mu({\bf x}).$$
\end{lemma}

 \section{Proof of Theorem \ref{Theorem1}}

In this section we prove Theorem \ref{Theorem1} (from Eq.(MB) to Eq.(FPL)). In the proof,
the entropy inequality and entropy dissipation $D_0^\vep(f^{\vep}(t))$ of Eq.(MB) will be sufficiently used to prove that the weak limit $f(t,{\bf v})$ of
$f^{\vep}(t,{\bf v})$ satisfies that the function $(t,{\bf v},{\bf v}_*)\mapsto \sqrt{f(t,{\bf v}) f(t,{\bf v}_*)/|{\bf v-v}_*|}$ has the weak projection gradient in ${\bf v}-{\bf v}_*\neq {\bf 0}$.
\medskip

\begin{proof}[Proof of Theorem \ref{Theorem1} (From Eq.(MB) to Eq.(FPL))] The proof consists of several steps. We first recall that $f^{\vep}\in L^{\infty}([0, \infty),
L^1_2\cap L\log L({\mR}^3))$ with the initial data $f_0^{\vep}\in L^1_2\cap L\log L({\mR}^3)$
are weak solutions of Eq.({\rm MB}) given in the theorem.

In the following, we assume without loss of generality that $0<\vep_0\le 1$, and we denote by $C_0$ the positive and finite constants that depend only on
the initial bounds $
 \sup\limits_{0<\vep\le \vep_0}\|f_0^{\vep}\|_{L^1_2}$ and $\sup\limits_{0<\vep\le \vep_0}\|f^{\vep}_0\log f^{\vep}_0\|_{L^1}$, and
 $C_0$ may have different value in different lines.
 As is well-known, we first have
\bes&& \sup_{0<\vep\le \vep_0, t\ge 0}\int_{{\mathbb R}^3}f^{\vep}(t,{\bf v})(1+|{\bf v}|^2+
|\log f^{\vep}(t,{\bf v})|){\rm d}{\bf v}\le C_0,\lb{3.2}\\
&& \sup_{0<\vep\le \vep_0}\int_{0}^{\infty}D_{0}^{\vep}(f^{\vep}(t)){\rm d}t\le C_0  \lb{3.3}.\ees
Given any sequence ${\cal E}=\{\vep_n\}_{n=1}^{\infty}\subset
(0, \vep_0]$ satisfying $\vep_n\to 0\,(n\to\infty)$.
\smallskip

{\bf Step 1.} We prove that
there exists a subsequence of ${\cal E}$, still
denote it by ${\cal E}$, and a function $f\in
L^{\infty}([0, \infty),L^1_2\cap L\log L({\mR}^3))$, such that \be
f^{\vep}(t,\cdot) \rightharpoonup f(t,\cdot)\quad (\,{\rm as}\,\,\,
{\cal E}\ni \vep\to 0\,)\quad {\rm weakly\,\,in}\,\,\, L^1({\mR}^3)\qquad
\forall\, t\ge 0\lb{3.weakConv}\ee
and $f$ conserves the mass, momentum, and energy.

We first prove that for any $t\in [0,\infty)$, any $\psi\in C_c^2({\mR}^3)$, and any
 measurable function $\zeta(r)$ on
${\mR}_{\ge 0}$ satisfying
$0\le \zeta(r)\le 1$ on $r\in [0,\infty)$,
\bes\lb{3.7}&&
\fr{1}{4}\int_{{\bRRS}}\zeta(|{\bf v-v}_*|)B_{0}^{\vep}({\bf v}- {\bf v}_*,\og)
\big|{f^{\vep}}'{f_*^{\vep}}'-{f^{\vep}}{f_*^{\vep}}\big|\big|\psi+\psi_*-\psi'-\psi_*'\big|
{\rm d}\og{\rm d}{\bf v}{\rm d}{\bf v}_*\\
&&\le\|D^2\psi\|_{*}
\bigg(\int_{{\mathbb R}^3\times {\mathbb R}^3}\zeta (|{\bf v}-{\bf v}_*|)
f^{\vep}f^{\vep}_*(1+|{\bf v}|^2+|{\bf v}_*|^2)^{1/2}|{\bf v}-{\bf v}_*|{\rm d}{\bf v}{\rm d}{\bf v}_*
\bigg)^{1/2}\sqrt{D_{0}^{\vep}(f^{\vep}(t))}.\nonumber \ees
In fact, using the inequality
$|a-b|\le\fr{1}{\sqrt{2}}\sqrt{a+b}
\sqrt{\Gm(a,b)}\,(a,b\ge 0)$ where $\Gm(a,b)$ is given in (\ref{1.Gamma}), we have
$$\big|{f^{\vep}}'{f_*^{\vep}}'-{f^{\vep}}{f_*^{\vep}}\big|\le\fr{1}{\sqrt{2}}\sqrt{
{f^{\vep}}'{f_*^{\vep}}'+f^{\vep}f_*^{\vep}}
\sqrt{\Gm\big({f^{\vep}}'{f_*^{\vep}}',\,f^{\vep}f_*^{\vep}\big)}.$$
By  using  the first inequality in Lemma \ref{Lemma2.2} and
Cauchy-Schwarz inequality  we deduce that
\beas{\rm the\,\, l.h.s.\,\, of\,\, (\ref{3.7})}
&\le &\fr{\|D^2\psi\|_{*}}{2}
\bigg(\int_{{\mathbb R}^3\times {\mathbb R}^3}\zeta (|{\bf v}-{\bf v}_*|)
f^{\vep}f^{\vep}_*(1+|{\bf v}|^2+|{\bf v}_*|^2)^{1/2}|{\bf v}-{\bf v}_*|^4\\
&\times& \int_{{\mathbb S}^2}
B_{0}^{\vep}({\bf v}-{\bf v}_*,\og)
\sin^2(\theta)\cos^2(\theta){\rm d}\og {\rm d}{\bf v}{\rm d}{\bf v}_*
\bigg)^{1/2}\sqrt{D_{0}^{\vep}(f^{\vep}(t))}\eeas
where $\theta=\arccos({\bf n}\cdot\og)$, ${\bf n}=({\bf v-v}_*)/|{\bf v-v}_*|$.
Then using (\ref{1.BB}) gives (\ref{3.7}).

Since $0\le \zeta(\cdot)\le 1$,
$(1+|{\bf v}|^2+|{\bf v}_*|^2)^{1/2}|{\bf v-v}_*|\le
(1+|{\bf v}|^2)(1+ |{\bf v}_*|^2)$, and
$f^{\vep}(t,{\bf v})$ conserves the mass and energy, it follows that
\beas\int_{{\mathbb R}^3\times {\mathbb R}^3}
\zeta (|{\bf v}-{\bf v}_*|)
f^{\vep}f^{\vep}_*(1+|{\bf v}|^2+|{\bf v}_*|^2)^{1/2}|{\bf v}-{\bf v}_*|
{\rm d}{\bf v}{\rm d}{\bf v}_*\le
\|f^{\vep}(t)\|_{L^1_2}^2=
\|f_0^{\vep}\|_{L^1_2}^2.\eeas
In particular for the case  $\zeta(r)\equiv 1$
we obtain from (\ref{3.7}) that
\bes\lb{3.EE}&&
\fr{1}{4}\int_{{\bRRS}}B_{0}^{\vep}({\bf v}- {\bf v}_*,\og)
\big|{f^{\vep}}'{f_*^{\vep}}'-{f^{\vep}}{f_*^{\vep}}\big|\big|\psi+\psi_*-\psi'-\psi_*'\big|
{\rm d}\og{\rm d}{\bf v}{\rm d}{\bf v}_*
\\
&&\le \|D^2\psi\|_{*}\|f_0^{\vep}\|_{L^1_2}\sqrt{D_{0}^{\vep}(f^{\vep}(t))}\nonumber\ees
and so for all $t_1, t_2\ge 0$
\bes\lb{3.EEE}&& \sup_{0<\vep\le \vep_0}\bigg|
\int_{{\mathbb R}^3}\psi({\bf v})f^{\vep}(t_1,{\bf v}){\rm d}{\bf v}-
\int_{{\mathbb R}^3}\psi({\bf v})f^{\vep}(t_2,{\bf v}){\rm d}{\bf v}\bigg|\\
&&\le \|D^2\psi\|_{*}\sup_{0<\vep\le \vep_0}
\|f^{\vep}_0\|_{L^1_2}\int_{t_1\wedge t_2}^{t_1\vee
t_2}\sqrt{D_{0}^{\vep}(f^{\vep}(t))}{\rm d}t \le
C_0\|D^2\psi\|_{*}\sqrt{|t_2-t_1|}.\nonumber\ees
From this uniform estimate and standard smooth
approximation, it is easily proved that for any $\psi\in
L^{\infty}({\mR}^3)$, \be \sup_{0<\vep\le \vep_0, |t_1-t_2|\le
\dt}\bigg|\int_{{\mR}^3}\psi({\bf v}) f^{\vep}(t_1,{\bf v}){\rm
d}{\bf v}-\int_{{\mR}^3}\psi({\bf v}) f^{\vep}(t_2,{\bf v}){\rm
d}{\bf v}\bigg| \to 0\quad (\dt\to 0).\lb{3.4}\ee
Now from (\ref{3.2}), (\ref{3.4}) and  Dunford-Petties criterion of
$L^1$-weakly relative compactness, we conclude that there exist a subsequence of ${\cal E}=\{\vep_n\}_{n=1}^{\infty}$, still
denote it by ${\cal E}$,  and a function
$f\in L^{\infty}([0, \infty),L^1_2\cap L\log L({\mR}^3))$, such that (\ref{3.weakConv})
hold true.

It is easily seen that the weak liming function $f$ conserves that mass and momentum.
To prove the conservation of energy,
we consider a smooth cutoff approximation: let $\zeta\in C^{\infty}_c({\mR}_{\ge 0})$ satisfies
$0\le \zeta\le 1 $ on ${\mR}_{\ge 0}$ and
$\zeta(r)=1$ for $r\in[0,1]$, $\zeta(r)=0$ for $r\ge 2$, and let
$$\psi_{\dt}({\bf v})=\vp_{\dt}(|{\bf v}|^2/2),\qquad
\vp_{\dt}(r)=r\zeta(\dt r),\qquad \dt>0.$$
Then $\psi_{\dt}\in C^{2}_c({\bR}), 0\le \psi_{\dt}({\bf v})
\le |{\bf v}|^2/2$, and  $\lim\limits_{\dt\to 0^+}\psi_{\dt}({\bf v})=
|{\bf v}|^2/2$.
Compute
\beas|\vp_{\dt}''(r)|r^{3/4}=
\big|2\zeta'(\dt r)(\dt r)^{3/4}+(\dt r)^{7/4}
\zeta''(\dt r)\big|\dt^{1/4}\le C\dt^{1/4}\eeas
with $C=\sup\limits_{r\ge 0}\big|2\zeta'(r)r^{3/4}+r^{7/4}
\zeta''(r)\big|.$ Using (\ref{2.DE}) gives $\|D^2\psi_{\dt}\|_{*}
\le 8\sup\limits_{r\ge 0}|\vp_{\dt}''(r)|r^{3/4}\le 8C\dt^{1/4}$ and so
by weak convergence and (\ref{3.EEE}) we have
\beas&& \bigg|
\int_{{\mathbb R}^3}\psi_{\dt}({\bf v})f(t,{\bf v}){\rm d}{\bf v}-
\int_{{\mathbb R}^3}\psi_{\dt}({\bf v})f_0({\bf v}){\rm d}{\bf v}\bigg|=
\lim_{{\cal E}\ni\vep\to 0}\bigg|
\int_{{\mathbb R}^3}\psi_{\dt}({\bf v})f^{\vep}(t,{\bf v}){\rm d}{\bf v}-
\int_{{\mathbb R}^3}\psi_{\dt}({\bf v})f^{\vep}_0({\bf v}){\rm d}{\bf v}\bigg|
\\
&&\le
C_0\|D^2\psi_{\dt}\|_{*}\sqrt{t}\le C_0\dt^{1/4}\sqrt{t}
\qquad \forall\, t\ge 0,\,\forall\, \dt>0.\eeas
Letting $\dt\to 0^+$ we obtain using dominated convergence theorem that
$$\int_{{\mathbb R}^3}\fr{|{\bf v}|^2}{2}f(t,{\bf v}){\rm d}{\bf v}-
\int_{{\mathbb R}^3}\fr{|{\bf v}|^2}{2}f_0({\bf v}){\rm d}{\bf v}
=\lim_{\dt\to 0^+}\bigg(\int_{{\mathbb R}^3}\psi_{\dt}({\bf v})f(t,{\bf v}){\rm d}{\bf v}-
\int_{{\mathbb R}^3}\psi_{\dt}({\bf v})f_0({\bf v}){\rm d}{\bf v}\bigg)=0$$
for all $t\in [0,\infty)$. So $f$ also conserves energy.

\vskip2mm
 {\bf Step 2.}
In what follows, the function $f(t,{\bf v})$  is always the $L^1$-weak
limit of $f^{\vep}(t,{\bf v})\,({\cal E}\ni \vep\to 0)$ obtained in Step 1.
  Let
\beas&& F(t,{\bf v,v}_*)=f(t,{\bf v})f(t,{\bf v}_*)\fr{1}{|{\bf v-v}_*|},\\
&&
F_{\dt}(t,{\bf v,v}_*)=f(t,{\bf v})f(t,{\bf v}_*)\fr{1}{|{\bf v-v}_*|} 1_{\{|{\bf v-v}_*|\ge \dt\}}+\dt e^{-(|{\bf
v}|^2+|{\bf v}_*|^2)},\\
&&
F^{\vep}_{\dt}(t,{\bf v,v}_*)=f^{\vep}(t,{\bf v})f^{\vep}(t,{\bf v}_*)\fr{1}{|{\bf v-v}_*|}1_{\{|{\bf v-v}_*|\ge \dt\}}+
\dt e^{-(|{\bf
v}|^2+|{\bf v}_*|^2)},\quad \dt>0.\eeas
By change of variables ${\bf v}=\fr{{\bf w}+{\bf z}}{2} $ and $
 {\bf v}_*=\fr{{\bf w}-{\bf z}}{2}$, we also denote
\bes\label{FbarF} \bar{F}(t,{\bf w},{\bf z})=F\big(t,\fr{{\bf w}+{\bf z}}{2},\fr{{\bf w}-{\bf z}}{2}\big),\quad
\bar{F}^{\vep}_{\dt}(t,{\bf w},{\bf z})=F^{\vep}_{\dt}\big(t,\fr{{\bf w}+{\bf z}}{2},\fr{{\bf w}-{\bf z}}{2}\big).\ees
Due to the $L^1$- weak convergence proved in Step 1, it is easily seen that
for any $\vp=\vp({\bf w}, {\bf z})\in L^{\infty}({\bRR})$ and any $\psi=\psi(t,{\bf w}, {\bf z})\in L^{\infty}([0,\infty)\times {\bRR})$,
\bes\lb{3.W1}&&\lim_{{\cal E}\ni \vep\to 0}\int_{{\bRR}}\vp
\bar{F}^{\vep}_{\dt}{\rm d}{\bf z}{\rm d}{\bf w}
=\int_{{\bRR}}\vp \bar{F}_{\dt}{\rm
d}{\bf z}{\rm d}{\bf w}\qquad \forall\, t\in[0,\infty), \ees
\bes\lb{3.W2}&&
\lim_{{\cal E}\ni \vep\to 0}\int_{0}^{\infty}e^{-t}\int_{{\bRR}}\psi\bar{F}^{\vep}_{\dt}{\rm d}{\bf z}{\rm d}{\bf
w}{\rm d}t=\int_{0}^{\infty}e^{-t}\int_{{\bRR}}\psi
\bar{F}_{\dt}{\rm d}{\bf z}{\rm d}{\bf w}{\rm d}t. \ees
For any $0<T<\infty$, using the above convergence to the function
$1_{[0,T]}(t) e^t  \psi(t, {\bf w}, {\bf
z})$ we have
\bes\lb{3.W2T}&&
\lim_{{\cal E}\ni \vep\to 0}\int_{0}^{T}\int_{{\bRR}}\psi\bar{F}^{\vep}_{\dt}{\rm d}{\bf z}{\rm d}{\bf
w}{\rm d}t =\int_{0}^{T}\int_{{\bRR}}\psi
\bar{F}_{\dt}{\rm d}{\bf z}{\rm d}{\bf w}{\rm d}t
\ees
for all $\psi\in L^{\infty}([0, T]\times {\bRR}).$

Our goal of this step is to prove the following weak
convergence: for any $\dt>0$ and any bounded Borel measurable function
$\psi(t,r,{\bf w,z},\sg)$ on  $[0, \infty)\times[0,\infty)\times {\bRRS}$ we have
\bes\lb{3.WF}&&\lim_{{\cal E}\ni \vep\to 0} \int_{0}^{\infty}e^{-t}\int_{0}^{\infty}\varrho(r)\int_{{\bRR}}\int_{{\mS}^1({\bf n})}\psi(t,r, {\bf w},{\bf z},\sg)
\bar{F}^{\vep}_{\dt}(t,{\bf w},{\bf z}_{\vep r}){\rm d}\sg
{\rm d}{\bf z}{\rm d}{\bf w}{\rm d}r {\rm d}t
\\
&&= \int_{0}^{\infty}e^{-t}\int_{{\bRR}}\bigg(\int_{0}^{\infty}\varrho(r)\int_{{\mS}^1({\bf n})}
\psi(t,r,{\bf w},{\bf z},\sg){\rm d}\sg{\rm d}r\bigg)
\bar{F}_{\dt}(t,{\bf w},{\bf z})
{\rm d}{\bf z}{\rm d}{\bf w}{\rm d}t\nonumber
\ees
where the function $\varrho(r)$ is given in (\ref{varrho}).

Before going on, we first note that if (\ref{3.WF}) holds true, then
for any $0<T<\infty$ and
any bounded Borel measurable function
$\psi(t,r,{\bf w,z},\sg)$ on  $[0, T]\times[0,\infty)\times {\bRRS}$,
using (\ref{3.WF}) to the bounded function
$1_{0,T]}(t) e^{t}\psi(t,r, {\bf w},{\bf z},\sg)$ gives the convergence
\bes\lb{3.WFT}&&\lim_{{\cal E}\ni \vep\to 0} \int_{0}^{T}\int_{0}^{\infty}\varrho(r)\int_{{\bRR}}\int_{{\mS}^1({\bf n})}\psi(t,r, {\bf w},{\bf z},\sg)
\bar{F}^{\vep}_{\dt}(t,{\bf w},{\bf z}_{\vep r}){\rm d}\sg
{\rm d}{\bf z}{\rm d}{\bf w}{\rm d}r {\rm d}t
\\
&&= \int_{0}^{T}\int_{{\bRR}}\bigg(\int_{0}^{\infty}\varrho(r)\int_{{\mS}^1({\bf n})}
\psi(t,r,{\bf w},{\bf z},\sg){\rm d}\sg{\rm d}r\bigg)
\bar{F}_{\dt}(t,{\bf w},{\bf z})
{\rm d}{\bf z}{\rm d}{\bf w}{\rm d}t.\nonumber
\ees

To prove (\ref{3.WF}) and for notational convergence, we will use
a Borel measure space
 $(X,{\mathscr B},\mu)$ where
\be X=[0, \infty)\times [0, \infty)\times{\bRRS}, \lb{3.X}\ee
$\mu$ is defined for every Borel set $E\subset X$ by
\be\mu(E):=
\int_{0}^{\infty}e^{-t}\int_{0}^{\infty}\varrho(r)
\int_{{\bRR}}\bigg(\int_{{\mS}^1({\bf n})}1_{E}(t,r,{\bf w},{\bf z},\sg)
{\rm d}\sg\bigg){\rm d}{\bf z}{\rm d}{\bf w}{\rm d}r{\rm d}t\lb{3.Measure}\ee
with ${\bf n}={\bf z}/|{\bf z}|$, and we denote ${\rm d}\mu={\rm d}\mu(t,r,{\bf w},{\bf z},\sg)$.
Note: one may use the complete measure space $(X,{\mathscr M},\bar{\mu})$ which is the completion of
$(X,{\mathscr B},\mu)$. Since, as is well-known, every ${\mathscr M}$-measurable function can be
modified on a Borel set of $\mu$-measure zero such that after the modification the function is
Borel measurable, we can assume that all functions appeared are Borel measurable. In other words,
 when concerning measurability and integrability, we do not distinguish between
a function and its modification on a set of measure zero.
By changing variable $\og=\cos(\theta){\bf n}+\sin(\theta)\sg$, we see that for any nonnegative or bounded Borel measurable function $\vp$ on
${\bS}$, it holds \be
  \int_{{\mS}^1({\bf n})} \vp(\sg){\rm d}\sg=\fr{1}{2}\int_{{\bS}}
  \vp\Big(\fr{\og-({\bf n}\cdot\og) {\bf n}}{\sqrt{1-({\bf n}\cdot\og)^2}}\Big){\rm d}\og \qquad \forall\, {\bf n}\in {\bS}\lb{3.*}\ee
where for the cases $\og=\pm {\bf n}$ we define for instance $\fr{\og-({\bf n}\cdot\og) {\bf n}}{\sqrt{1-({\bf n}\cdot\og)^2}}
=(1,0,0)^{\tau}$.  The equality (\ref{3.*}) shows that $\mu$ is really a Borel measure on
the product set $X$.
Also from $0\le \varrho \in L^1([0,\infty))$ one sees that $\mu(K)<\infty$ for any compact set
$K\subset X$ and thus by measure theory (see e.g. W. Rudin \cite{W.Rudin}, Chapter 2),
$\mu$ is a regular Borel measure.
This regularity allows us to use Lusin's theorem on continuous approximation and the criterion of $L^1$-weak compactness.

It is easy to see that for any $\mu$-integrable or nonnegative Borel measurable functions $F(t,r, {\bf w}, {\bf z},\sg),\\
 G(t,{\bf w},{\bf z})$ on $X$,
\beas&&\int_{X}F(t,r,{\bf w},{\bf z},\sg)
{\rm d}\mu
=\int_{0}^{\infty}e^{-t}\int_{0}^{\infty}\varrho(r)\int_{{\bRR}}\int_{{\mS}^1({\bf n})}F(t,r,{\bf w},{\bf z},\sg)
{\rm d}\sg {\rm d}{\bf z}{\rm d}{\bf w}{\rm d}r{\rm d}t,\\
\\
&&\int_{X}
G(t,{\bf w},{\bf z})
{\rm d}\mu
=2\pi\int_{0}^{\infty}e^{-t}\int_{{\bRR}}G(t,{\bf w},{\bf z}){\rm d}{\bf z}{\rm d}{\bf w}{\rm d}t
.\eeas
By using Lemma \ref{Prop2.2}, we derive that for any  Borel measurable
function $\vp\ge 0$ on $X\times {\bR}$,
\be
\int_{X}\vp(t,r,{\bf w}, {\bf
z},\sg,{\bf z}_{\vep
r}){\rm d}\mu=
\int_{X}\vp(t,r,{\bf w}, {\bf
z}_{\vep r}^*,\sg_{\vep r},{\bf z}){\rm d}\mu
.\lb{3.Exchange}\ee
And for any Borel set $E\subset X$, the set
$$E^{\vep}:=\{(t,r,{\bf w}, {\bf z}, \sg)\in X\,\,|\,\,
(t,r,{\bf w}, {\bf z}^*_{\vep r}, \sg_{\vep r})\in E\},\quad \vep>0$$
 is also a Borel set in $X$ and
\be \mu(E^{\vep})=\mu(E)\qquad \forall\, \vep>0.\lb{3.MZeta}\ee
In fact, the definition of
${\bf z}^*_{\vep r}, \sg_{\vep r}$ implies that $E^{\vep}$ is a Borel set, and the equality (\ref{3.MZeta})
follows from
$1_{E}(t,r,{\bf w},{\bf z}^*_{\vep r},\sg_{\vep r})=
1_{E^{\vep}}(t,r,{\bf w},{\bf z},\sg)$
and the formula (\ref{3.Exchange}).

Now for any bounded  Borel measurable function $\psi(t,r,{\bf w,z},\sg)$
on $X$,
applying $L^1$-weak convergence (\ref{3.W2}) to
bounded  Borel measurable function $(t,{\bf w,z})\mapsto
\int_{0}^{\infty}\varrho(r)\int_{{\mS}^1({\bf n})}\psi(t,r,{\bf w}, {\bf
z}, \sg){\rm d}\sg{\rm d}r$,
we have
\be\lim_{{\cal E}\ni \vep\to 0}
\int_{X}\psi(t,r,{\bf w,z},\sg)
\bar{F}^{\vep}_{\dt}(t,{\bf w},{\bf z}){\rm d}\mu
=\int_{X}\psi(t,r,{\bf w,z},\sg)
\bar{F}_{\dt}(t,{\bf w},{\bf z}){\rm d}\mu.\lb{3.WeakFF}
\ee
This weak convergence and the criterion of $L^1$-weak compactness
(Dunford-Pettis theorem) imply that
the set $\{\bar{F}^{\vep}_{\dt}\}_{\vep \in {\cal E}}$ is relatively weakly compact in
$L^1(X, {\rm d}\mu)$ and thus we have
the following uniform integrability (with $\dt>0$ fixed):
\be
\sup_{\vep\in{\cal E}, \,\mu(E)\le \eta}\int_{E}
\bar{F}^{\vep}_{\dt}(t,{\bf w},{\bf z}){\rm d}\mu\to 0\quad
{\rm as}\quad \eta\to 0^{+}\lb{3.MLim}\ee
where $E\subset X$ are Borel sets.
Note that from (\ref{3.Exchange})
and $1_{E}(t,r,{\bf w},{\bf z}^*_{\vep r},\sg_{\vep r})=
1_{E^{\vep}}(t,r,{\bf w},{\bf z},\sg)$ we have
\be \int_{E}
\bar{F}^{\vep}_{\dt}(t,{\bf w},{\bf z}_{\vep r}){\rm d}\mu
=\int_{E^{\vep}}
\bar{F}^{\vep}_{\dt}(t,{\bf w},{\bf z}){\rm d}\mu\qquad \forall\,{\rm Borel\,\, set\,\,}\, E\subset X.
\lb{3.MF}\ee
Since $\mu(E^{\vep})=\mu(E)$, it follows that
\be
\sup_{\vep\in{\cal E}, \,\mu(E)\le \eta}\int_{E}
\bar{F}^{\vep}_{\dt}(t,{\bf w},{\bf z}_{\vep r}){\rm d}\mu
\le
\sup_{\vep\in{\cal E}, \,\mu(E)\le \eta}\int_{E}
\bar{F}^{\vep}_{\dt}(t,{\bf w},{\bf z}){\rm d}\mu
\to 0\quad {\rm as}\quad \eta\to 0^{+}.\lb{3.MMLim}\ee

Let $\psi(t,r,{\bf w,z},\sg)$ be a bounded  Borel measurable function
 on  $X$.  For any $1<R<\infty$, let
$$\psi_R(t,r, {\bf w},{\bf z}, \sg)=\psi(t,r, {\bf w},{\bf z}, \sg)1_{\{|{\bf w}|^2+|{\bf z}|^2\le R^2\}}.$$
 Since $\mu$ is a regular  measure and
$\mu(\{(t,r,{\bf w},{\bf z},\sg)\in X\,|\,
|{\bf w}|^2+|{\bf z}|^2\le R^2\})<\infty$,
it follows from Lusin's theorem that there is
function $\vp_R\in C_c(X)$ such that
\beas \mu(E_R)<\fr{1}{R},\quad \|\vp_R\|_{\infty}\le \|\psi\|_{\infty}.\eeas
where
$$E_R=\big\{(t,r, {\bf w},{\bf z}, \sg)\in X\,\,|\,\,
\vp_{R}(t,r,{\bf w},{\bf z}, \sg)\neq \psi_R(t,r,{\bf w},{\bf z}, \sg)\big\}.$$ Then we have
\beas |\Dt_{\vep}(\psi)| &:= & \bigg|
\int_{X}
\psi(t,r,{\bf w},{\bf z},\sg)
\bar{F}^{\vep}_{\dt}(t,{\bf w},{\bf z}_{\vep r})
{\rm d}\mu-\int_{X}
\psi(t,r,{\bf w},{\bf z},\sg)
\bar{F}_{\dt}(t,{\bf w},{\bf z})
{\rm d}\mu\bigg|\\
&\le & \int_{X}
\big|\psi(t,r, {\bf w},{\bf z},\sg)-\psi_R(t,r, {\bf w},{\bf z},\sg)\big|
\bar{F}^{\vep}_{\dt}(t,{\bf w},{\bf z}_{\vep r}){\rm d}\mu\\
&+&\int_{X}
\big|\psi_R(t,r, {\bf w},{\bf z},\sg)-\vp_R(t,r, {\bf w},{\bf z},\sg)
\big|
\bar{F}^{\vep}_{\dt}(t,{\bf w},{\bf z}_{\vep r}){\rm d}\mu
\\
&+&\bigg|\int_{X}\vp_R(t,r, {\bf w},{\bf z},\sg)
\bar{F}^{\vep}_{\dt}(t,{\bf w},{\bf z}_{\vep r}){\rm d}\mu
-\int_{X}\vp_R(t,r, {\bf w},{\bf z},\sg)
\bar{F}_{\dt}(t,{\bf w},{\bf z}_{\vep r}){\rm d}\mu
\bigg|
\\
&+&\int_{X}
\big|\vp_R(t,r, {\bf w},{\bf z},\sg)-\psi_R(t,r, {\bf w},{\bf z},\sg)|
\bar{F}_{\dt}(t,{\bf w},{\bf z}){\rm d}\mu
\\
&+&\int_{X}
\big|\psi_R(t,r,{\bf w},{\bf z},\sg)-\psi(t,r,{\bf w},{\bf z},\sg)\big|
\bar{F}_{\dt}(t,{\bf w},{\bf z}){\rm d}\mu
\\
&=:& I_{\vep,R}+J_{\vep,R}+|\Dt_{\vep}(\vp_R)|+J_R+I_R.\eeas

 \noindent Estimate of $ I_{\vep,R}$:\,
Using  Lemma \ref{Prop2.2} again  we have
\beas&& I_{\vep, R} \le
2\pi \|\psi\|_{\infty}
\int_{0}^{\infty}e^{-t}\int_{{\bRR}}1_{\{|{\bf w}|^2+
|{\bf z}|^2>R^2\}}
\bar{F}^{\vep}_{\dt}(t,{\bf w},{\bf z}){\rm d}{\bf z}
{\rm d}{\bf w}{\rm d}t.
\eeas
From (\ref{3.W2}) we know that $\{\bar{F}^{\vep}_{\dt}\}_{\vep\in {\cal E}}$ is
weakly compact in $L^1([0,\infty)\times {\bRR}, e^{-t}{\rm d}{\bf z}{\rm d}{\bf w}{\rm d}t)$, which
implies that $$\sup_{\vep\in {\cal E}}I_{\vep,R}\le  2\pi \|\psi\|_{\infty}\sup_{\vep\in {\cal E}}
\int_{0}^{\infty}e^{-t}\int_{{\bRR}}1_{\{|{\bf w}|^2+
|{\bf z}|^2>R^2\}}
\bar{F}^{\vep}_{\dt}(t,{\bf w},{\bf z}){\rm d}{\bf z}
{\rm d}{\bf w}{\rm d}t\to 0$$
as $R\to \infty.$

 \noindent Estimate of $ I_{R}$:\,   From the integrability $\bar{F}_{\dt}\in
L^1([0,\infty)\times {\bRR}, e^{-t}{\rm d}{\bf z}{\rm d}{\bf w}{\rm d}t)$,
one has
  \beas&& I_{R}\le 2\pi\|\psi\|_{\infty}
\int_{0}^{\infty}e^{-t}\int_{{\bRR}}1_{\{|{\bf w}|^2+|{\bf z}|^2>R^2\}}
\bar{F}_{\dt}(t,{\bf w},{\bf z}){\rm d}{\bf z}{\rm d}{\bf w}{\rm d}t
\to 0\quad {\rm as}\quad R\to\infty.\eeas

 \noindent Estimate of $ J_{\vep,R}$:\,
Since
$\mu(E_R)<\fr{1}{R}$,
it follows from (\ref{3.MMLim}) that
\beas&&\sup_{\vep\in {\cal E}}J_{\vep,R}\le 2\|\psi\|_{\infty}
\sup_{\vep\in {\cal E},\,\mu(E)<1/R}\int_{E}
\bar{F}^{\vep}_{\dt}(t,{\bf w},{\bf z}_{\vep r}){\rm d}\mu
\to 0\quad {\rm as}\,\,\, R\to\infty.\eeas

\noindent Estimate of $ J_{R}$:\, From $\bar{F}_{\dt}\in L^1(X, {\rm d}\mu)$,
we have
\beas&&J_{R} \le
2\|\psi\|_{\infty}\int_{E_R}
\bar{F}_{\dt}(t,{\bf w},{\bf z}){\rm d}\mu
\to 0\quad {\rm as}\quad R\to\infty.\eeas

 \noindent Estimate of $ \Dt_{\vep}(\vp_R)$:\,   Using \eqref{2.13} we have
\beas|
\Dt_{\vep}(\vp_R)| &\le&
\int_{X}
\big|\vp_R(t,r, {\bf w},{\bf z}^*_{\vep r},\sg_{\vep r})-
\vp_R(t,r, {\bf w},{\bf z},\sg)\big|
\bar{F}^{\vep}_{\dt}(t,{\bf w},{\bf z}){\rm d}\mu
\\
&+&\bigg|\int_{X}
\vp_R(t,r, {\bf w},{\bf z},\sg)
\bar{F}^{\vep}_{\dt}(t,{\bf w},{\bf z}){\rm d}\mu
- \int_{X}
\vp_R(t,r,{\bf w},{\bf z},\sg)
\bar{F}_{\dt}(t,{\bf w},{\bf z}){\rm d}\mu
\bigg|
\\
&=:&K_{\vep,R}+ L_{\vep, R}.\eeas
By the $L^1$-weak convergence obtained above we have
$L_{\vep, R}\to 0$  as ${\cal E}\ni \vep\to 0.$
To estimate $K_{\vep,R}$, we make a decomposition as follows:
\beas
K_{\vep, R}
&=&
\int_{X}1_{\{|{\bf z}|<1/R\}}\{\cdots\}{\rm d}\mu
+
\int_{X}\big(1_{\{0<r<1/R\}}+1_{\{r>R\}}\big)1_{\{|{\bf z}|\ge 1/R\}}\{\cdots\}{\rm d}\mu
\\
&+&
\int_{X}1_{\{1/R\le r\le R\}}1_{\{|{\bf z}|\ge 1/R\}}\{\cdots\}{\rm d}\mu
:=K_{\vep, R}^{(1)}+K_{\vep, R}^{(2)}+K_{\vep, R}^{(3)}.\eeas
We have
\beas \sup_{\vep\in {\cal E}}
K_{\vep, R}^{(1)}
\le 4\pi\|\psi\|_{\infty}\sup_{\vep\in {\cal E}}
\int_{0}^{\infty}e^{-t}\int_{|{\bf z}|<1/R}\bigg(\int_{{\bR}}
\bar{F}^{\vep}_{\dt}(t,{\bf w},{\bf z}){\rm d}{\bf w}\bigg)
{\rm d}{\bf z} {\rm d}t\to 0
\quad {\rm as}\,\,\, R\to\infty\eeas
since
$\{\int_{{\bR}}\bar{F}^{\vep}_{\dt}(t,{\bf w},{\bf z})
{\rm d}{\bf w}\}_{\vep\in {\cal E}}$ is
 weakly compact in $L^1([0,\infty)\times {\bR}, e^{-t}{\rm d}{\bf z}{\rm d}t)$.
Next we have
\beas\sup_{\vep\in {\cal E}}K_{\vep, R}^{(2)}
&\le & 2\|\psi\|_{\infty}\bigg(\int_{0}^{1/R}\varrho(r){\rm d}r+
\int_{R}^{\infty}\varrho(r){\rm d}r\bigg)
2\pi \sup_{\vep\in {\cal E}}\int_{0}^{\infty}e^{-t}\int_{{\bRR}}\bar{F}^{\vep}_{\dt}(t,{\bf w},{\bf z})
{\rm d}{\bf z}{\rm d}{\bf w} {\rm d}t
\\
&\le & C_{\psi}\bigg(\int_{0}^{1/R}\varrho(r){\rm d}r+\int_{R}^{\infty}\varrho(r){\rm d}r\bigg)\to 0\quad {\rm as}\,\,\,R\to\infty.\eeas
For the term $K_{\vep, R}^{(3)}$,
since $\vp_R\in C_c(X)$  is uniformally continuous on $X$ and
since
$$ |{\bf z}^*_{\vep r}-{\bf z}|+|\sg_{\vep r}-\sg|
\le 2\vep r +\fr{2\vep r}{|{\bf z}|}
\le 2\vep R+2\vep R^2\le 4 R^2\vep$$
 for all $r\in [1/R, R], |{\bf z}|\ge 1/R,
\sg\in {\mS}^1({\bf n})$, it follows that
$ K_{\vep, R}^{(3)}\to 0$ as ${\cal E}\ni \vep\to 0.$

Thus first letting $\vep\to 0^+$  we obtain
\beas \limsup_{{\cal E}\ni \vep\to 0}|\Dt_{\vep}(\psi)|
\le \sup_{\vep\in {\cal E}}I_{\vep,R}+\sup_{\vep\in {\cal E}}
J_{\vep,R}+
\sup_{\vep\in {\cal E}}
K_{\vep, R}^{(1)}+\sup_{\vep\in {\cal E}}
K_{\vep, R}^{(2)}+J_R+I_R\eeas
and then letting $R\to\infty$ we conclude
$\limsup\limits_{{\cal E}\ni \vep\to 0}|\Dt_{\vep}(\psi)|=0.$
This proves the weak convergence (\ref{3.WF}).
\vskip2mm
{\bf Step 3.}
In this step we prove that the square root
$\sqrt{\bar{F}_{\dt}}$ has the weak projection gradient
   $\Pi({\bf z})\nabla_{{\bf z}}\sqrt{\bar{F}_{\dt}(t,{\bf w},{\bf
z})}$ in ${\bf z}\in{\bR}\setminus\{{\bf 0}\}$ (see Definition \ref{weak-diff-1}) and, for any $T\in(0, \infty)$,
\bes&& \sup_{\dt>0}\int_{0}^{T}\int_{{\bRR}}\Big|\Pi({\bf z})\nabla_{{\bf z}}\sqrt{\bar{F}_{\dt}(t,{\bf w},{\bf z})}
\Big|^2 {\rm d}{\bf z}{\rm d}{\bf w}{\rm d}t\le
\liminf_{{\cal E}\ni\vep\to 0}H(f^{\vep}_0)
-H(f(T)), \quad  \lb{3.AEntropy}\\
&& \sup_{\dt>0}\int_{0}^{\infty}
\int_{{\bRR}}\Big|\Pi({\bf z})\nabla_{{\bf z}}\sqrt{\bar{F}_{\dt}(t,{\bf
w},{\bf z})}\Big|^2 {\rm d}{\bf z}{\rm d}{\bf w}{\rm d}t\le C_0. \lb{3.BEntropy}\ees
We emphasize that $C_0$ depends only on the bounds
$\sup\limits_{0<\vep\le \vep_0}\|f_0^{\vep}\|_{L^1_2}$ and $\sup\limits_{0<\vep\le \vep_0}\|f^{\vep}_0\log f^{\vep}_0\|_{L^1}$.

In order to prove (\ref{3.AEntropy}),
we choose a subsequence $\wt{{\cal E}}=\{\vep_{n_k}\}_{k=1}^{\infty}$ of
${\cal E}= \{\vep_n\}_{n=1}^{\infty}$ such that
$$\lim_{\wt{{\cal E}}\ni\vep\to 0}H(f^{\vep}_0)
=\liminf_{{\cal E}\ni\vep\to 0}H(f^{\vep}_0).
$$
For notational convenience we still denote $\wt{{\cal E}}$ by ${\cal E}$, i.e. we may assume
that
$\lim\limits_{{\cal E}\ni\vep\to 0}H(f^{\vep}_0)$ exists.
Observe that for any
$a,b\ge 0$, $c>0$ and $0\le \eta\le 1$, we have (the case $a=b=0$ is obvious)
\beas\Gm(a,b)\ge 2\fr{(a-b)^2}{a+b}\eta
=2\fr{(a\eta-b\eta)^2}{a\eta +b\eta}\ge
2\fr{(a\eta +c-(b\eta +c))^2}{a\eta +c+b\eta +c}.\eeas
So using also the collision invariance
$|{\bf v}'-{\bf v}_*'|=|{\bf v}-{\bf v}_*|, |{\bf
v}'|^2+ |{\bf v}_*'|^2=|{\bf v}|^2+|{\bf v}_*|^2$ we get
\beas\Gm\big({f^{\vep}}'{f^{\vep}}_*',
f^{\vep}f^{\vep}_*\big) &\ge & 2\fr{\big({f^{\vep}}'{f^{\vep}}_*'1_{\{|{\bf v}-{\bf v}_*|\ge \dt\}}+c-(f^{\vep}f^{\vep}_*1_{\{|{\bf v}-{\bf v}_*|\ge \dt\}} +c)\big)^2}
{{f^{\vep}}'{f^{\vep}}_*'1_{\{|{\bf v}-{\bf v}_*|\ge \dt\}}+c+f^{\vep}f^{\vep}_*1_{\{|{\bf v}-{\bf v}_*|\ge \dt\}} +c}\\
&=&2|{\bf v}-{\bf v}_*|
\fr{\big(F^{\vep}_{\dt}(t,{\bf v}',{\bf v}_*')-
F^{\vep}_{\dt}(t,{\bf v},{\bf v}_*)\big)^2}{
F^{\vep}_{\dt}(t,{\bf v}',{\bf v}_*')+
F^{\vep}_{\dt}(t,{\bf v},{\bf v}_*)},
\eeas
where $c=|{\bf v}-{\bf v}_*|\dt e^{-(|{\bf v}|^2+|{\bf
v}_*|^2)}$ and we have taken $\eta=1_{\{|{\bf v}-{\bf v}_*|\ge\delta\}}$.   Then using the second equality (\ref{2.14*}) in Lemma \ref{Prop2.3} yields that
\beas
D_0^{\vep}({f^{\vep}}(t)) &=& \fr{1}{4}\int_{{\mathbb R}^3\times
{\mathbb R}^3\times {\mathbb S}^2}
B^{\vep}_0({\bf v}-{\bf v}_*,\og)\Gm\big({f^{\vep}}'{f^{\vep}}_*',
{f^{\vep}}{f^{\vep}}_*\big){\rm d}\og{\rm d}{\bf v}{\rm d}{\bf v}_*
\\
&\ge & \fr{1}{4}\int_{{\mathbb R}^3\times
{\mathbb R}^3\times {\mathbb S}^2}
B^{\vep}_0({\bf v}-{\bf v}_*,\og)2|{\bf v}-{\bf v}_*|
\fr{\big(F^{\vep}_{\dt}(t,{\bf v}',{\bf v}_*')-
F^{\vep}_{\dt}(t,{\bf v},{\bf v}_*)\big)^2}{
F^{\vep}_{\dt}(t,{\bf v}',{\bf v}_*')+
F^{\vep}_{\dt}(t,{\bf v},{\bf v}_*)}
{\rm d}\og{\rm d}{\bf v}{\rm d}{\bf v}_*
\\
&= & \fr{1}{2\pi}\int_{{\bRR}}
\int_{0}^{\infty}\varrho(r)\int_{{\mS}^1({\bf n})}\fr{\Big(\fr{
\bar{F}^{\vep}_{\dt}(t,{\bf w},{\bf z}_{\vep r})- \bar{F}^{\vep}_{\dt}(t,{\bf
w},{\bf z})}{2\vep r}\Big)^2}{ \bar{F}^{\vep}_{\dt}(t,{\bf w},{\bf z}_{\vep r})+
\bar{F}^{\vep}_{\dt}(t,{\bf w},{\bf z})}{\rm d}\sg {\rm d}r {\rm
d}{\bf z}{\rm d}{\bf w}.\eeas
Introduce $$D^{\vep}_{\dt}(t,r,{\bf w},{\bf z},\sg):=1_{\{\vep\le r\le 1/\vep\}}
\fr{\bar{F}^{\vep}_{\dt}(t,{\bf w},{\bf z}_{\vep r})- \bar{F}^{\vep}_{\dt}(t,{\bf
w},{\bf z})}{2\vep r}.$$ Then from the above inequality we have
\be \int_{{\bRR}} \int_{0}^{\infty} \varrho(r)
\int_{{\mS}^1({\bf n})}
\fr{\big(D^{\vep}_{\dt}(t,r,{\bf w},{\bf z},\sg)\big)^2}{ \bar{F}^{\vep}_{\dt}(t,{\bf
w},{\bf z}_{\vep r})+ \bar{F}^{\vep}_{\dt}(t,{\bf w},{\bf z})}
{\rm d}\sg{\rm d}r{\rm d}{\bf z}{\rm d}{\bf w} \le
2\pi D_0^{\vep}({f^{\vep}}(t)).\lb{3.DissP}\ee
In what follows, we will
use the  measure $\mu$  on
$X$ defined in Step 2.
Then from (\ref{3.DissP}), we have
$$\int_{X}\fr{\big(D^{\vep}_{\dt}(t,r,{\bf w},{\bf
z},\sg)\big)^2}{ \bar{F}^{\vep}_{\dt}(t,{\bf w},{\bf z}_{\vep r})+
\bar{F}^{\vep}_{\dt}(t,{\bf w},{\bf z})}{\rm d}\mu
\le 2\pi\int_{0}^{\infty}e^{-t}D_0^{\vep}({f^{\vep}}(t)){\rm d}t
\le 2\pi\int_{0}^{\infty}D_0^{\vep}({f^{\vep}}(t)){\rm d}t\le C_{0}$$ which implies that
  for any nonnegative  Borel measurable function
$\psi(t,r,{\bf w},{\bf z},\sg)$ on
$X$,
\bes\lb{3.DD} && \int_{X}
\psi(t,r,{\bf w},{\bf z},\sg)
|D^{\vep}_{\dt}(t,r,{\bf w},{\bf z},\sg)|{\rm d}\mu
 \le
\bigg(\int_{X}\fr{\big(D^{\vep}_{\dt}(t,r,{\bf w},{\bf
z},\sg)\big)^2}{ \bar{F}^{\vep}_{\dt}(t,{\bf w},{\bf z}_{\vep r})+
\bar{F}^{\vep}_{\dt}(t,{\bf w},{\bf z})}{\rm d}\mu\bigg)^{1/2}\\
&&\times
\bigg(\int_{X}
\psi(t,r,{\bf w},{\bf z},\sg)^2\Big(
\bar{F}^{\vep}_{\dt}(t,{\bf w},{\bf z}_{\vep r})+ \bar{F}^{\vep}_{\dt}(t,{\bf
w},{\bf z})\Big){\rm d}\mu
\bigg)^{1/2}\nonumber
\\
&&\le
\sqrt{C_0} \bigg(\int_{X}\big(\psi(t,r,{\bf w}, {\bf
z}_{\vep r}^*,\sg_{\vep r})^2+\psi(t,r,{\bf w}, {\bf
z},\sg)^2\big)\bar{F}^{\vep}_{\dt}(t,{\bf w},{\bf z}){\rm d}\mu
\bigg)^{1/2}.\nonumber\ees
Here in the last inequality we have used the equality
(\ref{3.Exchange}).
Recalling $\varrho\in L^1([0,\infty))$ and a fact from measure theory we know that there is a non-decreasing function $l(r)>0$ on $[0,\infty)$ such that
$$\lim_{r\to\infty}l(r)=\infty,\quad \int_{0}^{\infty}l(r)\varrho(r){\rm d}r<\infty.$$
[For instance, $l(r)=(A(r))^{-1/2}$ where $A(r)=\int_{r}^{\infty}
\varrho(s){\rm d}s + e^{-r}.$\, The function $(A(r))^{1/2}$ is absolutely
continuous and $A'(r)=-\varrho(r)-e^{-r}$ a.e. in $(0,\infty)$
and so $\int_{0}^{\infty}l(r)\varrho(r){\rm d}r
\le \int_{0}^{\infty}-(A(r))^{-1/2}A'(r){\rm d}r
=\lim\limits_{R\to\infty}-2(A(r))^{1/2}\big|_{r=0}^{r=R}<\infty$.] If we choose $\psi(t,r,{\bf w},{\bf z},\sg)=
(1+t+|{\bf w}|^2+|{\bf
z}|^2+l(r))^{1/2}$, then
using the identity $|{\bf
z}_{\vep r}^*|\equiv |{\bf z}|$ and $(1+t+l(r)+|{\bf w}|^2+|{\bf
z}|^2)\le (1+t)(1+l(r))(1+|{\bf w}|^2+|{\bf
z}|^2)$
we have from (\ref{3.DD}) that
\beas&&\sup_{\vep\in{\cal E}}\int_{X} (1+t+l(r)+|{\bf w}|^2+|{\bf
z}|^2)^{1/2}
|D^{\vep}_{\dt}(t,r,{\bf w},{\bf z},\sg)|{\rm d}\mu\\
&&\le C_1\sup_{\vep\in{\cal E}}\bigg(\sup_{t\ge 0}\int_{{\bRR}}(1+|{\bf
w}|^2+|{\bf z}|^2)
\bar{F}^{\vep}_{\dt}(t,{\bf w},{\bf z}){\rm d}{\bf z}{\rm d}{\bf w}\bigg)^{1/2}<\infty\eeas
where the inequality ``$<\infty$"  is due to $\fr{1}{|{\bf v}-{\bf v}_*|}1_{\{|{\bf v-v}_*|\ge \dt\}}\le 1/\dt$ and  $\sup\limits_{\vep\in{\cal E}}\|f^{\vep}(t)\|_{L^1_2}=
\sup\limits_{\vep\in{\cal E}}\|f^{\vep}_0\|_{L^1_2}<\infty$ (by conservation of mass
and energy).
Next for any Borel set $E\subset X$, applying
(\ref{3.DD}) to the function $\psi(t,r,{\bf w},{\bf z},\sg)=1_{E}(t,r,{\bf w},{\bf z},\sg)$
gives
\beas \int_{E}
|D^{\vep}_{\dt}(t,r,{\bf w},{\bf z},\sg)|{\rm d}\mu\le \sqrt{C_0} \bigg(\int_{E^{\vep}}
  \bar{F}^{\vep}_{\dt}(t,{\bf w},{\bf z})
 {\rm d}\mu+\int_{E}
  \bar{F}^{\vep}_{\dt}(t,{\bf w},{\bf z})
 {\rm d}\mu\bigg)^{1/2}.
\eeas Hence from $\mu(E^{\vep})=\mu(E)$ and (\ref{3.MLim}) we obtain that
\beas\sup_{\vep\in {\cal E}}\sup_{\mu(E)\le \eta}\int_{E}
|D^{\vep}_{\dt}(t,r,{\bf w},{\bf z},\sg)|{\rm d}\mu
\le
\sqrt{C_0} \bigg(2\sup_{\vep\in {\cal E}}\sup_{\mu(E)\le \eta}
 \int_{E} \bar{F}^{\vep}_{\dt}(t,{\bf w},{\bf z}) {\rm d}\mu\bigg)^{1/2}
\to 0\,\,\,\,{\rm as}\,\,\eta\to 0^+.\eeas
By the criterion of $L^1$-relatively weak compactness
(Dunford-Pettis theorem), with $\dt$ fixed,
there exist a subsequence $
{\cal E}_{\dt}:=\{\vep_{n_k}\}_{k=1}^{\infty}\subset {\cal E}$ (which may depend on $\dt$) and
 a function
$D_{\dt}(t,r,{\bf w},{\bf z},\sg)\in
L^1(X, {\rm d}\mu)$ such
that  for any bounded Borel measurable
function $\psi(t,r,{\bf w},{\bf z},\sg)$ on $X$,
\be\lim_{{\cal E}_{\dt}\ni \vep\to 0} \int_{X} \psi(t,r,{\bf w},{\bf z},\sg)
D^{\vep}_{\dt}(t,r,{\bf w},{\bf z},\sg){\rm d}\mu
=\int_{X}
\psi(t,r,{\bf w},{\bf z},\sg) D_{\dt}(t,r,{\bf w},{\bf
z},\sg){\rm d}\mu. \lb{3.Xconv}\ee
Let us define
$$D^{\vep}_{\dt}(t,{\bf w},{\bf z},\sg):=\int_{0}^{\infty}
\varrho(r)D^{\vep}_{\dt}(t,r,{\bf w},{\bf z},\sg) {\rm d}r,
\quad D_{\dt}(t,{\bf w},{\bf z},\sg):=\int_{0}^{\infty}
\varrho(r)D_{\dt}(t,r,{\bf w},{\bf z},\sg) {\rm d}r.$$ Then for any bounded Borel measurable function $\psi(t,{\bf w}, {\bf
z}, \sg)$ and for any $0<T<\infty$,  applying (\ref{3.Xconv})
with  the function $1_{[0,T]}(t) e^{t}\psi(t,{\bf w}, {\bf
z}, \sg)$, we obtain that
\bes\lb{3.WW} && \lim_{{\cal E}_{\dt}\ni \vep\to 0}\int_{0}^{T}\int_{{\bRR}} \int_{{\mS}^1({\bf n})} \psi(t,{\bf w},
{\bf z}, \sg) D^{\vep}_{\dt}(t,{\bf w},{\bf z},\sg)
{\rm d}\sg{\rm d}{\bf z}{\rm d}{\bf w}{\rm d}t\\
&&=\int_{0}^{T}\int_{{\bRR}} \int_{{\mS}^1({\bf n})} \psi(t,{\bf
w}, {\bf z}, \sg)D_{\dt}(t,{\bf w},{\bf z},\sg) {\rm d}\sg{\rm
d}{\bf z}{\rm d}{\bf w}{\rm d}t.\nonumber \ees

Now we are going to prove that $D_{\dt}(t,{\bf w},{\bf z},\sg)$ satisfies the
conditions in Lemma \ref{Lemma6.4} with $Y=[0,\infty)\times {\bR}$ and
${\bf y}=(t, {\bf w})$.
First by definition of $D_{\dt}(t,{\bf w},{\bf z},\sg)$ we have
$D_{\dt}\in L^1([0, T]\times {\bR}\times {\bRS},{\rm d}\nu )$ for all $0<T<\infty$,
where the measure $\nu$ is given in Lemma \ref{Lemma6.4}.
Take any $\psi\in {\cal T}_{c}([0,\infty)\times {\bR}\times
{\bRS})$ (recall \eqref{6.31} for $Y=[0,\infty)\times {\bR}$ and
${\bf y}=(t, {\bf w})$). By definition, there exists $0<T<\infty$ such that $\psi$ is supported in $[0,T]\times {\bRRS}$. Then, thanks to the cutoff $1_{\{\vep\le r\le 1/\vep\}}$,
 Lemma \ref{Prop2.2} yields that
\beas&& \int_{0}^{\infty}\int_{{\bRR}}
\int_{{\mS}^1({\bf n})} \psi(t,{\bf w},{\bf
z},\sg)D^{\vep}_{\dt}(t,{\bf w},{\bf z},\sg)
{\rm d}\sg{\rm d}{\bf z}{\rm d}{\bf w}{\rm d}t\\
&&= \int_{0}^{T}\int_{{\bRR}} \int_{{\mS}^1({\bf n})} \psi(t,{\bf w},{\bf
z},\sg) \bigg(\int_{0}^{\infty} \varrho(r)1_{\{\vep\le r\le 1/\vep\}}\fr{\bar{F}^{\vep}_{\dt}(t,{\bf w}, {\bf z}_{\vep r})-
\bar{F}^{\vep}_{\dt}(t,{\bf w}, {\bf z})}{2\vep r} {\rm d}r
\bigg){\rm d}\sg{\rm d}{\bf z}{\rm d}{\bf w}{\rm d}t\\
&&=\int_{0}^{T}\int_{{\bRR}} \int_{{\mS}^1({\bf n})} \bigg(\int_{0}^{\infty}
\varrho(r)1_{\{\vep\le r\le 1/\vep\}}\fr{\psi(t,{\bf w},{\bf
z}^*_{\vep r},\sg_{\vep r})- \psi(t,{\bf w},{\bf z},\sg)}{2\vep r}
{\rm d}r \bigg)\bar{F}^{\vep}_{\dt}(t,{\bf w}, {\bf z}){\rm d}\sg{\rm d}{\bf
w}{\rm d}{\bf z}{\rm d}t\\
&&=\int_{X}1_{[0,T]}(t)e^{t} {\mathscr D}_{\psi}^{\vep}(t,r,{\bf w,z},\sg)\bar{F}^{\vep}_{\dt}(t,{\bf w}, {\bf z})
{\rm d}\mu,\eeas
where
$${\mathscr D}_{\psi}^{\vep}(t,r,{\bf w,z},\sg):=
1_{\{\vep\le r\le 1/\vep\}}\fr{\psi(t,{\bf w},{\bf
z}^*_{\vep r},\sg_{\vep r})- \psi(t,{\bf w},{\bf z},\sg)}{2\vep r}.$$
For any
$(t,r,{\bf w,z},\sg)\in [0, T]\times (0,\infty)\times{\bR}\times({\bR}\times \{{\bf 0}\})\times {\bS}$,
let $\vep>0$ be so small that we can use local mean-value theorem of differentation to
get some $\wt{{\bf z}}_{\vep r}={\bf z}+\theta ({\bf z}^*_{\vep}-{\bf z}),
\wt{\sg}_{\vep r}=\sg+\theta( \sg_{\vep r}-\sg)$ with $0<\theta<1$ (depending on $t,{\bf w},{\bf
z},\sg,\vep, r$) such that (with (\ref{2.HK1})-(\ref{2.HK5}))
\beas
{\mathscr D}_{\psi}^{\vep}(t,r,{\bf w,z},\sg) &=&1_{\{\vep\le r\le 1/\vep\}}
\bigg(\nabla_{{\bf z}}\psi(t,{\bf w}, \wt{{\bf
z}}_{\vep r},\wt{\sg}_{\vep r})\cdot \fr{{\bf
z}^*_{\vep r}-{\bf z}}{2\vep r} + \nabla_{\sg}\psi(t,{\bf w},
\wt{{\bf z}}_{\vep r},\wt{\sg}_{\vep r})\cdot \fr{\sg_{\vep r}-\sg}{2\vep r}\bigg)
\\
&\to & - \nabla_{{\bf z}}\psi(t,{\bf w},{\bf z},\sg)\cdot \sg +
\nabla_{\sg}\psi(t,{\bf w},{\bf z},\sg)\cdot \fr{{\bf n}}{|{\bf z}|}
\quad {\rm as}\,\,\,\vep\to 0^{+}.\eeas
Also using Lemma \ref{Lemma6.new1} and (\ref{2.HK5}) we have
 \beas&&
|{\mathscr D}_{\psi}^{\vep}(t,r,{\bf w,z},\sg)|\le \|\nabla_{{\bf z}}\psi\|_{\infty}\fr{|{\bf z}^*_{\vep r}-{\bf z}|}{2\vep r}+4\sqrt{2}\|\nabla_{{\sg}}\psi(t,{\bf w,z},\cdot)\|_{{\bS}}
\fr{|\sg_{\vep r}-\sg|}{2\vep r}\\
&&\le \|\nabla_{{\bf z}}\psi\|_{\infty}+4\sqrt{2}\Big\|\fr{1}{|{\bf z}|}\nabla_{{\sg}}\psi(t,{\bf w,z},\cdot)\Big\|_{{\bS}}.\eeas
 Thus the function $1_{[0,T]}(t)e^{t}  {\mathscr D}_{\psi}^{\vep}(t,r,{\bf w,z},\sg)1_{\{{\bf z}\neq{\bf 0}, r>0\}}$ is  bounded on $X$.
Recalling the weak convergence
(\ref{3.WeakFF}) , i.e.
$\bar{F}^{\vep}_{\dt}\rightharpoonup \bar{F}_{\dt}\,
({\cal E}\ni \vep\to 0)$  weakly in
$L^1(X, {\rm d}\mu)$,
 we can use part (1) of Lemma \ref{Lemma2.5} to conclude the convergence:
\beas&&\lim_{{\cal E}\ni\vep\to 0}\int_{X}1_{[0,T]}(t)e^{t}  {\mathscr D}_{\psi}^{\vep}(t,r,{\bf w,z},\sg) \bar{F}^{\vep}_{\dt}(t,{\bf w}, {\bf z})
{\rm d}\mu
\\
&&=\int_{X}1_{[0,T]}(t)e^{t}  \Big(- \nabla_{{\bf z}}\psi(t,{\bf w},{\bf z},\sg)\cdot \sg +
\nabla_{\sg}\psi(t,{\bf w},{\bf z},\sg)\cdot \fr{{\bf n}}{|{\bf z}|}
\Big)\bar{F}_{\dt}(t,{\bf w}, {\bf z})
{\rm d}\mu.\eeas
Collecting the above results we derive that
\beas&& \lim_{{\cal E}_{\dt}\ni \vep\to 0}\int_{0}^{\infty}\int_{{\bRR}}
\int_{{\mS}^1({\bf n})} \psi(t,{\bf w},{\bf
z},\sg)D^{\vep}_{\dt}(t,{\bf w},{\bf z},\sg)
{\rm d}\sg{\rm d}{\bf z}{\rm d}{\bf w}{\rm d}t\\
&&=\int_{0}^{\infty}\int_{{\bRR}}
\int_{{\mS}^1({\bf n})}
\Big(- \nabla_{{\bf z}}\psi(t,{\bf w},{\bf z},\sg)\cdot \sg +
\nabla_{\sg}\psi(t,{\bf w},{\bf z},\sg)\cdot \fr{{\bf n}}{|{\bf z}|}
\Big)\bar{F}_{\dt}(t,{\bf w}, {\bf z}){\rm d}\sg{\rm d}{\bf z}{\rm d}{\bf w}{\rm d}t.
\eeas
From this  together  (\ref{3.WW}) we finally obtain that for any $\psi\in {\cal T}_{c}([0,\infty)\times {\bR}\times
{\bRS})$
\beas&&\int_{0}^{\infty}\int_{{\bRR}} \int_{{\mS}^1({\bf n})} \psi(t,{\bf
w}, {\bf z}, \sg)D_{\dt}(t,{\bf w},{\bf z},\sg) {\rm d}\sg{\rm
d}{\bf z}{\rm d}{\bf w}{\rm d}t\\
&&=-\int_{0}^{\infty}\int_{{\bRR}}
\int_{{\mS}^1({\bf n})}
\Big(\nabla_{{\bf z}}\psi(t,{\bf w},{\bf z},\sg)\cdot \sg -
\nabla_{\sg}\psi(t,{\bf w},{\bf z},\sg)\cdot \fr{{\bf n}}{|{\bf z}|}
\Big)\bar{F}_{\dt}(t,{\bf w}, {\bf z}){\rm d}\sg{\rm d}{\bf z}{\rm d}{\bf w}{\rm d}t.
\eeas
By  Lemma \ref{Lemma6.4} with $Y=[0,\infty)\times {\bR}$ and
${\bf y}=(t, {\bf w})$
we conclude that the weak projection gradient
$\Pi({\bf z})\nabla_{\bf z}\bar{F}_{\dt}(t,{\bf w},{\bf
z})$ exists and (up to a set of measure zero)
\bes \lb{3.30} && \Pi({\bf z})\nabla_{\bf z}\bar{F}_{\dt}(t,{\bf w},{\bf z})=\fr{1}{\pi}
\int_{{\mS}^1({\bf n})} D_{\dt}(t,{\bf w},{\bf z},\sg)\sg{\rm
d}\sg,\\
&&  D_{\dt}(t,{\bf w},{\bf z},\sg)=
(\Pi({\bf z})\nabla_{\bf z}\bar{F}_{\dt}(t,{\bf w},{\bf z}))\cdot\sg.\lb{3.31}\ees
And by Lemma \ref{Lemma6.8}, the weak projection gradient
$\Pi({\bf z})\nabla_{\bf z}\sqrt{\bar{F}_{\dt}(t,{\bf w},{\bf
z})}$ also exists and
\be\Pi({\bf z})\nabla_{\bf z}\sqrt{\bar{F}_{\dt}(t,{\bf w},{\bf
z})}=\fr{\Pi({\bf z})\nabla_{\bf z}\bar{F}_{\dt}(t,{\bf w},{\bf
z})}{2\sqrt{\bar{F}_{\dt}(t,{\bf w},{\bf
z})}}.\lb{3.Squ}\ee
Next in order to prove the entropy inequality (\ref{3.AEntropy}), we consider
convexity argument.  The function
$(x,y,z)\mapsto \fr{x^2}{y+z}$ is
convex in $(x,y,z)\in {\mR}_{\ge 0}\times{\mR}_{>0}^2$ so that we have
\be \lb{3.convex} \fr{x^2}{y+z}\ge \fr{x_0^2}{y_0+z_0}+\fr{2x_0}{y_0+z_0}(x-x_0)-\fr{x_0^2}{(y_0+z_0)^2}(y-y_0)
-\fr{x_0^2}{(y_0+z_0)^2}(z-z_0)\ee
for all $(x,y,z),(x_0,y_0,z_0)\in  {\mR}_{\ge 0}\times {\mR}_{>0}^2.$
For any $0<T<\infty, 0<R<\infty$, let
$$S_{R}=\big\{(t,{\bf w},{\bf z},\sg)\in [0,T]\times{\bRRS}\,\,\,|\,\,\,
|{\bf w}|^2+|{\bf z} |^2\le R^2,\, |D_{\dt}(t,{\bf w},{\bf z},\sg)|\le R\big\}.$$
Then using (\ref{3.convex}) to
$x=D^{\vep}_{\dt}(t,{\bf w},{\bf z},\sg), y=
\bar{F}^{\vep}_{\dt}(t,{\bf w},{\bf z}_{\vep r}), z=
\bar{F}^{\vep}_{\dt}(t,{\bf w},{\bf z}),
x_0= D_{\dt}(t,{\bf w},{\bf z},\sg), y_0=z_0=\bar{F}_{\dt}(t,{\bf w},{\bf z}) $
and then multiplying both sides of (\ref{3.convex}) by $1_{S_R}(t,{\bf w},{\bf z},\sg)$
and then  taking integration we have
\bes\lb{3.ConvD}&&\quad \int_{0}^{T}\int_{{\bRR}}
\int_{{\mS}^1({\bf n})}\int_{0}^{\infty}\varrho(r)1_{S_R}(t,{\bf w},{\bf z},\sg)
\fr{\big(D^{\vep}_{\dt}(t,{\bf w},{\bf z},\sg)\big)^2}
{\bar{F}^{\vep}_{\dt}(t,{\bf w},{\bf z}_{\vep r})+
\bar{F}^{\vep}_{\dt}(t,{\bf w},{\bf z})}{\rm d}r{\rm d}\sg{\rm d}{\bf z}{\rm d}{\bf w}{\rm d}t
\\
&&\ge \int_{0}^{T}\int_{{\bRR}}
\int_{{\mS}^1({\bf n})}1_{S_R}
\fr{\big(D_{\dt}(t,{\bf w},{\bf z},\sg)\big)^2}
{2\bar{F}_{\dt}(t,{\bf w},{\bf z})}{\rm d}\sg{\rm d}{\bf z}{\rm d}{\bf w}{\rm d}t
\nonumber\\
&& + \int_{0}^{T}\int_{{\bRR}}
\int_{{\mS}^1({\bf n})}1_{S_R}\fr{ D_{\dt}(t,{\bf w},{\bf z},\sg)}
{\bar{F}_{\dt}(t,{\bf w},{\bf z})}
\big(D^{\vep}_{\dt}(t,{\bf w},{\bf z},\sg)-D_{\dt}(t,{\bf w},{\bf z},\sg)\big){\rm d}\sg{\rm d}{\bf z}{\rm d}{\bf w}{\rm d}t\nonumber
\ees
\beas
&&- \int_{0}^{T}\int_{{\bRR}}
\int_{{\mS}^1({\bf n})}1_{S_R}\fr{ \big(D_{\dt}(t,{\bf w},{\bf z},\sg)\big)^2}
{\big(2\bar{F}_{\dt}(t,{\bf w},{\bf z})\big)^2}
\bigg(\int_{0}^{\infty}\varrho(r)
\bar{F}^{\vep}_{\dt}(t,{\bf w},{\bf z}_{\vep r}){\rm d}r-\bar{F}_{\dt}(t,{\bf w},{\bf z})\bigg)
{\rm d}\sg{\rm d}{\bf z}{\rm d}{\bf w}{\rm d}t\\
&&- \int_{0}^{T}\int_{{\bRR}}
\int_{{\mS}^1({\bf n})}1_{S_R}\fr{ \big(D_{\dt}(t,{\bf w},{\bf z},\sg)\big)^2}
{\big(2\bar{F}_{\dt}(t,{\bf w},{\bf z})\big)^2}
\big(\bar{F}^{\vep}_{\dt}(t,{\bf w},{\bf z})-\bar{F}_{\dt}(t,{\bf w},{\bf z})\big){\rm d}\sg{\rm d}{\bf z}{\rm d}{\bf w}{\rm d}t\eeas
where we have used $\int_{0}^{\infty}\varrho(r){\rm d}r=1$.
Since $$\inf\limits_{t\ge 0, |{\bf w}|^2+|{\bf z}|^2\le R^2}
\bar{F}_{\dt}(t,{\bf w},{\bf z})>0,$$
the functions
$1_{S_R}(t,{\bf w},{\bf z},\sg)\fr{D_{\dt}(t,{\bf w},{\bf z},\sg)}
{\bar{F}_{\dt}(t,{\bf w},{\bf z})},
1_{S_R}(t,{\bf w},{\bf z},\sg)\fr{(D_{\dt}(t,{\bf w},{\bf z},\sg))^2}
{(2\bar{F}_{\dt}(t,{\bf w},{\bf z}))^2}$ are all bounded in
$[0,T]\times {\bRRS}$. Thus letting ${\cal E}_{\dt}\ni \vep\to 0$  we
obtain from the weak convergences (\ref{3.WFT}), (\ref{3.W2T})  and  (\ref{3.WW}) that
(recall notation in (\ref{1.UpLw}))
\beas&& \liminf_{{\cal E}_{\dt}\ni \vep\to 0}\int_{0}^{T}\int_{{\bRR}}
\int_{{\mS}^1({\bf n})}\int_{0}^{\infty}\varrho(r)1_{S_R}(t,{\bf w},{\bf z},\sg)
\fr{\big(D^{\vep}_{\dt}(t,{\bf w},{\bf z},\sg)\big)^2}
{
\bar{F}^{\vep}_{\dt}(t,{\bf w},{\bf z}_{\vep r})+
\bar{F}^{\vep}_{\dt}(t,{\bf w},{\bf z})}{\rm d}\sg{\rm d}{\bf z}{\rm d}{\bf w}{\rm d}t
\\
&&\ge
\int_{0}^{T}\int_{{\bRR}}\int_{{\mS}^1({\bf n})}
1_{S_R}(t,{\bf w},{\bf z},\sg)
\fr{\big(D_{\dt}(t,{\bf w},{\bf z},\sg)\big)^2}
{2\bar{F}_{\dt}(t,{\bf w},{\bf z})}{\rm d}\sg {\rm d}{\bf z}{\rm d}{\bf w}{\rm d}t.\eeas
This together with (\ref{3.DissP}) and
$\liminf\limits_{{\cal E}_{\dt}\ni \vep\to 0}\le \limsup\limits_{{\cal E}\ni \vep\to 0}$
gives
\beas
\int_{0}^{T}\int_{{\bRR}}\int_{{\mS}^1({\bf n})}
1_{S_R}(t,{\bf w},{\bf z},\sg)
\fr{\big(D_{\dt}(t,{\bf w},{\bf z},\sg)\big)^2}
{2\bar{F}_{\dt}(t,{\bf w},{\bf z})} {\rm d}\sg {\rm d}{\bf z}{\rm d}{\bf w}{\rm d}t
\le
2\pi \limsup_{{\cal E}\ni \vep\to 0}\int_{0}^{T}D_0^{\vep}({f^{\vep}}(t)){\rm d}t.
\eeas
Letting $R\to\infty$ we conclude from the monotone convergence theorem that
\beas
\int_{0}^{T}\int_{{\bRR}}\int_{{\mS}^1({\bf n})}
\fr{\big(D_{\dt}(t,{\bf w},{\bf z},\sg)\big)^2}
{2\bar{F}_{\dt}(t,{\bf w},{\bf z})} {\rm d}\sg {\rm d}{\bf z}{\rm d}{\bf w}{\rm d}t
\le
2\pi \limsup_{{\cal E}\ni \vep\to 0}\int_{0}^{T}D_0^{\vep}({f^{\vep}}(t)){\rm d}t\qquad \forall\, T>0.
\eeas
Since $D_{\dt}(t,{\bf w},{\bf z},\sg)=
(\Pi({\bf z})\nabla_{\bf z}\bar{F}_{\dt}(t,{\bf w},{\bf z}))\cdot\sg$ and (using the second equality in (\ref{6.30}))
$$\int_{{\mS}^1({\bf n})} \big(\big(\Pi({\bf z})\nabla_{{\bf z}}\bar{F}_{\dt}(t,{\bf w},{\bf z})\big)\cdot \sg
\big)^2{\rm d}\sg=\pi \big|\Pi({\bf z})\nabla_{{\bf z}}\bar{F}_{\dt}(t,{\bf
w},{\bf z}) \big|^2,$$ it follows that
\beas \fr{1}{4}\int_{0}^{T}
\int_{{\bRR}}\fr{|\Pi({\bf z})\nabla_{{\bf z}}\bar{F}_{\dt}(t,{\bf
w},{\bf z}) |^2}{\bar{F}_{\dt}(t,{\bf w},{\bf z})}{\rm d}{\bf w}{\rm
d}{\bf z}{\rm d}t\ \le\limsup_{{\cal E}\ni \vep\to 0}
 \int_{0}^{T}D_0^{\vep}({f^{\vep}}(t)){\rm d}t\qquad \forall\, T>0.\eeas
Then using the equality (\ref{3.Squ}) we conclude
\be \sup_{\dt>0}\int_{0}^{T}
\int_{{\bRR}}\Big|\Pi({\bf z})\nabla_{{\bf z}}\sqrt{\bar{F}_{\dt}(t,{\bf
w},{\bf z})}\Big|^2 {\rm d}{\bf z}{\rm d}{\bf w}{\rm d}t
\le\limsup_{{\cal E}\ni \vep\to 0}
 \int_{0}^{T}D_0^{\vep}({f^{\vep}}(t)){\rm d}t\quad \forall\, T>0.\lb{3.EntropyD}\ee
On the other hand, from the  entropy inequality (\ref{entropy-inequality})
we have
$$\limsup\limits_{{\cal E}\ni \vep\to 0}
 \int_{0}^{T}D_0^{\vep}({f^{\vep}}(t)){\rm d}t\le
\limsup\limits_{{\cal E}\ni \vep\to 0}
H(f^{\vep}_0)-\liminf\limits_{{\cal E}\ni \vep\to 0}
H(f^{\vep}(T)).$$
Thanks to the $L^1({\bR})$-weak convergence $f^{\vep}(t,\cdot)\rightharpoonup f(t,\cdot)$ (as ${\cal E}\ni\vep\to 0$ for every $t\ge 0$) and that $ f\mapsto f \log f $ is convex,  one has
$\liminf\limits_{{\cal E}\ni \vep\to 0}
H(f^{\vep}(T))\ge H(f(T)).$
Since the limit $\lim\limits_{{\cal E}\ni \vep\to 0}
H(f^{\vep}_0)$ already exists, this gives
$$\limsup_{{\cal E}\ni \vep\to 0}
 \int_{0}^{T}D_0^{\vep}({f^{\vep}}(t)){\rm d}t
\le\lim_{{\cal E}\ni \vep\to 0}
H(f^{\vep}_0)-H(f(T)).$$
This together with (\ref{3.EntropyD}) gives (\ref{3.AEntropy}).
From (\ref{3.AEntropy}) and letting $T\to\infty$  we also obtain (\ref{3.BEntropy}).
\vskip2mm

{\bf Step 4}.  In this step we prove that
the weak projection gradient
$\Pi({\bf z})\nabla_{{\bf z}}\sqrt{\bar{F}(t,{\bf w},{\bf
z})}$ exists and satisfies
\be \int_{0}^{T}{\rm d}t
\int_{{\bRR}}\Big|\Pi({\bf z})\nabla_{{\bf z}}\sqrt{\bar{F}(t,{\bf w},{\bf
z})}\Big|^2 {\rm d}{\bf z}{\rm d}{\bf w} \le \liminf_{{\cal E}\ni \vep\to 0}
H(f^{\vep}_0)-H(f(T))\qquad \forall\, T>0\lb{3.Entropy}\ee
(which corresponds to the condition (ii) in Definition \ref{weakFPL} ) and  moreover  it holds that
\bes \lb{3.Psi}&& \lim_{\dt\to 0^{+}}
\int_{0}^{\infty} \int_{{\bRR}}\Psi_{\dt}(t,{\bf w},{\bf z})\cdot
\Pi({\bf z})\nabla_{{\bf z}}\sqrt{\bar{F}_{\dt}(t,{\bf w},{\bf
z})}\,{\rm d}{\bf z}{\rm d}{\bf w}{\rm d} t \\
&&
=\int_{0}^{\infty}
\int_{{\bRR}} \Psi(t,{\bf w},{\bf z})\cdot\Pi({\bf
z})\nabla_{{\bf z}}\sqrt{\bar{F}(t,{\bf w},{\bf z})}\,
{\rm d}{\bf z}{\rm d}{\bf w}{\rm d}t\nonumber \ees for all $\Psi_{\dt}, \Psi\in L^2([0,\infty)\times
{\bRR}, {\bR})$ satisfying $\Psi_{\dt}\to \Psi $ in $L^2([0,\infty)\times{\bRR}, {\bR})$
as $\dt\to 0^{+}$.

To do this we first take any sequence $\{\dt_n\}_{n=1}^{\infty}$
satisfying $0<\dt_n\to 0\,(n\to\infty)$.  From the $L^2$-boundedness
(\ref{3.BEntropy}) and the criterion
of $L^2$-weak compactness,   there exists a subsequence
$\{\dt_{n_k}\}_{k=1}^{\infty}$ and a vector-valued
function ${\bf D}\in L^2([0,\infty)\times {\bRR}, {\bR})$ such
that
\bes \lb{3.L2Weak} &&\lim_{k\to\infty}
\int_{0}^{\infty} \int_{{\bRR}}\Psi(t,{\bf w},{\bf z})\cdot
\Pi({\bf z})\nabla_{{\bf z}}\sqrt{\bar{F}_{\dt_{n_k}}(t,{\bf w},{\bf
z})}\, {\rm d}{\bf z}{\rm d}{\bf w}{\rm d}t
\\
&&
 =\int_{0}^{\infty}
\int_{{\bRR}} \Psi(t,{\bf w},{\bf z})\cdot {\bf D}(t,{\bf
w},{\bf z}) {\rm d}{\bf z}{\rm d}{\bf w}{\rm d}t
\quad \forall\,\Psi\in L^2([0,\infty)\times {\bRR}, {\bR}). \qquad
\nonumber \ees
By convexity of ${\bf x}\mapsto |{\bf x}|^2$, this weak convergence together with (\ref{3.AEntropy})
imples that
\be \int_{0}^{T}\int_{{\bRR}}|{\bf
D}(t,{\bf w},{\bf z})|^2 {\rm d}{\bf z}{\rm d}{\bf w}{\rm d}t \le
\liminf_{{\cal E}\ni \vep\to 0}
H(f^{\vep}_0)-H(f(T))\quad \forall\, T>0.\lb{3.DEntropy}\ee
Next for any $\Psi\in {\cal T}_{c}([0,\infty)\times {\bRR}, {\bR})$, we have from (\ref{3.L2Weak})
and the result of Step 3 that
\bes \lb{3.***}  && \int_{0}^{\infty} \int_{{\bRR}}
 \Psi(t,{\bf w},{\bf z})\cdot {\bf D}(t,{\bf w},{\bf z}) {\rm
d}{\bf z}{\rm d}{\bf w}{\rm d} t
 \\
&&=-\lim_{k\to\infty}\int_{0}^{\infty}
\int_{{\bRR}}\sqrt{\bar{F}_{\dt_{n_k}}(t,{\bf w},{\bf z})}\,\nabla_{{\bf
z}}\cdot \Pi({\bf z})\Psi(t,{\bf w},{\bf z}) {\rm d}{\bf z}{\rm
d}{\bf w}{\rm d} t.\nonumber\ees
Since
$$\bar{F}_{\dt_{n_k}}(t,{\bf w},{\bf z})=\bar{F}(t,{\bf w},{\bf z})
1_{\{|{\bf
z}|\ge\dt_{n_k}\}}+ \dt_{n_k} e^{-\fr{1}{2}(|{\bf w}|^2+|{\bf
z}|^2)}$$ the pointwise convergence
$\lim\limits_{k\to\infty}|{\bf z}|\bar{F}_{\dt_{n_k}}(t,{\bf w},{\bf z})
=|{\bf z}|\bar{F}(t,{\bf w},{\bf z})$ holds
for all $(t,{\bf
w},{\bf z})\in [0,\infty)\times{\bR}\times({\bR}\setminus\{{\bf 0}\}).$
If we denote
\beas g_k(t,{\bf w},{\bf z}):=\Big|\sqrt{|{\bf z}|\bar{F}_{\dt_{n_k}}(t,{\bf w},{\bf z})}-
\sqrt{|{\bf z}|\bar{F}(t,{\bf w},{\bf z})}\Big|,\quad
\psi_*(t,{\bf w},{\bf z}):=
\fr{1}{\sqrt{|{\bf z}|}}\big|\nabla_{{\bf z}}\cdot \Pi({\bf z})\Psi(t,{\bf w},{\bf z}) \big|.\eeas
then, by \eqref{FbarF}, we have
\beas&&\lim_{k\to\infty}g_k(t,{\bf w},{\bf z})=0\quad \forall\, (t,{\bf w,z})\in
[0,\infty)\times{\bR}\times {\bR}\setminus\{{\bf 0}\}, \\
&&
g_k(t,{\bf w},{\bf z})^2\le f\big(t,\fr{{\bf w+ z}}{2}\big)f\big(t,\fr{{\bf w- z}}{2}\big)+|{\bf z}|e^{-\fr{1}{2}
(|{\bf w}|^2+|{\bf z}|^2)}=: G(t,{\bf w},{\bf z}).\eeas
By definition of $\Psi\in {\cal T}_{c}([0,\infty)\times {\bRR}, {\bR})$(see
(\ref{6.10}) with $Y=[0,\infty)\times{\bR}$)
one sees that $\Psi$ has compact support in $[0,T]\times [-R,R]^3\times [-R, R]^3$ for
some $0<T,R<\infty$  so that $\psi_*^2\in L^1([0,T]\times {\bRR})$.
Also we have $G\in  L^1([0,T]\times {\bRR})$. Then using Cauchy-Schwarz inequality and the dominated convergence theorem we deduce
\beas&&\int_{0}^{\infty}
\int_{{\bRR}}\Big|\sqrt{\bar{F}_{\dt_{n_k}}}-\sqrt{\bar{F}} \Big| \big|\nabla_{{\bf z}}\cdot
\Pi({\bf z})\Psi\big| {\rm d}{\bf z}{\rm
d}{\bf w}{\rm d} t
=\int_{0}^{T} \int_{{\bRR}}g_k\psi_*{\rm d}{\bf z}{\rm d}{\bf w}{\rm d} t\\
&&\le \bigg(\int_{0}^{T} \int_{{\bRR}}g_k^2 {\rm d}{\bf z}{\rm d}{\bf w}{\rm
d} t \bigg)^{1/2}\bigg(\int_{0}^{T} \int_{{\bRR}}\psi_*^2 {\rm d}{\bf z}{\rm d}{\bf w}{\rm d} t\bigg)^{1/2}
\to 0\quad {\rm as}\quad k\to\infty.\eeas
This together with (\ref{3.***}) leads to
\beas&& \int_{0}^{\infty}
\int_{{\bRR}}\Psi(t,{\bf w},{\bf z})\cdot {\bf D}(t,{\bf
w},{\bf z}){\rm d}{\bf z}{\rm d}{\bf w}{\rm d} t
\\
&&
=-\int_{0}^{\infty} \int_{{\bRR}}\sqrt{\bar{F}(t,{\bf w},{\bf
z})}\nabla_{{\bf z}}\cdot \Pi({\bf z})\Psi(t,{\bf w},{\bf z})
{\rm d}{\bf z}{\rm d}{\bf w}{\rm d} t\quad \forall\,
\Psi\in {\cal T}_{c}([0,\infty)\times {\bRR}, {\bR}). \eeas By Defintion \ref{weak-diff-1},
this means
that the weak projection gradient $\Pi({\bf z})\nabla_{{\bf z}}\sqrt{\bar{F}(t,{\bf w},{\bf z})}$
exists and is equal to ${\bf D}(t,{\bf w},{\bf z}).$
Thus the inequality (\ref{3.DEntropy})
is rewritten
 \be\int_{0}^{T} \int_{{\bRR}}|\Pi({\bf
z})\nabla_{{\bf z}}\sqrt{\bar{F}(t,{\bf w},{\bf z})} |^2 {\rm d}{\bf
w}{\rm d}{\bf z}{\rm d}t \le
\liminf_{{\cal E}\ni \vep\to 0}
H(f^{\vep}_0)-H(f(T))\quad \forall\, T>0.\lb{3.FEntropy}\ee
Since the weak limit ${\bf D}(t,{\bf w},{\bf z})=\Pi({\bf z})\nabla_{{\bf z}}\sqrt{\bar{F}(t,{\bf w},{\bf z})}$ is unique, it does not depend on the choice of $\{\dt_n\}_{n=1}^{\infty}$. It follows from
the arbitrariness of $\{\dt_n\}_{n=1}^{\infty}$ and
(\ref{3.L2Weak}) that
\be\lb{3.Star}\lim_{\dt\to 0^{+}}
\int_{0}^{\infty} \int_{{\bRR}} \Psi\cdot
\Pi({\bf z})\nabla_{{\bf z}}\sqrt{\bar{F}_{\dt}}\,{\rm d}{\bf z}{\rm d}{\bf w}{\rm d}t
=\int_{0}^{\infty}
\int_{{\bRR}}\Psi\cdot\Pi({\bf
z})\nabla_{{\bf z}}\sqrt{\bar{F}}\, {\rm d}{\bf
w}{\rm d}{\bf z}{\rm d}t\ee
holds for all $\Psi\in L^2([0,\infty)\times
{\bRR}, {\bR}).$

Finally for any  $\Psi_{\dt},\Psi\in L^2([0,\infty)\times {\bRR},{\bR} )$ satisfying
$\Psi_{\dt}\to \Psi \,(\dt\to 0^{+})$ in $L^2([0,\infty)\times {\bRR},{\bR})$, we have
\beas&&
\bigg|  \int_{0}^{\infty} \int_{{\bRR}}
\Psi_{\dt}\cdot\Pi({\bf z})\nabla_{{\bf
z}}\sqrt{\bar{F}_{\dt}}\,{\rm d}{\bf w}
{\rm d}{\bf z}{\rm d} t -\int_{0}^{\infty} \int_{{\bRR}}\Psi\cdot\Pi({\bf z})\nabla_{{\bf z}}\sqrt{\bar{F}}\,
{\rm d}{\bf z}{\rm d}{\bf w}{\rm d}t\bigg|\\
&&\le\int_{0}^{\infty} \int_{{\bRR}}\Big|\big(
\Psi_{\dt}-\Psi\big)\cdot\Pi({\bf z})\nabla_{{\bf
z}}\sqrt{\bar{F}_{\dt}}\,\Big|{\rm d}{\bf w}
{\rm d}{\bf z}{\rm d} t\\
&&+\bigg|\int_{0}^{\infty} \int_{{\bRR}} \Psi\cdot\Pi({\bf z})\nabla_{{\bf z}}\sqrt{\bar{F}_{\dt}}\,
{\rm d}{\bf z}{\rm d}{\bf w}{\rm d}t-\int_{0}^{\infty} \int_{{\bRR}} \Psi\cdot\Pi({\bf z})\nabla_{{\bf z}}\sqrt{\bar{F}}\,
{\rm d}{\bf z}{\rm d}{\bf w}{\rm d}t
\bigg|
\\
&&:=I_{\dt}+J_{\dt}\to 0\quad {\rm as}\quad\dt\to 0^{+}\eeas
which is because of \eqref{3.Star} and
\beas&& I_{\dt}
\le
\left(\int_{0}^{\infty} \int_{{\bRR}}|
\Psi_{\dt}-\Psi|^2{\rm d}{\bf w}
{\rm d}{\bf z}{\rm d} t\right)^{1/2}
\left(\int_{0}^{\infty} \int_{{\bRR}}
\big|\Pi({\bf z})\nabla_{{\bf
z}}\sqrt{\bar{F}_{\dt}\,}\big|^2{\rm d}{\bf w}
{\rm d}{\bf z}{\rm d} t\right)^{1/2}
\\
&&\le C_0\left(\int_{0}^{\infty} \int_{{\bRR}}|
\Psi_{\dt}-\Psi|^2
|{\rm d}{\bf w}
{\rm d}{\bf z}{\rm d} t\right)^{1/2}
\to 0\quad (\dt\to0^{+}).\eeas
This ends the proof of the general
weak convergence (\ref{3.Psi}).
\vskip2mm
{\bf Step 5}. In this step we prove the existence of the weak projection gradient of
$ \sqrt{ff_*/|{\bf v-v}_*|}$ and the entropy inequalities
(\ref{entropyinequality2}) and (\ref{entropyinequality1}).
\,So far we have proved that the weak projection gradient $\Pi({\bf z})\nabla_{{\bf z}}\sqrt{\bar{F}(t,{\bf w},{\bf z})}$
exists and belongs to $L^2([0,\infty)\times {\bRR})$.
Applying Lemma \ref{Lemma6.2} to the functions $\sqrt{\bar{F}(t,{\bf w},{\bf z})}$ and
$\sqrt{|{\bf z}|}$, we see that the weak  projection gradient
$\Pi({\bf z})\nabla_{{\bf z}}\sqrt{|{\bf z}|\bar{F}(t,{\bf w},{\bf z})}$ also exists and
\be\Pi({\bf z})\nabla_{{\bf z}}\sqrt{|{\bf z}|\bar{F}(t,{\bf w},{\bf z})}
=\sqrt{|{\bf z}|}
\Pi({\bf z})\nabla_{{\bf z}}\sqrt{\bar{F}(t,{\bf w},{\bf z})}.\lb{3.1/2}\ee
Also since $\sqrt{|{\bf z}|\bar{F}(t,{\bf w},{\bf z})}=\sqrt{f(t,\fr{{\bf w+z}}{2}) f(t,\fr{{\bf w-z}}{2})}$ belongs to
$L^2([0,T]\times {\bRR})\,\,(\forall\, 0<T<\infty)$, we conclude from Lemma \ref{Lemma6.7}  that the weak projection gradient $\Pi({\bf z})\nabla_{{\bf z}}(|{\bf z}|\bar{F}(t,{\bf w},{\bf z}))$ also exists
and \footnote{Note that here we do not claim the existence of the weak  projection gradient
$\Pi({\bf z})\nabla_{{\bf z}}\bar{F}(t,{\bf w},{\bf z})$.}
\be \Pi({\bf z})\nabla_{{\bf z}}(|{\bf z}|\bar{F}(t,{\bf w},{\bf z}))=2\sqrt{|{\bf z}|\bar{F}(t,{\bf w},{\bf z})}
\Pi({\bf z})\nabla_{{\bf z}}\sqrt{|{\bf z}|\bar{F}(t,{\bf w},{\bf z})}.\lb{3.Sq}\ee
Since
$f(t,{\bf v})f(t,{\bf v}_*)/|{\bf v-v}_*|=F(t,{\bf v},{\bf v}_*)$,
it follows from Definition \ref{weak-diff-2} that the weak projection gradients
$\Pi({\bf v-v}_*)\nabla_{{\bf v-v}_*}
\sqrt{ff_*/|{\bf v-v}_*|},\,
\Pi({\bf v-v}_*)\nabla_{{\bf v-v}_*}
\sqrt{ff_*}$, and $
\Pi({\bf v-v}_*)\nabla_{{\bf v-v}_*}
\big(ff_*\big)$ in ${\bf v}-{\bf v}_*\neq {\bf 0}$ all exist
and
\bes&& \Pi({\bf v-v}_*)\nabla_{{\bf v-v}_*}
\sqrt{f(t,{\bf v})f(t,{\bf v}_*)/|{\bf v-v}_*|}=2\Pi({\bf z})\nabla_{{\bf z}}\sqrt{\bar{F}(t,{\bf w},{\bf z})}\big|_{{\bf w}={\bf v+v}*, {\bf z}={\bf v-v}_*},\quad \lb{3.Sq2}
\\
&&\Pi({\bf v-v}_*)\nabla_{{\bf v-v}_*}
\sqrt{f(t,{\bf v})f(t,{\bf v}_*)}=2\Pi({\bf z})\nabla_{{\bf z}}\sqrt{|{\bf z}|\bar{F}(t,{\bf w},{\bf z})}\big|_{{\bf w}={\bf v+v}*, {\bf z}={\bf v-v}_*},\quad  \lb{3.Sq3}
\\
&&\Pi({\bf v-v}_*)\nabla_{{\bf v-v}_*}
\big(f(t,{\bf v})f(t,{\bf v}_*)\big)=2\Pi({\bf z})\nabla_{{\bf z}}\big(|{\bf z}|\bar{F}(t,{\bf w},{\bf z})\big)\big|_{{\bf w}={\bf v+v}*, {\bf z}={\bf v-v}_*}.\quad  \lb{3.Sq4}
\ees
These together with (\ref{3.Sq}) imply that
\be\Pi({\bf v-v}_*)\nabla_{{\bf v-v}_*}
\big(ff_*\big)=2\sqrt{ff_*}
\Pi({\bf v-v}_*)\nabla_{{\bf v-v}_*}
\sqrt{ff_*}.\lb{3.Sq5}\ee
And again by change of variables and using (\ref{3.1/2}) and (\ref{3.Sq3}), we have
\beas&&
\int_{{\bRR}}\Big|\Pi({\bf z})\nabla_{{\bf z}}\sqrt{\bar{F}(t,{\bf w},{\bf
z})}\Big|^2 {\rm d}{\bf z}{\rm d}{\bf w}=8
\int_{{\bRR}}\Big|\Pi({\bf z})\nabla_{{\bf z}}\sqrt{\bar{F}(t,{\bf w},{\bf
z})}\,\Big|_{ {\bf w}={\bf v+v}*, {\bf z}={\bf v-v}_*}\Big|^2
{\rm d}{\bf v}{\rm d}{\bf v}_*
\\
&&=2
\int_{{\bRR}}
\big|\Pi({\bf v-v}_*)\nabla_{{\bf v-v}_*}\sqrt{f(t,{\bf v})f(t,{\bf
v}_*)/|{\bf v-v}_*|}\big|^2
{\rm d}{\bf v}{\rm d}{\bf v}_*=D_L(f(t)),\eeas
Connecting this with (\ref{3.FEntropy}) gives
\be\lb{3.FEntropy*} \int_{0}^{T}D_L(f(t)){\rm d}t
\le\liminf_{{\cal E}\ni \vep\to 0}
H(f^{\vep}_0)-H(f(T))\qquad \forall\, T>0.\ee
This implies that $\int_{0}^{\infty}D_L(f(t)){\rm d}t<\infty$ and
the inequality
(\ref{entropyinequality2}) holds true.
Also from $\int_{0}^{\infty}D_L(f(t)){\rm d}t<\infty$,  (\ref{3.Sq5}), Cauchy-Schwarz inequality, and
$|{\bf v-v}_*|\le \la {\bf v}\ra\la{\bf v}_*\ra$
 we have
\beas&&\int_{{\bRR}} \big|\Pi({\bf v-v}_*)\nabla_{{\bf v-v}_*}
\big(f(t,{\bf v})f(t,{\bf v}_*)\big)
\big|{\rm d}{\bf v}{\rm d}{\bf v}_*\nonumber
\\
&&\le \sqrt{2} \sqrt{\int_{{\bRR}}f(t,{\bf v})f(t,{\bf v}_*)|{\bf v-v}_*|
{\rm d}{\bf v}{\rm d}{\bf v}_*}\sqrt{D_L(f(t))}\nonumber
\\
&&\le \sqrt{2}\|f(t)\|_{L^1_1}\sqrt{D_L(f(t))}\le \sqrt{2}\sqrt{\|f(t)\|_{L^1}\|f(t)\|_{L^1_2}}
\sqrt{D_L(f(t))}\eeas
so that the integral in the righthand side of
(\ref{H-solu}) is absolutely convergent for almost all $t\in(0,\infty)$.

Finally we prove the entropy inequality (\ref{entropyinequality1}) under the additional assumption
in the theorem. We will use the following
inequality
\be\lb{logab} |a\log a-b\log b|\le C_p|a-b|^{1/p}+|a-b|\log^{+}(|a-b|)+2\sqrt{a\wedge b}\sqrt{|a-b|}
\ee
for $a,b\ge 0$ and $1<p<\infty$, where $C_p= \fr{p}{e(p-1)}.$
[To prove (\ref{logab}) one may assume that $a>b>0$. Then $|a\log a-b\log b|=
|(a-b)\log((a-b)\fr{a}{a-b})+b\log(\fr{a}{b})|
\le (a-b)|\log (a-b)|+(a-b)\log(\fr{a}{a-b})+
b\log(\fr{a}{b})$, and so (\ref{logab}) follows from
the inequalities
$|\log x|\le \fr{q}{e} x^{-1/q}$ for $0<x\le 1$ (where
$q=p/(p-1)$) and $\log x\le \sqrt{x-1}$ for $x\ge 1$.]
\,Now assume that
$|f_0^{\vep}-f_0|\big(1+\log^{+}(|f_0^{\vep}-f_0|)\big)\to 0$ in $L^1({\bR})$ as
$\vep\to 0^{+}$ for some $0\le f_0\in L^1_2({\bR})\cap L\log L ({\bR})$.
Using the inequality (\ref{logab}) with $p=5/4$ and Cauchy-Schwarz inequality, we have
\beas&&|
H(f^{\vep}_0)-H(f_0)|
\le \int_{{\bR}}\big|f^{\vep}_0({\bf v})\log(f^{\vep}_0({\bf v}))
-f_0({\bf v})\log(f_0({\bf v}))
\big|{\rm d}{\bf v}
\\
&&\le \int_{{\bR}}\Big(
|f^{\vep}_0-f_0|\log^{+}(|f^{\vep}_0-f_0|)
+2\sqrt{f_0}
\sqrt{|f^{\vep}_0-f_0|}
+C|f^{\vep}_0-f_0|^{4/5}
\Big){\rm d}{\bf v}
\\
&&\le
\int_{{\bR}}
|f^{\vep}_0-f_0|\log^{+}(|f^{\vep}_0-f_0|){\rm d}{\bf v}
+2\left(\int_{{\bR}}f_0{\rm d}{\bf v}\right)^{1/2}
\left(\int_{{\bR}}|f^{\vep}_0-f_0|{\rm d}{\bf v}\right)^{1/2}
\\
&&+C\left(\int_{{\bR}}|f^{\vep}_0-f_0|(1+|{\bf v}|){\rm d}{\bf v}
\right)^{4/5}
\left(\int_{{\bR}}(1+|{\bf v}|)^{-4}{\rm d}{\bf v}\right)^{1/5}.\eeas
Then, by the $L^1$-weak convergence  $f^{\vep}(t,\cdot)\rightharpoonup f(t,\cdot)
\,({\cal E}\ni \vep\to 0$) for all $t\ge 0$, we have
$f(0,\cdot)=f_0$, and since
$\sup\limits_{0<\vep\le \vep_0}\|f^{\vep}_0\|_{L^1_2}<\infty$, it follows from the above estimate that
$$\lim_{\vep\to 0^{+}}H(f^{\vep}_0)=H(f_0)=H(f(0)).$$
Thus in this case we see from
(\ref{entropyinequality2}) or (\ref{3.FEntropy*}) that  $f(t,\cdot)$ satisfies the entropy inequality (\ref{entropyinequality1}):
$$H(f(t))+
\int_{0}^{t}D_L(f(s)){\rm d}s\le  H(f(0))\qquad \forall\, t>0.$$

 {\bf Step 6.} The goal of this step is to prove \eqref{Lweak} so that $f$ is an $H$-solution of Eq.(FPL).
 Let
 \beas{\cal
Q}^{\vep}[\psi](t):= \fr{1}{4}\int_{{\bRRS}}B_{0}^{\vep}({\bf v}- {\bf v}_*,\og)
\big({f^{\vep}}'{f_*^{\vep}}'-{f^{\vep}}{f_*^{\vep}}\big)\big(\psi+\psi_*-\psi'-\psi_*'\big)
{\rm d}{\bf v}{\rm d}{\bf v}_*{\rm d}\og.\eeas
Recall the  definition of weak solutions for $f^{\vep}$, we have, for any $\psi\in C_c^2({\bR})$,
$$
\int_{{\mathbb R}^3}\psi({\bf v})f^{\vep}(T,{\bf v}){\rm d}{\bf v}
=
\int_{{\mathbb R}^3}\psi({\bf v})f^{\vep}_0({\bf v}){\rm d}{\bf v}
+\int_{0}^{T}{\cal
Q}^{\vep}[\psi](t) {\rm d}t\quad \forall\, T\in (0,\infty).$$
For any $\dt>0$, consider a decomposition
$${\cal Q}^{\vep}[\psi](t)={\cal Q}^{\vep}_{<\dt}[\psi](t)+{\cal Q}^{\vep}_{\ge \dt}[\psi](t)$$
with
\beas&&{\cal Q}^{\vep}_{<\dt}[\psi](t)
:=\fr{1}{4}\int_{{\bRRS}}1_{\{|{\bf v-v}_*|<\dt\}}B_{0}^{\vep}
\big({f^{\vep}}'{f_*^{\vep}}'-{f^{\vep}}{f_*^{\vep}}\big)\big(\psi+\psi_*-\psi'-\psi_*'\big)
{\rm d}{\bf v}{\rm d}{\bf v}_*{\rm d}\og,\\
&&{\cal Q}^{\vep}_{\ge \dt}[\psi](t)
:=\fr{1}{4}\int_{{\bRRS}}1_{\{|{\bf v-v}_*|\ge \dt\}}B_{0}^{\vep}
\big({f^{\vep}}'{f_*^{\vep}}'-{f^{\vep}}{f_*^{\vep}}\big)\big(\psi+\psi_*-\psi'-\psi_*'\big)
{\rm d}{\bf v}{\rm d}{\bf v}_*{\rm d}\og.\eeas
Then using (\ref{3.7}) with $\zeta(r)=1_{\{r<\dt\}},  1_{\{r\ge \dt\}}$
respectively we have
\bes&&|{\cal Q}^{\vep}_{<\dt}[\psi](t)|
\le C\|D^2\psi\|_{*} \sqrt{D_{0}^{\vep}(f^{\vep}(t))}\sqrt{\dt}\qquad \forall\, t\ge 0, \lb{3.Q4}
\\
&&
|{\cal Q}^{\vep}_{\ge \dt}[\psi](t)|\le C \|D^2\psi\|_{*}\sqrt{D_{0}^{\vep}(f^{\vep}(t))}
\qquad \forall\, t\ge 0.  \lb{3.Q2}
\ees
Here the constant $0<C<\infty$
depends only on $\sup\limits_{0<\vep\le \vep_0}\|f^{\vep}_0\|_{L^1_2}.$

We claim that for any $t\ge 0$,
\be{\cal Q}^{\vep}_{\ge \dt}[\psi](t)=\int_{{\bRR}}1_{\{|{\bf v-v}_*|\ge \dt\}}L^{\vep}_{0}[\psi]({\bf v}, {\bf v}_*)
f^{\vep}(t,{\bf v})f^{\vep}(t,{\bf v}_*){\rm d}{\bf v}{\rm d}{\bf v}_*.\lb{3.QQ}\ee
To prove it, we consider a cutoff kernel
\beas B_{0}^{\vep,\eta}({\bf v}- {\bf v}_*,\og)=1_{\{|\sin(\theta)\cos(\theta)|\ge \eta\}}
B_{0}^{\vep}({\bf v}- {\bf v}_*,\og),\quad \eta>0\eeas
with
$\theta=\arccos({\bf n}\cdot\og)\in[0,\pi]$ and ${\bf n}=({\bf v-v}_*)/|{\bf v-v}_*|$.
Correspondingly let $L^{\vep,\eta}_{0}[\psi]({\bf v}, {\bf v}_*)$ and
${\cal Q}^{\vep,\eta}_{\ge \dt}[\psi](t)$
be defined in (\ref{1.LLL}) (with $\ld=0$) and in (\ref{3.QQ}) respectively
by replacing $ B_{0}^{\vep}({\bf v}- {\bf v}_*,\og)$ with
$B_{0}^{\vep,\eta}({\bf v}- {\bf v}_*,\og)$.
Then on the one hand we see from the inequality (\ref{3.7}) with $\zeta(r)=1_{\{r\ge \dt\}}$
and the dominated convergence theorem that
$$\lim\limits_{\eta\to 0^{+}}{\cal Q}^{\vep,\eta}_{\ge \dt}[\psi](t)={\cal Q}^{\vep}_{\ge \dt}[\psi](t).$$
On the other hand, using  the first inequality in
Lemma \ref{Lemma2.2} and (\ref{1.BB}), we have
\beas&&
\fr{1}{4}\int_{{\bRRS}}1_{\{|{\bf v-v}_*|\ge \dt\}}B_{0}^{\vep,\eta}({\bf v}- {\bf v}_*,\og)
{f^{\vep}}'{f_*^{\vep}}'\big|\psi+\psi_*-\psi'-\psi_*'\big|
{\rm d}{\bf v}{\rm d}{\bf v}_*{\rm d}\og\\
&&=
\fr{1}{4}\int_{{\bRRS}}1_{\{|{\bf v-v}_*|\ge \dt\}}B_{0}^{\vep,\eta}({\bf v}- {\bf v}_*,\og)
{f^{\vep}}{f_*^{\vep}}\big|\psi+\psi_*-\psi'-\psi_*'\big|
{\rm d}{\bf v}{\rm d}{\bf v}_*{\rm d}\og
\\
&&\le \fr{\|D^2\psi\|_{\infty}}{4\dt\eta}
\int_{{\bRR}}\bigg(|{\bf v-v}_*|^3
\int_{{\bS}}B_{0}^{\vep}({\bf v}- {\bf v}_*,\og)\sin^2(\theta)\cos^2(\theta){\rm d}\og\bigg)
{f^{\vep}}{f_*^{\vep}}{\rm d}{\bf v}{\rm d}{\bf v}_*
\\
&&\le \fr{\|D^2\psi\|_{\infty}}{\dt\eta}
\int_{{\bRR}}
{f^{\vep}}{f_*^{\vep}}{\rm d}{\bf v}{\rm d}{\bf v}_*<\infty.
\eeas
From this, we derive that (\ref{3.QQ}) holds for the collision kernel
$B_{0}^{\vep,\eta}({\bf v}- {\bf v}_*,\og)$, i.e.
\be{\cal Q}^{\vep,\eta}_{\ge \dt}[\psi](t)=
\int_{{\bRR}}1_{\{|{\bf v-v}_*|\ge \dt\}}L^{\vep,\eta}_{0}[\psi]({\bf v}, {\bf v}_*)
f^{\vep}(t,{\bf v})f^{\vep}(t,{\bf v}_*){\rm d}{\bf v}{\rm d}{\bf v}_*.\lb{3.QQQ}\ee
By Lemma \ref{Lemma2.2} and (\ref{1.BB}), we have for all ${\bf v,v}_*$ satisfying
$|{\bf v-v}_*|\ge \dt$ that
\beas\sup_{0\le \eta<1}|L^{\vep,\eta}_{0}[\psi]({\bf v}, {\bf v}_*)|
\le \fr{1}{\dt}32\|D^2\psi\|_{\infty},\quad
\lim_{\eta\to 0^{+}}L^{\vep,\eta}_{0}[\psi]({\bf v}, {\bf v}_*)=
L^{\vep}_{0}[\psi]({\bf v}, {\bf v}_*).\eeas
Thus using dominated convergence theorem and letting $\eta\to 0^+$, we conclude from (\ref{3.QQQ})
that
$$\lim_{\eta\to 0^{+}}{\cal Q}^{\vep,\eta}_{\ge \dt}[\psi](t)=\int_{{\bRR}}1_{\{|{\bf v-v}_*|\ge \dt\}}L^{\vep}_{0}[\psi]({\bf v}, {\bf v}_*)
f^{\vep}(t,{\bf v})f^{\vep}(t,{\bf v}_*){\rm d}{\bf v}{\rm d}{\bf v}_*. $$
This proves (\ref{3.QQ}).

Now let
\bes&&\psi\in C^2_c({\bR}),\quad
\Psi({\bf v,v}_*)=\nabla \psi({\bf v})-\nabla \psi({\bf v}_*),\quad
\bar{\Psi}({\bf w},{\bf z})=\Psi\big(\fr{{\bf w+z}}{2}, \fr{{\bf w-z}}{2}
\big),\lb{6.LLL}\\
&&{\cal Q}_{L}[\psi](t):=-\fr{1}{4}
\int_{{\bRR}}\sqrt{\bar{F}(t,{\bf w},{\bf z})}\bar{\Psi}({\bf w},{\bf z})\cdot \Pi({\bf z})\nabla_{{\bf z}}\sqrt{\bar{F}(t,{\bf w},{\bf z})}\,
{\rm d}{\bf z}{\rm d}{\bf w},\nonumber\\
&& {\cal Q}_{L,\dt}[\psi](t):=-\fr{1}{4}
\int_{{\bRR}}\sqrt{\bar{F}_{\dt}(t,{\bf w},{\bf z})}\bar{\Psi}({\bf w},{\bf z})\cdot \Pi({\bf z})\nabla_{{\bf z}}\sqrt{\bar{F}_{\dt}(t,{\bf w},{\bf z})}\,
{\rm d}{\bf z}{\rm d}{\bf w}.\nonumber\ees
Before going further, we need to show that
\be\lb{3.QL} {\cal Q}_{L}[\psi](t)=\int_{{\bR}}\psi({\bf v})Q_L(f)(t,{\bf v}){\rm d}{\bf v},\quad t\ge 0,\ee
where the righthand side is defined in (\ref{H-solu}).
  In fact, from the existence of the weak projection gradient
in ${\bf v-v}_*\neq {\bf 0}$ for the product $f(t,{\bf v})f(t,{\bf v}_*)$ proved in Step 5,
(\ref{3.Sq}), and (\ref{3.Sq4}) we have  (note that all integrals below have been proven to be absolutely convergent for almost all $t>0$)
\beas&&\int_{{\bRR}}\sqrt{\bar{F}(t,{\bf w},{\bf z})}\bar{\Psi}({\bf w},{\bf z})\cdot \Pi({\bf z})\nabla_{{\bf z}}\sqrt{\bar{F}(t,{\bf w},{\bf z})}\,
{\rm d}{\bf z}{\rm d}{\bf w}
\\
&&=8\int_{{\bRR}}\fr{1}{|{\bf z}|}\bar{\Psi}({\bf w},{\bf z})\cdot \sqrt{|{\bf z}|\bar{F}(t,{\bf w},{\bf z})}\sqrt{|{\bf z}|}\Pi({\bf z})\nabla_{{\bf z}}\sqrt{\bar{F}(t,{\bf w},{\bf z})}\Big|_{{\bf w}={\bf v+v}_*,
 {\bf z}={\bf v-v}_*} \,
{\rm d}{\bf v}{\rm d}{\bf v}_*
\\
&&=2\int_{{\bRR}}\fr{\nabla \psi({\bf v})-\nabla \psi({\bf v}_*)}{|{\bf v-v}_*|}\cdot
\Pi({\bf v-v}_*)\nabla_{{\bf v-v}_*}
\big(f(t,{\bf v})f(t,{\bf v}_*)\big)
{\rm d}{\bf v}{\rm d}{\bf v}_*.\eeas
Therefore
$${\cal Q}_{L}[\psi](t)=-\fr{1}{2}
\int_{{\bRR}}
\fr{\nabla \psi({\bf v})-\nabla \psi({\bf v}_*)}{|{\bf v-v}_*|}
\cdot \Pi({\bf v-v}_*)\nabla_{{\bf v-v}_*}
\big(f(t,{\bf v})f(t,{\bf v}_*)\big)
{\rm d}{\bf v}{\rm d}{\bf v}_*$$
and so (\ref{3.QL}) holds by definition of
$\int_{{\bR}}\psi({\bf v})Q_L(f)(t,{\bf v}){\rm d}{\bf v}$ in (\ref{H-solu}),

Next recalling (\ref{1.LL}) for $L[\psi]({\bf v}, {\bf v}_*)$ and
using the relation in (\ref{6.LLL}),  we have
\beas L[\psi]({\bf v}, {\bf v}_*)\Big|_{{\bf v}=\fr{{\bf w+z}}{2}, {\bf v}_*=\fr{{\bf w-z}}{2}}
=\fr{1}{|{\bf z}|}\nabla_{{\bf z}}\cdot \Pi({\bf z})\bar{\Psi}({\bf w},{\bf z}),\eeas
which implies
\beas&&\big(f(t,{\bf v})f(t,{\bf v}_*)1_{\{|{\bf v-v}_*|\ge \dt\}}+
|{\bf v}-{\bf v}_*|\dt e^{-(|{\bf v}|^2+|{\bf v}_*|^2)}\big)L[\psi]({\bf v}, {\bf v}_*)\Big|_{{\bf v}=\fr{{\bf w+z}}{2}, {\bf v}_*=\fr{{\bf w-z}}{2}}
\\
&&=\bar{F}_{\dt}(t,{\bf w},{\bf z})\nabla_{{\bf z}}\cdot \Pi({\bf z})\bar{\Psi}({\bf w},{\bf z})\eeas
and taking integration with change of variable $({\bf v},{\bf v}_*)=(\fr{{\bf w+z}}{2},\fr{{\bf w-z}}{2})$
 \bes\lb{QLdelta}
&&\int_{{\bRR}}\big(f(t,{\bf v})f(t,{\bf v}_*)1_{\{|{\bf v-v}_*|\ge \dt\}}+
|{\bf v}-{\bf v}_*|\dt e^{-(|{\bf v}|^2+|{\bf v}_*|^2)}\big)L[\psi]({\bf v}, {\bf v}_*)
{\rm d}{\bf v}{\rm d}{\bf v}_*\\
&&=\fr{1}{8}\int_{{\bRR}}\bar{F}_{\dt}\nabla_{{\bf z}}\cdot \Pi({\bf z})\bar{\Psi}{\rm d}{\bf z}{\rm d}{\bf w} =-\fr{1}{8}\int_{{\bRR}}\bar{\Psi}\cdot \Pi({\bf z})
\nabla_{{\bf z}}\bar{F}_{\dt}\, {\rm d}{\bf z}{\rm d}{\bf w}\nonumber\\
&&=-\fr{1}{4}\int_{{\bRR}}\sqrt{\bar{F}_{\dt}(t,{\bf w},{\bf z})}\bar{\Psi}({\bf w},{\bf z})\cdot\Pi({\bf z})
\nabla_{{\bf z}}\sqrt{\bar{F}_{\dt}(t,{\bf w},{\bf z})}\, {\rm d}{\bf z}{\rm d}{\bf w}
={\cal Q}_{L,\dt}[\psi](t)\nonumber\ees
where we have used the equality (\ref{3.Squ}).
Note that here $\bar{\Psi}$ may not
have compact support, but the second equality sign in \eqref{QLdelta} is easily proven to hold true by using smooth cutoff approximation $\bar{\Psi}_{n}({\bf w}, {\bf z})=
\zeta_{n}(|{\bf w}|^2+|{\bf z}|^2)\bar{\Psi}({\bf w}, {\bf z})$
and the equality
$\zeta_{n}(|{\bf w}|^2+|{\bf z}|^2)
\nabla_{{\bf z}}\cdot \Pi({\bf z})\bar{\Psi}({\bf w},{\bf z})=\nabla_{{\bf z}}\cdot \Pi({\bf z})\bar{\Psi}_{n}({\bf w},{\bf z})$ (see \eqref{6.19}), where
$\zeta_{n}(r)=\zeta(r/n), n\in{\mN},$ and $\zeta\in C^{\infty}_c({\mR}_{\ge 0})$ is the function used in the second part of Step 1.

Now we compute for any $0<T<\infty$ and $\dt>0$
\beas&&\bigg|\int_{0}^{T}{\cal Q}^{\vep}[\psi](t){\rm d}t
-\int_{0}^{T}{\cal Q}_{L}[\psi](t){\rm d}t
\bigg|\\
&&=\bigg|\int_{0}^{T}{\cal Q}_{<\dt}^{\vep}[\psi](t){\rm d}t
+\int_{0}^{T}{\cal Q}_{\ge\dt}^{\vep}[\psi](t){\rm d}t
-\int_{0}^{T}\int_{{\bRR}}1_{\{|{\bf v-v}_*|\ge \dt\}}L[\psi]
ff_*{\rm d}{\bf v}{\rm d}{\bf v}_*{\rm d}t
\\
&&+\int_{0}^{T}\int_{{\bRR}}1_{\{|{\bf v-v}_*|\ge \dt\}}L[\psi]
ff_*{\rm d}{\bf v}{\rm d}{\bf v}_*{\rm d}t
-\int_{0}^{T}{\cal Q}_{L}[\psi](t){\rm d}t
\bigg|
\\
&&\le \int_{0}^{T}|{\cal Q}_{<\dt}^{\vep}[\psi](t)|{\rm d}t
\\
&& +\bigg|
\int_{0}^{T}\int_{{\bRR}}1_{\{|{\bf v-v}_*|\ge \dt\}}L_0^{\vep}[\psi]
f^{\vep}f^{\vep}_*{\rm d}{\bf v}{\rm d}{\bf v}_*{\rm d}t
-\int_{0}^{T}\int_{{\bRR}}1_{\{|{\bf v-v}_*|\ge \dt\}}L[\psi]
ff_*{\rm d}{\bf v}{\rm d}{\bf v}_*{\rm d}t\bigg|
\\
&&+\bigg|\int_{0}^{T}\int_{{\bRR}}1_{\{|{\bf v-v}_*|\ge \dt\}}L[\psi]ff_*{\rm d}{\bf v}{\rm d}{\bf v}_*{\rm d}t
-\int_{0}^{T}{\cal Q}_{L}[\psi](t){\rm d}t
\bigg|
\\
&&:= I_{\vep,\dt}+J_{\vep,\dt}+K_{\dt}.\eeas
By \eqref{3.Q4}, we  first have
\beas&&
\sup_{\vep\in {\cal E}}I_{\vep,\dt}
\le C_0\|D^2\psi\|_{*} \sup_{\vep\in {\cal E}}\int_{0}^{T}\sqrt{D_{0}^{\vep}(f^{\vep}(t))}{\rm d}t\sqrt{\dt}
\le C_0\sqrt{T}\sqrt{\dt}.\eeas
And using Lemma \ref{Lemma2.4} and Lemma \ref{Lemma2.5} yield
$\lim\limits_{{\cal E}\ni \vep\to 0}J_{\vep,\dt}=0$ for all $\dt>0.$
Finally using \eqref{QLdelta} we have
\bes \lb{3.14} K_{\dt} &\le &
 \bigg|\int_{0}^{T}{\cal Q}_{L,\dt}[\psi](t){\rm d}t-\int_{0}^{T}{\cal Q}_{L}[\psi](t){\rm d}t
\bigg| \\
&+& \dt\int_{0}^{T}\int_{{\bRR}}|{\bf v}-{\bf v}_*| e^{-(|{\bf v}|^2+|{\bf v}_*|^2)}|L[\psi]|
ff_*{\rm d}{\bf v}{\rm d}{\bf v}_*{\rm d}t.
\nonumber \ees
The first term in the righthand of (\ref{3.14}) tends to zero as
$\dt\to 0^{+}$ by (\ref{3.Psi}) in Step 4 for the $L^2$-integrable functions
$\Psi_{\dt}(t,{\bf w},{\bf z})=1_{[0,T]}(t)\sqrt{\bar{F}_{\dt}(t,{\bf w},{\bf z})}\,
\bar{\Psi}({\bf w},{\bf z}),
\Psi(t,{\bf w},{\bf z})=1_{[0,T]}(t)\sqrt{\bar{F}(t,{\bf w},{\bf z})}\,
\bar{\Psi}({\bf w},{\bf z})$.
The second term in the righthand of (\ref{3.14}) obviously tends to zero as
$\dt\to 0^{+}$ by (\ref{2.20}) in Lemma \ref{Lemma2.4}.
We thus conclude the convergence
\be \lim_{{\cal E}\ni\vep\to 0}\int_{0}^{T}{\cal Q}^{\vep}[\psi](t){\rm d}t
=\int_{0}^{T}{\cal Q}_{L}[\psi](t){\rm d}t\qquad \forall\, T\in (0,\infty).\lb{3.19}\ee
From the above convergence  we obtain that
\beas&&
\int_{{\mathbb R}^3}\psi({\bf v}) f(T,{\bf v}){\rm d}{\bf v}
=\lim_{{\cal E}\ni\vep\to 0}\int_{{\mathbb R}^3}\psi({\bf v})f^{\vep}(T,{\bf v}){\rm d}{\bf v}\\
&&=\lim_{{\cal E}\ni\vep\to 0}\int_{{\mathbb R}^3}\psi({\bf v})f^{\vep}_0({\bf v}){\rm d}{\bf v}
+\lim_{{\cal E}\ni\vep\to 0}\int_{0}^{T}{\cal Q}^{\vep}[\psi](t){\rm d}t
\\
&&=\int_{{\mathbb R}^3}\psi({\bf v}) f_0({\bf v}){\rm d}{\bf v}
+\int_{0}^{T}{\cal Q}_{L}[\psi](t){\rm d}t \qquad \forall\, T\in [0,\infty).
\eeas
Since $t\mapsto {\cal Q}_{L}[\psi](t) $ belongs to $L^1_{loc}([0,\infty))$, this ensures that
 $t\mapsto\int_{{\mathbb R}^3}\psi({\bf v})f(t,{\bf v}){\rm d}{\bf v}$ is absolutely continuous on $[0,\infty)$ and \eqref{Lweak} holds true:
 $$\fr{{\rm d}}{{\rm d}t}\int_{{\mathbb R}^3}\psi({\bf v}) f(t,{\bf v}){\rm d}{\bf v}=
 {\cal Q}_{L}[\psi](t)\qquad {\rm a.e.}\quad t\in [0,\infty).$$
This together with the conservation law proved in Step 1
shows that $f$ is an $H$-solution of Eq.(FPL) and completes the proof of Theorem \ref{Theorem1}.
\end{proof}

 \section{Preliminary lemmas for Isotropic Functions}

This section is devoted to some preliminary lemmas for isotropic functions which are crucial
for proving the convergence from Eq.(FD) to Eq.(FPL) and from Eq.(BE) to Eq.(FPL).

\subsection{Useful lemmas}

First of all, under the change of variable (\ref{2.OG}) and using (\ref{2.VV}) (for the $\og$-representation (\ref{1.Omega})), one has a thoroughly detailed representation of $|{\bf v}'|^2, |{\bf v}_*'|^2$:
\be
\left\{\begin{array}
{ll}\displaystyle
|{\bf v}'|^2=|{\bf v}|^2\sin^2(\theta)+|{\bf v}_*|^2\cos^2(\theta)
-2\sqrt{|{\bf v}|^2|{\bf v}_*|^2-({\bf v}\cdot {\bf v}_*)^2}\,\cos(\theta)
\sin(\theta){\bf k}\cdot \sg,\\ \\
\displaystyle
|{\bf v}_*'|^2=|{\bf v}|^2\cos^2(\theta)+|{\bf v}_*|^2\sin^2(\theta)+2
\sqrt{|{\bf v}|^2|{\bf v}_*|^2-({\bf v}\cdot{\bf v}_*)^2}\,
\cos(\theta)\sin(\theta){\bf k}\cdot \sg
\end{array}\right.\lb{4.1}\ee
where
$$
{\bf k}=\fr{{\bf h}-({\bf h}\cdot {\bf n}) {\bf n}}{\sqrt{1-({\bf h}\cdot{\bf n})^2}}\in
{\mS}^1({\bf n})\quad{\rm with}\quad  {\bf h}=\fr{{\bf v}+{\bf v}_*}{|{\bf v}+{\bf v}_*|}$$
and for the case that ${\bf v}, {\bf v}_*$ are linearly dependent,
${\bf k}$ is chosen a fixed element in ${\mS}^1({\bf n})$ independent of
$\sg$.

\begin{lemma}\lb{Lemma4.1}For any $y,z\ge 0$, $\og, \og_*\in {\bS},$
let ${\bf v}=\sqrt{2y}\og, {\bf v}_*=\sqrt{2z}\og_*$,
$x=|{\bf v}'|^2/2$ and   $x_*=|{\bf v}_*'|^2/2$.
Then for any $\vp\in C_c^2({\mR}_{\ge 0})$ and   $\theta\in [0,\pi/2]$, we have with the
representation (\ref{4.1})
that
\be \lb{4.2}
\bigg|\int_{{\mathbb S}^1({\bf n})}\big(
\vp(x)+\vp(x_*)-\vp(y)-\vp(z)
\big){\rm d}\sg\bigg| \le C_0\|\vp''\|_{L^{\infty}}\big((y-z)^2
+yz \big(1-(\og\cdot\og_*)^2\big)\cos^2(\theta)\sin^2(\theta),\ee
\bes  \lb{4.3} &&\int_{{\mathbb S}^1({\bf n})}\big(
\vp(x)+\vp(x_*)-\vp(y)-\vp(z)
\big){\rm d}\sg=-2\pi \big(\vp'(y)-\vp'(z)\big)(y-z)\cos^2(\theta)
\\
&&+2\pi\big(\vp''(y)+\vp''(z)\big)yz \big(1-(\og\cdot\og_*)^2\big)
\cos^2(\theta) +R(y,z, \og, \og_*,\theta),\nonumber\ees
where  $R(y,z, \og, \og_*,\theta)$ is bounded by
\bes \lb{4.RR} &&
|R(y,z, \og, \og_*,\theta)|\le C_0\Big(\|\vp''\|_{L^{\infty}}\cos^2(\theta)+
\Lambda_{\vp}\big((\sqrt{y}+\sqrt{z})|{\bf v}-{\bf v}_*||\cos(\theta)|\big)
\Big)\\
&&\times \big((y-z)^2
+yz(1-(\og\cdot\og_*)^2)\big) \cos^2(\theta),\nonumber\ees
$C_0\in(0,\infty)$ is an absolute constant, and
$$\Lambda_{\vp}(\dt)=\sup_{x,y\ge 0,|x-y|\le \dt}|\vp''(x)-\vp''(y)|,\quad \dt\ge 0.$$

\end{lemma}

\begin{proof} Using $x_*-z= -(x-y)$ we have
\bes \lb{4.4} &&\vp(x)+\vp(x_*)-\vp(y)-\vp(z)=
\vp(x)-\vp(y)+\vp(x_*)-\vp(z)
\\
&&=\big(\vp'(y)-\vp'(z)\big)(x-y)+\fr{1}{2}\big(\vp''(y)+\vp''(z)\big)(x-y)^2\nonumber\\
&&
+\int_{0}^{1}(1-t)\Big[\vp''(y+t(x-y))-\vp''(y)
+\vp''(z-t(x-y)-\vp''(z)\Big] (x-y)^2{\rm d}t
.\nonumber \ees
From (\ref{4.1}) together with the facts $y=|{\bf v}|^2/2$ and $z=|{\bf v}_*|^2/2$, we deduce that
$$x-y=(z-y)\cos^2(\theta) -2\sqrt{yz}\sqrt{1-(\og\cdot\og_*)^2}\,
\cos(\theta)\sin(\theta) {\bf k}\cdot\sg.$$
Inserting this into the first and the second terms in the righthand side of (\ref{4.4}),
taking integration and using the fact that
$\int_{{\mathbb S}^1({\bf n})} {\bf k}\cdot \sg {\rm d}\sg=0,
\int_{{\mathbb S}^1({\bf n})}({\bf k}\cdot\sg)^2 {\rm d}\sg=\pi,$
we obtain (\ref{4.3}) with
\beas&&
R(y,z,\og,\og_*,\theta)
=\pi\big(\vp''(y)+\vp''(z)\big) (z-y)^2\cos^4(\theta)
\\
&&+\int_{{\mathbb S}^1({\bf n})}
\int_{0}^{1}(1-t)\Big[
\vp''(y+t(x-y))-\vp''(y)+\vp''(z-t(x-y)-\vp''(z)\Big](x-y)^2{\rm d}t
{\rm d}\sg.
\eeas
Since
$(x-y)^2\le 4\big((z-y)^2 +yz(1-(\og\cdot\og_*)^2)\big)\cos^2(\theta)$, $|x-y|=\fr{1}{2}\big||{\bf v}'|^2-|{\bf v}|^2\big|
\le 2(\sqrt{y}+\sqrt{z})|{\bf v}-{\bf v}_*|\cos(\theta)$
and $\Lambda_{\vp}(2\dt)\le 2\Lambda_{\vp}(\dt)$, we obtain
(\ref{4.RR})  with $C_0=16\pi$.

Next, due to (\ref{4.4}), we have
\be \lb{4.5}
\bigg|\int_{{\mathbb S}^1({\bf n})}\big(
\vp(x)+\vp(x_*)-\vp(y)-\vp(z)
\big){\rm d}\sg
\bigg|\le 52\pi\|\vp''\|_{\infty}
\Big((z-y)^2 +yz(1-(\og\cdot\og_*)^2)\Big)\cos^2(\theta).\ee
By replacing $\theta$ by $\fr{\pi}{2}-\theta \,(\theta\in [0,\pi/2])$, and using
the equality (\ref{2.Sym}), we see that (\ref{4.5}) holds if
  $\cos^2(\theta)$ is replaced by $\sin^2(\theta)$.
Since $\min\{\cos^2(\theta), \sin^2(\theta)\}\le 2\cos^2(\theta)\sin^2(\theta)$, we derive (\ref{4.2}).
\end{proof}
\vskip2mm

Before introducing the next lemma, we state the following facts which are easy to prove. That is,
for any $-1<\alpha<1$ and any $y>0, z>0, y\neq z,$ we have
\bes\lb{4.F1} &&\int_{{\bSS}}\fr{1}{|{\bf v}-{\bf v}_*|^{3+\alpha}}\Big|_{{\bf v}=\sqrt{2y}\og,
{\bf v}_*=\sqrt{2z}\og_*}{\rm d}\og{\rm d}\og_*
=4\pi^2\fr{(\sqrt{y}+\sqrt{z})^{1+\alpha}-|\sqrt{y}-\sqrt{z}|^{1+\alpha}}
{2^{\fr{1+\alpha}{2}}(1+\alpha)\sqrt{yz} |y-z|^{1+\alpha}},
\\
&&\lb{4.F2} \int_{{\bSS}}\fr{1-(\og\cdot\og_*)^2}{|{\bf v}-{\bf v}_*|^{3+\alpha}}\Big|_{{\bf v}=\sqrt{2y}\og,
{\bf v}_*=\sqrt{2z}\og_*}
{\rm d}\og{\rm d}\og_*
\\
&& =16\pi^2\bigg(
\fr{(\sqrt{y}+\sqrt{z})^{3-\alpha}-|\sqrt{y}-\sqrt{z}|^{3-\alpha}}{2^{\fr{3+\alpha}{2}}
(1-\alpha^2)(3-\alpha)y^{3/2}z^{3/2}}
-\fr{(\sqrt{y}+\sqrt{z})^{1-\alpha}+|\sqrt{y}-
\sqrt{z}|^{1-\alpha}}
{2^{\fr{3+\alpha}{2}}(1-\alpha^2)yz}\bigg).\nonumber
\ees
In particular for the case $\alpha=0$, we
have
\bes&&\int_{{\bSS}}\fr{1}{|{\bf v}-{\bf v}_*|^{3}}\Big|_{{\bf v}=\sqrt{2y}\og,
{\bf v}_*=\sqrt{2z}\og_*}{\rm d}\og{\rm d}\og_*
=\fr{4\pi^2\sqrt{2}}{\sqrt{y\vee z}\,|y-z|},\lb{4.F3}\\
&&\int_{{\bSS}}\fr{1-(\og\cdot\og_*)^2}{|{\bf v}-{\bf v}_*|^{3}}
\Big|_{{\bf v}=\sqrt{2y}\og,
{\bf v}_*=\sqrt{2z}\og_*}{\rm d}\og{\rm d}\og_*
=\fr{8\pi^2\sqrt{2}}{3(y\vee z)^{\fr{3}{2}}}.\lb{4.F4}\ees

We are in a position to prove
\begin{lemma}\lb{Lemma4.2}Let $0\le \alpha<1, y\ge 0, z\ge 0, {\bf v}=\sqrt{2y}\og,\,{\bf v}_*=\sqrt{2z}\og_*$. Then
\be \int_{{\mS}^2\times {\mS}^2}\fr{(y-z)^2
+ yz(1-(\og\cdot\og_*)^2)}{|{\bf v}-{\bf v}_*|^{3+\alpha}}
{\rm d}\og
{\rm d}\og_*\le \fr{C_0}{1-\alpha}\big(y^{\fr{1-\alpha}{2}}+z^{\fr{1-\alpha}{2}}\big),\lb{4.10}\ee
\be \int_{{\mS}^2\times {\mS}^2}\fr{(y-z)^2
+ yz(1-(\og\cdot\og_*)^2)}{|{\bf v}-{\bf v}_*|^{3}(|{\bf v}-{\bf v}_*|^{\alpha}\wedge 1)}
{\rm d}\og
{\rm d}\og_*\le \fr{C_0}{1-\alpha}\big(y^{\fr{1}{2}}+y^{\fr{1-\alpha}{2}}+
z^{\fr{1}{2}}+z^{\fr{1-\alpha}{2}}\big),\lb{4.11}\ee
where $C_0$ is  an absolute constant.
\end{lemma}

\begin{proof}In what follows, the notaion $C_0$ is denoted by any
absolute constant   and it may vary in
different lines.

We first prove (\ref{4.10}). By symmetry and approximation (Fatou's Lemma)
we may assume that $y<z$. Due to the formula (\ref{4.F1}), (\ref{4.F2})
and the mean-value theorem, we get
\beas&&\int_{{\bSS}}\fr{(y-z)^2}{|{\bf v}-{\bf v}_*|^{3+\alpha}}{\rm d}\og{\rm d}\og_*
=(z-y)^24\pi^2\fr{(\sqrt{y}+\sqrt{z})^{1+\alpha}-(\sqrt{z}-\sqrt{y})^{1+\alpha}}
{2^{\fr{1+\alpha}{2}}(1+\alpha)\sqrt{yz}(z-y)^{1+\alpha}}\\
&&
=\fr{(z-y)^24\pi^2\xi_1^{\alpha}2\sqrt{y}}
{2^{\fr{1+\alpha}{2}}\sqrt{yz}(z-y)^{1+\alpha}}\le
\fr{(z-y)^24\pi^22^{\alpha} z^{\fr{\alpha}{2}}2\sqrt{y}}
{2^{\fr{1+\alpha}{2}}\sqrt{yz}(z-y)^{1+\alpha}}
\le 4\pi^22^{\fr{1+\alpha}{2}} z^{\fr{\alpha}{2}-\fr{1}{2}}
(z-y)^{1-\alpha}
 \le  4\pi^22^{\fr{1+\alpha}{2}}z^{\fr{1-\alpha}{2}},\\
&&\int_{{\mS}^2\times {\mS}^2}\fr{
yz(1-(\og\cdot\og_*)^2)}{|{\bf v}-{\bf v}_*|^{3+\alpha}}
{\rm d}\og\le
16\pi^2yz
\fr{(\sqrt{y}+\sqrt{z})^{3-\alpha}-(\sqrt{z}-\sqrt{y})^{3-\alpha}}{2^{\fr{3+\alpha}{2}}
(1-\alpha^2)(3-\alpha)y^{3/2}z^{3/2}}
\\
&&=16\pi^2yz
\fr{\xi_2^{2-\alpha}2\sqrt{y}}{2^{\fr{3+\alpha}{2}}
(1-\alpha^2)y^{3/2}z^{3/2}}
\le 16\pi^2z
\fr{2^{2-\alpha}z^{\fr{2-\alpha}{2}}2}{2^{\fr{3+\alpha}{2}}
(1-\alpha^2)z^{3/2}}=\fr{16\pi^2 2^{\fr{3}{2}(1-\alpha)}}{1-\alpha^2}
z^{\fr{1-\alpha}{2}}.\eeas
This proves (\ref{4.10}).
Applying (\ref{4.10}) with the cases $\alpha=0$ and $0<\alpha<1$ respectively,
we have
\beas&& {\rm l.h.s.\,of\,\, (\ref{4.11})}
=\int_{{\mS}^2\times {\mS}^2}\fr{(y-z)^2
+ yz(1-(\og\cdot\og_*)^2)}{|{\bf v}-{\bf v}_*|^{3+\alpha}}
\Big|_{|{\bf v}-{\bf v}_*|\le 1} {\rm d}\og{\rm d}\og_*\\
&&
+
\int_{{\mS}^2\times {\mS}^2}\fr{(y-z)^2
+ yz(1-(\og\cdot\og_*)^2)}{|{\bf v}-{\bf v}_*|^{3}}
\Big|_{|{\bf v}-{\bf v}_*|> 1} {\rm d}\og{\rm d}\og_*
\\
&&
\le\fr{C_0}{1-\alpha}
\big(y^{\fr{1-\alpha}{2}}+z^{\fr{1-\alpha}{2}}\big)
+C_0
\big(y^{\fr{1}{2}}+z^{\fr{1}{2}}\big)
\le \fr{C_0}{1-\alpha}
\big(y^{\fr{1-\alpha}{2}}+z^{\fr{1-\alpha}{2}}+
y^{\fr{1}{2}}+z^{\fr{1}{2}}\big).\eeas
This completes the proof.
\end{proof}

\subsection{Crucial lemma for the convergence of isotropic solutions} The following lemma deals with
the convergence of the quadratic terms and zero limit of the cubic terms for isotropic Eq.(FD) and Eq.(BE).
\begin{lemma}\lb{Lemma4.3} Suppose that
$B^{\vep}_{\ld}({\bf v}-{\bf v}_*,\og)$ is given by (\ref{ker1}),(\ref{ker2}),(\ref{ker3})
with $\ld\in\{-1,+1\}$.
Recall that $L[\vp](y,z), {\cal J}_{\ld}^{\vep}[\vp](y,z)$ and $ {\cal K}_{\ld}^{\vep}[\vp](x,y,z)$ are defined in (\ref{Lker1}), (\ref{JPhi}) and (\ref{KPhi}).
Then for any $\vp\in C_c^2({\mR}_{\ge 0}), F\in {\cal B}_1^{+}({\mR}_{\ge 0})$
we have (for any $\ld\in\{-1,+1\}$ )
\be \sup_{\vep>0}\big|{\cal J}_{\ld}^{\vep}[\vp](y,z)\big| \le
C_{0}\|\vp''\|_{\infty}(y^{1/2}+y^{1/4}+z^{1/2}+z^{1/4})
\qquad \forall\,(y,z)\in{\mR}_{\ge 0}^2,\lb{4.13}\ee
\be \lim_{\vep\to 0^+}\sup_{(y,z)\in[0,R]^2}
\big|{\cal J}_{\ld}^{\vep}[\vp](y,z)-L[\vp](y,z)
\big|=0\qquad \forall\, 0<R<\infty, \lb{4.JL}\ee
\be \bigg|\vep^3
\int_{{\mR}_{\ge 0}^3}{\cal K}_{\ld}^{\vep}[\vp](x,y,z){\rm d}F(x){\rm d}F(y){\rm d}F(z)\bigg|\le
C_0\|\vp''\|_{\infty}A_{\widehat{\phi}}\|F\|_0^3\vep\quad \forall\, \vep>0\lb{4.K}\ee
where $0< C_0<\infty$ is an absolute constant. Consequently for any constant $0<M_0<\infty$ and for the set
${\mathscr F}_0=\{F\in {\cal B}_1^{+}({\mR}_{\ge 0})\,\,|\,\, \|F\|_0\le M_0\}$
we have
\be \lb{4.K*}\lim_{\vep\to 0^+}\sup_{F\in {\mathscr F}_0}\bigg| \vep^3
\int_{{\mR}_{\ge 0}^3}{\cal K}_{\ld}^{\vep}[\vp](x,y,z){\rm d}F(x){\rm d}F(y){\rm d}F(z)\bigg|=0.\ee
\end{lemma}

\begin{proof} The proof consists of three steps.
\vskip1mm
{\bf Step 1. Proof of (\ref{4.13}).}
Thanks to Lemma 5.2 in \cite{CaiLu}, for any $y>0, z>0$, we have
\beas&&\fr{(4\pi)^2}{\sqrt{2}\sqrt{yz}}
\int_{0}^{y+z}{\rm d}x
\int_{|\sqrt{y}-\sqrt{x}|\vee |\sqrt{x_*}-\sqrt{z}|}^{(\sqrt{y}+\sqrt{x})\wedge (\sqrt{z}+
\sqrt{x_*})}
\int_{0}^{2\pi}\Phi_{\ld}^{\vep}\big(\sqrt{2} s, \sqrt{2} Y_*\big)\Dt\vp(x,y,z)
{\rm d}\theta {\rm d}s\\
&&=
\int_{{\mS}^2\times{\mS}^2}{\rm d}\sg{\rm d}\sg_*
\int_{{\mS}^2}|({\bf v}-{\bf v}_*)\cdot\og|\Phi_{\ld}^{\vep}
\big(|{\bf v}'-{\bf v}|, |{\bf v}_*'-{\bf v}|\big)\\
&&\times \Big(\vp(
|{\bf v}'|^2/2)+\vp(|{\bf v}_*'|^2/2)-\vp(|{\bf v}|^2/2)-
\vp(|{\bf v}_*|^2/2)\Big)\Big|_{{\bf v}=\sqrt{2y}\sg,\,
{\bf v}_*=\sqrt{2z}\sg_*}
{\rm d}\og.\eeas
Then by the   definition of
${\cal J}_{\ld}^{\vep}[\vp]$,
for any $y>0, z>0$, we have
\bes \lb{4.JJ} &&{\cal J}_{\ld}^{\vep}[\vp](y,z)=\fr{1}{4\pi\sqrt{2}}
\int_{{\mS}^2\times{\mS}^2}{\rm d}\sg{\rm d}\sg_*
\int_{{\mS}^2}|({\bf v}-{\bf v}_*)\cdot\og|
\Phi_{\ld}^{\vep}
\big(|{\bf v}'-{\bf v}|, |{\bf v}_*'-{\bf v}|\big)\\
&&\times \Dt\vp\big(
|{\bf v}'|^2/2, |{\bf v}|^2/2, |{\bf v}_*|^2/2\big)\big|_{{\bf v}=\sqrt{2y}\sg,\,
{\bf v}_*=\sqrt{2z}\sg_*}
{\rm d}\og\nonumber\\
&&=\fr{2}{4\pi\sqrt{2}}
\int_{{\mS}^2\times{\mS}^2}L^{\vep}_{0}[\vp(|\cdot|^2/2)]({\bf v}, {\bf v}_*)\big|_{{\bf v}=\sqrt{2y}\sg,\,
{\bf v}_*=\sqrt{2z}\sg_*}
{\rm d}\sg{\rm d}\sg_*\nonumber\\
&&+
\fr{2}{4\pi\sqrt{2}}
\int_{{\mS}^2\times{\mS}^2}{\ld}E^{\vep}
[\vp(|\cdot|^2/2)]({\bf v}, {\bf v}_*)\big|_{{\bf v}=\sqrt{2y}\sg,\,
{\bf v}_*=\sqrt{2z}\sg_*}
{\rm d}\sg{\rm d}\sg_*,\nonumber
\ees
where we have used (\ref{2.LE}), i.e., $L^{\vep}_{\ld}[\vp]=L^{\vep}_{0}[\vp]
+{\ld}E^{\vep}[\vp]$.

Let ${\bf z}={\bf v}-{\bf v}_*$ and write ${\bf v}=\sqrt{2y}\og,\,
{\bf v}_*=\sqrt{2z}\og_*$,
$x=|{\bf v}'|^2/2, x_*=|{\bf v}_*'|^2/2$.
From (\ref{2.L}), (\ref{2.E}) and the representation (\ref{4.1}) we have
\beas&& L^{\vep}_0[\vp(|\cdot|^2/2)]({\bf v}, {\bf v}_*)
=2
\int_{0}^{\pi/2}
\fr{|{\bf z}|\cos(\theta)\sin(\theta)}{\vep^4}
\widehat{\phi}\Big(\fr{|{\bf z}|\cos(\theta)}{\vep}\Big)^2
\\
&&\times\bigg( \int_{{\mS}^1({\bf n})}\big(
\vp(x)+\vp(x_*)-\vp(y)-\vp(z)
\big){\rm d}\sg\bigg){\rm d}\theta,\\ \\
&& {\ld}E^{\vep}[\vp(|\cdot|^2/2)]({\bf v}, {\bf v}_*)=\ld 2\int_{0}^{\pi/2}
\fr{|{\bf z}|\cos(\theta)\sin(\theta)}{\vep^4}
\widehat{\phi}\Big(\fr{|{\bf z}|\cos(\theta)}{\vep}\Big)
\widehat{\phi}\Big(\fr{|{\bf z}|\sin(\theta)}{\vep}\Big)\\
&&\times\bigg(\int_{{\mS}^1({\bf n})}
\big(\vp(x)+\vp(x_*)-\vp(y)-\vp(z)
\big)
{\rm d}\sg\bigg){\rm d}\theta.
\eeas

\noindent Estimate of $E^{\vep}[\vp(|\cdot|^2/2)]$:\,  We first prove that
\bes \lb{4.EE} &&
\fr{2}{4\pi\sqrt{2}}
\int_{{\mS}^2\times{\mS}^2}\big|E^{\vep}[\vp(|\cdot|^2/2)]({\bf v}, {\bf v}_*)\big|
\big|_{{\bf v}=\sqrt{2y}\og,\,
{\bf v}_*=\sqrt{2z}\og_*}
{\rm d}\og{\rm d}\og_*\\
&&\le C_0\|\vp''\|_{\infty}A_{\wh{\phi},2}^*(\vep)
\big(y^{\fr{1}{2}}+y^{\fr{1}{4}}+
z^{\fr{1}{2}}+z^{\fr{1}{4}}\big),\quad y>0, z>0.\nonumber
\ees
For any $y>0, z>0$, thanks to Lemma \ref{Lemma4.1} and Cauchy-Schwartz inequality (observing that  $\cos^3(\theta)\sin^3(\theta)=\cos^{5/2}(\theta)\sqrt{\sin(\theta)}\cdot
\sin^{5/2}(\theta)\sqrt{\cos(\theta)}$), in the case ${\bf z}\neq {\bf 0}$,
we have
\beas&&\big|
E^{\vep}[\vp(|\cdot|^2/2)]({\bf v}, {\bf v}_*)\big|\\
&&\le C_0\int_{0}^{\pi/2}
\fr{|{\bf z}|\cos^3(\theta)\sin^3(\theta)}{\vep^4}\Big|
\widehat{\phi}\Big(\fr{|{\bf z}|\cos(\theta)}{\vep}\Big)\Big|\Big|
\widehat{\phi}\Big(\fr{|{\bf z}|\sin(\theta)}{\vep}\Big)\Big|{\rm d}\theta
\\
&&\times \|\vp''\|_{\infty}
\big((y-z)^2
+yz(1-(\og\cdot\og_*)^2)\big)
\\
&&\le C_0
\int_{0}^{1}\fr{|{\bf z}|t^5}{\vep^4}
\Big|
\widehat{\phi}\Big(\fr{|{\bf z}|t}{\vep}\Big)\Big|^2{\rm d}t
\|\vp''\|_{\infty}\big((y-z)^2
+yz(1-(\og\cdot\og_*)^2)\big)
\\
&&\le C_0\|\vp''\|_{\infty} A_{\phi, 2}^*(\vep)
\fr{(y-z)^2
+yz(1-(\og\cdot\og_*)^2)}{|{\bf z}|^3(|{\bf z}|^{1/2}\wedge 1)},
\eeas
where we have used  Lemma \ref{Lemma4.2} with $\alpha=2$.
Thus by \eqref{4.11}, one has
\bes\lb{Elambdae}&&\fr{2}{4\pi\sqrt{2}}
\int_{{\mS}^2\times{\mS}^2}\big|E^{\vep}[\vp(|\cdot|^2/2)]
({\bf v}, {\bf v}_*)\big|\big|_{{\bf v}=\sqrt{2y}\og,\,
{\bf v}_*=\sqrt{2z}\og_*}
{\rm d}\og{\rm d}\og_*\\
&&\le C_0\|\vp''\|_{\infty}A_{\wh{\phi},2}^*(\vep)\int_{{\mS}^2\times{\mS}^2}
\fr{(y-z)^2
+yz\big(1-(\og\cdot\og_*)^2\big)}{
|{\bf z}|^3 (|{\bf z}|^{1/2}\wedge 1)}
\Big|_{{\bf v}=\sqrt{2y}\og,\,{\bf v}_*=\sqrt{2z}\og_*} {\rm d}\og{\rm d}\og_*\nonumber
\\
&&\le C_0\|\vp''\|_{\infty}A_{\wh{\phi},2}^*(\vep)
\big(y^{\fr{1}{2}}+y^{\fr{1}{4}}+
z^{\fr{1}{2}}+z^{\fr{1}{4}}\big),\quad y>0, z>0.\nonumber\ees

\noindent Estimate of $L^{\vep}_0[\vp(|\cdot|^2/2)]$:\,
Using $2|ab|\le a^2+b^2$ and Lemma \ref{Lemma4.1}, we have
\beas&&|L^{\vep}_0[\vp(|\cdot|^2/2)]({\bf v}, {\bf v}_*)| \le 4
\int_{0}^{\pi/2}
\fr{|{\bf z}|\cos(\theta)\sin(\theta)}{\vep^4}
\widehat{\phi}\Big(\fr{|{\bf z}|\cos(\theta)}{\vep}\Big)^2
\\
&&\times\bigg|\int_{{\mS}^1({\bf n})}\big(
\vp(x)+\vp(x_*)-\vp(y)-\vp(z)\big)
{\rm d}\sg\bigg|
\\
&&\le
C_0\|\vp''\|_{L^{\infty}}\big((y-z)^2
+yz \big(1-(\og\cdot\og_*)^2\big)
\int_{0}^{\pi/2}
\fr{|{\bf z}|\cos^3(\theta)\sin^3(\theta)}{\vep^4}
\widehat{\phi}\Big(\fr{|{\bf z}|\cos(\theta)}{\vep}\Big)^2
{\rm d}\theta\\
&&\le C_0\|\vp''\|_{L^{\infty}}\fr{(y-z)^2
+yz (1-(\og\cdot\og_*)^2)}{|{\bf z}|^3}
.\eeas
Thus for any $y>0, z>0$, by Lemma \ref{Lemma4.2},
we have
\bes \lb{EL0e}  && \fr{2}{4\pi\sqrt{2}}
\int_{{\mS}^2\times{\mS}^2}|L^{\vep}_{0}[\vp(|\cdot|^2/2)]({\bf v}, {\bf v}_*)|
{\rm d}\og{\rm d}\og_*
\\
& & \le C_0\|\vp''\|_{\infty}
\int_{{\mS}^2\times{\mS}^2}
\fr{(y-z)^2
+yz(1-(\og\cdot\og_*)^2)}{|{\bf z}|^3}
{\rm d}\og{\rm d}\og_*
\le C_0\|\vp''\|_{\infty} \big( y^{1/2}+z^{1/2}\big).\nonumber\ees

Putting together  the estimates \eqref{Elambdae} and \eqref{EL0e}, we derive  (\ref{4.13})  thanks to \eqref{4.JJ}. It is easy to see that
  (\ref{4.13}) holds also for $y=0$ or $z=0$.
\vskip1mm

{\bf Step 2. Proof of (\ref{4.JL}).}
Applying (\ref{2.L}) to
$L^{\vep}_{0}[\vp(|\cdot|^2/2)]({\bf v}, {\bf v}_*)$ and recalling
$x:=|{\bf v}'|^2/2, x_*:=|{\bf v}'_*|^2/2$ and (\ref{4.1}) we have
\beas&&{\cal J}_{0}^{\vep}[\vp](y,z):=\fr{2}{4\pi\sqrt{2}}
\int_{{\mS}^2\times{\mS}^2}L^{\vep}_{0}[\vp(|\cdot|^2/2)]({\bf v}, {\bf v}_*)\Big|_{{\bf v}=\sqrt{2y}\og,\,{\bf v}_*=\sqrt{2z}\og_*}
{\rm d}\og{\rm d}\og_*\\
&&=
\fr{2}{4\pi\sqrt{2}}
\int_{{\mS}^2\times{\mS}^2}2\Bigg\{
\int_{0}^{\pi/2}
\fr{|{\bf z}|\cos(\theta)\sin(\theta)}{\vep^4}
\widehat{\phi}\Big(\fr{|{\bf z}|\cos(\theta)}{\vep}\Big)^2
{\rm d}\theta\\
&&\times \int_{{\mS}^1({\bf n})}
\big(\vp(x)+\vp(x_*)-\vp(y)-\vp(z)\big){\rm d}\sg
\Bigg\}{\rm d}\og{\rm d}\og_*\\
&&=
\fr{1}{\pi\sqrt{2}}
\int_{{\mS}^2\times{\mS}^2}
\int_{0}^{\pi/2}
\fr{|{\bf z}|\cos^3(\theta)\sin(\theta)}{\vep^4}
\widehat{\phi}\Big(\fr{|{\bf z}|\cos(\theta)}{\vep}\Big)^2
{\rm d}\theta\\
&&\times \Big(-2\pi\big(\vp'(y)-\vp'(z)\big)(y-z)
+2\pi\big(\vp''(y)+\vp''(z)\big)yz(1-(\og\cdot\og_*)^2)
\Big)
{\rm d}\og{\rm d}\og_*\\
&&+
\fr{1}{\pi\sqrt{2}}
\int_{{\mS}^2\times{\mS}^2}
\int_{0}^{\pi/2}
\fr{|{\bf z}|\cos(\theta)\sin(\theta)}{\vep^4}
\widehat{\phi}\Big(\fr{|{\bf z}|\cos(\theta)}{\vep}\Big)^2
 R(y,z,\og,\og_*,\theta){\rm d}\theta{\rm d}\og{\rm d}\og_*
\\
&&=:I_{1}^{\vep}(y,z)+I_{2}^{\vep}(y,z)\eeas
where $R(y,z,\og,\og_*,\theta)$ is estimated in (\ref{4.RR}).

\noindent Estimate of $I_{1}^{\vep}(y,z)$:\,
Using  Lemma \ref{Lemma4.2}, (\ref{4.F3}), (\ref{4.F4}) and recalling that ${\bf z}=
{\bf v}-{\bf v}_*, {\bf v}=\sqrt{2y}\og,\,
{\bf v}_*=\sqrt{2z}\og_*$, we have
\beas&&
I_{1}^{\vep}(y,z)=
\fr{1}{\pi\sqrt{2}}
\int_{{\mS}^2\times{\mS}^2}\Big(\fr{1}{2\pi|{\bf z}|^3}-R_{\vep}(|{\bf z}|)
\Big)\\
&&\times \Big(-2\pi\big(\vp'(y)-\vp'(z)\big)(y-z)
+2\pi\big(\vp''(y)+\vp''(z)\big)yz(1-(\og\cdot\og_*)^2)
\Big)
{\rm d}\og{\rm d}\og_*\\
&&=\fr{1}{\pi\sqrt{2}}
\int_{{\mS}^2\times{\mS}^2}
\fr{1}{|{\bf z}|^3}
 \Big(-\big(\vp'(y)-\vp'(z)\big)(y-z)
+\big(\vp''(y)+\vp''(z)\big)yz(1-(\og\cdot\og_*)^2)
\Big)
{\rm d}\og{\rm d}\og_*
\\
&&+\sqrt{2}
\int_{{\mS}^2\times{\mS}^2}R_{\vep}(|{\bf z}|)
 \Big(\big(\vp'(y)-\vp'(z)\big)(y-z)
-\big(\vp''(y)+\vp''(z)\big)yz(1-(\og\cdot\og_*)^2)
\Big)
{\rm d}\og{\rm d}\og_*\\
&&=L[\vp](y,z) +\sqrt{2}
\int_{{\mS}^2\times{\mS}^2}R_{\vep}(|{\bf z}|)
 \Big(\big(\vp'(y)-\vp'(z)\big)(y-z)
-\big(\vp''(y)+\vp''(z)\big)yz(1-(\og\cdot\og_*)^2)
\Big)
{\rm d}\og{\rm d}\og_*.\eeas
This yields that
\beas&& \big|I_{1}^{\vep}(y,z)-L[\vp](y,z)\big|
\le C_0\|\vp''\|_{\infty}
\int_{{\mS}^2\times{\mS}^2}R_{\vep}(|{\bf z}|)\big((y-z)^2+yz(1-(\og\cdot\og_*)^2)\big)
{\rm d}\og{\rm d}\og_*
\\
&&\le C_0\|\vp''\|_{\infty}
A_{\widehat{\phi}}^*(\vep)\int_{{\mS}^2\times{\mS}^2}
\fr{(y-z)^2+yz(1-(\og\cdot\og_*)^2)}{|{\bf z}|^3(|{\bf z}|^{1/2}\wedge 1)}
{\rm d}\og{\rm d}\og_*\\
&&\le C_0\|\vp''\|_{\infty}A_{\widehat{\phi}}^*(\vep)
\big(y^{\fr{1}{2}}+y^{\fr{1}{4}}+z^{\fr{1}{2}}+z^{\fr{1}{4}}\big).\eeas

\noindent Estimate of $I_{2}^{\vep}(y,z)$:\, It is not difficult to see that
\beas&& |I_{2}^{\vep}(y,z)|\le
\fr{1}{\pi\sqrt{2}}
\int_{{\mS}^2\times{\mS}^2}
\int_{0}^{\pi/2}
\fr{|{\bf z}|\cos(\theta)\sin(\theta)}{\vep^4}
\widehat{\phi}\Big(\fr{|{\bf z}|\cos(\theta)}{\vep}\Big)^2
 |R(y,z,\og,\og_*,\theta)|{\rm d}\theta{\rm d}\og{\rm d}\og_*
\\
&&\le C_0
\int_{{\mS}^2\times{\mS}^2}
\int_{0}^{\pi/2}
\fr{|{\bf z}|\cos^3(\theta)\sin(\theta)}{\vep^4}
\widehat{\phi}\Big(\fr{|{\bf z}|\cos(\theta)}{\vep}\Big)^2
 \\
 &&\times \Big(\|\vp''\|_{L^{\infty}}\cos^2(\theta)+
\Lambda_{\vp}\big((\sqrt{y}+\sqrt{z})|{\bf z}||\cos(\theta)|\big)
\Big)\big((y-z)^2
+yz(1-(\og\cdot\og_*)^2)\big){\rm d}\theta{\rm d}\og{\rm d}\og_*
\\
&&= C_0\|\vp''\|_{L^{\infty}}
\int_{{\mS}^2\times{\mS}^2}\bigg(
\int_{0}^{\pi/2}
\fr{|{\bf z}|\cos^5(\theta)\sin(\theta)}{\vep^4}
\widehat{\phi}\Big(\fr{|{\bf z}|\cos(\theta)}{\vep}\Big)^2
 {\rm d}\theta
\bigg)\\
&&\times \big((y-z)^2
+yz(1-(\og\cdot\og_*)^2)\big){\rm d}\og{\rm d}\og_*
\\
&&+C_0
\int_{{\mS}^2\times{\mS}^2}\bigg(
\int_{0}^{\pi/2}
\fr{|{\bf z}|\cos^3(\theta)\sin(\theta)}{\vep^4}
\widehat{\phi}\Big(\fr{|{\bf z}|\cos(\theta)}{\vep}\Big)^2
\Lambda_{\vp}\big((\sqrt{y}+\sqrt{z})|{\bf z}||\cos(\theta)|\big)
{\rm d}\theta
\bigg)\\
&&\times\big((y-z)^2
+yz(1-(\og\cdot\og_*)^2)\big){\rm d}\og{\rm d}\og_*
\\
&&\le C_0\|\vp''\|_{L^{\infty}}A_{\wh{\phi},2}^*(\vep)
\int_{{\mS}^2\times{\mS}^2}
\fr{(y-z)^2
+yz(1-(\og\cdot\og_*)^2)}{|{\bf z}|^3(|{\bf z}|^{1/2}\wedge 1)}
{\rm d}\og{\rm d}\og_*\\
&&+C_0
\int_{{\mS}^2\times{\mS}^2}\fr{(y-z)^2
+yz(1-(\og\cdot\og_*)^2)}{|{\bf z}|^3}
{\rm d}\og{\rm d}\og_*\Og_{\wh{\phi},\vp}(\sqrt{y}+\sqrt{z})\vep
\\
&&\le C_0\Big(\|\vp''\|_{L^{\infty}}A_{\wh{\phi},2}^*(\vep)
+\Og_{\wh{\phi},\vp}\big((\sqrt{y}+\sqrt{z})\vep\big)
\Big)\big(y^{\fr{1}{2}}+y^{\fr{1}{4}}+z^{\fr{1}{2}}+z^{\fr{1}{4}}
\big),\eeas
where
$$\Og_{\phi, \vp}(\dt)=
\int_{0}^{\infty}r^3\widehat{\phi}(r)^2
\Lambda_{\vp}\big(\dt r\big){\rm d}r,\quad \dt\ge 0.$$
Then for any $y>0, z>0$, we have
\beas&&\big|{\cal J}_{0}^{\vep}[\vp](y,z)
-L[\vp](y,z)\big|
\\
&&\le C_0\Big(\|\vp''\|_{L^{\infty}}\big(A_{\widehat{\phi}}^*(\vep)+A_{\wh{\phi},2}^*(\vep)\big)
+\Og_{\wh{\phi},\vp}\big((\sqrt{y}+\sqrt{z})\vep\big)
\Big)\big(y^{\fr{1}{2}}+y^{\fr{1}{4}}+z^{\fr{1}{2}}+z^{\fr{1}{4}}
\big).\eeas
From this together with (\ref{4.JJ}) and (\ref{4.EE}), we derive that
\beas&&\big|{\cal J}_{\ld}^{\vep}[\vp](y,z)
-L[\vp](y,z)\big|
\\
&&\le C_0\Big(\|\vp''\|_{L^{\infty}}\big(A_{\widehat{\phi}}^*(\vep)+A_{\wh{\phi},2}^*(\vep)\big)
+\Og_{\wh{\phi},\vp}\big((\sqrt{y}+\sqrt{z})\vep\big)
\Big)\big(y^{\fr{1}{2}}+y^{\fr{1}{4}}+z^{\fr{1}{2}}+z^{\fr{1}{4}}
\big).\eeas
This inequality holds also for $y=0$ or $z=0$ by continuity. From the above inequality and Lemma \ref{Lemma4.2}, we complete the proof of
(\ref{4.JL}).
\vskip1mm

{\bf Step 3. Proof of (\ref{4.K}).}
We will sufficiently use  symmetric and cancellation properties
of the cubic terms for isotropic measures.
For notational convenience, we set
$$\Dt\vp(x,y,z)=\vp(x)+\vp(x_*)-\vp(y)-\vp(z),\quad x_*=(y+z-x)_{+}.$$
It is easily deduced that
\be |\Dt\vp(x,y,z)|\le \|\vp''\|_{\infty}|x-y||x-z|,\quad x,y,z\ge 0,\, x\le y+z.\lb{4.SS}\ee
To prove (\ref{4.K}), we note that since the function $(x,y,z)\mapsto {\cal K}_{\ld}^{\vep}[\vp](x,y,z)=
W_{\Phi_{\ld}^{\vep}}(x,y,z)\Dt\vp(x,y,z)$ is continuous and bounded on
${\mR}_{\ge 0}^3$, we can use approximation (see Lemma \ref{Lemma5.1} below) to
 assume first that the measure $F$ is absolutely continuous with respect to the
Lebesgue measure on ${\mR}_{\ge 0}$,
that is, there is $0\le g\in L^1({\mR}_{\ge 0})$ such that
${\rm d}F(x)=g(x){\rm d}x.$ Then $ F\otimes F\otimes F$ is also
absolutely continuous with respect to the Lebesgue measure
on ${\mR}_{\ge 0}^3$. This enables us to avoid discussing tedious behavior of
integration of ${\cal K}_{\ld}^{\vep}[\vp](x,y,z)$ on some
sets of Lebesgue measure zero. So we have
\beas&&
\vep^3
\int_{{\mR}_{\ge 0}^3}{\cal K}_{\ld}^{\vep}[\vp]{\rm d}^3F
=\vep^3
\int_{{\mR}_{>0}^3}{\cal K}_{\ld}^{\vep}[\vp]{\rm d}^3F=\vep^3
\int_{{\mR}_{>0}^3}W_{\Phi_{\ld}^{\vep}}(x,y,z)\Dt\vp(x,y,z){\rm d}^3F\\
&&=\vep^34\pi
\int_{{\mR}_{>0}^3, y+z>x}\fr{\Dt\vp(x,y,z) }{\sqrt{xyz}}
\int_{|\sqrt{x}-\sqrt{y}|\vee |\sqrt{x_*}-\sqrt{z}|}
^{(\sqrt{x}+\sqrt{y})\wedge(\sqrt{x_*}+\sqrt{z})}{\rm d}s\\
&&\times
\int_{0}^{2\pi}\fr{1}{\vep^4}\bigg\{\widehat{\phi}\Big(\fr{\sqrt{2} s}{\vep}\Big)^2
+\widehat{\phi}\Big(\fr{\sqrt{2} Y_*}{\vep}\Big)^2
+\ld 2\widehat{\phi}\Big(\fr{\sqrt{2} s}{\vep}\Big)
\widehat{\phi}\Big(\fr{\sqrt{2} Y_*}{\vep}\Big)\bigg\}
{\rm d}\theta {\rm d}^3F.
 \eeas
Here we have used the fact that $W_{\Phi_{\ld}^{\vep}}(x,y,z)=0$ for $y+z\le x$.
By Lemma 5.3 in \cite{CaiLu}, we have
$$
\int_{|\sqrt{x}-\sqrt{y}|\vee |\sqrt{x_*}-\sqrt{z}|}
^{(\sqrt{x}+\sqrt{y})\wedge(\sqrt{x_*}+\sqrt{z})}{\rm d}s
\int_{0}^{2\pi}\widehat{\phi}\Big(\fr{\sqrt{2} Y_*}{\vep}\Big)^2{\rm d}\theta
=2\pi
\int_{|\sqrt{x}-\sqrt{z}|\vee |\sqrt{x_*}-\sqrt{y}|}
^{(\sqrt{x}+\sqrt{z})\wedge(\sqrt{x_*}+\sqrt{y})}
\widehat{\phi}\Big(\fr{\sqrt{2} s}{\vep}\Big)^2{\rm d}s.$$
Thus by exchanging variables $y\leftrightarrow z $ and using the symmetry
 $\Dt\vp(x,z,y)=\Dt\vp(x,y,z)$,  we have
\beas&&\int_{{\mR}_{>0}^3, y+z>x}
\fr{\Dt\vp(x,y,z)}{\sqrt{xyz}}
\int_{|\sqrt{x}-\sqrt{y}|\vee |\sqrt{x_*}-\sqrt{z}|}
^{(\sqrt{x}+\sqrt{y})\wedge(\sqrt{x_*}+\sqrt{z})}{\rm d}s
\int_{0}^{2\pi}\widehat{\phi}\Big(\fr{\sqrt{2} Y_*}{\vep}\Big)^2
{\rm d}\theta{\rm d}^3F\\
\\
&&=2\pi\int_{{\mR}_{>0}^3, y+z>x}
\fr{\Dt\vp(x,y,z)}{\sqrt{xyz}}
\int_{|\sqrt{x}-\sqrt{y}|\vee |\sqrt{x_*}-\sqrt{z}|}
^{(\sqrt{x}+\sqrt{y})\wedge(\sqrt{x_*}+\sqrt{z})}
\widehat{\phi}\Big(\fr{\sqrt{2}s}{\vep}\Big)^2{\rm d}s
{\rm d}^3F.
\eeas
Then
\beas&&\vep^3\int_{{\mR}_{\ge 0}^3}{\cal K}_{\ld}^{\vep}[\vp]{\rm d}^3F
=\fr{16\pi^2}{\vep}\int_{{\mR}_{>0}^3, y+z>x}
\fr{\Dt\vp(x,y,z)}{\sqrt{xyz}}
\int_{|\sqrt{x}-\sqrt{y}|\vee |\sqrt{x_*}-\sqrt{z}|}
^{(\sqrt{x}+\sqrt{y})\wedge(\sqrt{x_*}+\sqrt{z})}
\widehat{\phi}\Big(\fr{\sqrt{2}s}{\vep}\Big)^2{\rm d}s
{\rm d}^3F\\
&&+\ld
\fr{8\pi}{\vep}\int_{{\mR}_{>0}^3, y+z>x}
\fr{\Dt\vp(x,y,z)}{\sqrt{xyz}}
\int_{|\sqrt{x}-\sqrt{y}|\vee |\sqrt{x_*}-\sqrt{z}|}
^{(\sqrt{x}+\sqrt{y})\wedge(\sqrt{x_*}+\sqrt{z})}{\rm d}s
\int_{0}^{2\pi}\widehat{\phi}\Big(\fr{\sqrt{2}s}{\vep}\Big)\widehat{\phi}\Big(\fr{\sqrt{2}Y_*}{\vep}\Big)
{\rm d}\theta {\rm d}^3F
\\
&&=:K^{\vep}_{1}[\vp]+K^{\vep}_{2}[\vp].\eeas

\noindent Estimate of $K^{\vep}_{1}[\vp]$:\,  Notice that
$ \Dt\vp(x,y,z)|_{x=y}=\Dt\vp(x,y,z)|_{x=z}=0.$
We have the following decomposition for
$K^{\vep}_{1}[\vp]$ according to
$x<y\wedge z, y\wedge z<x<y\vee z$ and $y\vee z<x<y+z$ respectively:
\beas&&
K^{\vep}_{1}[\vp]\\
&&
= \fr{16\pi^2}{\vep}\bigg(\int_{0<x<y\le z}+\int_{0<x<z<y}+
\int_{0<y<x\le z}+\int_{0<z<x<y}
+\int_{0<y\le z<x<y+z}+\int_{0<z<y<x<y+z}
\bigg)\\
&&\times \fr{\Dt\vp(x, y,z)}{\sqrt{xyz}}
\int_{|\sqrt{y}-\sqrt{x}|\vee|\sqrt{x_*}-\sqrt{z}|}^{(\sqrt{y}+\sqrt{x})\wedge(\sqrt{x_*}+\sqrt{z})}
\widehat{\phi}\Big(\fr{\sqrt{2}\,s}{\vep}\Big)^2{\rm d}s{\rm d}^3F\\
&&=:K^{\vep}_{1,1}[\vp]+K^{\vep}_{1,2}[\vp]+K^{\vep}_{1,3}[\vp]+
K^{\vep}_{1,4}[\vp]+K^{\vep}_{1,5}[\vp]+K^{\vep}_{1,6}[\vp].\eeas
In the following estimates we will use an equality:
\be
\int_{\sqrt{b}-\sqrt{a}}^{\sqrt{b}+\sqrt{a}}s^{-2}{\rm d}s
=\fr{2\sqrt{a}}{b-a}\qquad {\rm for}\quad 0<a<b.\lb{4.AB}\ee
We first compute the sum $K^{\vep}_{1,1}[\vp]+K^{\vep}_{1,3}[\vp]$.
For $K^{\vep}_{1,1}[\vp]$, it is easy to check that if
$0<x<y\le z$, then $
|\sqrt{y}-\sqrt{x}|\vee|\sqrt{x_*}-\sqrt{z}|=\sqrt{y}-\sqrt{x},
(\sqrt{y}+\sqrt{x})\wedge(\sqrt{x_*}+\sqrt{z})=\sqrt{y}+\sqrt{x}$.
Thus
\beas K^{\vep}_{1,1}[\vp]=\fr{16\pi^2}{\vep}\int_{0<x<y\le z}\fr{\Dt\vp(x,y,z)}{\sqrt{xyz}}
\int_{\sqrt{y}-\sqrt{x}}^{\sqrt{y}+\sqrt{x}}
\widehat{\phi}\Big(\fr{\sqrt{2}\,s}{\vep}\Big)^2{\rm d}s
{\rm d}^3F.\eeas
For  $K^{\vep}_{1,3}[\vp]$, if $0<y<x\le z$, then $
|\sqrt{y}-\sqrt{x}|\vee|\sqrt{x_*}-\sqrt{z}|=\sqrt{x}-\sqrt{y},
(\sqrt{y}+\sqrt{x})\wedge(\sqrt{x_*}+\sqrt{z})=\sqrt{y}+\sqrt{x}$.
By exchanging $x \leftrightarrow y$, we have
\beas K^{\vep}_{1,3}[\vp]
=\fr{16\pi^2}{\vep}\int_{0<x<y\le z}\fr{\Dt\vp(y,x,z)}{\sqrt{xyz}}
\int_{\sqrt{y}-\sqrt{x}}^{\sqrt{y}+\sqrt{x}}
\widehat{\phi}\Big(\fr{\sqrt{2}\,s}{\vep}\Big)^2{\rm d}s{\rm d}^3F.\eeas
Let
$$\Dt_{\rm sym}\vp(x,y,z):= \Dt\vp(x,y,z)+\Dt\vp(y,x, z)
=\vp(z+y-x)+\vp(z+x-y)-2\vp(z).$$
Then
$$|\Dt_{\rm sym}\vp(x,y,z)|\le \|\vp''\|_{\infty}(y-x)^2.$$
By the second inequality in the  assumption (\ref{ker4}) which implies
$\widehat{\phi}(\fr{\sqrt{2}\,s}{\vep})^2\le A_{\widehat{\phi}}^2\fr{\vep^2}{2 s^2}$, we obtain that
\beas \big|K^{\vep}_{1,1}[\vp]+ K^{\vep}_{1,3}[\vp]\big|
&\le&\fr{16\pi^2}{\vep}
\int_{0<x<y\le z}
\fr{|\Dt_{\rm sym}\vp(x,y,z)|}{\sqrt{xyz}}
\int_{\sqrt{y}-\sqrt{x}}^{\sqrt{y}+\sqrt{x}}A_{\widehat{\phi}}^2\fr{\vep^2}{2 s^2}
{\rm d}s
{\rm d}^3F\\
&\le &16\pi^2\|\vp''\|_{\infty}A_{\widehat{\phi}}^2\vep
\int_{0< x<y\le z}\fr{(y-x)^2}{\sqrt{xyz}}
\fr{1}{2}
\int_{\sqrt{y}-\sqrt{x}}^{\sqrt{y}+\sqrt{x}}s^{-2}{\rm d}s
{\rm d}^3F\\
&=& 16\pi^2\|\vp''\|_{\infty}A_{\widehat{\phi}}^2\vep
\int_{0<x<y\le z}\fr{y-x}{\sqrt{yz}}{\rm d}^3F\le 16\pi^2\|\vp''\|_{\infty}A_{\widehat{\phi}}^2\|F\|_0^3\vep
.\eeas
For  $K^{\vep}_{1,2}[\vp]$, if
$ 0<x<z<y$, then $
|\sqrt{y}-\sqrt{x}|\vee|\sqrt{x_*}-\sqrt{z}|=\sqrt{y}-\sqrt{x},\,
(\sqrt{y}+\sqrt{x})\wedge(\sqrt{x_*}+\sqrt{z})=\sqrt{y}+\sqrt{x}.$
Using (\ref{4.SS}),
$\widehat{\phi}(\fr{\sqrt{2}\,s}{\vep})^2\le A_{\widehat{\phi}}^2\fr{\vep^2}{2 s^2}$
and (\ref{4.AB}), we have
\beas\big|K^{\vep}_{1,2}[\vp]\big|&\le&
 16\pi^2\|\vp''\|_{\infty}A_{\widehat{\phi}}^2\vep\int_{0<x<z<y}\fr{(y-x)(z-x)}{\sqrt{xyz}}
\fr{1}{2}
\int_{\sqrt{y}-\sqrt{x}}^{\sqrt{y}+\sqrt{x}}s^{-2}
{\rm d}s
{\rm d}^3F\\
&=& 16\pi^2\|\vp''\|_{\infty}A_{\widehat{\phi}}^2\vep\int_{0<x<z<y}\fr{z-x}{\sqrt{yz}}
{\rm d}^3F\le 16\pi^2\|\vp''\|_{\infty}A_{\widehat{\phi}}^2\|F\|_0^3\vep
.\eeas
For $K^{\vep}_{1,4}[\vp]$, if
$0<z<x<y  $, then $
|\sqrt{y}-\sqrt{x}|\vee|\sqrt{x_*}-\sqrt{z}|=\sqrt{x_*}-\sqrt{z},
(\sqrt{y}+\sqrt{x})\wedge(\sqrt{x_*}+\sqrt{z})=\sqrt{x_*}+\sqrt{z}$.
Thanks to facts (\ref{4.SS}),
$\widehat{\phi}(\fr{\sqrt{2}\,s}{\vep})^2\le A_{\widehat{\phi}}^2\fr{\vep^2}{2 s^2}$,
 (\ref{4.AB}) and   $x_*-z=y-x$,  we have
\beas\big|K^{\vep}_{1,4}[\vp]\big|
&\le& 16\pi^2\|\vp''\|_{\infty}A_{\widehat{\phi}}^2\vep\int_{0<z<x<y}\fr{(y-x)(x-z)}{\sqrt{xyz}}
\fr{1}{2}
\int_{\sqrt{x_*}-\sqrt{z}}^{\sqrt{x_*}+\sqrt{z}}s^{-2}{\rm d}s
{\rm d}^3F
\\
&=& 16\pi^2\|\vp''\|_{\infty}A_{\widehat{\phi}}^2\vep\int_{0<z<x<y}\fr{x-z}{\sqrt{xy}}{\rm d}^3F
\le 16\pi^2\|\vp''\|_{\infty}A_{\widehat{\phi}}^2\|F\|_0^3\vep.\eeas
For  $K^{\vep}_{1,5}[\vp]$, if
$0<y\le z<x<y+z$, then $ |\sqrt{y}-\sqrt{x}|\vee|\sqrt{x_*}-\sqrt{z}|=\sqrt{z}-\sqrt{x_*},\,
(\sqrt{y}+\sqrt{x})\wedge(\sqrt{x_*}+\sqrt{z})=\sqrt{x_*}+\sqrt{z}$.
Due to  (\ref{4.SS}),
$\widehat{\phi}(\fr{\sqrt{2}\,s}{\vep})^2\le A_{\widehat{\phi}}^2\fr{\vep^2}{2 s^2}$, (\ref{4.AB}) and   $z-x_*=x-y $, we have
\beas \big|K^{\vep}_{1,5}[\vp]\big|
&\le& 16\pi^2\|\vp''\|_{\infty}A_{\widehat{\phi}}^2\vep\int_{0<y\le z<x<y+z}\fr{(x-y)(x-z)}{\sqrt{xyz}}
\fr{1}{2}
\int_{\sqrt{z}-\sqrt{x_*}}^{\sqrt{x_*}+\sqrt{z}}s^{-2}{\rm d}s{\rm d}^3F
\\
&=& 16\pi^2\|\vp''\|_{\infty}A_{\widehat{\phi}}^2\vep\int_{0< y\le z<x<y+z}\fr{ (x-z)\sqrt{x_*}}{\sqrt{xyz}}
{\rm d}^3F\\
&\le& 16\pi^2\|\vp''\|_{\infty}A_{\widehat{\phi}}^2\vep\int_{0< y\le z<x<y+z}\fr{y}{\sqrt{xz}}
{\rm d}^3F \quad ({\rm because}\,\, x-z<y,\, x_*< y )
\\
&\le&  16\pi^2\|\vp''\|_{\infty}A_{\widehat{\phi}}^2\|F\|_0^3\vep
.\eeas
For the last term $K^{\vep}_{1,6}[\vp]$, if
$0<z<y<x<y+z$, then $
|\sqrt{y}-\sqrt{x}|\vee|\sqrt{x_*}-\sqrt{z}|=\sqrt{z}-\sqrt{x_*},\,
(\sqrt{y}+\sqrt{x})\wedge(\sqrt{x_*}+\sqrt{z})=\sqrt{x_*}+\sqrt{z}$.
Using (\ref{4.SS}),
$\widehat{\phi}(\fr{\sqrt{2}\,s}{\vep})^2\le A_{\widehat{\phi}}^2\fr{\vep^2}{2 s^2}$,
  (\ref{4.AB}) and  $z-x_*=x-y $, we obtain that
\beas \big|K^{\vep}_{1,6}[\vp]\big|
&\le& 16\pi^2\|\vp''\|_{\infty}A_{\widehat{\phi}}^2\vep\int_{0<z<y<x<y+z}
\fr{(x-y)(x-z)}{\sqrt{xyz}}
\int_{\sqrt{z}-\sqrt{x_*}}^{\sqrt{x_*}+\sqrt{z}}s^{-2}{\rm d}s
{\rm d}^3F\\
&=& 16\pi^2\|\vp''\|_{\infty}A_{\widehat{\phi}}^2\vep\int_{0<z<y<x<y+z}\fr{(x-z)\sqrt{x_*}}{\sqrt{xyz}}
{\rm d}^3F\\
&\le& 16\pi^2\|\vp''\|_{\infty}A_{\widehat{\phi}}^2\vep\int_{0<z<y<x<y+z}\fr{y}{\sqrt{xy}}
{\rm d}^3F\quad ({\rm because}\,\, x-z<y,\, x_*<z )\\
&\le&  16\pi^2\|\vp''\|_{\infty}A_{\widehat{\phi}}^2\|F\|_0^3\vep.\eeas
Patching together all the above estimates, we get
\beas\big|K^{\vep}_{1}[\vp]\big|\le \big|K^{\vep}_{1,1}[\vp]+K^{\vep}_{1,3}[\vp]\big|
+\sum_{2\le i\le 6, i\neq 3}\big|K^{\vep}_{1,i}[\vp]\big|\le
80\pi^2\|\vp''\|_{\infty}A_{\widehat{\phi}}^2\|F\|_0^3\vep.\eeas

\noindent Estimate of $K^{\vep}_{2}[\vp]$:\, We observe that
\beas&& Y_*=Y_*(x,y,z,s,\theta)=\bigg|\sqrt{z-\fr{(x-y+s^2)^2}{4s^2}
}+e^{{\rm i}\theta}\sqrt{x-\fr{(x-y+s^2)^2}{4s^2}}\bigg|\ge |\sqrt{z}-\sqrt{x}|
\\
&&{\rm for\,\,all}\,\,\,  s\in \big[|\sqrt{x}-\sqrt{y}|\vee |\sqrt{x_*}-\sqrt{z}|,
\, (\sqrt{x}+\sqrt{y})\wedge(\sqrt{x_*}+\sqrt{z})\big]\eeas
and
$$ (\sqrt{x}+\sqrt{y})\wedge(\sqrt{x_*}+\sqrt{z})-|\sqrt{x}-\sqrt{y}|\vee |\sqrt{x_*}-\sqrt{z}|
=2\min\{\sqrt{x},\sqrt{x_*},\sqrt{y},\sqrt{z}\}.$$
Using (\ref{4.SS}) and
$|\widehat{\phi}(r)|\le A_{\widehat{\phi}}\fr{1}{r}$ for $ r>0$, we have
\beas
\big|K^{\vep}_{2}[\vp]\big| &\le&
\fr{8\pi}{\vep}\|\vp''\|_{\infty}
\int_{{\mR}_{>0}^3, y+z>x, |y-x||z-x|>0}
\fr{|x-y||x-z|}{\sqrt{xyz}}
\int_{|\sqrt{x}-\sqrt{y}|\vee |\sqrt{x_*}-\sqrt{z}|}
^{(\sqrt{x}+\sqrt{y})\wedge(\sqrt{x_*}+\sqrt{z})}{\rm d}s\\
&\times&
\int_{0}^{2\pi}\Big|\widehat{\phi}\Big(\fr{\sqrt{2}s}{\vep}\Big)\Big|
\Big|\widehat{\phi}\Big(\fr{\sqrt{2}Y_*}{\vep}\Big)\Big|
{\rm d}\theta {\rm d}^3F\\
&\le&
4\pi\|\vp''\|_{\infty}A_{\widehat{\phi}}^2\vep
\int_{{\mR}_{>0}^3, y+z>x, |y-x||z-x|>0}
\fr{|y-x||z-x|}{\sqrt{xyz}}
\int_{|\sqrt{x}-\sqrt{y}|\vee |\sqrt{x_*}-\sqrt{z}|}
^{(\sqrt{x}+\sqrt{y})\wedge(\sqrt{x_*}+\sqrt{z})}\int_{0}^{2\pi}\fr{1}{s Y_*}{\rm d}s
{\rm d}\theta {\rm d}^3F
\\
&\le& 16\pi^2\|\vp''\|_{\infty}A_{\widehat{\phi}}^2\vep\int_{{\mR}_{> 0}^3,  y+z>x, |y-x||z-x|>0}
\fr{|y-x||z-x|}{\sqrt{xyz}}\fr{\min\{\sqrt{x},\sqrt{x_*},\sqrt{y},\sqrt{z}\}}
{|\sqrt{x}-\sqrt{y}||\sqrt{z}-\sqrt{x}|} {\rm d}^3F\\
&=& 16\pi^2\|\vp''\|_{\infty}A_{\widehat{\phi}}^2\vep\int_{{\mR}_{>0}^3, y+z>x, |y-x||z-x|>0}
(\sqrt{y}+\sqrt{x})(\sqrt{z}+\sqrt{x})
\fr{\min\{\sqrt{x},\sqrt{x_*},\sqrt{y},\sqrt{z}\}}{\sqrt{xyz}}
{\rm d}^3F\\
&\le& 16\pi^2\|\vp''\|_{\infty}A_{\widehat{\phi}}^2 4\sqrt{2}\|F\|_0^3\vep,\eeas
where we have used  the inequality
\beas x,y,z>0, y+z>x \Longrightarrow (\sqrt{y}+\sqrt{x})(\sqrt{x}+\sqrt{z})
\fr{\min\{\sqrt{x},\sqrt{x_*},\sqrt{y},\sqrt{z}\}}{\sqrt{xyz}}
\le 4\sqrt{2}.\eeas
Collecting the above estimates and setting   $C_0=(80+ 64\sqrt{2})\pi^2$, we obtain
\beas\bigg|
\vep^3
\int_{{\mR}_{\ge 0}^3}{\cal K}_{\ld}^{\vep}[\vp]{\rm d}^3F
\bigg|
\le C_0\|\vp''\|_{\infty}A_{\widehat{\phi}}^2\|F\|_0^3\vep.\eeas
 \smallskip
Finally we use approximation to extend the above result  to the general measure $F$.
Let $0\le \zeta\in C_{c}^{\infty}({\mR})$ be an even function such that
${\rm supp}\zeta \subset [-1,1]$, $\int_{{\mR}}\zeta(x){\rm d}x=1$ (i.e.
$\int_{{\mR}_{\ge 0}}\zeta(x){\rm d}x=\fr{1}{2}$).
Let
$$g_n(x)=\int_{{\mR}_{> 0}}n\zeta(n(x-y)){\rm d}F(y)+2F(\{0\})n\zeta(nx),\quad x\in
{\mR}_{\ge 0},\, n\in {\mN};\quad
{\rm d}F_n(x)=g_n(x){\rm d}x.$$
Then $0\le g_n\in L^1({\mR}_{\ge 0})$ and
$F_n$ is absolutely continuous with respect the Lebesgue measure on ${\mR}_{\ge 0}$.
For any nonnegative  Borel measurable function $\vp$ on ${\mR}_{\ge 0}$ or any
$\vp\in C_b({\mR}_{\ge 0})$,
we have, by Fubini theorem, that
\beas
\int_{{\mR}_{\ge 0}}\vp(x){\rm d}F_n(x)
=\int_{{\mR}_{> 0}}\bigg(\int_{-ny}^{\infty}\vp\big(y+\fr{u}{n}\big)\zeta(u)
{\rm d}u\bigg){\rm d}F(y)+2F(\{0\})
\int_{0}^{\infty}\vp\big(\fr{u}{n}\big)\zeta(u)
{\rm d}u.\eeas
Taking $\vp(x)\equiv 1$ and $\vp(x)\equiv 1+x$, we have
\beas\|F_n\|_0\le \|F\|_0,\quad \|F_n\|_1\le 2\|F\|_1.\eeas
Using dominated convergence theorem we have
\beas\lim_{n\to\infty}
\int_{{\mR}_{\ge 0}}\vp(x){\rm d}F_n(x)
=
\int_{{\mR}_{\ge 0}}\vp(y){\rm d}F(y)\qquad \forall\, \vp\in C_b({\mR}_{\ge 0}).\eeas
For any $\vp\in C_c^2({\mR}_{\ge 0})$, since the function
${\cal K}_{\ld}^{\vep}[\vp]$ belongs to $C_b({\mR}_{\ge 0}^3)$, it follows from
Lemma \ref{Lemma5.1} with $d=3, s=1$ and $\Psi_n=\Psi= \vep^3{\cal K}_{\ld}^{\vep}[\vp]$
that
$$\lim_{n\to\infty}\vep^3
\int_{{\mR}_{\ge 0}^3}{\cal K}_{\ld}^{\vep}[\vp]{\rm d}^3F_n
=\vep^3
\int_{{\mR}_{\ge 0}^3}{\cal K}_{\ld}^{\vep}[\vp]{\rm d}^3F.$$
Since we have proved that $\big|\vep^3
\int_{{\mR}_{\ge 0}^3}{\cal K}_{\ld}^{\vep}[\vp]{\rm d}^3F_n
\big|\le C_0\|\vp''\|_{\infty}A_{\widehat{\phi}}^2\|F_n\|_0^3\vep$ and
$\|F_n\|_0\le \|F\|_0$ for all $n\in {\mN}$, it follows that
\beas\bigg|
\vep^3
\int_{{\mR}_{\ge 0}^3}{\cal K}_{\ld}^{\vep}[\vp]{\rm d}^3F
\bigg|
=\lim_{n\to\infty}\bigg|\vep^3
\int_{{\mR}_{\ge 0}^3}{\cal K}_{\ld}^{\vep}[\vp]{\rm d}^3F_n
\bigg|
\le C_0\|\vp''\|_{\infty}A_{\widehat{\phi}}^2\|F\|_0^3\vep.\eeas
This proves (\ref{4.K}) and completes the proof of the lemma.
\end{proof}

 \section{Proof of  Theorem \ref{Theorem2} and Theorem \ref{Theorem3} }
In this section we prove Theorem \ref{Theorem2} (from Eq.(FD) to Eq.(FPL)) and
Theorem \ref{Theorem3} (from Eq.(BE) to Eq.(FPL)). We begin by introducing two general known results:

\begin{lemma}\lb{Lemma5.1}{\rm(\cite{Lu2004})}.  Given $s\ge 0$. Let
$\mu_n,\mu\in {\cal B}_s^{+}({\mR}_{\ge 0})$, $n=1,2,3,...$, satisfy
$$\sup_{n\ge 1}\|\mu\|_s=\sup_{n\ge 1}\int_{{\mR}_{\ge 0}}(1+x)^s{\rm d}\mu_n(x)<\infty,$$
$$\lim_{n\to\infty}\int_{{\mR}_{\ge 0}}\psi(x){\rm d}\mu_n(x)=
\int_{{\mR}_{\ge 0}}\psi(x){\rm d}\mu(x)\qquad \forall\, \psi\in C^{\infty}_c({\mR}_{\ge 0}).$$
Let $d\in{\mN}$ and let $\Psi_n, \Psi\in C({\mR}_{\ge 0}^d)$
satisfy
\beas&&\lim_{x_1+x_2+\cdots+x_d\to \infty}\sup_{n\ge 1}
\fr{|\Psi_n(x_1,x_2, ...,x_d)|}{\prod_{i=1}^d (1+x_i)^s}=0,\quad
\lim_{x_1+x_2+\cdots+x_d\to \infty}\fr{|\Psi(x_1,x_2, ...,x_d)|}{\prod_{i=1}^d (1+x_i)^s}=0,\\ \\
&&\lim_{n\to \infty}\sup_{(x_1,x_2, ..., x_d)\in [0, R]^N}|\Psi_n(x_1,x_2, ...,x_d)
-\Psi(x_1,x_2, ...,x_d)|=0\qquad \forall\, 0<R<\infty.\eeas
Then
$$\lim_{n\to \infty}
\int_{{\mR}_{\ge 0}^d}
\Psi_n(x_1,x_2, ...,x_d){\rm d}\mu_n(x_1){\rm d}\mu_n(x_2)\cdots
{\rm d}\mu_n(x_d)$$
$$=\int_{{\mR}_{\ge 0}^d}
\Psi(x_1,x_2, ...,x_d){\rm d}\mu(x_1){\rm d}\mu(x_2)\cdots
{\rm d}\mu(x_d).$$
\end{lemma}

\begin{corollary}\lb{Lemma5.2}{\rm(\cite{Lu2004})}.  Given $s\ge 0$. Let
$\{\mu_t\}_{t\ge 0}\subset {\cal B}_s^{+}({\mR}_{\ge 0})$ satisfy
$\sup_{t\ge 0}\|\mu_t\|_s<\infty$ and, for every $\vp\in C_c^{\infty}({\mR}_{\ge 0})$,
the function $t\mapsto \int_{{\mR}_{\ge 0}}\vp(x){\rm d}\mu_t(x)
$ is continuous on $[0,\infty)$.
Let $d\in{\mN}$ and  $\Psi\in C([0,\infty)\times {\mR}_{\ge 0}^d)$ satisfy
\beas\lim_{x_1+x_2+\cdots+x_d\to \infty}\sup_{t\ge 0}
\fr{|\Psi(t, x_1,x_2, ...,x_d)|}{\prod_{i=1}^d (1+x_i)^s}=0.\eeas
Then the function
$$t\mapsto \int_{{\mR}_{\ge 0}^d}
\Psi(t, x_1,x_2, ...,x_d){\rm d}\mu_t(x_1){\rm d}\mu_t(x_2)\cdots
{\rm d}\mu_t(x_d)\,\,\, is\,\,continuous\,\, on\,\,\, [0,\infty).$$
\end{corollary}

The following lemma gives an equivalent definition for measure-valued isotropic weak solution
of Eq.(FPL).

\begin{lemma}\lb{Lemma5.3}
Let $\{F_t\}_{t\ge 0}\subset {\cal B}_{1}^{+}({\mathbb R}_{\ge 0})$ satisfy
\bes&&
\sup\limits_{t\ge 0}\|F_t\|_1<\infty,\lb{5.FF}\\
&& \int_{{\mR}_{\ge 0}}\vp{\rm d}F_t
=\int_{{\mR}_{\ge 0}}\vp{\rm d}F_0
+\int_{0}^{t}{\rm d}s\int_{{\mR}_{\ge 0}^2}L[\vp]{\rm d}^2F_s\quad \forall\, \vp\in
C_c^2({\mR}_{\ge 0}),\,\,\,\forall\, t\ge 0.\lb{5.InL}\ees
Then $F_t$ is  measure-valued isotropic weak solution
of Eq.(FPL). In particular if ${\rm d}F_t(x)=f(t, x)\sqrt{x}{\rm d}x$ for some
$0\le f\in L^{\infty}([0,\infty), L^1({\mR}_{\ge 0},(1+x)\sqrt{x}{\rm d}x))$, then
$f$ is an isotropic weak solution
of Eq.(FPL).
\end{lemma}

\begin{proof}Recalling \eqref{Lker2} that for any $\vp\in C^2_{c}({\mR}_{\ge 0})$ we have
\be |L[\vp](x,y)|\le C_0\|\vp''\|_{\infty}(\sqrt{x}+\sqrt{y})\qquad \forall\,(x,y)\in{\mR}_{\ge 0}^2.
\lb{5.L}\ee
We first prove that $F_t$ conserves the mass and energy.
Following   the proof of
Theorem \ref{Theorem1}, we consider smooth cutoff approximation
$$\vp_R(x)=(a+bx)\zeta(\fr{x}{R}),\quad x\in {\mR}_{\ge 0},\quad R\ge 1.$$
Here $a,b\in\{0,1\}$, $\zeta\in C_c^{\infty}({\mR}_{\ge 0})$ satisfies
$0\le \zeta(x)\le 1 $ for $ x\in{\mR}_{\ge 0},$ $\zeta(x)=1$ if $x\in [0,1];
\,\zeta(x)=0$ if $x\ge 2$.
We have
\be\vp_R\in C^{\infty}_c({\mR}_{\ge 0}),\quad 0\le \vp_R(x)\le a+bx,
\quad \lim_{R\to\infty}\vp_R(x)=a+bx,\qquad \forall\, x\in {\mR}_{\ge 0}\lb{5.phi}\ee
 for all $R\ge 1$
 and
\be
|\vp_R''(x)|\le b|\zeta'(\fr{x}{R})|\fr{2}{R}+a|\zeta''(\fr{x}{R})|\fr{1}{R^2}
+b\Big|\fr{x}{R}\zeta''(\fr{x}{R})\Big|\fr{1}{R}
\le \fr{C_{\zeta}}{R}\qquad \forall\, x\in{\mR}_{\ge 0},\quad \forall\, R\ge 1.\lb{5.RR}\ee
Here $C_{\zeta}=\sup\limits_{x\ge 0}(2|\zeta'(x)|+(1+x)|\zeta''(x)|).$
 Let $C_*=\sup\limits_{t\ge 0}\|F_t\|_1$.
From (\ref{5.InL}), (\ref{5.L}) and (\ref{5.FF}) we have
\bes&& \int_{{\mR}_{\ge 0}}\vp_R{\rm d}F_t=
\int_{{\mR}_{\ge 0}}\vp_R{\rm d}F_0
+\int_{0}^{t}{\rm d}s\int_{{\mR}_{\ge 0}^2}L[\vp_R]{\rm d}^2F_s,\qquad \forall\, t\in [0,\infty),
\lb{5.Int}\\
&&
\bigg|\int_{0}^{t}{\rm d}s\int_{{\mR}_{\ge 0}^2}L[\vp_R]{\rm d}^2F_s
\bigg|\le \fr{C_0C_{\zeta}C_*^2}{R} t\to 0,\quad as\,\,\, R\to\infty
\qquad \forall\, t\in [0,\infty). \nonumber\ees
Thus letting $R\to\infty$ we obtain from (\ref{5.phi}), (\ref{5.Int}) and dominated convergence theorem that
$$\int_{{\mR}_{\ge 0}}(a+bx){\rm d}F_t(x)=
\int_{{\mR}_{\ge 0}}(a+bx){\rm d}F_0(x)\qquad \forall\, t\in [0,\infty).$$
Since $a,b\in\{0,1\}$ are arbitrary, this proves the conservation of mass and energy for $F_t$.

Next from the equation (\ref{5.InL}) and the bounds (\ref{5.FF}), (\ref{5.L}) we see that
for any $\vp\in C_c^2({\mR}_{\ge 0})$,
the function $t\mapsto \int_{{\mR}_{\ge 0}}\vp{\rm d}F_t$ is continuous
on $[0,\infty)$. Then for any $\vp\in C_c^2({\mR}_{\ge 0})$,
using (\ref{5.L}) and Corollary \ref{Lemma5.2} with $d=2, s=1$, $\mu_t=F_t$
and $\Psi(t,x,y)=L[\vp](x,y)$, we conclude that
$t\mapsto \int_{{\mR}_{\ge 0}^2}L[\vp]{\rm d}^2F_t $ is continuous on $[0,\infty)$, and thus
it follows from (\ref{5.InL}) that the function $t\mapsto \int_{{\mR}_{\ge 0}}\vp{\rm d}F_t$
belongs to $C^1([0,\infty))$.  We therefore conclude  from (\ref{5.InL}) and
Definition \ref{measoluFPL} that $F_t$ is a measure-valued isotropic weak solution
of Eq.(FPL).
\end{proof}
\vskip2mm

Now we are in a position to prove the convergence  from Eq.(FD) to Eq.(FPL) for isotropic weak solutions.

\begin{proof} [Proof of Theorem \ref{Theorem2} (From Eq.(FD) to Eq.(FPL))] Let ${\cal E}=\{\vep_n\}_{n=1}^{\infty}
\subset (0,\vep_0]$ satisfy $\vep_n\to 0\,(n\to\infty)$.
 By
isotropic weak form of Eq.(FD) in (\ref{weak1}) with $\ld=-1$,
we have for any $\vp\in C_c^2({\mR}_{\ge 0})$ and $\vep\in {\cal E}$  that
\bes\lb{5.6}&&\int_{{\mR}_{\ge 0}}\vp(x)f^{\vep}(t,x)\sqrt{x}{\rm d}x \\
&&=
\int_{{\mR}_{\ge 0}}\vp(x)f_0^{\vep}(x)\sqrt{x}{\rm d}x
+\int_{0}^{t}{\rm d}s\int_{{\mR}_{\ge 0}^2}{\cal J}_{-1}^{\vep}
[\vp](y,z)f^{\vep}(y,s)f^{\vep}
(z,s)\sqrt{yz}
{\rm d}y{\rm d}z\nonumber \\
&&-\int_{0}^{t}{\rm d}s\int_{{\mR}_{\ge 0}^3}\vep^3{\cal K}_{-1}^{\vep}
[\vp](x,y,z) f^{\vep}(s,x)f^{\vep}(s,y)f^{\vep}(s,z)\sqrt{xyz}
{\rm d}x{\rm d}y{\rm d}z.\nonumber \ees
Next,  using
the entropy identity (\ref{entropy-identity})-(\ref{1.endi}) with $\ld=-1$ and noting that
$0\le 1-\vep^3 f^{\vep}\le 1$ and using the same proof
as in the proof of Theorem \ref{Theorem1}, we have
for the isotropic solutions $(t, {\bf v})\mapsto f^{\vep}(t,|{\bf v}|^2/2)$ that
\be \sup_{\vep\in{\cal E}, t\ge 0}\int_{{\mR}_{\ge 0}}f^{\vep}(t,x)
(1+x+|\log f^{\vep}(t,x)|)\sqrt{x}\,{\rm d}x<\infty\lb{5.7}\ee
and, for any $\vp\in L^{\infty}({\mR}_{\ge 0})$,
\be \Og_{\vp}(\eta):=\sup_{\vep\in{\cal E}, |t_1-t_2|\le \eta}\bigg|\int_{{\mR}_{\ge 0}}\vp(x)
f^{\vep}(x,t_1)\sqrt{x}{\rm d}x-\int_{{\mR}^3}\vp(x)
f^{\vep}(x,t_2)\sqrt{x}{\rm d}x\bigg|
\to 0\quad (\eta\to 0)\lb{5.8}\ee
Applying Lemma \ref{Lemma4.3} to the measures $F_t^{\vep}$ defined by
${\rm d}F_t^{\vep}(x)=f^{\vep}(t,x)\sqrt{x}{\rm d}x$ and  $M_0=\sup\limits_{\vep\in{\cal E}}\int_{{\mR}_{\ge 0}}f^{\vep}_0(x)\sqrt{x}\,{\rm d}x$ and recalling that $F_t^{\vep}$
conserve the mass, we also have for all $\vp\in C_c^2({\mR}_{\ge 0})$ that
\be \lim_{{\cal E}\ni\vep\to 0}\sup_{t\ge 0}\bigg|\vep^3\int_{{\mR}_{\ge 0}^3}{\cal K}_{-1}^{\vep}
[\vp](x,y,z) f^{\vep}(t,x)f^{\vep}(t,y)f^{\vep}(t,z)\sqrt{xyz}
{\rm d}x{\rm d}y{\rm d}z\bigg|=0.\lb{5.9}\ee
From (\ref{5.7}), (\ref{5.8}) and  Dunford-Petties criterion of
$L^1$-weakly relative compactness we conclude that there exist a subsequence of
$\{f^{\vep}\}_{\vep\in{\cal E}}$, still denote it as $\{f^{\vep}\}_{\vep\in{\cal E}}$, and a
function $0\le f\in L^{\infty}([0,\infty), L^1({\mR}_{\ge 0},(1+x)\sqrt{x}{\rm d}x))$, such that
$$f^{\vep}(t,\cdot) \rightharpoonup f(t,\cdot)\quad ({\cal E}\ni\vep\to 0)\quad {\rm weakly\,\,in}
\,\,\, L^1({\mR}_{\ge 0},\sqrt{x}{\rm d}x)\qquad
\forall\, t\ge 0.$$
By the way, it is easily deduced from (\ref{5.7}), (\ref{5.8}) and the weak convergence that
\be \sup_{t\ge 0}\int_{{\mR}_{\ge 0}}f(t,x)
(1+x+|\log f(t,x)|)\sqrt{x}\,{\rm d}x<\infty,\lb{5.10}\ee
\be \sup_{|t_1-t_2|\le \eta}\bigg|\int_{{\mR}_{\ge 0}}\vp(x)
f(t_1, x)\sqrt{x}{\rm d}x-\int_{{\mR}^3}\vp(x)
f(t_2,x)\sqrt{x}{\rm d}x\bigg|\le \Og_{\vp}(\eta)
\to 0\quad (\eta\to 0)\lb{CONL}\ee
for all $\vp\in L^{\infty}({\mR}_{\ge 0})$.

Next using Lemma \ref{Lemma4.3} and \eqref{5.L} we have for all $\vp\in C_c^2({\mR}_{\ge 0})$ that
\beas&& \sup_{\vep\in{\cal E}}|{\cal J}_{-1}^{\vep}[\vp](x,y)|, \, |L[\vp](x,y)|\le
C_{0}\|\vp''\|_{\infty}(x^{1/2}+y^{1/2}+1)
,\quad (x,y)\in{\mR}_{\ge 0}^2,\\
&&
\lim_{{\cal E}\ni\vep\to 0}\sup_{(x,y)\in[0,R]^2}|{\cal J}_{-1}^{\vep}[\vp](x,y)
-L[\vp](x,y)|=0,
\quad \forall\, 0<R<\infty.\eeas
For any $t\ge 0$ fixed, applying Lemma \ref{Lemma5.1}
with $d=2, s=1$ and ${\rm d}\mu_n(x)=
f^{\vep_n}(t,x)\sqrt{x}{\rm d}x, {\rm d}\mu(x)=
f(t,x)\sqrt{x}{\rm d}x,$ and
$\Psi_n(x,y)={\cal J}_{-1}^{\vep_n}[\vp](x,y), \Psi(x,y)=L[\vp](x,y)$,
we have
\beas\lim_{{\cal E}\ni\vep\to 0}\int_{{\mR}_{\ge 0}^2}{\cal J}_{-1}^{\vep}[\vp](x,y)
f^{\vep}(t,x)f^{\vep}(t,y)\sqrt{xy}{\rm d}x{\rm d}y
=\int_{{\mR}_{\ge 0}^2}L[\vp](x,y)f(t,x)f(t,y)
\sqrt{xy}{\rm d}x{\rm d}y\eeas
for all $t\ge 0$,
and so it follows from dominated convergence theorem that
\beas&&
\lim_{{\cal E}\ni\vep\to 0}\int_{0}^{t}{\rm d}s\int_{{\mR}_{\ge 0}^2}{\cal J}_{-1}^{\vep}[\vp](x,y)
f^{\vep}(s,x)f^{\vep}(s,y)
\sqrt{xy}{\rm d}x{\rm d}y
\\
&&=\int_{0}^{t}{\rm d}s\int_{{\mR}_{\ge 0}^2}L[\vp](x,y)f(s,x)f(s,y)
\sqrt{xy}{\rm d}x{\rm d}y\quad \forall\, t\ge 0.\eeas
Thus taking the limit $\lim\limits_{{\cal E}\ni\vep\to 0}$ to the equation (\ref{5.6}) we obtain
\beas\int_{{\mR}_{\ge 0}}\vp(x)f(t,x)\sqrt{x}{\rm d}x=
\int_{{\mR}_{\ge 0}}\vp(x)f_0(x)\sqrt{x}{\rm d}x
+\int_{0}^{t}{\rm d}s\int_{{\mR}_{\ge 0}^2}L[\vp](x,y)f(s,x)f(s,y)
\sqrt{xy}{\rm d}x{\rm d}y\eeas
for all $\vp\in C_c^2({\mR}_{\ge 0})$ and all $t\ge 0$.
Now applying Lemma \ref{Lemma5.3} to the measdure $F_t$ defined by ${\rm d}F_t(x)=f(t,x)\sqrt{x}{\rm d}x$, we conclude that
$f$ is an isotropic weak solution of Eq.(FPL).
\end{proof}

Finally we prove the convergence from Eq.(BE) to Eq.(FPL) for measure-valued isotropic
weak solutions.

\begin{proof}[Proof of Theorem \ref{Theorem3} (From Eq.(BE) to Eq.(FPL))]
Let ${\cal E}=\{\vep_n\}_{n=1}^{\infty}
\subset (0,\vep_0]$ satisfy $\vep_n\to 0\,(n\to\infty)$. Recall that
for any $\vp\in C_c^2({\mR}_{\ge 0})$, $\vep\in{\cal E}$, we have
\be \int_{{\mR}_{\ge 0}}\vp{\rm d}F^{\vep}_t=
\int_{{\mR}_{\ge 0}}\vp{\rm d}F_0
+\int_{0}^{t}{\rm d}s\int_{{\mR}_{\ge 0}^2}{\cal J}_{{+1}}^{\vep}
[\vp]{\rm d}^2F^{\vep}_s
+\int_{0}^{t}{\rm d}s\int_{{\mR}_{\ge 0}^3}\vep^3{\cal K}_{{+1}}^{\vep}
[\vp]{\rm d}^3F^{\vep}_s.
\lb{5.12}\ee
 We need to prove that
\be \sup_{\vep\in{\cal E}, |t_1-t_2|\le \dt}\bigg|
\int_{{\mR}_{\ge 0}}\vp{\rm d}F^{\vep}_{t_1}
-\int_{{\mR}_{\ge 0}}\vp{\rm d}F^{\vep}_{t_2}
\bigg|\to 0\quad {\rm as}\,\,\dt\to 0\quad \forall\, \vp\in C_c({\mR}_{\ge 0}).
\lb{CONLL}\ee
In fact from Lemma \ref{Lemma4.3} and the conservation of mass and energy we have
for any $\vp\in C_c^2({\mR}_{\ge 0})$ that
\bes&&\sup_{\vep\in{\cal E}, t\ge 0}
\int_{{\mR}_{\ge 0}^2}|{\cal J}_{{+1}}^{\vep}
[\vp]|{\rm d}^2F^{\vep}_t
\le C_0\|\vp''\|_{\infty}\sup_{\vep\in{\cal E}}(\|F_0^{\vep}\|_1)^2<\infty,\lb{5.14}\\
&&\sup_{t\ge 0}\bigg|\vep^3\int_{{\mR}_{\ge 0}^3}{\cal K}_{{+1}}^{\vep}
[\vp]{\rm d}^3F^{\vep}_t
\bigg|
\le C_0\|\vp''\|_{\infty}A_{\widehat{\phi}}^2(\|F_0^{\vep}\|_0)^3\vep
\to 0\quad ({\cal E}\ni\vep\to 0). \lb{5.15}\ees
These imply
\be \sup_{\vep\in{\cal E}}\bigg|
\int_{{\mR}_{\ge 0}}\vp{\rm d}F^{\vep}_{t_1}
-\int_{{\mR}_{\ge 0}}\vp{\rm d}F^{\vep}_{t_2}
\bigg|\le C_{\vp}|t_1-t_2|\qquad \forall\, t_1,t_2\ge 0,\,\,\, \forall\, \vp\in C_c^2({\mR}_{\ge 0}).
\lb{5.16}\ee
where  $C_{\vp}$
depends only on $\vp$, $\sup\limits_{\vep\in{\cal E}}\|F_0^{\vep}\|_1 $, and $A_{\widehat{\phi}}^2$.
Now for any $\vp\in C_c({\mR}_{\ge 0})$, using
(\ref{5.16}) and a standard smooth approximation we have
\be \sup_{\vep\in{\cal E}, |t_1-t_2|\le \dt}\bigg|
\int_{{\mR}_{\ge 0}}\vp{\rm d}F^{\vep}_{t_1}
-\int_{{\mR}_{\ge 0}}\vp{\rm d}F^{\vep}_{t_2}
\bigg|\le C_{\vp}\Lambda_{\vp}^*(\dt)\qquad \forall\, \dt>0\ee
where $\Lambda_{\vp}^*(\dt)=
\sup\limits_{|x_1-x_2|\le \dt}|\vp(x_1)-\vp(x_2)|.$ This proves
(\ref{CONLL}).

As is well-known, the bound $\sup\limits_{\vep\in{\cal E}, t\ge 0}\|F^{\vep}_t\|_1=
\sup\limits_{\vep\in{\cal E}}\|F^{\vep}_0\|_1<\infty$ and (\ref{CONLL}) imply that
 there exists a subsequence of
${\cal E}=\{\vep_n\}_{n=1}^{\infty}$, still denote it as
${\cal E}$, and a family $\{F_t\}_{t\ge 0}\subset {\cal B}_{1}^{+}
({\mR}_{\ge 0})$, such that
\bes && \sup_{t\ge 0}\|F_t\|_1\le \|F_0\|_1,\lb{5.18}\\
&&
\int_{{\mR}_{\ge 0}}\vp{\rm d}F_{t}=\lim_{{\cal E}\ni\vep\to 0}
\int_{{\mR}_{\ge 0}}\vp{\rm d}F^{\vep}_{t}\qquad \forall\, \vp\in C_c({\mR}_{\ge 0}),\quad \forall\,
t\ge 0.\lb{5.19}\ees
On the other hand from Lemma \ref{Lemma4.3} we have
\beas&& \sup_{\vep\in{\cal E}}|{\cal J}_{+1}^{\vep}[\vp](x,y)|, \, |L[\vp](x,y)|\le
C_{\vp}(x^{1/2}+y^{1/2})
,\quad (x,y)\in{\mR}_{\ge 0}^2,\\
&&
\lim_{{\cal E}\ni\vep\to 0}\sup_{(x,y)\in[0,R]^2}\big|{\cal J}_{+1}^{\vep}[\vp](x,y)-L[\vp](x,y)\big|=0
\qquad \forall\, 0<R<\infty.\eeas
Thus for any $t\ge 0$ fixed, using Lemma \ref{Lemma5.1} with $d=2, s=1$ and $\mu_n=F^{\vep_n}_t$,
$\mu=F_t$, and $\Psi_n(x,y)={\cal J}_{+1}^{\vep_n}[\vp](x,y), \Psi(x,y)=L[\vp](x,y)$
we obtain
$$\lim_{{\cal E}\ni\vep\to 0}\int_{{\mR}_{\ge 0}^2}
{\cal J}_{+1}^{\vep}[\vp]{\rm d}^2F^{\vep}_t
=\int_{{\mR}_{\ge 0}^2}L[\vp]{\rm d}^2F_t$$
and so the dominated convergence theorem yields that
\be\lim_{{\cal E}\ni\vep\to 0}\int_{0}^{t}{\rm d}s\int_{{\mR}_{\ge 0}^2}
{\cal J}_{+1}^{\vep}[\vp]{\rm d}^2F^{\vep}_s
=\int_{0}^{t}{\rm d}s\int_{{\mR}_{\ge 0}^2}L[\vp]{\rm d}^2F_s\qquad \forall\,t\ge 0.\lb{5.21}\ee
Taking the limit $\lim\limits_{{\cal E}\ni\vep\to 0}$ to the equation (\ref{5.12})
and using (\ref{5.19}),
(\ref{5.21}) and (\ref{5.15}) we derive  for
any $\vp\in C_c^2({\mR}_{\ge 0})$  that
$$\int_{{\mR}_{\ge 0}}\vp{\rm d}F_t
=\int_{{\mR}_{\ge 0}}\vp{\rm d}F_0
+\int_{0}^{t}{\rm d}s\int_{{\mR}_{\ge 0}^2}L[\vp]{\rm d}^2F_s \qquad \forall\,t\ge 0.$$
Then it follows from (\ref{5.18}) and Lemma \ref{Lemma5.3} that
$F_t$ is a  measure-valued weak solution of Eq.(FPL).\end{proof}

 \section{Appendix}

In this appendix we prove some general properties that have been used in the previous sections,
and establish the {\it weak projection gradient.}

\subsection{Equivalence of definition for weak formula}
 We will prove
\begin{proposition}\lb{Prop6.1}  Let $d, m\in{\mN}$,
$X={\bRd}$ or $X={\mR}_{\ge 0}^d$, and let
$\{\mu_t\}_{t\ge 0}$ be a family of positive
  measures on $X$ satisfying
$\sup\limits_{t\ge 0}\mu_t(X)<\infty$ and that
for any $\psi\in C^m_c(X)$,
$t\mapsto \int_{X}\psi({\bf x}){\rm d}\mu_t({\bf x})$ is absolutely continuous
on $[0,\infty)$ and
\be \fr{{\rm d}}{{\rm d}t}\int_{X}\psi({\bf x}){\rm d}\mu_t({\bf x})
={\cal Q}_t[\psi],\quad {\rm a.e.}\,\,\, t\in [0,\infty)\lb{6.1}\ee
where ${\cal Q}_t[\psi]$ has the properties: there exist a function
$0\le  M\in L^1_{loc}([0,\infty))$ and a null set $Z_0\subset [0,\infty)$ such that
for any $t\in[0,\infty)\setminus Z_0$,  $\psi\mapsto
{\cal Q}_t[\psi]$ is linear on $C^m_c(X)$ and
\be |{\cal Q}_t[\psi]|\le \|\psi\|_{m,\infty}M(t)\qquad \forall\, \psi\in C_c^m(X),\quad
\forall\,t\in[0,\infty)\setminus Z_0\lb{6.2}\ee
where $\|\psi\|_{m,\infty}=\sup\limits_{\,{\bf x}\in X}
\sum\limits_{|\alpha|\le m}
|{\p}^{\alpha} \psi({\bf x})|,
{\p}^{\alpha}={\p}^{\alpha}_{{\bf x}}={\p}_{x_1}^{\alpha_1}{\p}_{x_2}^{\alpha_2}\cdots
{\p}_{x_d}^{\alpha_d}, |\alpha|=\alpha_1+\alpha_2+\cdots+\alpha_d .$

Then for any $\psi\in C^m_c([0,\infty)\times X)$,
$t\mapsto \int_{X}\psi(t,{\bf x}){\rm d}\mu_t({\bf x})$ is absolutely continuous
$[0, \infty)$, $t\mapsto {\cal Q}_t[\psi(t)]$
belongs to $L^1_{loc}([0,\infty))$, and
\be \fr{{\rm d}}{{\rm d}t}\int_{X}\psi(t,{\bf x}){\rm d}\mu_t({\bf x})=
\int_{X}{\p}_t\psi(t,{\bf x}){\rm d}\mu_t({\bf x})
+{\cal Q}_t[\psi(t)],\quad {\rm a.e.}\,\,\, t\in [0,\infty).\lb{6.3}\ee
\end{proposition}

\begin{proof} The proof consists of two steps.

 {\bf Step 1.} We prove that for any $\psi\in C_c([0,\infty)\times X)$,
$t\mapsto \int_{X}\psi(t,{\bf x}){\rm d}\mu_t({\bf x})$ is continuous on $[0,\infty)$.

Suppose first that $\psi\in C_c(X)$. By extension theorem of continuous functions we
can assume that $\psi\in C_c({\bRd})$.
 Let $0\le J\in C^{\infty}_c({\bRd})$ satisfy
$\int_{{\bRd}}J({\bf x}){\rm d}{\bf x}=1$ and
$J$ is supported on the closed unite ball centered at the origin.
Let $J_{\dt}({\bf x})=\dt^{-d}J(\dt^{-1}{\bf x}),\dt>0$, and
 $\psi_{\dt}=\psi*J_{\dt}$ (convolution).
Then $\psi_{\dt}\in C^{\infty}_c({\bRd})$ and
$$\|\psi_{\dt}-\psi\|_{\infty}\le \Og_{\psi}(\dt):=
\sup_{{\bf x}, {\bf y}\in {\bRd}, |{\bf x}-{\bf y}|\le
 \dt}|\psi({\bf x})-\psi({\bf y})|\to 0\quad (\dt\to 0^{+}).$$
Let $C_{\mu}=\sup\limits_{t\ge 0}\mu_t(X)$.  For any $t\in[0,\infty)$,  we have
\beas\bigg|\int_{X}\psi{\rm d}\mu_s-
\int_{X}\psi{\rm d}\mu_{t}
 \bigg|&\le& 2C_{\mu}\Og_{\psi}(\dt)+\bigg|\int_{X}\psi_{\dt}{\rm d}\mu_s
-\int_{X}\psi_{\dt}{\rm d}\mu_{t}\bigg|\\
&\le& 2C_{\mu}\Og_{\psi}(\dt)+\|\psi_{\dt}\|_{m,\infty} \int_{t\wedge s}^{t\vee s}M(\tau){\rm d}\tau
.\eeas
By first letting $s\to t$ and then letting $\dt\to 0^{+}$ we conclude that
$t\mapsto \int_{X}\psi({\bf x}){\rm d}\mu_t({\bf x})$ is continuous at $t\in [0,\infty)$.

For general case, let $\psi\in C_c([0,\infty)\times X)$
and take any $t\in[0,\infty)$.  By the uniform continuity of
$\psi$ on $[0,\infty)\times X$ and the above result we see that as $s\to t$
\beas&&
\bigg|\int_{X}\psi(s,{\bf x}){\rm d}\mu_s({\bf x})
-\int_{X}\psi(t,{\bf x}){\rm d}\mu_{t}({\bf x})\bigg|
\\
&&\le C_{\mu}\|\psi(s)-\psi(t)\|_{\infty}+\bigg|\int_{X}\psi(t,{\bf x}){\rm d}\mu_{s}({\bf x})
-\int_{X}\psi(t,{\bf x}){\rm d}\mu_{t}({\bf x})\bigg|\to 0.\eeas
So $t\mapsto \int_{X}\psi(t,{\bf x}){\rm d}\mu_t({\bf x})$ is continuous at $t\in [0,\infty)$.

{\bf Step 2.} From (\ref{6.1}) we see that for any $\psi\in  C^m_c(X)$,
$t\mapsto  {\cal Q}_t[\psi]$ is Lebesgue measurable on $[0,\infty)$. From this and using the following
estimate (\ref{6.4})
it is not difficult to prove that
for any $\psi\in C^m_c([0,\infty)\times X)$,
$t\mapsto{\cal Q}_t[\psi(t)]$ is also Lebesgue measurable
on $[0,\infty)$.
From the assumption (\ref{6.2}) we have
\be\big|{\cal Q}_{t}[\psi(t)]-{\cal Q}_t[\psi(s)]\big|\le
\|\psi(t)-\psi(s)\|_{m,\infty}M(t)
\le C_{\psi}|t-s|M(t)\lb{6.4}\ee
for all $s\in [0,\infty)$ and all $t\in [0, \infty)\setminus Z_0$.
Here and below $C_{\psi}=\sup\limits_{t\ge 0}\big(
\|\psi(t)\|_{m,\infty}
+\|D_1\psi(t)\|_{m,\infty}\big).$
Since
$$|{\cal Q}_t[\psi(t)]|\le C_{\psi}M(t)\qquad \forall\, t\in [0, \infty)\setminus Z_0$$
this implies that $t\mapsto {\cal Q}_t[\psi(t)]$
belongs to $L^1_{loc}([0,\infty))$. Also
we have for any $0\le t_1<t_2<\infty$
\beas
\bigg|\int_{X}\psi(t_2,{\bf x}){\rm d}\mu_{t_2}({\bf x})
-\int_{X}\psi(t_1,{\bf x}){\rm d}\mu_{t_1}({\bf x})\bigg|
\le C_{\mu} C_{\psi}|t_2-t_1|+C_{\psi}\int_{t_1}^{t_2}M(s){\rm d}s.\eeas
This implies that
$t\mapsto \int_{X}\psi(t,{\bf x}){\rm d}\mu_t({\bf x})$ is absolutely
continuous on $[0,\infty)$.

Since $t\mapsto {\cal Q}_t[\psi(t)]$
also belongs to $L^1_{loc}([0,\infty))$, there is a null set $Z_1\subset [0,\infty)$
such that
\be
\lim_{t+h\ge 0, h\to 0}\fr{1}{h}\int_{t}^{t+h}{\cal Q}_{\tau}[\psi(\tau)]
{\rm d}\tau={\cal Q}_{t}[\psi(t)]
\quad \forall\, t\in [0,\infty)\setminus Z_1.\lb{6.5}\ee
Now for any $t\in [0,\infty)\setminus Z_1$ and $0\neq h\in [-t, \infty)$, look at
the difference quotient
\beas&&\fr{1}{h}\bigg(
\int_{X}\psi(t+h,{\bf x}){\rm d}\mu_{t+h}({\bf x})-
\int_{X}\psi(t,{\bf x}){\rm d}\mu_{t}({\bf x})
\bigg)\\
&&
=\int_{X}\fr{\psi(t+h,{\bf x})-\psi(t,{\bf x})}{h}{\rm d}\mu_{t+h}({\bf x})
+\fr{1}{h}\int_{t}^{t+h}{\cal Q}_{\tau}[\psi(t)]
{\rm d}\tau.
\eeas
Denoting
$D_1\psi(t, {\bf x})={\p}_t\psi(t,{\bf x})$
we have
\beas&&\bigg|\int_{X}\fr{\psi(t+h,{\bf x})-\psi(t,{\bf x})}{h}{\rm d}\mu_{t+h}({\bf x})
-\int_{X}D_1\psi(t,{\bf x}){\rm d}\mu_{t}({\bf x})
\bigg|
\\
&&\le C_{\mu}\Lambda_{D_1\psi}(|h|)+\bigg|\int_{X}D_1\psi(t+h,{\bf x}){\rm d}\mu_{t+h}({\bf x})
-\int_{X}D_1\psi(t,{\bf x}){\rm d}\mu_{t}({\bf x})
\bigg|\to 0\eeas
as $h\to 0$, where
$\Lambda_{D_1\psi}(\dt)=\sup\limits_{t_1,t_2\ge 0, |t_1-t_2|\le \dt}\|D_1\psi(t_1)-
D_1\psi(t_2)\|_{\infty}\to 0$ as $\dt\to 0^+$.
Also from (\ref{6.4}) we have
\beas&&\bigg|\fr{1}{h}\int_{t}^{t+h}{\cal Q}_{\tau}[\psi(t)]
{\rm d}\tau-\fr{1}{h}\int_{t}^{t+h}{\cal Q}_{\tau}[\psi(\tau)]
{\rm d}\tau\bigg|\le C_{\psi}\bigg|\int_{t}^{t+h}M(\tau){\rm d}\tau\bigg|\to 0
\eeas
as $h\to 0$. From these and (\ref{6.5}) we conclude that
the function
$t\mapsto \int_{X}\psi(t,{\bf x}){\rm d}\mu_t({\bf x})$ is differentiable at $t\in [0,\infty)
\setminus Z_1$ and satisfies the equality (\ref{6.3}).
\end{proof}
\vskip1mm

\subsection{Construction of the potential function $\phi$} It is given in the following

\begin{proposition}\lb{Prop6.2} Let $\widehat{\phi}(r)$ satisfy (\ref{ker2}), (\ref{1.po1}) with a constant
$0\le r_0<\infty$. Then the limiting function $\phi(\rho)$ given by (\ref{1.po2}) is well-defined, and  $\xi\mapsto \widehat{\phi}(|\xi|)$ is the Fourier transform of the function ${\bf x}\mapsto \phi (|{\bf x}|)$ in $\mathcal{S}'$ where ${\cal S}={\cal S}({\bR})$ is the class of Schwartz functions on ${\bR}$. Furthermore if $r_0=0$ and $\widehat{\phi}(r)\ge 0$
in $(0,\infty)$, then  $\phi (\rho)\ge 0$ in $(0,\infty).$
\end{proposition}

\begin{proof} To prove the existence of the limit in (\ref{1.po2}), we consider
$$\phi_R(\rho)=\fr{1}{2\pi^2\rho}\int_{0}^{R}r\widehat{\phi}(r)\sin(\rho r){\rm d}r,\quad \rho>0, \, R>0.$$
Recalling assumptions (\ref{1.po1}) and (\ref{ker2}) we know that
$r\mapsto r\wh{\phi}(r)$ is monotone in $(r_0,\infty)$ and $ r\wh{\phi}(r)\to 0$ as $r\to\infty$.
Therefore using the second mean-value formula of integrals we have for all $R_2> R_1>r_0$ that
\be\bigg|\int_{R_1}^{R_2}r\wh{\phi}(r)\sin(\rho r){\rm d}r\bigg|
\le \fr{2}{\rho}\Big(|R_1\wh{\phi}(R_1)|+|R_2\wh{\phi}(R_2)|\Big) \lb{6.6}\ee
which implies that for every $\rho>0$ the limit
$\phi(\rho):=\lim\limits_{R\to\infty}\phi_R(\rho) $
exists. Also from (\ref{ker2}) and (\ref{6.6}) we deduce that there is a constants $0<C<\infty$ such that
$$\sup_{R>0}|\phi_R(\rho)|\le C \big(1+\fr{1}{\rho^2}\big)\quad {\rm hence}\quad
|\phi(\rho)|\le C \big(1+\fr{1}{\rho^2}\big)\qquad \forall\, \rho>0.$$
Thus for any $\vp\in {\cal S}({\bR})$, the dominated convergence theorem can be used and we
conclude that $\phi (|\cdot|)\widehat{\vp}\in L^1({\bR})$ and
\be \int_{{\mR}^3}\phi (|{\bf x}|)\widehat{\vp}({\bf x}){\rm d}{\bf x}
=\int_{0}^{\infty}\int_{{\bS}}\rho^2 \phi(\rho)\wh{\vp}(\rho \sg){\rm d}\sg{\rm d}\rho
=\lim\limits_{R\to\infty}\int_{0}^{\infty}\int_{{\bS}}\rho^2 \phi_R(\rho)\wh{\vp}(\rho \sg){\rm d}\sg{\rm d}\rho.
\lb{6.7}\ee
On the other hand for any $0<R<\infty$ we compute
\beas&& \int_{|\xi|\le R}\widehat{\phi} (|\xi|)\vp(\xi){\rm d}\xi
=\int_{0}^{R}r^2\wh{\phi}(r)\int_{{\bS}}(2\pi)^{-3}\int_{{\bR}}
\wh{\vp}({\bf x})e^{{\rm i}{\bf x}\cdot r\og }{\rm d}{\bf x}
{\rm d}\og{\rm d}r  \\
&&=\int_{0}^{R}r^2\wh{\phi}(r)(2\pi)^{-3}\int_{{\bR}}
\wh{\vp}({\bf x})\fr{4\pi}{r|{\bf x}|}\sin(r|{\bf x}|){\rm d}{\bf x}
{\rm d}r=\int_{0}^{\infty}\int_{{\bS}}\rho^2 \phi_R(\rho)\wh{\vp}(\rho \sg){\rm d}\sg{\rm d}\rho.
\nonumber\eeas
Since $\widehat{\phi} (|\cdot|)\vp\in L^1({\bR})$,
letting $R\to\infty$ we deduce from this and (\ref{6.7}) that
$$
\int_{{\mR}^3}\widehat{\phi} (|\xi|)\vp(\xi){\rm d}\xi=\int_{{\mR}^3}\phi (|{\bf x}|)\widehat{\vp}({\bf x}){\rm d}{\bf x}\qquad \forall\, \vp\in {\cal S}({\mR}^3).$$
According to the Fourier transform of generalized functions, $\widehat{\phi}(|\xi|)$ is the Fourier transform of $\phi (|{\bf x}|)$ in $\mathcal{S}'$.

Now we turn to the second part of the proposition. Assume that $r_0=0$ and $\widehat{\phi}(r)\ge 0$
for all $r\in (0,\infty)$. Then $r\mapsto r\wh{\phi}(r)$ is monotone in $(0,\infty)$. Since
$\lim\limits_{r\to\infty}r\wh{\phi}(r)=0$ and $\wh{\phi}(r)\not\equiv 0$, this implies that
$r\mapsto r\wh{\phi}(r)$ is monotone non-increasing in $(0,\infty)$. Therefore
using the second mean-value formula of integrals we have
for any $R>\dt>0$ that
\be
 \int_{\dt}^{R}r\wh{\phi}(r)\sin(\rho r){\rm d}r\ge
 \fr{\dt\wh{\phi}(\dt)}{\rho} (\cos(\rho \dt)-1)\ge -\,\fr{\rho}{2}\dt^3\wh{\phi}(\dt).\lb{6.8}\ee
 Since the assumption (\ref{ker2}) implies that  $r\mapsto r^{3/2}\wh{\phi}(r)$ belongs to $L^1((0,1))$, it follows that there is a sequence $0<\dt_n\to 0\,(n\to\infty)$ such that
$\dt_n^3\wh{\phi}(\dt_n)\to 0\,(n\to\infty).$ From this and (\ref{6.8}) we obtain
$$
 \int_{0}^{R}r\wh{\phi}(r)\sin(\rho r){\rm d}r=
\lim_{n\to\infty} \int_{\dt_n}^{R}r\wh{\phi}(r)\sin(\rho r){\rm d}r
 \ge 0\qquad \forall\,\rho>0,\, R>0.$$
Thus we conclude $\phi (\rho)=\lim\limits_{R\to\infty}\phi_R(\rho)\ge 0$ for all $\rho\in (0,\infty).$
\end{proof}
\vskip1mm

\subsection{Weak projection gradient: definition and properties}
Let us first list some basic notations which are widely used throughout this section.
\smallskip

    $\bullet$ We recall that   $\Pi({\bf z})=\Pi({\bf n}):={\rm I}-{\bf n}\otimes {\bf n}\in
{\mR}^{3\times 3}$ where $ {\bf n}={\bf z}/|{\bf z}|$.

   $\bullet$  The vertical plane ${\mR}^2({\bf z})$ is defined by
\beas {\mR}^2({\bf z})=\{ {\bf h}\in {\bR}\,\,|\,\,{\bf z}\cdot {\bf h}=0\},\quad {\bf z}\neq {\bf 0}. \eeas

 $\bullet$  $Y$ is a Borel set in ${\mR}^N$. In our research problems shown above, $Y$
 has been taken $[0,\infty)\times{\bR}$ and $[0,\infty)$.

 Test function spaces for
 scalar and vector-valued cases are defined by
\bes
 && \lb{6.9}{\cal T}_{c}(Y\times {\bR})=\big\{ \psi\in C_c(Y\times {\bR})\,\,\big |
\,\, {\rm for\,\,any\,\, }\,{\bf y}\in Y,\,\,\psi({\bf y}, {\bf 0})=0,\,\, {\bf z}\mapsto \psi({\bf y},{\bf z})\in C^1({\bR})
, \\
&&{\rm and}\,\, \sup_{{\bf y}\in Y,{\bf z}\in {\bR}}|\nabla_{{\bf z}}\psi({\bf y},{\bf z})|<\infty
\big\},\nonumber\\
 && \lb{6.10}
{\cal T}_{c}(Y\times{\bR}, {\bR})=\big\{ \Psi=(\Psi_1, \Psi_2, \Psi_3)^{\tau}\in C_c(Y\times {\bR}, {\bR})\,\,\big|
\,\, {\rm for\,\,any\,\, }\,{\bf y}\in Y,\,\,\Psi({\bf y}, {\bf 0})={\bf 0},\\
&& {\bf z}\mapsto \Psi({\bf y},{\bf z})\in C^1({\bR}, {\bR}),\,{\rm and}\,\,
 \sup_{{\bf y}\in Y,{\bf z}\in {\bR}}|\Psi'_{{\bf z}}({\bf y},{\bf z})|<\infty
 \big\}\nonumber\ees
 where  $\Psi'_{{\bf z}}({\bf y},{\bf z})=\big(\fr{\p \Psi_i({\bf y},{\bf z})}{\p z_j}\big)_{3\times 3}$,
 $|\Psi'_{{\bf z}}({\bf y},{\bf z})|=\Big(\sum\limits_{i=1}^3\sum\limits_{j=1}^3|\fr{\p \Psi_i({\bf y},{\bf z})}{\p z_j}|^2\Big)^{1/2}.$

\begin{remark} {\rm (1)
Denote by $\nabla_{\bf z}=(\p_{z_1}, \p_{z_2},\p_{z_3})^{\tau}$ the
gradient operator of functions on ${\bR}$. Using the projection $\Pi({\bf z})$ we have an orthogonal decomposition
$\nabla_{\bf z} =\Pi({\bf z})\nabla_{\bf z}+({\bf n}\cdot\nabla_{\bf z} ){\bf n}$, i.e.
\be
\nabla_{\bf z} F({\bf z}) =\Pi({\bf z})\nabla_{\bf z} F({\bf z})+({\bf n}\cdot\nabla_{\bf z} F({\bf z}) ){\bf n},\quad \Pi({\bf z})\nabla_{\bf z} F({\bf z})\perp {\bf n}\lb{6.11}\ee
for all $F\in C^1({\bR})$.  This gives the differentiation of $F$ at ${\bf z}$ along the vertical plane ${\mR}^2({\bf z})$:
\be F({\bf z}+{\bf h})-F({\bf z})=\Pi({\bf z})\nabla_{\bf z} F({\bf z})\cdot {\bf h} +
 o({\bf h}),\quad {\bf h}\in {\mR}^2({\bf z})\lb{6.12}\ee
In view of (\ref{6.12}), we call $\Pi({\bf z})\nabla_{\bf z} F({\bf z})$ the gradient of $F$ at
${\bf z}$ along the vertical plan ${\mR}^2({\bf z})$, or simply, the {\it projection gradient} of $F$ at ${\bf z}$.
Let $\Psi\in C^1_c({\bR},{\bR})$. Using the identity
\be{\bf a}\cdot \Pi({\bf z}){\bf b}= \Pi({\bf z}){\bf a}\cdot {\bf b}={\bf b}\cdot
\Pi({\bf z}){\bf a},\quad {\bf a},{\bf b}\in {\bR}\lb{6.13}\ee
and
 integration by parts  we have
$$\int_{{\bR}}\Psi({\bf z})\cdot\Pi({\bf z})\nabla_{\bf z}F({\bf z}) {\rm d}{\bf z}
=-\int_{{\bR}}F({\bf z})\nabla_{\bf z}\cdot \Pi({\bf z})\Psi({\bf z})
{\rm d}{\bf z}.$$

(2) For all $\Psi\in {\cal T}_{c}(Y\times{\bR}, {\bR})$, we have
\be
\nabla_{{\bf z}} \cdot \Pi({\bf z})\Psi({\bf y}, {\bf z})=\nabla_{{\bf z}} \cdot \Psi({\bf y},{\bf z})-{\bf n}^{\tau}\Psi'_{{\bf z}}({\bf y},{\bf z}){\bf n}-
2\fr{\Psi({\bf y},{\bf z})}{|{\bf z}|}\cdot {\bf n},\lb{6.14}\ee
where
${\bf n}={\bf z}/|{\bf z}|,\,{\bf z}\in{\bR}\setminus\{{\bf 0}\}$.
Since $\Psi({\bf y}, {\bf 0})={\bf 0}$,  the quotient  $\fr{\Psi({\bf y},{\bf z})}{|{\bf z}|}$ makes sense and is bounded:
\beas \sup_{{\bf y}\in Y,{\bf z}\in {\bR}\setminus\{{\bf 0}\}}
\fr{|\Psi({\bf y},{\bf z})|}{|{\bf z}|}\le \sup_{{\bf y}\in Y,{\bf z}\in {\bR}}|\Psi_{{\bf z}}'({\bf y},{\bf z})|<\infty.\eeas
Thus $\nabla_{{\bf z}} \cdot \Pi({\bf z})\Psi({\bf y}, {\bf z})$ is bounded and has compact support
in $Y\times {\mR}^3$.}
\end{remark}

\subsubsection{Definition of the weak projection gradients}

\begin{definition}\lb{weak-diff-1}
 Let $F\in L^1_{loc}(Y\times {\bR})$. We say that $F({\bf y,z})$
 has a weak projection gradient ${\bf D}({\bf y},{\bf z})\in {\mR}^2({\bf z})$
 in ${\bf z}\in{\bR}\setminus\{{\bf 0}\}$,  if ${\bf D}
\in L^1_{loc}
(Y\times {\bR}, {\bR})$ and
\be\int_{Y\times{\bR}} \Psi({\bf y}, {\bf z})\cdot{\bf D}({\bf y}, {\bf z}) {\rm d}{\bf z}
{\rm d}{\bf y}
=-\int_{Y\times{\bR}}F({\bf y}, {\bf z})\nabla_{\bf z} \cdot \Pi({\bf z})\Psi({\bf y}, {\bf z})
{\rm d}{\bf z}{\rm d}{\bf y}
\lb{6.15}\ee
for all $\Psi\in {\cal T}_{c}(Y\times{\bR}, {\bR}).$
In this case, we denote
$${\bf D}({\bf y}, {\bf z})=
\Pi({\bf z})\nabla_{\bf z} F({\bf y}, {\bf z}).$$
\end{definition}

Several remarks are in order.
\begin{remark}
\lb{6.remark}{\rm (1) The perpendicular property ${\bf D}({\bf y},{\bf z})\in {\mR}^2({\bf z})$ is necessary for ${\bf D} \in L^1_{loc}(Y\times {\bR}, {\bR})$ satisfying (\ref{6.15}).
In fact, by inserting
$\Psi_{\vep}({\bf y,z}):=\big({\rm I}-\fr{{\bf z}{\bf z}^{\tau}}{\vep^2+|{\bf z}|^2}
\big)\Psi({\bf y,z})$ with $\Psi\in  {\cal T}_{c}(Y\times{\bR}, {\bR})$ into (\ref{6.15}) and letting $\vep\to 0^{+}$ we conclude ${\bf D}({\bf y},{\bf z})\cdot {\bf z}=0$ for a.e. $({\bf y},{\bf z})\in Y\times {\bR}$.

(2) It is obvious that, up to a set of measure zero, ${\bf D}({\bf y}, {\bf z})=\Pi({\bf z})\nabla_{\bf z} F({\bf y}, {\bf z})$ is
uniquely determined by the identity (\ref{6.15}).

(3) From (\ref{6.14}) one sees that the condition $\Psi({\bf y},{\bf 0})={\bf 0}$ for the test functions is only used to bound $\nabla_{{\bf z}} \cdot \Pi({\bf z})\Psi({\bf y}, {\bf z})$; it can be removed if $F$ satisfies that  $F({\bf y},{\bf z})/|{\bf z}|\in L^1_{loc}(Y\times{\bR})$.
}
\end{remark}

To adapt Definition \ref{weak-diff-1} to our problem, we observe that for
a smooth function $F({\bf v}, {\bf v}_*)$, let
$\bar{F}({\bf w,z})=F\big(\fr{\bf w+z}{2}, \fr{\bf w-z}{2}\big),$ then
 it holds the following relation:
$$\Pi({\bf v-v}_*)\big(\nabla_{{\bf v}}-\nabla_{{\bf v}_*}\big)
F({\bf v}, {\bf v}_*)=2\Pi({\bf z})\nabla_{{\bf z}}\bar{F}({\bf w,z})\big|_{
{\bf w}={\bf v+v}_*,\,{\bf z}={\bf v-v}_*}.$$

\begin{definition}\lb{weak-diff-2} Let $F\in L^1_{loc}
(Y\times {\bRR})$. We say that $F({\bf y},{\bf v}, {\bf v}_*)$
 has the weak projection gradient $\Pi({\bf v-v}_*)\nabla_{{\bf v-v}_*}F({\bf y},{\bf v}, {\bf v}_*)
$ in ${\bf v -v}_*\neq {\bf 0}$, if
the function
$\bar{F}({\bf y}, {\bf w},{\bf z}):=F\big({\bf y},\fr{{\bf w+z}}{2},\fr{{\bf w-z}}{2}\big)$
 has the weak projection gradient $\Pi({\bf z})\nabla_{\bf z} \bar{F}({\bf y},{\bf w}, {\bf z})$  in ${\bf z}\in{\bR}\setminus\{{\bf 0}\}$. In this case we define
\be \Pi({\bf v-v}_*)\nabla_{{\bf v-v}_*}F({\bf y},{\bf v}, {\bf v}_*)
:=2\Pi({\bf z})\nabla_{\bf z} \bar{F}({\bf y},{\bf w}, {\bf z})\big|_{{\bf w}={\bf v+v}_*,\,{\bf z}={\bf v-v}_*}.\lb{6.16}\ee
\end{definition}
\vskip1mm

From the relation between the two types of weak projection gradients one sees that
 in order to study the weak projection gradients
defined in the second definition, it needs only to study the weak projection gradients
defined in the first definition.

 Let us first show that the vector-valued test function space
 in Definition \ref{weak-diff-1} can be reduced to the scalar test function space.

\begin{lemma}\lb{Lemma6.1}A function $F\in L^1_{loc}
(Y\times {\bR})$ has the weak projection gradient $\Pi({\bf z})\nabla_{{\bf z}}F({\bf y}, {\bf z})$ in ${\bf z}\in{\bR}\setminus\{{\bf 0}\}$,  if and only if  there exists  a vector-valued function
${\bf D}({\bf y},{\bf z})\in {\mR}^2({\bf z})$
such that  ${\bf D}\in L^1_{loc}(Y\times {\bR}, {\bR})$ and the equality
\be \int_{Y\times{\bR}} \psi({\bf y}, {\bf z}){\bf D}({\bf y}, {\bf z}) {\rm d}{\bf z}
{\rm d}{\bf y}
=-\int_{Y\times{\bR}}F({\bf y},{\bf z})\Big(\Pi({\bf z})\nabla_{\bf z} \psi({\bf y},{\bf z})-
2\fr{\psi({\bf y},{\bf z})}{|{\bf z}|}{\bf n}
\Big) {\rm d}{\bf z}{\rm d}{\bf y} \lb{6.17}\ee
holds for all $\psi\in {\cal T}_{c}(Y\times {\bR}).$
Moreover, if (\ref{6.17}) holds for all $\psi\in {\cal T}_{c}(Y\times {\bR})$, then ${\bf D}({\bf y}, {\bf z})=
\Pi({\bf z})\nabla_{\bf z} F({\bf y}, {\bf z})$ a.e. on $Y\times {\bR}$.
\end{lemma}

  The proof  follows easily from Definition \ref{weak-diff-1}
 and the implication relation that if $\psi\in {\cal T}_{c}(Y\times {\bR})$
 then $\psi({\bf y},{\bf z}){\bf a}$ belongs to
${\cal T}_{c}(Y\times{\bR}, {\bR})$ for all constant vector ${\bf a}\in {\bR}$. We omit the details here.
\vskip1mm

By definition of the projection $\Pi({\bf z})$ we have
$\Pi({\bf z}){\bf z}={\bf 0}$ for all ${\bf z}\in {\bR}\setminus\{{\bf 0}\}$. From this
it is easily deduced the following equalities that are often used:
For any $a\in C^1((0,\infty))$ and any ${\bf z}\in {\bR}\setminus\{{\bf 0}\}$,
\bes\lb{6.18}&& \Pi({\bf z})\nabla_{{\bf z}}\big(a(|{\bf z}|)\psi({\bf z})\big)=a(|{\bf z}|)
\Pi({\bf z})\nabla_{{\bf z}} \psi({\bf z})\qquad\forall\, \psi\in C^1({\bR})\\
&&\lb{6.19}
\nabla_{{\bf z}}\cdot \Pi({\bf z})a(|{\bf z}|)\psi({\bf z})=a(|{\bf z}|)
\nabla_{{\bf z}}\cdot \Pi({\bf z})\psi({\bf z})\qquad\forall\, \psi\in C^1({\bR},{\bR}).\ees
\vskip1mm

\subsubsection{Some properties of weak projection gradients}

The first part of the following lemma shows that
if in addition that $ F({\bf y}, {\bf z})/|{\bf z}|\in L^1_{loc}(Y\times{\bR})$, then the
condition $\psi({\bf y}, {\bf 0})=0$ in test function space ${\cal T}_{c}(Y\times {\bR})$ can be removed.

\begin{lemma}\lb{Lemma6.2} Suppose $F\in L^1_{loc}(Y\times {\bR})$
has the weak projection gradient $\Pi({\bf z})\nabla_{\bf z} F({\bf y}, {\bf z})$ in ${\bf z}\in{\bR}\setminus\{{\bf 0}\}$.

 (a) If $F({\bf y}, {\bf z})/|{\bf z}|\in L^1_{loc}(Y\times{\bR})$, then for any $\psi\in C_c^1({\bR})$, there is a
null set $Z_{\psi}\subset Y$ such that
for all ${\bf y}\in Y\setminus Z_{\psi}$
\be \int_{{\bR}}\psi({\bf z})\Pi({\bf z})\nabla_{\bf z} F({\bf y}, {\bf z})
{\rm d}{\bf z}
=-\int_{{\bR}}F({\bf y}, {\bf z})\Big(\Pi({\bf z})\nabla \psi({\bf z})-
2\fr{\psi({\bf z})}{|{\bf z}|}{\bf n}\Big){\rm d}{\bf z}. \lb{6.20}\ee

(b) Let $a\in C^1((0,\infty))$. If $a(|{\bf z}|)F({\bf y},{\bf z})$ and
$a(|{\bf z}|)\Pi({\bf z})\nabla_{{\bf z}} F({\bf y},{\bf z})$ both
belong to $L^1_{loc}(Y\times{\bR})$, then $a(|{\bf z}|)F({\bf y},{\bf z})$
also has the
weak projection gradient $\Pi({\bf z})\nabla_{{\bf z}}\big(a(|{\bf z}|)
F({\bf y},{\bf z})\big)$ in
${\bf z}\in{\bR}\setminus\{{\bf 0}\}$ and it holds
\be\lb{6.21} \Pi({\bf z})\nabla_{{\bf z}}\big(a(|{\bf z}|)
F({\bf y},{\bf z})\big)=a(|{\bf z}|)\Pi({\bf z})\nabla_{{\bf z}}F({\bf y},{\bf z}).\ee

\end{lemma}

 \begin{proof} (a): First of all  from the assumptions we have
for any $R_1,R_2\in {\mN}$ that
\beas&&
\int_{|{\bf y}|<R_1}\bigg(\int_{|{\bf z}|<R_2}\big|\Pi({\bf z})\nabla_{\bf z} F({\bf y}, {\bf z})
\big|{\rm d}{\bf z}\bigg){\rm d}{\bf y}=
\int_{|{\bf y}|<R_1, |{\bf z}|<R_2}\big|\Pi({\bf z})\nabla_{\bf z} F({\bf y}, {\bf z})
\big|{\rm d}{\bf z}{\rm d}{\bf y}<\infty,\\
&&
\int_{|{\bf y}|<R_1}\bigg(\int_{|{\bf z}|<R_2}|F({\bf y}, {\bf z})|\big(1+\fr{1}{|{\bf z}|}\big)
{\rm d}{\bf z}\bigg){\rm d}{\bf y}
= \int_{|{\bf y}|<R_1, |{\bf z}|<R_2}|F({\bf y}, {\bf z})|\big(1+\fr{1}{|{\bf z}|}\big)
{\rm d}{\bf z}{\rm d}{\bf y}<\infty.\eeas
Therefore all integrals in the rest of the proof are absolutely convergent  and
\beas\int_{|{\bf z}|<R_2}\big|\Pi({\bf z})\nabla_{\bf z} F({\bf y}, {\bf z})
\big|{\rm d}{\bf z}<\infty,\quad \int_{|{\bf z}|<R_2}|F({\bf y}, {\bf z})|\big(1+\fr{1}{|{\bf z}|}\big)
{\rm d}{\bf z}<\infty\eeas
for all $R_2\in{\mN}$ and all ${\bf y}\in Y\setminus Z_0$, where
$Z_0$ is a common null set in $Y$.

Let $\zeta\in C^{\infty}({\mR}_{\ge 0})$ satisfy
$0\le \zeta(\cdot)\le 1$ on ${\mR}_{\ge 0}$, and $\zeta(r)=0$ for all $r\in[0,1]$,
$\zeta(r)=1$ for all $r\ge 2$. Let
$\zeta_{\dt}(|{\bf z}|)=\zeta(|{\bf z}|/\dt),\dt>0$.
Take any $\vp\in C_c(Y)$ and $\psi\in C_c^1({\bR})$.
Let
$\psi_{\dt}({\bf y},{\bf z})=\vp({\bf y})\zeta_{\dt}(|{\bf z}|)\psi({\bf z}).$
Then $\psi_{\dt}\in {\cal T}_{c}(Y\times {\bR})$ and thanks to the identity \eqref{6.18} we have
$\Pi({\bf z})\nabla_{\bf z} \psi_{\dt}({\bf y},{\bf z})
=\vp({\bf y})\zeta_{\dt}(|{\bf z}|)\Pi({\bf z})
\nabla\psi({\bf z})$ and so, by
   Lemma \ref{Lemma6.1},
\beas&& \int_{Y\times{\bR}}\vp({\bf y})\zeta_{\dt}(|{\bf z}|)\psi({\bf z})
\Pi({\bf z})\nabla_{\bf z}
F({\bf y}, {\bf z}) {\rm d}{\bf z}
{\rm d}{\bf y}\\
&&
=-\int_{Y\times{\bR}}\vp({\bf y})\zeta_{\dt}(|{\bf z}|)F({\bf y},{\bf z})\Big(\Pi({\bf z})\nabla_{\bf z} \psi({\bf z})-
2\fr{\psi({\bf z})}{|{\bf z}|}{\bf n}
\Big) {\rm d}{\bf z}{\rm d}{\bf y}.\eeas
By definition of the test functions and the fact
that $\lim\limits_{\dt\to 0^{+}}\zeta_{\dt}(|{\bf z}|)=1$ for all
${\bf z}\in {\bR}\setminus\{{\bf 0}\}$,  dominated convergence theorem yields that \beas\int_{Y\times{\bR}}
\vp({\bf y})\psi({\bf z})
\Pi({\bf z})\nabla_{\bf z}
F({\bf y}, {\bf z}) {\rm d}{\bf z}
{\rm d}{\bf y}
=-\int_{Y\times{\bR}}F({\bf y},{\bf z})\vp({\bf y})\Big(
\Pi({\bf z})\nabla\psi({\bf z})-
2\fr{\psi({\bf z})}
{|{\bf z}|}{\bf n}
\Big) {\rm d}{\bf z}{\rm d}{\bf y}.\eeas
Since $\vp\in C_c(Y)$ is arbitrary, there is a null set $Z_{\psi}\subset Y$
such that (\ref{6.20}) holds for all ${\bf y}\in Y\setminus Z_{\psi}$.

(b): For any $\Psi\in {\cal T}_{c}(Y\times{\bR},{\bR})$, the function
$\zeta_{\dt}(|{\bf z}|)a(|{\bf z}|)\Psi({\bf y}, {\bf z})$
($\dt>0$) belongs to ${\cal T}_{c}(Y\times{\bR},{\bR})$ and using (\ref{6.18}) we have
$\nabla_{\bf z} \cdot \Pi({\bf z})\zeta_{\dt}(|{\bf z}|)a(|{\bf z}|)\Psi({\bf y}, {\bf z})
=\zeta_{\dt}(|{\bf z}|)a(|{\bf z}|)\nabla_{\bf z} \cdot\Pi({\bf z})\Psi({\bf y}, {\bf z})$
for all ${\bf z}\in{\bR}\setminus\{{\bf 0}\}$ and so
$$\int_{Y\times{\bR}} \Psi\zeta_{\dt}(|{\bf z}|)a(|{\bf z}|)\cdot \Pi({\bf z})\nabla_{{\bf z}}
F{\rm d}{\bf z}
{\rm d}{\bf y}=-\int_{Y\times{\bR}}F\zeta_{\dt}(|{\bf z}|)a(|{\bf z}|)\nabla_{\bf z} \cdot \Pi({\bf z})\Psi
{\rm d}{\bf z}{\rm d}{\bf y}.$$
By the assumptions in (b) and recalling that $\Psi\in {\cal T}_{c}(Y\times{\bR},{\bR})$,
letting $\dt\to 0^+$ we obtain
$$\int_{Y\times{\bR}} \Psi({\bf y}, {\bf z}) a(|{\bf z}|)\cdot \Pi({\bf z})\nabla_{{\bf z}}
F({\bf y}, {\bf z}){\rm d}{\bf z}
{\rm d}{\bf y}=
-\int_{Y\times{\bR}}F({\bf y}, {\bf z}) a(|{\bf z}|)\nabla_{\bf z} \cdot \Pi({\bf z})\Psi({\bf y}, {\bf z})
{\rm d}{\bf z}{\rm d}{\bf y}.$$
This proves \eqref{6.21} by definition of weak projection gradients.
\end{proof}
\vskip1mm

Note that for the case $a(r)=r^{\alpha}$ with $\alpha>0$,  the condition ``
$|{\bf z}|^{\alpha}F({\bf y},{\bf z})$ and
$|{\bf z}|^{\alpha}\Pi({\bf z})\nabla_{{\bf z}} F({\bf y},{\bf z})$ both
belong to $L^1_{loc}(Y\times{\bR})$"   is satisfied automatically and thus for this case the
conclusion of the lemma holds true.
Whereas for the case $a(r)=r^{-\alpha}$ ($\alpha>0$),  the condition ``
$|{\bf z}|^{-\alpha}F({\bf y},{\bf z})$ and
$|{\bf z}|^{-\alpha}\Pi({\bf z})\nabla_{{\bf z}} F({\bf y},{\bf z})$ both
belong to $L^1_{loc}(Y\times{\bR})$" should be checked out in practice.
\smallskip

The following two lemmas are important:

\begin{lemma}\lb{Lemma6.7} Suppose that $F\in L^2_{loc}(Y\times{\bR})$ has the weak
projection gradient $\Pi({\bf z})\nabla_{{\bf z}} F({\bf y},{\bf z})$ in
${\bf z}\in{\bR}\setminus\{{\bf 0}\}$ and that
$\Pi({\bf z})\nabla_{{\bf z}} F({\bf y},{\bf z})$ belongs to $L^2_{loc}(Y\times{\bR}, {\bR})$.
Then $F^2$ also has the weak
projection gradient $\Pi({\bf z})\nabla_{{\bf z}}F^2({\bf y},{\bf z})$ in
${\bf z}\in{\bR}\setminus\{{\bf 0}\}$
and
$$\Pi({\bf z})\nabla_{{\bf z}}F^2({\bf y},{\bf z})=2F({\bf y},{\bf z})\Pi({\bf z})\nabla_{{\bf z}}F({\bf y},{\bf z}).$$
\end{lemma}

\begin{lemma}\lb{Lemma6.8} Suppose that $0< F\in L^1_{loc}(Y\times{\bR})$ has the weak
projection gradient $\Pi({\bf z})\nabla_{{\bf z}} F({\bf y},{\bf z})$ in ${\bf z}\in {\bR}\setminus\{{\bf 0}\}$ and satisfies
$$\inf_{{\bf y}\in Y, |{\bf y}|\le R_1, |{\bf z}|\le R_2}F({\bf y}, {\bf z})>0
\qquad \forall\, 0<R_1, R_2<\infty, $$
\be\lb{6.**} \fr{F({\bf y},{\bf z})}{|{\bf z}|}\,\,{\rm belongs \,\,to}\,\,\in L^1_{loc}(Y\times{\bR}). \ee
Then $\sqrt{F({\bf y},{\bf z})}$ has also the weak projection gradient $\Pi({\bf z})\nabla_{{\bf z}} \sqrt{F({\bf y},{\bf z})}$ in ${\bf z}\in {\bR}\setminus\{{\bf 0}\}$ and it holds
$$\Pi({\bf z})\nabla_{{\bf z}} \sqrt{F({\bf y},{\bf z})}=\fr{\Pi({\bf z})\nabla_{{\bf z}} F({\bf y},{\bf z})}{2\sqrt{F({\bf y},{\bf z})}}.$$
\end{lemma}

To prove the two lemmas, we first make a common preparation. Note that the assumption
$F\in L^2_{loc}(Y\times{\bR})$ implies that $F$ satisfies \eqref{6.**}.

\begin{lemma}\label{Fdelta}   Let $j(t)=c\exp(\fr{-1}{1-t})1_{\{t<1\}}$ and
$J({\bf u})=j(|{\bf u}|^2)$ with
$c>0$ such that
$\int_{{\bR}}J({\bf u}){\rm d}{\bf u}=1.$
Set
$J_{\dt}({\bf u})=\dt^{-3}J(\dt^{-1}{\bf u}), \dt>0.$
Suppose $F\in L^1_{loc}(Y\times{\bR})$  satisfies \eqref{6.**} and has the weak
projection gradient $\Pi({\bf z})\nabla_{{\bf z}} F({\bf y},{\bf z})$ in ${\bf z}\in {\bR}\setminus\{{\bf 0}\}$. Let
$F_{\dt}({\bf y},{\bf z})=(J_{\dt}*F({\bf y},\cdot))({\bf z}).$
Then for any any $\dt>0$ and any ${\bf z}\in{\bR}\setminus\{{\bf 0}\}$, there
is a null set $Z_{\dt,{\bf z}}\subset Y$, such that for all ${\bf y}\in Y\setminus Z_{\dt,{\bf z}}$,
\bes\lb{6.22}&&\qquad \Pi({\bf z})\nabla_{\bf z} F_{\dt}({\bf y},{\bf z})-\Pi({\bf z})\nabla_{\bf z} F({\bf y},{\bf z})\\
&&=
\int_{|{\bf u}|\le 1}J({\bf u})\Pi({\bf z})\Big(\Pi({\bf z}-\dt{\bf u})\nabla_{\bf z} F({\bf y},{\bf z}-\dt{\bf u})
-\Pi({\bf z})\nabla_{\bf z} F({\bf y},{\bf z})\Big) {\rm d}{\bf u}\nonumber\\
&&+
 \int_{|{\bf u}|\le 1}\big(F({\bf y},{\bf z}-\dt{\bf u})-F({\bf y},{\bf z})\big)\Big\{\big(\nabla J({\bf u})
\cdot \fr{{\bf z}-\dt{\bf u}}{|{\bf z}-\dt{\bf u}|}\big)\fr{-\Pi({\bf z}){\bf u}}{|{\bf z}-\dt{\bf u}|}
-2J({\bf u})\Pi({\bf z})\fr{{\bf z}-\dt{\bf u}}{|{\bf z}-\dt{\bf u}|^2}\Big\}{\rm d}{\bf u},\nonumber\ees
and
\bes\lb{6.23}&&\qquad \Pi({\bf z})\nabla_{\bf z} F_{\dt}({\bf y},{\bf z})-\Pi({\bf z})\nabla_{\bf z} F({\bf y},{\bf z}) \\
&& =
\int_{|{\bf u}|\le 1}J({\bf u})\Pi({\bf z})\Big(\Pi({\bf z}-\dt{\bf u})\nabla_{\bf z} F({\bf y},{\bf z}-\dt{\bf u})
-\Pi({\bf z})\nabla_{\bf z} F({\bf y},{\bf z})\Big) {\rm d}{\bf u} \nonumber
\\
&&+
 \int_{|{\bf u}|\le 1}\Big(\fr{F({\bf y},{\bf z}-\dt{\bf u})}{|{\bf z}-\dt{\bf u}|}-
 \fr{F({\bf y},{\bf z})}{|{\bf z}|}\Big)\Big\{\big(\nabla J({\bf u})
\cdot \fr{{\bf z}-\dt{\bf u}}{|{\bf z}-\dt{\bf u}|}\big)\big(-\Pi({\bf z}){\bf u}\big)
-2J({\bf u})\Pi({\bf z})\fr{{\bf z}-\dt{\bf u}}{|{\bf z}-\dt{\bf u}|}\Big\}{\rm d}{\bf u}.\nonumber\ees
\end{lemma}
\begin{proof}  Since $({\bf y},{\bf u})\mapsto
\fr{F({\bf y},{\bf u})}{|{\bf u}|}$
belongs to $L^1_{loc}(Y\times {\bR})$, applying Lemma \ref{Lemma6.2}(a) to $({\bf y}, {\bf u})\mapsto
F({\bf y}, {\bf u})$ and taking
${\bf u}\mapsto J_{\dt}({\bf z}-{\bf u})$ as a test function
with ${\bf z}, \dt$ fixed,
 we have, for a null set $Z_{\dt,{\bf z}}\subset Y$ that
\beas&&\int_{{\bR}}
J_{\dt}({\bf z}-{\bf u})\Pi({\bf u})\nabla_{\bf u} F({\bf y},{\bf u}) {\rm d}{\bf u}
=-\int_{{\bR}}F({\bf y},{\bf u})\Big(\Pi({\bf u})\nabla_{{\bf u}}(
J_{\dt}({\bf z}-{\bf u}))
-2\fr{J_{\dt}({\bf z}-{\bf u})}{|{\bf u}|^2}{\bf u}\Big){\rm d}{\bf u}\\
&&=\int_{{\bR}}F({\bf y},{\bf u})
\nabla_{{\bf z}}J_{\dt}({\bf z}-{\bf u}){\rm d}{\bf u}-\int_{{\bR}}F({\bf y},{\bf u})
\Big(\big(\nabla_{{\bf z}}J_{\dt}({\bf z}-{\bf u})\cdot \fr{{\bf u}}{|{\bf u}|}\big)\fr{{\bf u}}{|{\bf u}|}-2J_{\dt}({\bf z}-{\bf u})\fr{{\bf u}}{|{\bf u}|^2}\Big){\rm d}{\bf u}
\eeas for all ${\bf y}\in Y\setminus Z_{\dt,{\bf z}}$.
This gives
\beas&&
\nabla_{{\bf z}} F_{\dt}({\bf y},{\bf z})=\int_{{\bR}}F({\bf y},{\bf u})
\nabla_{{\bf z}}J_{\dt}({\bf z}-{\bf u}){\rm d}{\bf u}
\\
&&=
\int_{|{\bf u}|<1}J({\bf u})\Pi({\bf z}-\dt{\bf u})\nabla_{\bf z}F({\bf y},{\bf z}-\dt{\bf u}) {\rm d}{\bf u}
\\
&& +
 \int_{|{\bf u}|<1}F({\bf y},{\bf z}-\dt{\bf u})\Big(\big(\dt^{-1}\nabla J({\bf u})
\cdot \fr{{\bf z}-\dt{\bf u}}{|{\bf z}-\dt{\bf u}|}\big)\fr{{\bf z}-\dt{\bf u}}{|{\bf z}-\dt{\bf u}|}
-2J({\bf u})
\fr{{\bf z}-\dt{\bf u}}{|{\bf z}-\dt{\bf u}|^2}\Big){\rm d}{\bf u}.
\eeas
and so
\beas&&\Pi({\bf z})
\nabla_{{\bf z}} F_{\dt}({\bf y},{\bf z})-\Pi({\bf z})
\nabla_{{\bf z}} F({\bf y},{\bf z})\\
&&=
\int_{|{\bf u}|<1}J({\bf u})\Pi({\bf z})\Big(\Pi({\bf z}-\dt{\bf u})\nabla_{\bf z} F({\bf y},{\bf z}-\dt{\bf u})
-\Pi({\bf z})\nabla_{\bf z} F({\bf y},{\bf z})\Big) {\rm d}{\bf u}
\\
&&+
 \int_{|{\bf u}|<1}F({\bf y},{\bf z}-\dt{\bf u})\Big\{\big(\nabla J({\bf u})
\cdot \fr{{\bf z}-\dt{\bf u}}{|{\bf z}-\dt{\bf u}|}\big)\dt^{-1}\Pi({\bf z})\fr{{\bf z}-\dt{\bf u}}{|{\bf z}-\dt{\bf u}|}
-2J({\bf u})\Pi({\bf z})
\fr{{\bf z}-\dt{\bf u}}{|{\bf z}-\dt{\bf u}|^2}\Big\}{\rm d}{\bf u}\eeas
for all ${\bf y}\in Y\setminus Z_{\dt,{\bf z}}$. Here we have used $\Pi({\bf z})^2=\Pi({\bf z})$.

To prove \eqref{6.22} and \eqref{6.23}, we then need only to prove the following claim: for any  Borel measurable function
$G(r,s,t) $ on $[0,1]\times {\mR}\times [0,\infty)$ satisfying
$|G(r,s, t)|\le C\fr{1}{t^{\alpha}}$ for all $(r,s,t)\in[0,1]\times {\mR}
\times (0,\infty)$
for some constants $0<C<\infty, \alpha<3$, we have
\be
 \int_{|{\bf u}|\le 1}G\big(|{\bf u}|, {\bf u}
\cdot ({\bf z}-\dt{\bf u}), |{\bf z}-\dt{\bf u}|\big)\Pi({\bf z})({\bf z}-\dt{\bf u})
{\rm d}{\bf u}={\bf 0},\quad {\bf z}\neq {\bf 0},\dt>0.\lb{6.24}\ee
In fact, from the assumption on $G$ we see that the integral below is absolutely
convergent, and recalling ${\bf n}={\bf z}/|{\bf z}|$, $\Pi({\bf z}){\bf z}={\bf 0}$, and \eqref{2.Sym1} we
compute
\beas&&\int_{|{\bf u}|\le 1}
G\big(|{\bf u}|, {\bf u}
\cdot ({\bf z}-\dt{\bf u}), |{\bf z}-\dt{\bf u}|\big)\Pi({\bf z})({\bf z}-\dt{\bf u})
{\rm d}{\bf u}
=-\dt\int_{0}^{1}r^3\int_{0}^{\pi}\sin^2(\theta)\\&&\times
G\big(r, r|{\bf z}|\cos(\theta)-\dt r^2, \sqrt{|{\bf z}|^2-
2\dt |{\bf z}|r \cos(\theta)+\dt^2r^2}
\big) \bigg(\int_{{\mS}^1({\bf n})}\sg
{\rm d}\sg\bigg) {\rm d}\theta{\rm d}r
={\bf 0},\eeas
where in the last equality we have used $\int_{{\mS}^1({\bf n})}\sg
{\rm d}\sg={\bf 0}.$

Now since
$J({\bf u})=j(|{\bf u}|^2), \nabla J({\bf u})=
2j'(|{\bf u}|^2){\bf u}$, applying (\ref{6.24}) gives
\beas\\
&&
\int_{|{\bf u}|<1}\Big\{\big(\nabla J({\bf u})
\cdot \fr{{\bf z}-\dt{\bf u}}{|{\bf z}-\dt{\bf u}|}\big)\dt^{-1}\Pi({\bf z})\fr{{\bf z}-\dt{\bf u}}{|{\bf z}-\dt{\bf u}|}
-2J({\bf u})\Pi({\bf z})
\fr{{\bf z}-\dt{\bf u}}{|{\bf z}-\dt{\bf u}|^2}\Big\}{\rm d}{\bf u}
={\bf 0},\\
&&
 \int_{|{\bf u}|\le 1}\Big\{\big(\nabla J({\bf u})
\cdot \fr{{\bf z}-\dt{\bf u}}{|{\bf z}-\dt{\bf u}|}\big)\dt^{-1}\Pi({\bf z})({\bf z}-\dt {\bf u})
-2J({\bf u})\Pi({\bf z})\fr{{\bf z}-\dt{\bf u}}{|{\bf z}-\dt{\bf u}|}\Big\}{\rm d}{\bf u}
={\bf 0}.\eeas
These together with $\dt^{-1}\Pi({\bf z})({\bf z}-\dt {\bf u})=
\Pi({\bf z})(-{\bf u})$ yield \eqref{6.22} and \eqref{6.23}.  \end{proof}

We also need the following elementary property:

\begin{lemma}\lb{Lemma6.6} Let $1\le p<\infty,  l\in {\mN}, g\in L^p_{loc}(Y\times{\bR}, {\mR}^l)$. Then
for all $0<R<\infty$ we have
\beas&&
\int_{{\bf y}\in Y, |{\bf y}|\le R, |{\bf z}|\le R}|g({\bf y},{\bf z}+{\bf h})-g({\bf y},{\bf z})|^p{\rm d}{\bf z}{\rm d}{\bf y}\to 0\quad {\rm as}\quad
{\bf h}\to {\bf 0},\\
&& \sup_{|{\bf h}|<1}
\int_{{\bf y}\in Y, |{\bf y}|\le R, |{\bf z}|\le R}
|g({\bf y},{\bf z}+{\bf h})-g({\bf y},{\bf z})|^p{\rm d}{\bf z}{\rm d}{\bf y}\le
2^p\int_{{\bf y}\in Y, |{\bf y}|\le R, |{\bf z}|\le R+1}
|g({\bf y}, {\bf z})|^p{\rm d}{\bf z}{\rm d}{\bf y}.\eeas
\end{lemma}
\vskip1mm

\begin{proof}[ Proof of Lemma \ref{Lemma6.7}] Let $F$ satisfy the assumptions in Lemma \ref{Lemma6.7} and let $F_{\dt}({\bf y},{\bf z})=(J_{\dt}*F({\bf y},\cdot))({\bf z})$. Using Lemma \ref{Lemma6.6}, one has
\beas&&\sup_{0<\dt<1}
\int_{{\bf y}\in Y, |{\bf y}|\le R, |{\bf z}|\le R}|F_{\dt}({\bf y},{\bf z})|^2{\rm d}{\bf z}{\rm d}{\bf y}
\le \int_{{\bf y}\in Y, |{\bf y}|\le R,|{\bf z}|\le
R+1}|F({\bf y},{\bf z})|^2{\rm d}{\bf z}{\rm d}{\bf y},\\
&&
\int_{{\bf y}\in Y, |{\bf y}|\le R, |{\bf z}|\le R}|F_{\dt}({\bf y},{\bf z})-F({\bf y},{\bf z})|^2{\rm d}{\bf z}{\rm d}{\bf y}\to 0\quad {\rm as}\quad \dt\to 0^{+}.\eeas
Since $F_{\dt}({\bf y},{\bf z})$ is smooth in ${\bf z}$, we have
 $\nabla_{{\bf z}}(F_{\dt}({\bf y},{\bf z})^2)
=2F_{\dt}({\bf y},{\bf z})\nabla_{{\bf z}}F_{\dt}({\bf y},{\bf z})$
and so for any $\Psi\in {\cal T}_{c}(Y\times{\bR}, {\bR})$
\be
\int_{Y\times{\bR}}
\Psi({\bf y}, {\bf z})\cdot 2F_{\dt}({\bf y},{\bf z})\Pi({\bf z})\nabla_{{\bf z}}F_{\dt}({\bf y},{\bf z})
{\rm d}{\bf z}{\rm d}{\bf y}
=-
\int_{Y\times{\bR}} F^2_{\dt}({\bf y},{\bf z})\nabla_{{\bf z}}\cdot \Pi({\bf z})\Psi({\bf y},{\bf z})
{\rm d}{\bf z}{\rm d}{\bf y}.\lb{6.25}\ee

\noindent Convergence of l.h.s of \eqref{6.25}:\, By change of variables and simple estimates
it is not difficult to prove that there is a constant $0<C<\infty$ such that
$$\int_{|{\bf u}|<1}\Big|\fr{\Pi({\bf z}){\bf u}}{|{\bf z}-\dt{\bf u}|}\Big|^2{\rm d}{\bf u}
\le \fr{C}{|{\bf z}|^2},\quad \int_{|{\bf u}|<1}\Big|
 \Pi({\bf z})\fr{{\bf z}-\dt{\bf u}}{|{\bf z}-\dt{\bf u}|^2}\Big|^2{\rm d}{\bf u}
\le \fr{C}{|{\bf z}|^2},
\quad \forall\, \dt>0,\, {\bf z}\neq {\bf 0}.$$
Then using (\ref{6.22}) and
$|\Pi({\bf z}){\bf h}|\le |{\bf h}|$,
we obtain that, for all ${\bf z}\in{\bR}\setminus\{{\bf 0}\}, 0<\dt<1$ and all
${\bf y}\in Y\setminus Z_{\dt,{\bf z}}$,
\beas&&\big|\Pi({\bf z})\nabla_{\bf z} F_{\dt}({\bf y},{\bf z})-\Pi({\bf z})\nabla_{\bf z} F({\bf y},{\bf z})\big|^2
\\
&&\le C\int_{|{\bf u}|\le 1}\big|\Pi({\bf z}-\dt{\bf u})\nabla_{\bf z} F({\bf y},{\bf z}-\dt{\bf u})
-\Pi({\bf z})\nabla_{\bf z} F({\bf y},{\bf z})\big|^2 {\rm d}{\bf u}\\
&&
 +\fr{C}{|{\bf z}|^2}\int_{|{\bf u}|\le 1}\big|F({\bf y},{\bf z}-\dt{\bf u})-F({\bf y},{\bf z})\big|^2
{\rm d}{\bf u}.\eeas
From this and definition of $\Psi\in {\cal T}_{c}(Y\times{\bR}, {\bR})$ which implies that
$|\Psi({\bf y},{\bf z})|\le C_{\Psi}|{\bf z}|$,  we have
\beas&&
\int_{Y\times {\bR}}
|\Psi({\bf y}, {\bf z})|^2
\big|\Pi({\bf z})\nabla_{{\bf z}}F_{\dt}({\bf y},{\bf z})
-\Pi({\bf z})\nabla_{{\bf z}}F({\bf y},{\bf z})\big|^2
{\rm d}{\bf z}{\rm d}{\bf y}
\\
&&\le C_{\Psi}\int_{|{\bf u}|\le 1}
\bigg(\int_{|{\bf y}|<R, |{\bf z}|<R}
\big|\Pi({\bf z}-\dt{\bf u})\nabla_{\bf z} F({\bf y},{\bf z}-\dt{\bf u})
-\Pi({\bf z})\nabla_{\bf z} F({\bf y},{\bf z})\big|^2{\rm d}{\bf z}{\rm d}{\bf y}\bigg) {\rm d}{\bf u}
\\
&&+C_{\Psi}\int_{|{\bf u}|\le 1}
\bigg(\int_{|{\bf y}|\le R, |{\bf z}|\le R}
\big|F({\bf y},{\bf z}-\dt{\bf u})-F({\bf y},{\bf z})\big|^2{\rm d}{\bf z}{\rm d}{\bf y}\bigg) {\rm d}{\bf u}
 \to 0\quad (\dt\to 0^{+}),\eeas
 where we have used Lemma \ref{Lemma6.6}.
This implies the desired convergence:
\beas&&\bigg|
\int_{Y\times{\bR}}
\Psi\cdot 2F_{\dt}\Pi({\bf z})\nabla_{{\bf z}}F_{\dt}
{\rm d}{\bf z}{\rm d}{\bf y}-
\int_{Y\times{\bR}}
\Psi\cdot 2F\Pi({\bf z})\nabla_{{\bf z}}F
{\rm d}{\bf z}{\rm d}{\bf y}\bigg|
\\
&& \le
2
\int_{|{\bf y}|\le R, |{\bf z}|\le R}
|\Psi|
|F_{\dt}|\big|\Pi({\bf z})\nabla_{{\bf z}}F_{\dt}
-\Pi({\bf z})\nabla_{{\bf z}}F\big|
{\rm d}{\bf z}{\rm d}{\bf y}\\
&& +2
\int_{|{\bf y}|\le R, |{\bf z}|\le R}
|\Psi|
\big|F_{\dt}
-F\big|\big|\Pi({\bf z})\nabla_{{\bf z}}F
\big|
{\rm d}{\bf z}{\rm d}{\bf y}\\
&& \le 2\bigg(
\int_{|{\bf y}|\le R, |{\bf z}|\le R}
|F_{\dt}|^2{\rm d}{\bf z}\bigg)^{1/2}
\bigg(
\int_{|{\bf y}|\le R, |{\bf z}|\le R}
|\Psi|^2
\big|\Pi({\bf z})\nabla_{{\bf z}}F_{\dt}
-\Pi({\bf z})\nabla_{{\bf z}}F\big|^2
{\rm d}{\bf z}{\rm d}{\bf y}\bigg)^{1/2}\\
&& +2C_{\Psi}\bigg(
\int_{|{\bf y}|\le R, |{\bf z}|\le R}
\big|F_{\dt}-F\big|^2
{\rm d}{\bf z}{\rm d}{\bf y}\bigg)^{1/2}\bigg(
\int_{|{\bf y}|\le R, |{\bf z}|\le R}
\big|\Pi({\bf z})\nabla_{{\bf z}}F
\big|^2
{\rm d}{\bf z}{\rm d}{\bf y}
\bigg)^{1/2}\\
&&\to 0\quad (\dt\to 0^{+}).\eeas

\noindent Convergence of r.h.s of \eqref{6.25}:\, We also have
\beas&&\bigg|
\int_{Y\times{\bR}} F^2_{\dt}\nabla_{{\bf z}}\cdot \Pi({\bf z})\Psi
{\rm d}{\bf z}{\rm d}{\bf y}
-
\int_{Y\times{\bR}} F^2\nabla_{{\bf z}}\cdot \Pi({\bf z})\Psi
{\rm d}{\bf z}{\rm d}{\bf y}\bigg|
\\
&&\le C_{\Psi}
\int_{|{\bf y}|\le R, |{\bf z}|\le R}\big| F^2_{\dt}-F^2\big|
{\rm d}{\bf z}{\rm d}{\bf y}\\
&&\le C_{\Psi}\bigg(
\int_{|{\bf y}|\le R, |{\bf z}|\le R}\big| F_{\dt}-F\big|^2{\rm d}{\bf z}{\rm d}{\bf y}\bigg)^{1/2}\bigg(\int_{|{\bf y}|\le R, |{\bf z}|\le R+1}
\big|F\big|^2
{\rm d}{\bf z}{\rm d}{\bf y}\bigg)^{1/2}\\
&&\to 0\quad (\dt\to 0^{+}).\eeas
\smallskip
Thus from these two estimates,   (\ref{6.25}) yields that
\beas
\int_{Y\times{\bR}}
\Psi({\bf y}, {\bf z})\cdot 2F({\bf y},{\bf z})\Pi({\bf z})\nabla_{{\bf z}}F({\bf y},{\bf z})
{\rm d}{\bf z}{\rm d}{\bf y}
=-
\int_{Y\times{\bR}}
 F^2({\bf y},{\bf z})\nabla_{{\bf z}}\cdot \Pi({\bf z})\Psi({\bf y},{\bf z})
{\rm d}{\bf z}{\rm d}{\bf y}.\eeas
By Defintion \ref{weak-diff-1},  $F^2$ has the
weak projection gradient $\Pi({\bf z})\nabla_{{\bf z}}F^2({\bf y},{\bf z})$ in
${\bf z}\in{\bR}\setminus\{{\bf 0}\}$
and
$\Pi({\bf z})\nabla_{{\bf z}}F^2({\bf y},{\bf z})=2F({\bf y},{\bf z})\Pi({\bf z})\nabla_{{\bf z}}F({\bf y},{\bf z}).$ \end{proof}

\begin{proof}[Proof of Lemma \ref{Lemma6.8}] Suppose that $F$ satisfies the assumption in Lemma \ref{Lemma6.8} and let $F_{\dt}({\bf y},{\bf z})=(J_{\dt}*F({\bf y},\cdot))({\bf z})=\int_{|{\bf u}\le 1}
J({\bf u})F({\bf y},{\bf z}-\dt{\bf u}){\rm d}{\bf u}$ be defined in Lemma \ref{Fdelta}.
Then
\be \lb{6.26} c_{R}:=\inf_{{\bf y}\in Y, |{\bf y}|\le R, |{\bf z}|\le R+1}F({\bf y},{\bf z})>0,\quad
\inf_{{\bf y}\in Y, |{\bf y}|\le R, |{\bf z}|\le R, 0<\dt<1}F_{\dt}({\bf y},{\bf z})
\ge c_{R}.\ee
Thus, as $\dt\to 0^{+}$,
\beas &&\int_{|{\bf y}|\le R,|{\bf z}|\le R}\Big|\sqrt{F_{\dt}}-\sqrt{F}\Big|{\rm d}{\bf z}{\rm d}{\bf y}
\le \fr{1}{2\sqrt{c_R}}\int_{ |{\bf y}|\le R,|{\bf z}|\le R}|F_{\dt}-F|
{\rm d}{\bf z}{\rm d}{\bf y}\to 0,\\&&\int_{|{\bf y}|\le R, |{\bf z}|\le R}
\Big|\fr{1}{\sqrt{F_{\dt}}}-
\fr{1}{\sqrt{F}}\Big|
{\rm d}{\bf z}{\rm d}{\bf y}\le \fr{1}{2c_R\sqrt{c_R}}\int_{ |{\bf y}|\le R, |{\bf z}|\le R}
|F_{\dt}-F|
{\rm d}{\bf z}{\rm d}{\bf y}\to 0.\eeas
Since $F_{\dt}({\bf y},{\bf z})$ is smooth in ${\bf z}$
and has positive lower bound on any bounded set,  it follows that
\beas&&\nabla_{\bf z} \sqrt{F_{\dt}({\bf y},{\bf z})}
=\fr{\nabla_{\bf z} F_{\dt}({\bf y},{\bf z})}{ 2\sqrt{F_{\dt}({\bf y},{\bf z})} }.
\eeas
Thus for any $\Psi\in {\cal T}_{c}(Y\times{\bR}, {\bR})$
\be\int_{Y\times{\bR}}\Psi({\bf y},{\bf z})\cdot\fr{\Pi({\bf z})\nabla_{\bf z} F_{\dt}({\bf y},{\bf z})}{2\sqrt{F_{\dt}({\bf y},{\bf z})}}
{\rm d}{\bf z}{\rm d}{\bf y}
=-\int_{Y\times{\bR}}\sqrt{F_{\dt}({\bf y},{\bf z})}\,
\nabla_{\bf z}\cdot \Pi({\bf z})\Psi({\bf y},{\bf z})
{\rm d}{\bf z}{\rm d}{\bf y}.\lb{6.27}\ee

\noindent Convergence of l.h.s of \eqref{6.27}:\,
Using (\ref{6.23}) we have for all ${\bf z}\in{\bR}\setminus\{{\bf 0}\},
0<\dt<1$ and ${\bf y}\in Y\setminus Z_{\dt,{\bf z}}$ that
\beas&&\big|\Pi({\bf z})\nabla_{\bf z} F_{\dt}({\bf y},{\bf z})-\Pi({\bf z})\nabla_{\bf z} F({\bf y},{\bf z})
\big|\\
 &&\le
\int_{|{\bf u}|\le 1}J({\bf u})\big|\Pi({\bf z}-\dt{\bf u})\nabla_{\bf z} F({\bf y},{\bf z}-\dt{\bf u})
-\Pi({\bf z})\nabla_{\bf z} F({\bf y},{\bf z})\big| {\rm d}{\bf u}\\
&&
+C
 \int_{|{\bf u}|\le 1}\bigg|\fr{F({\bf y},{\bf z}-\dt{\bf u})}{|{\bf z}-\dt{\bf u}|}-
 \fr{F({\bf y},{\bf z})}{|{\bf z}|}\bigg|
{\rm d}{\bf u}.\eeas
Then, thanks to  Lemma\,\ref{Lemma6.6},
\beas&&\int_{|{\bf y}|\le R, |{\bf z}|\le R}\big|\Pi({\bf z})\nabla_{\bf z} F_{\dt}({\bf y},{\bf z})-\Pi({\bf z})\nabla_{\bf z} F({\bf y},{\bf z})
\big|{\rm d}{\bf z}{\rm d}{\bf y}
\\
&&\le \int_{|{\bf u}|\le 1}J({\bf u})\bigg(\int_{|{\bf y}|\le R, |{\bf z}|\le R}
\big|\Pi({\bf z}-\dt{\bf u})\nabla_{\bf z} F({\bf y},{\bf z}-\dt{\bf u})
-\Pi({\bf z})\nabla_{\bf z} F({\bf y},{\bf z})\big|{\rm d}{\bf z}{\rm d}{\bf y}\bigg)
{\rm d}{\bf u}
\\
&&+C\int_{|{\bf u}|\le 1}\bigg(\int_{|{\bf y}|\le R, |{\bf z}|\le R}
 \bigg|\fr{F({\bf y},{\bf z}-\dt{\bf u})}{|{\bf z}-\dt{\bf u}|}-
 \fr{F({\bf y},{\bf z})}{|{\bf z}|}\bigg|
{\rm d}{\bf z}{\rm d}{\bf y}\bigg){\rm d}{\bf u}\to 0\quad (\dt\to 0^{+}).\eeas
Next we have
\beas
&&\int_{|{\bf y}|\le R, |{\bf z}|\le R}
\bigg|\fr{\Pi({\bf z})\nabla_{\bf z} F_{\dt}({\bf y},{\bf z})}{2\sqrt{F_{\dt}({\bf y},{\bf z})}}-
\fr{\Pi({\bf z})\nabla_{\bf z} F({\bf y},{\bf z})}{2\sqrt{F({\bf y},{\bf z})}}\bigg|
{\rm d}{\bf z}{\rm d}{\bf y}\\
&&\le \fr{1}{2}\int_{|{\bf y}|\le R, |{\bf z}|\le R}
\fr{|\Pi({\bf z})\nabla_{\bf z} F_{\dt}({\bf y},{\bf z})-\Pi({\bf z})\nabla_{\bf z} F({\bf y},{\bf z})|}{\sqrt{F_{\dt}({\bf y},{\bf z})}}
{\rm d}{\bf z}{\rm d}{\bf y}\\
&&+\fr{1}{2}\int_{|{\bf y}|\le R, |{\bf z}|\le R}|\Pi({\bf z})\nabla_{\bf z} F({\bf y},{\bf z})|\bigg|\fr{1}{\sqrt{F_{\dt}({\bf y},{\bf z})}}-
\fr{1}{\sqrt{F({\bf y},{\bf z})}}\bigg|
{\rm d}{\bf z}{\rm d}{\bf y} =: \fr{1}{2}A_{\dt}+\fr{1}{2}B_{\dt}\eeas
and from the lower bound (\ref{6.26}) we have
\beas A_{\dt}
\le \fr{1}{\sqrt{c_R}}\int_{ |{\bf y}|\le R, |{\bf z}\le R}
|\Pi({\bf z})\nabla_{\bf z} F_{\dt}({\bf y},{\bf z})-\Pi({\bf z})\nabla_{\bf z} F({\bf y},{\bf z})|
{\rm d}{\bf z}{\rm d}{\bf y}\to 0\quad (\dt\to 0^{+}).\eeas
For $B_{\dt}$,  we consider a decomposition: for any $0<M<\infty$,
\bes \lb{6.28}&&
B_{\dt}
=
\int_{ |{\bf y}|\le R, |{\bf z}|\le R}|\Pi({\bf z})\nabla_{\bf z} F|1_{\{|\Pi({\bf z})\nabla_{\bf z} F|\le M\}}\bigg|\fr{1}{\sqrt{F_{\dt}}}-
\fr{1}{\sqrt{F}}\bigg|
{\rm d}{\bf z}{\rm d}{\bf y}
\\
&&+\int_{ |{\bf y}|\le R, |{\bf z}|\le R}
|\Pi({\bf z})\nabla_{\bf z} F({\bf y},{\bf z})|
1_{\{|\Pi({\bf z})\nabla_{\bf z} F|> M\}}
\bigg|\fr{1}{\sqrt{F_{\dt}}}-
\fr{1}{\sqrt{F}}\bigg|
{\rm d}{\bf z}{\rm d}{\bf y}\nonumber\\
&&\le
M\int_{|{\bf y}|\le R, |{\bf z}|\le R}\bigg|\fr{1}{\sqrt{F_{\dt}}}-
\fr{1}{\sqrt{F}}\bigg|
{\rm d}{\bf z}{\rm d}{\bf y} + \fr{1}{\sqrt{c_R}}\int_{|{\bf y}|\le R, |{\bf z}|\le R}
|\Pi({\bf z})\nabla_{\bf z} F|
1_{\{|\Pi({\bf z})\nabla_{\bf z} F|>M\}}
{\rm d}{\bf z}{\rm d}{\bf y}.\nonumber
\ees
By definition of $F_{\dt}$, the first term in the righthand side of (\ref{6.28}) tends to zero as $\dt\to 0^+$, while by the
$L^1_{loc}$ integrability of $|\Pi({\bf z})\nabla_{\bf z} F({\bf y},{\bf z})|$, the second term in the righthand side of (\ref{6.28})
tends to zero as $M\to \infty$.
Thus by first letting $\dt\to 0^{+}$ and then letting $M\to\infty$ we
conclude $\lim\limits_{\dt\to 0^{+}}B_{\dt}=0$.

Thus
$$\lim_{\dt\to 0^{+}}\int_{ |{\bf y}|\le R, |{\bf z}|\le R}
\bigg|\fr{\Pi({\bf z})\nabla_{\bf z} F_{\dt}({\bf y},{\bf z})}{2\sqrt{F_{\dt}({\bf y},{\bf z})}}-
\fr{\Pi({\bf z})\nabla_{\bf z} F({\bf y},{\bf z})}{2\sqrt{F({\bf y},{\bf z})}}\bigg|
{\rm d}{\bf z}{\rm d}{\bf y}=0.$$
From this we deduce the convergence:
\beas&&\bigg|\int_{Y\times{\bR}}\Psi({\bf y},{\bf z})\cdot\fr{\Pi({\bf z})\nabla_{\bf z} F_{\dt}({\bf y},{\bf z})}{2\sqrt{F_{\dt}({\bf y},{\bf z})}}
{\rm d}{\bf z}{\rm d}{\bf y}
-\int_{Y\times{\bR}}\Psi({\bf y},{\bf z})\cdot\fr{\Pi({\bf z})\nabla_{\bf z} F({\bf y},{\bf z})}{2\sqrt{F({\bf y},{\bf z})}}
{\rm d}{\bf z}{\rm d}{\bf y}\bigg|
\\
&&\le C_{\Psi}\int_{ |{\bf y}|\le R, |{\bf z}|\le R}\bigg| \fr{\Pi({\bf z})\nabla_{\bf z} F_{\dt}({\bf y},{\bf z})}{2\sqrt{F_{\dt}({\bf y},{\bf z})}}-\fr{\Pi({\bf z})\nabla_{\bf z} F({\bf y},{\bf z})}{2\sqrt{F({\bf y},{\bf z})}}\bigg|
{\rm d}{\bf z}{\rm d}{\bf y}
\to 0\quad (\dt\to 0^{+}).\eeas

\noindent Convergence of r.h.s of \eqref{6.27}:\, Since $\Psi\in {\cal T}_{c}(Y\times{\bR}, {\bR})$, we have
\beas&&\bigg|\int_{Y\times{\bR}} \sqrt{F_{\dt}({\bf y},{\bf z})}
\,\nabla_{\bf z}\cdot \Pi({\bf z})\Psi({\bf y},{\bf z})
{\rm d}{\bf z}{\rm d}{\bf y}
-\int_{Y\times{\bR}} \sqrt{F({\bf y},{\bf z})}\,
\nabla_{\bf z}\cdot \Pi({\bf z})\Psi({\bf y},{\bf z})
{\rm d}{\bf z}{\rm d}{\bf y}\bigg|\\
&&\le C_{\Psi}\int_{ |{\bf y}|\le R, |{\bf z}|\le R} \big|\sqrt{F_{\dt}({\bf y},{\bf z})}-
\sqrt{F({\bf y},{\bf z})}\big|
{\rm d}{\bf z}{\rm d}{\bf y}\to 0
\quad (\dt\to 0^{+}).\eeas
\smallskip
Letting $\dt\to 0^{+}$ in (\ref{6.27}) we finally obtain  that
\beas\int_{Y\times{\bR}}\Psi({\bf y},{\bf z})\cdot\fr{\Pi({\bf z})\nabla_{\bf z} F({\bf y},{\bf z})}{2\sqrt{F({\bf y},{\bf z})}}
{\rm d}{\bf z}{\rm d}{\bf y}
=-\int_{Y\times{\bR}} \sqrt{F({\bf y},{\bf z})}
\,\nabla_{\bf z}\cdot \Pi({\bf z})\Psi({\bf y},{\bf z})
{\rm d}{\bf z}{\rm d}{\bf y}.\eeas
Thus $\Pi({\bf z})\nabla_{\bf z} \sqrt{F({\bf y},{\bf z})}$ exists and
$\Pi({\bf z})\nabla_{\bf z} \sqrt{F({\bf y},{\bf z})}=\fr{\Pi({\bf z})\nabla_{\bf z} F({\bf y},{\bf z})}{2\sqrt{F({\bf y},{\bf z})}}.$
\end{proof}
\vskip2mm

\subsubsection{Application to Boltzmann collision operator} We begin with an algebraic lemma.

\begin{lemma}\lb{Prop6.3} Let ${\bf n}\in {\bS}, {\mS}^{1}({\bf n})=\{\sg\in {\bS}\,|\,\sg\,\bot\, {\bf n}\,\}, A\in {\mR}^{3\times 3}$, and ${\bf a}, {\bf b}\in {\bR}$. Then
\bes&&\fr{1}{\pi}\int_{{\mS}^1({\bf n})}\sg^{\tau} A\sg{\rm d}\sg={\rm Tr}(A)-{\bf n}^{\tau}A{\bf n},\lb{6.29}\\
&&  \lb{6.30} \fr{1}{\pi}\int_{{\mS}^1({\bf n})}({\bf a}\cdot \sg )\sg{\rm d}\sg=
\Pi({\bf n}){\bf a},
\quad \fr{1}{\pi}\int_{{\mS}^1({\bf n})}( {\bf a}\cdot\sg)({\bf b}\cdot\sg){\rm d}\sg=
{\bf a}\cdot \Pi({\bf n}){\bf b}.\ees
\end{lemma}

\begin{proof} Suppose
${\bf i},  {\bf j}\in {\mathbb S}^2$  and $\{{\bf i}, {\bf j},{\bf n}\}$ is an orthonormal base of ${\mathbb R}^3$.
Then for the orthogonal matrix $T= ( {\bf i}\,\, {\bf j}\,\, {\bf n})^{\tau}\in {\mathbb R}^{3\times 3}$
we have
$T{\bf i}={\bf e}_1=(1,0,0)^{\tau},\, T{\bf j}={\bf e}_2=(0,1,0)^{\tau},\,
T{\bf n}={\bf e}_3=(0,0,1)^{\tau}.$
It is easy to see that
${\bf i}^{\tau}A{\bf i}+{\bf j}^{\tau}A{\bf j}+{\bf n}^{\tau}A{\bf n}
=\sum_{i=1}^3{\bf e}_i^{\tau}TAT^{-1}{\bf e}_i=
{\rm Tr}(TAT^{-1})={\rm Tr}(A)$. Hence
$\fr{1}{\pi}
\int_{{\mathbb S}^1({\bf z})}
\sg^{\tau}A\sg{\rm d}\sg=\fr{1}{\pi}
\int_{0}^{2\pi}
(\cos(\theta){\bf i}+\sin(\theta){\bf j})^{\tau}A
(\cos(\theta){\bf i}+\sin(\theta){\bf j})
{\rm d}\theta
= {\bf i}^{\tau}A{\bf i}+{\bf j}^{\tau}A{\bf j}
={\rm Tr}(A)-{\bf n}^{\tau}A{\bf n}.$
It is not difficult to check that
$\fr{1}{\pi}
\int_{{\mS}^1({\bf n})}({\bf a}\cdot \sg)\sg
 {\rm d}\sg=\fr{1}{\pi}
\int_{0}^{2\pi}({\bf a}\cdot (\cos(\theta){\bf i}+\sin(\theta){\bf j}))
(\cos(\theta){\bf i}+\sin(\theta){\bf j})
 {\rm d}\theta
=({\bf a}\cdot{\bf i}) {\bf i}+({\bf a}\cdot{\bf j}) {\bf j}
={\bf a}-({\bf a}\cdot{\bf n}){\bf n}
=\Pi({\bf n}){\bf a},$
which yields that
$\fr{1}{\pi}\int_{{\mS}^1({\bf n})}( {\bf a}\cdot \sg)({\bf b}\cdot \sg){\rm d}\sg=
 \big(\fr{1}{\pi}\int_{{\mS}^1({\bf n})}( {\bf a}\cdot \sg)\sg{\rm d}\sg
 \big)\cdot {\bf b}=\Pi({\bf n}){\bf a}\cdot {\bf b}={\bf a}\cdot
 \Pi({\bf n}){\bf b}.$
\end{proof}

Since the Boltzmann collision operator involves integration on the unit sphere ${\bS}$,  to apply the definition of {\it weak projection gradient} to the weak coupling limit, we need to introduce the
following space of test functions:
\bes\label{6.31} &&
{\cal T}_{c}(Y\times {\bRS})\\
&&=\bigg\{ \psi\in  C_c(Y\times {\bR}\times {\bS})\,\,\Big|
\,\,\psi({\bf y},\cdot,\cdot)\in C^1({\bR}\times{\bS}),\,\psi({\bf y},{\bf 0},\sg)= 0\,\, {\rm for\,\,all\,\,}
{\bf y}\in Y,\sg\in {\bS},\nonumber\\
&&\sup_{({\bf y},{\bf z},\sg)\in Y\times {\bRS},
}|\nabla_{{\bf z}}\psi({\bf y},{\bf z},\sg)|
<\infty,\,\,\sup_{({\bf y}, {\bf z},\sg)\in Y\times ({\bR}\setminus \{{\bf 0}\})\times
{\bS}}\fr{|\nabla_{\sg}\psi({\bf y},{\bf z},\sg)|}{|{\bf z}|
}<\infty\bigg\}.\nonumber\ees
Here $\psi({\bf y},\cdot,\cdot)\in C^1({\bR}\times{\bS})$ can be understood as there is an open set $\Og\supset {\bS}$ (it may depend on ${\bf y}$) such that $\psi({\bf y},\cdot,\cdot)\in C^1({\bR}\times\Og)$, and so
$\sg\mapsto \nabla_{\sg}\psi({\bf y},{\bf z},\sg)$ is the usual gradient taken
in open set.
\vskip1mm

The following lemma is used to estimate bounds of difference quotients of functions
on the sphere.

\begin{lemma}\lb{Lemma6.new1}
Let $n\ge 2$ and $\psi\in C^1({\mS}^{n-1})$ (which means that there is an open set $\Og\supset {\mS}^{n-1}$
such that $\psi\in C^1(\Og)$). Then
$$|\psi(\sg_1)-\psi(\sg_2)|\le 4\sqrt{2}\|\nabla \psi\|_{{\mS}^{n-1}}|\sg_1-\sg_2|\qquad
\forall\, \sg_1, \sg_2\in {\mS}^{n-1}$$
where $\|\nabla \psi\|_{{\mS}^{n-1}}=\max\limits_{\sg\in {\mS}^{n-1}}|\nabla \psi(\sg)|$.
\end{lemma}

\begin{proof}Denote $\|\nabla \psi\|=\|\nabla \psi\|_{{\mS}^{n-1}}$.
Take any $\sg_1, \sg_2\in {\mS}^{n-1}$.

{\bf Case 1:} $\sg_1\cdot \sg_2\ge 0.$  Consider
$t\mapsto \gm(t)={\bf a}(t)/|{\bf a}(t)|\in {\mS}^{n-1}$ with
${\bf a}(t)=t\sg_1+(1-t)\sg_2, t\in [0,1]$. We have
$|{\bf a}(t)|\ge\fr{1}{\sqrt{2}}, {\bf a}'(t)=\sg_1-\sg_2,$ for  all $t\in [0,1].$
This gives
$|\gm'(t)|\le 2|{\bf a}'(t)|/|{\bf a}(t)|
\le 2\sqrt{2}|\sg_1-\sg_2|$ for all $t\in [0,1].$  So we have for some $t\in (0,1)$
$$|\psi(\sg_1)-\psi(\sg_2)|=|\psi\circ\gm(1)-\psi\circ\gm(0)|
=|\nabla \psi(\sg)\cdot \gm'(t)|\big|_{\sg=\gm(t)}\le 2\sqrt{2}\|\nabla \psi\||\sg_1-\sg_2|.$$

{\bf Case 2:} $\sg_1\cdot\sg_2<0.$  In this case we have
$|\sg_1-\sg_2|\ge \sqrt{2}$ and we can choose some $\sg_0\in {\mS}^{n-1}$ such that
$\sg_1\cdot \sg_0\ge 0, \sg_2\cdot \sg_0\ge 0$. [For instance
if $\sg_1+\sg_2={\bf 0}$, then take
$\sg_0\in {\mS}^{n-1}$ such that $\sg_0\cdot\sg_1=0$;
if $\sg_1+\sg_2\neq {\bf 0}$, then choose
$\sg_0=(\sg_1+\sg_2)/|\sg_1+\sg_2|$.]
Applying Case 1 to $\sg_1, \sg_0 $ and $\sg_0,\sg_2$ respectively and
noting that $|\sg_1-\sg_0|,|\sg_0-\sg_2|\le \sqrt{2}\le |\sg_1-\sg_2|$
we obrain
$$|\psi(\sg_1)-\psi(\sg_2)|\le
|\psi(\sg_1)-\psi(\sg_0)|+|\psi(\sg_0)-\psi(\sg_2)|
\le 4\sqrt{2}\|\nabla \psi\||\sg_1-\sg_2|.$$
\end{proof}
\vskip1mm

\begin{lemma}\lb{Lemma6.3}For any $\Psi\in{\cal T}_{c}(Y\times{\bR}, {\bR})$, let
$\psi({\bf y},{\bf z},\sg)=\Psi({\bf y},{\bf z})\cdot\sg $. Then
$\psi\in {\cal T}_{c}(Y\times {\bRS})$ and
\beas&&
\nabla_{{\bf z}}\psi({\bf y},{\bf z},\sg)\cdot\sg=
\sg^{\tau}\Psi'_{{\bf z}}({\bf y},{\bf z})\sg,\quad
\nabla_{\sg}\psi({\bf y},{\bf z},\sg)
=\Psi({\bf y},{\bf z}),\\
\\
&&
\fr{1}{\pi}\int_{{\mS}^1({\bf n})} \Big(
\nabla_{{\bf z}}\psi({\bf y},{\bf z},\sg)\cdot\sg
-\fr{\nabla_{\sg}\psi({\bf y},{\bf z},\sg)}{|{\bf z}|}\cdot {\bf n}
\Big){\rm d}\sg=\nabla_{\bf z} \cdot \Pi({\bf z})\Psi({\bf y},{\bf z}).\eeas
\end{lemma}

\begin{proof} That
$\psi\in {\cal T}_{c}(Y\times {\bRS})$ and the first two equalities are obvious.
Then applying Lemma \ref{Prop6.3}, we have
\beas&&\fr{1}{\pi}\int_{{\mathbb S}^1({\bf z})}
\sg^{\tau}\Psi'_{\bf z}({\bf y},{\bf z})\sg {\rm d}\sg
={\rm Tr}(\Psi'_{\bf z}({\bf y},{\bf z}))-{\bf n}^{\tau}\Psi'_{\bf z}({\bf y},{\bf z}){\bf n}
=\nabla_{{\bf z}}\cdot\Psi({\bf y},{\bf z})-{\bf n}^{\tau}\Psi'_{\bf z}({\bf y},{\bf z}){\bf n},\\ \\
\mathbb{}&&\fr{1}{\pi}
\int_{{\mS}^1({\bf n})} \fr{
\nabla_{\sg}\psi({\bf y},{\bf z},\sg)}{|{\bf z}|}\cdot{\bf n} {\rm d}\sg
=2\fr{\Psi({\bf y},{\bf z})}{|{\bf z}|}\cdot{\bf n}\eeas
and so using (\ref{6.14}) yields the third equality.
\end{proof}
\vskip1mm

 Let $\nu$ be a Borel measure on $Y\times{\bRS}$ defined by
\be\lb{6.32} \nu(E)=\int_{Y\times{\bR}}
\int_{{\mS}^1({\bf n})}
 1_{E}({\bf y},{\bf z},\sg){\rm d}\sg{\rm d}{\bf z}{\rm d}{\bf y},\quad E\subset Y\times{\bRS}\ee
 with ${\bf n}={\bf z}/|{\bf z}| $.
It is easily seen that the following equality
\be\lb{6.33}\int_{Y\times{\bRS}}g({\bf y},{\bf z},\sg){\rm d}\nu({\bf y},{\bf z},\sg)
 =\int_{Y\times{\bR}}
\int_{{\mS}^1({\bf n})}
 g({\bf y},{\bf z},\sg){\rm d}\sg{\rm d}{\bf z}{\rm d}{\bf y}\ee
 holds for all
nonnegative Borel measurable or $\nu$- integrable functions $g$ on $Y\times{\bRS}$.

The last lemma below shows that the weak projection gradient
$\Pi({\bf z})\nabla_{{\bf z}}F({\bf y}, {\bf z})$ can be characterized by
a vector-integration on the perpendicular circle  ${\mS}^1({\bf n})$,
and plays an important role in the proof of Theorem \ref{Theorem1} (From Eq.(MB) to Eq.(FPL)).

\begin{lemma}\lb{Lemma6.4}  A function $F\in L^1_{loc}(Y\times {\bR})$ has the weak projection gradient
$\Pi({\bf z})\nabla_{{\bf z}}F({\bf y}, {\bf z})$ in ${\bf z}\in{\bR}\setminus\{{\bf 0}\}$,  if and only if  there is a Borel measurable
function $D({\bf y},{\bf z},\sg)$ on $Y\times {\bRS}$ such that
$ D\in L^1_{loc}(Y\times{\bRS},{\rm d}\nu )$ and it holds for all
$\psi\in {\cal T}_{c}(Y\times {\bRS})$ that
\bes\lb{6.34} && \int_{Y\times{\bR}}
\int_{{\mS}^1({\bf n})} \psi({\bf y},{\bf z},\sg)
 D({\bf y},{\bf z},\sg){\rm d}\sg{\rm d}{\bf z}{\rm d}{\bf y} \\
&& = -\int_{Y\times{\bR}}F({\bf y},{\bf z})
\int_{{\mS}^1({\bf n})} \Big(
\nabla_{{\bf z}}\psi({\bf y},{\bf z},\sg)\cdot\sg
-\fr{\nabla_{\sg}\psi({\bf y},{\bf z},\sg)}{|{\bf z}|}\cdot{\bf n}
\Big){\rm d}\sg{\rm d}{\bf z}{\rm d}{\bf y}.
\nonumber\ees
Moreover, if $(\ref{6.34})$ holds for such a function $D$
(equivalently if $\Pi({\bf z})\nabla_{{\bf z}}F({\bf y}, {\bf z})$ exists), then
\beas
&&\Pi({\bf z})\nabla_{{\bf z}}F({\bf y}, {\bf z})
=\fr{1}{\pi}
\int_{{\mS}^1({\bf n})}
 D({\bf y},{\bf z},\sg)\sg{\rm d}\sg\quad {\rm a.e.}\quad on\quad Y\times {\bR},\\
 &&
 D({\bf y},{\bf z},\sg)= \Pi({\bf z})\nabla_{{\bf z}}F({\bf y}, {\bf z})\cdot\sg
\qquad {\rm a.e.}\quad[\nu]\quad on\quad Y\times{\bRS}.\eeas
\end{lemma}

\begin{proof} First of all we note that the function $D({\bf y},{\bf z},\sg)$
(up to a set of $\nu$-measure zero) is uniquely determined by the identity
(\ref{6.34}).

As before, we use smooth approximation. Let
$j(t)=c\exp\big(\fr{-1}{1-t^2}\big)1_{\{|t|<1\}}$ with $c>0$
satisfy  $\int_{{\mR}}j(t){\rm d}t=1$, let
$j_{\dt}(t)=\dt^{-1} j(\dt^{-1}t), \dt>0$, and let
$\zeta_{\dt}(|{\bf z}|)=\zeta(|{\bf z}|/\dt)$ be the smooth function used in the proof of
Lemma \ref{Lemma6.2}.
Suppose  that $F$ has the weak projection gradient
$\Pi({\bf z})\nabla_{{\bf z}}F({\bf y}, {\bf z})$
in ${\bf z}\in{\bR}\setminus\{{\bf 0}\}$. Take
any $\psi\in {\cal T}_{c}(Y\times {\bRS})$. Note that by definition of
${\cal T}_{c}(Y\times {\bRS})$ there are no problems of integrability for all integrands below.
By continuity of $\sg\mapsto
\psi({\bf y},{\bf z},\sg)
\Pi({\bf z})\nabla_{\bf z} F({\bf y},{\bf z})\cdot\sg $ on ${\bS}$,
it is easily shown that (with ${\bf n}={\bf z}/|{\bf z}|$)
\beas&&
\int_{Y\times{\bR}}
\int_{{\mS}^1({\bf n})} \psi({\bf y},{\bf z},\sg)
\Pi({\bf z})\nabla_{\bf z} F({\bf y},{\bf z})\cdot\sg {\rm d}\sg{\rm d}{\bf z}{\rm d}{\bf y}
\\
&&=\lim_{\dt\to 0^{+}}
\int_{Y\times{\bR}}\int_{{\bS}}\zeta_{\dt}(|{\bf z}|) j_{\dt}({\bf n}\cdot\sg)\psi({\bf y},{\bf z},\sg)\sg\cdot
\Pi({\bf z})\nabla_{{\bf z}} F({\bf y},{\bf z}){\rm d}\sg {\rm d}{\bf z}{\rm d}{\bf y}
.\eeas
Thanks to the smooth cutoff $\zeta_{\dt}(|{\bf z}|)$,  the function
$({\bf y},{\bf z})\mapsto  \zeta_{\dt}(|{\bf z}|)j_{\dt}({\bf n}\cdot\sg)\psi({\bf y},{\bf z},\sg)\sg $ belongs to ${\cal T}_{c}(Y\times{\bR}, {\bR}).$
Hence, using also \eqref{6.19}, we have
\beas&& \int_{{\bS}}
\int_{Y\times{\bR}} \zeta_{\dt}(|{\bf z}|)j_{\dt}({\bf n}\cdot\sg)\psi({\bf y},{\bf z},\sg)\sg\cdot
\Pi({\bf z})\nabla_{{\bf z}} F({\bf y},{\bf z}) {\rm d}{\bf z}{\rm d}{\bf y}
{\rm d}\sg\\
&&=-\int_{{\bS}}
\int_{Y\times{\bR}}F({\bf y},{\bf z})\zeta_{\dt}(|{\bf z}|)\nabla_{{\bf z}} \cdot\big(\Pi({\bf z}) j_{\dt}({\bf n}\cdot\sg)\psi({\bf y},{\bf z},\sg)\sg\big){\rm d}{\bf z}{\rm d}{\bf y}{\rm d}\sg.
\eeas
By elementary calculation we have (for ${\bf z}\in {\bR}\setminus\{{\bf 0}\}$)
\beas&&\nabla_{{\bf z}} \cdot\big(\Pi({\bf z})j_{\dt}({\bf n}\cdot\sg)\psi({\bf y},{\bf z},\sg)\sg\big)
 =j_{\dt}'({\bf n}\cdot\sg)
\psi({\bf y},{\bf z},\sg)\fr{1-({\bf n}\cdot\sg) ^2}{|{\bf z}|}\\
&&
+j_{\dt}({\bf n}\cdot\sg)\Big(\nabla_{{\bf z}}\psi({\bf y},{\bf z},\sg)\cdot \sg
-( {\bf n}\cdot\nabla_{{\bf z}} \psi({\bf y},{\bf z},\sg)){\bf n}\cdot\sg
-2\fr{\psi({\bf y},{\bf z},\sg)}{|{\bf z}|}\sg\cdot {\bf n}
\Big)\eeas and then we compute using \eqref{2.Sym1} that
\beas
&&\int_{{\bS}}\nabla_{{\bf z}} \cdot\big(\Pi({\bf z}) j_{\dt}({\bf n}\cdot\sg)\psi({\bf y},{\bf z},\sg)\sg\big){\rm d}\sg
\\
&&=\dt^{-1}
\int_{-1/\dt}^{1/\dt} j'(s)\bigg(\int_{{\mS}^1({\bf n})}\fr{
\psi\big({\bf y},{\bf z},{\bf n}\dt s+\sqrt{1-(\dt s)^2}\sg\big)}{|{\bf z}|}{\rm d}\sg\bigg)
{\rm d}s
\\
&&-\dt\int_{-1/\dt}^{1/\dt} s^2j'(s)\bigg(\int_{{\mS}^1({\bf n})}
\fr{\psi\big({\bf y},{\bf z},{\bf n}\dt s+\sqrt{1-(\dt s)^2}\sg\big)}{|{\bf z}|}{\rm d}\sg\bigg){\rm d}s
\\
&&+
\int_{-1/\dt}^{1/\dt}j(s)\bigg( \int_{{\mS}^1({\bf n})}\nabla_{{\bf z}}\psi\big({\bf y},{\bf z},{\bf n}\dt s+\sqrt{1-(\dt s)^2}\sg\big)\cdot \sg{\rm d}\sg\bigg) \sqrt{1-(\dt s)^2}{\rm d}s
\\
&&-
2\dt
\int_{-1/\dt}^{1/\dt}sj(s)\bigg( \int_{{\mS}^1({\bf n})}\fr{\psi\Big({\bf y},{\bf z},{\bf n}\dt s+\sqrt{1-(\dt s)^2}\sg\bigg)}{|{\bf z}|}
{\rm d}\sg\Big){\rm d}s
\\
&&=: I_{1,\dt}({\bf y,z})+I_{2,\dt}({\bf y,z})+I_{3,\dt}({\bf y,z})+I_{4,\dt}({\bf y,z}).\eeas

\noindent Estimate of $I_{1,\dt}({\bf y,z})$:\, Using
$\int_{-1/\dt}^{1/\dt}
j'(s){\rm d}s=0$, the first term becomes
\beas I_{1,\dt}({\bf y,z})
=
 \int_{-1/\dt}^{1/\dt} j'(s)\bigg(\int_{{\mS}^1({\bf n})}\fr{
\psi\Big({\bf y},{\bf z},{\bf n}\dt s+\sqrt{1-(\dt s)^2}\sg\Big)-\psi({\bf y},{\bf z},\sg)}{\dt |{\bf z}|}{\rm d}\sg
\bigg)
{\rm d}s.
\eeas
For any ${\bf y}\in Y,{\bf z}\in {\bR}\setminus\{{\bf 0}\},\sg\in{\mS}^1({\bf n})$, and $s\in {\mR}$,
we have
$$\lim_{\dt\to 0^+}\fr{
\psi\big({\bf y},{\bf z},{\bf n}\dt s+\sqrt{1-(\dt s)^2}\sg\big)-\psi({\bf y},{\bf z},\sg)}{\dt |{\bf z}|}
=s\fr{
 \nabla_{\sg}\psi({\bf y},{\bf z},\sg)}{|{\bf z}|}\cdot {\bf n}.$$
Also, applying Lemma \ref{Lemma6.new1} to ${\bS}\ni \sg\mapsto \psi({\bf y,z},\sg)$, there exists $0<C_{\psi}<\infty$ such that for all $\dt>0$
\beas \sup_{y\in Y, {\bf z}\in {\bR}\setminus \{{\bf 0}\},
\sg\in{\mS}^1({\bf n})}\bigg| \fr{
\psi\big({\bf y},{\bf z},{\bf n}\dt s+\sqrt{1-(\dt s)^2}\sg\big)
-\psi({\bf y},{\bf z},\sg)}{\dt |{\bf z}|}
\bigg|
\le C_{\psi}|s|\quad \forall\, s\in[-1/\dt, 1/\dt].
\eeas
It follows from dominated convergence theorem and $\int_{-\infty}^{\infty}s j'(s){\rm d}s=-1$ that
\beas \lim_{\dt\to 0^{+}}I_{1,\dt}({\bf y,z})
=\int_{-\infty}^{\infty}j'(s)\bigg(\int_{{\mS}^1({\bf n})}s\fr{
 \nabla_{\sg}\psi}{|{\bf z}|}\cdot {\bf n}{\rm d}\sg\bigg){\rm d}s
 =-\int_{{\mS}^1({\bf n})}\fr{
 \nabla_{\sg}\psi({\bf y},{\bf z},\sg)}{|{\bf z}|}\cdot {\bf n}{\rm d}\sg.\eeas
Moreover we have
 $\sup\limits_{{\bf y}\in Y,\,{\bf z}\in {\bR}\setminus\{{\bf 0}\},\dt>0}
|I_{1,\dt}({\bf y,z})|<\infty$.
\vskip1mm

\noindent Estimates of $I_{2,\dt}({\bf y,z})$, $I_{3,\dt}({\bf y,z})$, $I_{4,\dt}({\bf y,z})$:\,
It is easily seen that
\beas&& \sup_{{\bf y}\in Y,\,{\bf z}\in {\bR}\setminus\{{\bf 0}\}}\big(
|I_{2,\dt}({\bf y,z})|+|I_{4,\dt}({\bf y,z})|\big)\le C_{\psi}\dt,\quad \forall\,\dt>0;\\
&&\sup_{{\bf y}\in Y,\,{\bf z}\in {\bR}\setminus\{{\bf 0}\}, \dt>0}|I_{3,\dt}({\bf y,z})|
<\infty,\quad \lim_{\dt\to 0^{+}}I_{3,\dt}({\bf y,z})
=\int_{{\mS}^1({\bf n})}\nabla_{{\bf z}}\psi({\bf y},{\bf z},\sg)\cdot \sg{\rm d}\sg.
\eeas
\smallskip
Thus it follows from dominated convergence theorem  that
\beas&&\lim_{\dt\to 0^{+}}
\int_{{\bS}}\nabla_{{\bf z}} \cdot\Big(\Pi({\bf z}) j_{\dt}({\bf n}\cdot\sg)\psi({\bf y},{\bf z},\sg)\sg\Big){\rm d}\sg\\
&&=
 \int_{{\mS}^1({\bf n})}\bigg( -\fr{
 \nabla_{\sg}\psi({\bf y},{\bf z},\sg)}{|{\bf z}|}\cdot {\bf n}
 +\nabla_{{\bf z}}\psi({\bf y},{\bf z},\sg)\cdot \sg\bigg){\rm d}\sg,\\
 &&\lim_{\dt\to 0^{+}}
\int_{Y\times{\bR}}\zeta_{\dt}(|{\bf z}|)F({\bf y},{\bf z})\bigg(\int_{{\bS}}\nabla_{{\bf z}} \cdot\Big(\Pi({\bf z}) j_{\dt}({\bf n}\cdot\sg)\psi({\bf y},{\bf z},\sg)\sg\Big){\rm d}\sg\bigg){\rm d}{\bf z}{\rm d}{\bf y}\\
&&=\int_{Y\times{\bR}}F({\bf y},{\bf z})
 \int_{{\mS}^1({\bf n})}\Big( \nabla_{{\bf z}}\psi({\bf y},{\bf z},\sg)\cdot \sg-\fr{
 \nabla_{\sg}\psi({\bf y},{\bf z},\sg)}{|{\bf z}|}\cdot {\bf n}
 \Big){\rm d}\sg
{\rm d}{\bf z}{\rm d}{\bf y}.
\eeas
Collecting the above results we conclude that
\beas&&
\int_{Y\times{\bR}}
\int_{{\mS}^1({\bf n})} \psi({\bf y},{\bf z},\sg)
\Pi({\bf z})\nabla_{\bf z} F({\bf y},{\bf z})\cdot\sg {\rm d}\sg{\rm d}{\bf z}{\rm d}{\bf y}
\\
&&=-\int_{Y\times{\bR}}F({\bf y},{\bf z})
 \int_{{\mS}^1({\bf n})}\Big( \nabla_{{\bf z}}\psi({\bf y},{\bf z},\sg)\cdot \sg-\fr{
 \nabla_{\sg}\psi({\bf y},{\bf z},\sg)}{|{\bf z}|}\cdot {\bf n}
 \Big){\rm d}\sg{\rm d}{\bf z}{\rm d}{\bf y}.\eeas
This shows that the function $D({\bf y},{\bf z},\sg):= \Pi({\bf z})\nabla_{{\bf z}}F({\bf y}, {\bf z})\cdot\sg$
satisfies (\ref{6.34}).

Conversely, suppose that there is a Borel measurable
function $D({\bf y},{\bf z},\sg)$ on $Y\times{\bRS}$ such that
$D\in L^1_{loc}(Y\times{\bRS},{\rm d}\nu )$ and
(\ref{6.34}) holds. Let
$${\bf D}({\bf y},{\bf z})=\fr{1}{\pi}
\int_{{\mS}^1({\bf n})}
 D({\bf y},{\bf z},\sg)\sg{\rm d}\sg,\quad {\bf n}={\bf z}/|{\bf z}|.$$
 Then ${\bf D}({\bf y},{\bf z})\in {\mR}^2({\bf z})\,({\bf z}\neq{\bf 0})$ and
 ${\bf D}\in L^1_{loc}(Y\times {\bR}, {\bR})$.
For any $\Psi\in {\cal T}_{c}(Y\times{\bR}, {\bR})$, let $\psi({\bf y},{\bf z},\sg)=
\Psi({\bf y},{\bf z})\cdot \sg$, then $\psi\in {\cal T}_{c}(Y\times {\bRS})$ and by (\ref{6.34}) and Lemma \ref{Lemma6.3} we have
\beas&&\int_{Y\times{\bR}}  \Psi({\bf y}, {\bf z})\cdot{\bf D}({\bf y}, {\bf z}) {\rm d}{\bf z}{\rm d}{\bf y}
=\fr{1}{\pi}\int_{Y\times{\bR}}\int_{{\mS}^1({\bf n})}\psi({\bf y}, {\bf z},\sg)
 D({\bf y},{\bf z},\sg){\rm d}\sg {\rm d}{\bf z}{\rm d}{\bf y}
\\
&&=
-\fr{1}{\pi}\int_{Y\times{\bR}}F({\bf y},{\bf z})
\int_{{\mS}^1({\bf n})} \Big(
\nabla_{{\bf z}}\psi({\bf y},{\bf z},\sg)\cdot \sg
-\fr{
\nabla_{\sg}\psi({\bf y},{\bf z},\sg)}{|{\bf z}|}\cdot {\bf n}
\Big){\rm d}\sg{\rm d}{\bf z}{\rm d}{\bf y}\\
&&=-\int_{Y\times{\bR}}F({\bf y},{\bf z})
\nabla_{\bf z} \cdot \Pi({\bf z})\Psi({\bf y},{\bf z})
{\rm d}{\bf z}{\rm d}{\bf y}.\eeas
Thus, by definition, $F$ has the weak projection gradient
$\Pi({\bf z})\nabla_{{\bf z}}F({\bf y}, {\bf z})$
in ${\bf z}\in{\bR}\setminus\{{\bf 0}\}$ and
$$ \Pi({\bf z})\nabla_{{\bf z}}F({\bf y}, {\bf z})=
{\bf D}({\bf y},{\bf z})=\fr{1}{\pi}
\int_{{\mS}^1({\bf n})}
 D({\bf y},{\bf z},\sg)\sg{\rm d}\sg.$$
Finally using (\ref{6.30}) we also have
$$\fr{1}{\pi}
\int_{{\mS}^1({\bf n})}\big(\Pi({\bf z})\nabla_{{\bf z}}F({\bf y}, {\bf z})
\cdot \sg\big)\sg{\rm d}\sg =\Pi({\bf z})\nabla_{{\bf z}}F({\bf y}, {\bf z}). $$
Since $D({\bf y},{\bf z},\sg)$ is unique, it follows that
$ D({\bf y},{\bf z},\sg)= \Pi({\bf z})
\nabla_{{\bf z}} F({\bf y}, {\bf z})\cdot\sg$\, a.e. $[\nu]$ on $Y\times{\bRS}.$ This completes the proof.
\end{proof}
\vskip3mm

{\bf Acknowledgment.}
This work was started while M. Pulvirenti was visiting
Tsinghua University in 2016, and was supported by National Natural Science Foundation
of China under Grant No.11771236.

\end{document}